\documentclass{imsart}

\startlocaldefs
\usepackage{color}
\usepackage{mathptmx}       % selects Times Roman as basic font
\usepackage{helvet}         % selects\textbf{\textbf{•}} Helvetica as sans-serif font
\usepackage{courier}        % selects Courier as typewriter font
\usepackage{type1cm}        % activate if the above 3 fonts are
\usepackage{makeidx}         % allows index generation
\usepackage{graphicx}        % standard LaTeX graphics tool
\usepackage{multicol}        % used for the two-column index
\usepackage[bottom]{footmisc}% places footnotes at page bottom
\usepackage{amsmath}
\usepackage{amsfonts}
\usepackage{amssymb}
\usepackage{amsthm}
\usepackage{amsopn}
\usepackage{subcaption}
\usepackage{comment}
\newtheorem{theorem}{Theorem}
\newtheorem{corollary}[theorem]{Corollary}
\newtheorem{lemma}[theorem]{Lemma}

\newtheorem{proposition}[theorem]{Proposition}

\newtheorem{remark}[theorem]{Remark}
\newtheorem{definition}[theorem]{Definition}

\usepackage[top=2cm, bottom=2cm, left=2cm, right=2cm]{geometry}
\numberwithin{theorem}{subsection}
\numberwithin{figure}{subsection}
\numberwithin{equation}{subsection}
\DeclareMathOperator{\CR}{CR}
\DeclareMathOperator{\cov}{cov}
\DeclareMathOperator{\dist}{dist}
\DeclareMathOperator{\SLE}{SLE}
\DeclareMathOperator{\CLE}{CLE}
\DeclareMathOperator{\GFF}{GFF}
\DeclareMathOperator{\inte}{int}
\DeclareMathOperator{\dimH}{dim}

\endlocaldefs

\begin{document}

\begin{frontmatter}

\title{Level Lines of Gaussian Free Field I: Zero-Boundary GFF}
\runtitle{Level Lines of GFF I}

% indicate corresponding author with \corref{}
% \author{\fnms{John} \snm{Smith}\corref{}\ead[label=e1]{smith@foo.com}\thanksref{t1}}
% \thankstext{t1}{Thanks to somebody} 
% \address{line 1\\ line 2\\ printead{e1}}
% \affiliation{Some University}

\author{\fnms{Menglu} \snm{Wang}\ead[label=e1]{mengluw@math.mit.edu}}
\address{Department of Mathematics, Massachusetts Institute of Technology \\\printead{e1}}
%\affiliation{Department of Mathematics, Massachusetts Institute of Technology\printead{e1}}
\and
\author{\fnms{Hao} \snm{Wu}\ead[label=e2]{hao.wu.proba@gmail.com}}
\address{NCCR/SwissMAP, Geneva University, Switzerland \textit{and} Yau Mathematical Sciences Center, Tsinghua University, China\\\printead{e2}}
%\affiliation{NCCR/SwissMAP, Geneva University, Switzerland \textit{and} Yau Mathematical Sciences Center, Tsinghua University, China\printead{e2}}

%\runauthor{???}

\begin{abstract}
We study level lines of Gaussian Free Field $h$ emanating from boundary points. The article has two parts.  In the first part, we show that the level lines are random continuous curves which are variants of $\SLE_4$ path.
We show that the level lines with different heights satisfy the same monotonicity behavior as the level lines of smooth functions.
We prove that the time-reversal of the level line coincides with the level line of $-h$. This implies that the time-reversal of $\SLE_4(\underline{\rho})$ process is still an $\SLE_4(\underline{\rho})$ process. We prove that the level lines satisfy ``target-independent" property. In the second part, we discuss the relation between Gaussian Free Field and Conformal Loop Ensemble (CLE). A CLE is a collection of disjoint $\SLE$-loops. Since the level lines of $\GFF$ are $\SLE_4$ paths, the collection of level loops of $\GFF$ corresponds to $\CLE_4$. We study the coupling between $\GFF$ and $\CLE_4$ with time parameter which sheds lights on the conformal invariant metric on $\CLE_4$.
\end{abstract}

\begin{keyword}[class=MSC]
\kwd[Primary ]{60G60}
%\kwd{}
\kwd[; secondary ]{60J67}
\end{keyword}

\begin{keyword}
\kwd{Gaussian Free Field}
\kwd{Schramm Loewner Evolution}
\kwd{Conformal Loop Ensemble}
\end{keyword}

\end{frontmatter}

\newcommand{\eps}{\epsilon}
\newcommand{\ov}{\overline}
\newcommand{\U}{\mathbb{U}}
\newcommand{\T}{\mathbb{T}}
\newcommand{\HH}{\mathbb{H}}
\newcommand{\LA}{\mathcal{A}}
\newcommand{\LC}{\mathcal{C}}
\newcommand{\LF}{\mathcal{F}}
\newcommand{\LK}{\mathcal{K}}
\newcommand{\LE}{\mathcal{E}}
\newcommand{\LL}{\mathcal{L}}
\newcommand{\LU}{\mathcal{U}}
\newcommand{\LV}{\mathcal{V}}
\newcommand{\LZ}{\mathcal{Z}}
\newcommand{\R}{\mathbb{R}}
\newcommand{\C}{\mathbb{C}}

\newcommand{\N}{\mathbb{N}}
\newcommand{\Z}{\mathbb{Z}}
\newcommand{\E}{\mathbb{E}}
\newcommand{\PP}{\mathbb{P}}
\newcommand{\QQ}{\mathbb{Q}}
\newcommand{\A}{\mathbb{A}}
\newcommand{\bn}{\mathbf{n}}
\newcommand{\cond}{\,|\,}
\newcommand{\la}{\langle}
\newcommand{\ra}{\rangle}
\newcommand{\tree}{\Upsilon}
\section{Introduction}
The two dimensional Gaussian Free Field (GFF) is a natural time analog of Brownian motion \cite{SheffieldGFFMath} that has been extensively used as a basic building block in Conformal Field Theories. Like Brownian motion, it plays an important role in statistical physics, random surfaces, and quantum field theory. The geometry of the two-dimensional GFF---the fact that one can describe geometric lines in this very irregular distribution---has been discovered recently \cite{DubedatSLEFreeField, SchrammSheffieldDiscreteGFF, SchrammSheffieldContinuumGFF, MillerSheffieldIG1}, and led to a number of recent developments. The GFF also corresponds to the scaling limit of discrete models, for instance, the height function of dimer models \cite{KenyonDimer}.
In the current paper, we focus on the level lines of GFF in the upper half plane $\HH$. This is the first in a two-paper series that also includes \cite{WangWuLevellinesGFFII}. The latter paper will study the level lines of GFF in the whole-plane. Before we talk about the level lines of GFF, we need to introduce two other important random planar objects: Schramm Loewner Evolution (SLE) and Conformal Loop Ensemble (CLE).

Oded Schramm's SLE was introduced to understand the scaling limits of discrete models \cite{SchrammFirstSLE}. A chordal SLE is a random non-self-traversing curve in simply connected domains joining two distinct boundary points. It is the only one-parameter family of random curves (usually indexed by a positive real number $\kappa$) that satisfies the conformal invariance and domain Markov property (the precise meaning is recalled in Section \ref{subsec::chordal_sle}). Since its introduction, SLE curves have been proved to be the scaling limits of many discrete models, for instance, $\SLE_2$ is the scaling limit of loop-erased random walk \cite{LawlerSchrammWernerLERWUST}, $\SLE_3$ is the scaling limit of the interface in critical Ising model \cite{ChelkakSmirnovIsing, CDCHKSConvergenceIsingSLE}, $\SLE_4$ is the scaling limit of the level line of discrete GFF \cite{SchrammSheffieldDiscreteGFF}.

CLE was introduced when one tries to understand the scaling limit of the ``entire" discrete models (in contrast with one interface which turns out to be SLE curves). A simple CLE \cite{SheffieldExplorationTree, SheffieldWernerCLE} is a random countable collection of disjoint simple loops in simply connected domains (non-empty, other than $\C$) that are non-nested. It is the only one-parameter family of random collection of loops that satisfies the conformal invariance and domain Markov property (the precise meaning is recalled in Section \ref{subsec::cle}). In \cite{SheffieldWernerCLE}, the authors prove that each loop in simple CLE is an $\SLE_{\kappa}$-type loop for $\kappa\in (8/3,4]$.

In \cite{SchrammSheffieldContinuumGFF}, the authors show that, for a special constant $\lambda=\pi/2$, if boundary conditions of the GFF are set to be $+\lambda$ on $\R_+$ and $-\lambda$ on $\R_-$, then one can make sense of the zero level line of the GFF whose law is chordal SLE$_4$; furthermore, the zero level line is a path-valued function of the field. Therefore, we say that $\SLE_4$ is the level line of GFF. 
The current paper has two parts. In the first part, we generalize this method to introduce the level lines of GFF whose boundary value is piecewise constant: Theorems \ref{thm::boundary_levelline_gff_coupling} and \ref{thm::boundary_levelline_gff_deterministic}. We show that the level lines of GFF are continuous curves: Theorem \ref{thm::sle_chordal_continuity_transience}. We explain the interaction between two level lines: Theorem \ref{thm::boundary_levelline_gff_interacting}. We show that the time-reversal of the level line of GFF $h$ is the level line of $-h$: Theorem \ref{thm::sle_chordal_reversibility}.
We prove the ``target-independent" property of the level lines of GFF: Theorem \ref{thm::boundary_levellines_targetindependence}. In a series of papers \cite{MillerSheffieldIG1, MillerSheffieldIG2, MillerSheffieldIG3, MillerSheffieldIG4} by Miller and Sheffield, they study the flow lines and the counterflow lines of GFF. These curves in GFF correspond to $\SLE_{\kappa}$ for $\kappa\in (0,4)\cup(4,\infty)$. The first part of the  current paper study the properties of level lines of GFF, and this part can be viewed as a make up for \cite{MillerSheffieldIG1, MillerSheffieldIG2, MillerSheffieldIG3} for $\kappa=4$.  The relation between the current paper and these papers is discussed in detail in Section \ref{subsec::relation_previous_works}.

In the second part of the paper, we discuss the relation between GFF and $\CLE_4$. Since the level lines of $\GFF$ are $\SLE_4$ paths, the collection of level loops of $\GFF$ corresponds to $\CLE_4$. 
In \cite{WernerWuCLEExploration}, the authors introduce a conformally invariant growing mechanism of loops and it leads to a conformal invariant way to describe distances between loops in $\CLE$. 
We show that there exists a coupling between $\GFF$ and this conformally invariant growing process of $\CLE_4$. In this coupling, the loops in $\CLE_4$ corresponds to a certain collection of level loops of $\GFF$, and the structure of $\GFF$ gives in turn better understanding of this growing process and sheds lights on the conformal invariant metric between loops in $\CLE_4$.  

\subsection{Boundary emanating level lines of $\GFF$}
For convenience and concreteness, we state our results in the upper half plane $\HH$. Recall that $\SLE_{\kappa}$ is the random curve satisfying the conformal invariance and domain Markov property. Oded Schramm found that the Loewner evolution is suitable to describe the domain Markov property; and $\SLE$ curves can also be defined through Loewner evolution. 
Suppose that $\gamma$ is a continuous curve in $\HH$ starting from 0 targeted at $\infty$ (parameterized appropriately), let $g_t$ be the conformal map from $\HH\setminus\gamma[0,t]$ onto $\HH$ such that $\lim_{z\to\infty}|g_t(z)-z|=0$. Then the family $(g_t,t\ge 0)$ satisfies the Loewner evolution
\[\partial_t g_t(z)=\frac{2}{g_t(z)-W_t},\]
where $W_t$ is the image of the tip of the curve $\gamma(t)$ under $g_t$. In fact, the curve $\gamma$ is determined by the process $W$; and we also say that $\gamma$ is the Loewner chain driven by $W$. Chordal $\SLE_{\kappa}$ is the Loewner chain driven by $W=\sqrt{\kappa}B$ where $B$ is a one-dimensional Brownian motion. 

More generally, an $\SLE_{\kappa}(\underline{\rho})$ process is a variant of $\SLE_{\kappa}$ where one keeps track of multiple additional points, which are called force points. Suppose that $\underline{x}^L=(x^{l,L}<\cdots<x^{1,L}\le 0)$ and $\underline{x}^R=(0\le x^{1,R}<\cdots<x^{r,R})$ are the force points of which we want to keep track, where the superscripts $L,R$ mean ``left" and ``right" respectively. Associate with each force point $x^{i,q}$, for $q\in\{L,R\}$, a weight $\rho^{i,q}\in\R$.
We denote by $\underline{\rho}=(\underline{\rho}^L;\underline{\rho}^R)$ the vector of weights.
An $\SLE_{\kappa}(\underline{\rho})$ process is a variant of $\SLE_{\kappa}$ process which can be well-defined up until the ``continuation threshold".
It is a measure on continuously growing compact hulls $K_t$---compact subsets of $\overline{\HH}$ so that $\HH\setminus K_t$ is simply connected.
We will provide more discussion of $\SLE_{\kappa}(\underline{\rho})$ process in Section \ref{subsec::chordal_sle}.

\begin{theorem}\label{thm::boundary_levelline_gff_coupling}
Let $(K_t,t\ge 0)$ be the Loewner chain of the $\SLE_4(\underline{\rho}^L;\underline{\rho}^R)$ process in $\HH$ starting from 0 targeted at $\infty$ with force points $(\underline{x}^L;\underline{x}^R)$. Let $(g_t,t\ge 0)$ be the sequence of corresponding conformal maps and set $f_t=g_t-W_t$. There exists a coupling $(h,K)$ where $h$ is a zero-boundary $\GFF$ on $\HH$ such that the following is true. Suppose that $\tau$ is any finite stopping time less than the continuation threshold for $K$. Let $\eta_t^0$ be the harmonic function in $\HH$ with boundary values
\[\left\{
                 \begin{array}{ll}
                   -\lambda(1+\sum_0^j\rho^{i,L}), & \text{if } x\in [f_t(x^{j+1,L}),f_t(x^{j,L})); \\
                   \lambda(1+\sum_0^j\rho^{i,R}), & \text{if } x\in [f_t(x^{j,R}),f_t(x^{j+1,R}))
                 \end{array}
               \right.
\]
where $\rho^{0,L}=\rho^{0,R}=0$, $x^{0,L}=0^-$, $x^{l+1,L}=-\infty$, $x^{0,R}=0^+$, $x^{r+1,R}=\infty$. See Figure \ref{fig::thm_boundary_levelline_gff_coupling}. Define
\[\eta_t(z)=\eta_t^0(f_t(z)).\]
Then the conditional law of $(h+\eta_0)|_{\HH\setminus K_{\tau}}$ given $K_{\tau}$ is equal to the law of $\eta_{\tau}+h\circ f_{\tau}$.
\end{theorem}

\begin{figure}[ht!]
\begin{center}
\includegraphics[width=0.63\textwidth]{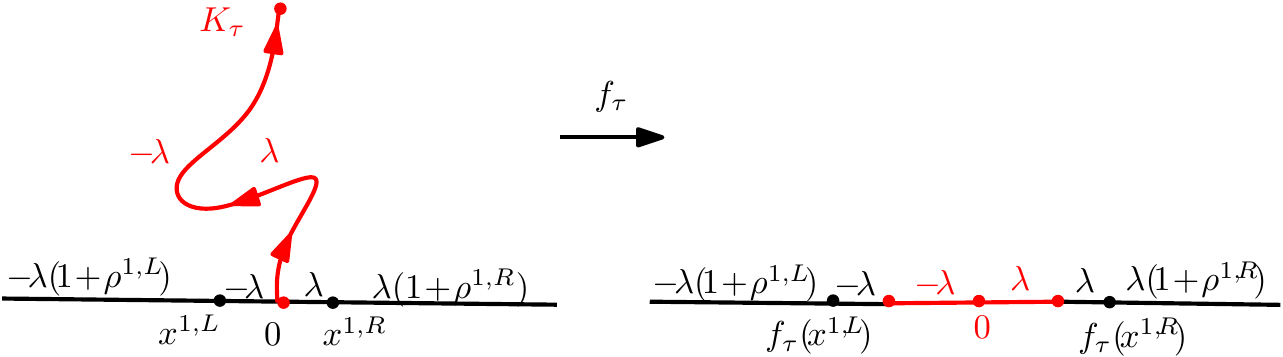}
\end{center}
\caption{\label{fig::thm_boundary_levelline_gff_coupling} The function $\eta_{\tau}^0$ in Theorem \ref{thm::boundary_levelline_gff_coupling} is the harmonic function with boundary values depicted in the right panel. The function $\eta_{\tau}=\eta_{\tau}^0(f_{\tau})$ is the harmonic function with boundary values depicted in the left panel.}
\end{figure}

\begin{theorem}\label{thm::boundary_levelline_gff_deterministic}
Suppose that $h$ is a $\GFF$ on $\HH$ and $\gamma$ is an $\SLE_4(\underline{\rho})$ process. If $(h,\gamma)$ are coupled as in Theorem \ref{thm::boundary_levelline_gff_coupling}, then $\gamma$ is almost surely determined by $h$.
\end{theorem}

Throughout this paper, we focus on $\GFF$ with piecewise constant boundary value that changes only finitely many times, as in Theorem \ref{thm::boundary_levelline_gff_coupling}. To be concise, we will use ``piecewise constant boundary value" to indicate piecewise constant boundary value that changes only finitely many times. 

If the $\SLE_4(\underline{\rho})$ process $\gamma$ and the GFF $h$ are coupled as in Theorem \ref{thm::boundary_levelline_gff_coupling}, we call $\gamma$ the \textbf{level line} of $h+\eta_0$. Generally, for any $u\in\R$, we call $\gamma$ the level line of a GFF $h$ with height $u$ if it is the level line of $h+u$. From the coupling of $\SLE_4$ paths with GFF, we can prove the continuity of the level lines which implies the continuity of $\SLE_4(\underline{\rho})$ process.

\begin{theorem}\label{thm::sle_chordal_continuity_transience}
Suppose that $h$ is a $\GFF$ on $\HH$ whose boundary value is piecewise constant. Then the level line of $h$ is almost surely continuous up to and including the continuation threshold.

In particular, this implies the continuity and the transience of $\SLE_4(\underline{\rho})$ process. Suppose that
$\gamma$ is an $\SLE_4(\underline{\rho})$ process in $\HH$ starting from 0 targeted at $\infty$. Then $\gamma$ is almost surely continuous up to and including the continuation threshold. On the event that the continuation threshold is not hit before $\gamma$ reaches $\infty$, we have that $\gamma$ is almost surely transient:  $\lim_{t\to\infty}\gamma(t)=\infty$.
\end{theorem}

We also study the interaction between two level lines with different heights and starting points. In contrast with the case that $h$ is smooth, these level lines can bounce off of each other, but they still have the same monotonicity behavior in their starting points and heights as in the smooth case. See Figure \ref{fig::thm_boundary_levelline_gff_interacting}.

\begin{figure}[ht!]
\begin{subfigure}[b]{0.48\textwidth}
\begin{center}
\includegraphics[width=0.625\textwidth]{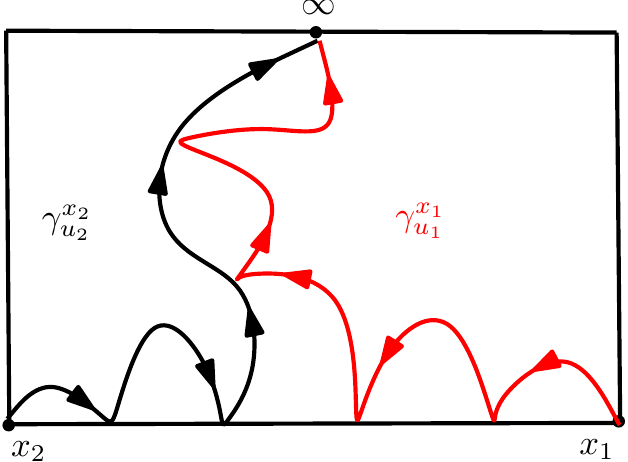}
\end{center}
\caption{If $u_2>u_1$, the level line $\gamma_{u_2}^{x_2}$ stays to the right of $\gamma_{u_1}^{x_1}$.}
\end{subfigure}
$\quad$
\begin{subfigure}[b]{0.48\textwidth}
\begin{center}\includegraphics[width=0.62\textwidth]{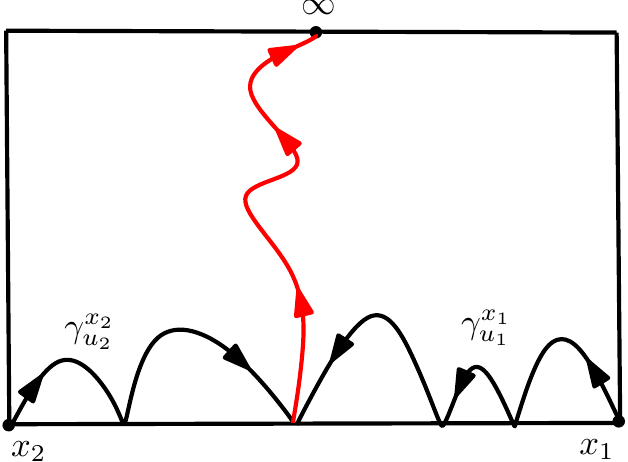}
\end{center}
\caption{If $u_2=u_1$, the two level lines merge upon intersecting.}
\end{subfigure}
\caption{\label{fig::thm_boundary_levelline_gff_interacting} The level lines of $\GFF$ satisfy the same monotonicity behavior as in the smooth case.}
\end{figure}

\begin{theorem}\label{thm::boundary_levelline_gff_interacting}
Suppose that $h$ is a $\GFF$ on $\HH$ whose boundary value is piecewise constant. For each $u\in\R$ and $x\in\partial\HH$, let $\gamma_u^x$ be the level line of $h$ with height $u$ starting from $x$. Fix $x_2\le x_1$.
\begin{enumerate}
\item [(1)] If $u_2>u_1$, then $\gamma_{u_2}^{x_2}$ almost surely stays to the left of $\gamma_{u_1}^{x_1}$.
\item [(2)] If $u_2=u_1$, then $\gamma_{u_2}^{x_2}$ may intersect $\gamma_{u_1}^{x_1}$ and, upon intersecting, the two curves merge and never separate.
\end{enumerate}
\end{theorem}

\begin{remark} 
Assume the same notations as in Theorem \ref{thm::boundary_levelline_gff_interacting}. Fix $u_2>u_1$. 
We have the following facts about the intersection of the level lines $\gamma^{x_1}_{u_1}$ and $\gamma^{x_2}_{u_2}$.
\begin{enumerate}
\item [(1)]  When $u_2-u_1\ge 2\lambda$, the level lines $\gamma^{x_1}_{u_1}$ and $\gamma^{x_2}_{u_2}$ do not intersect each other almost surely.
\item [(2)] When $0<u_2-u_1<2\lambda$, the Hausdorff dimension of the intersection is given by the following.
\[\dimH_H\left(\gamma^{x_1}_{u_1}\cap\gamma^{x_2}_{u_2}\cap\HH\right)=2-\frac{1}{8}\left((u_2-u_1)/\lambda+2\right)^2,\]
almost surely on the event $\left\{\gamma^{x_1}_{u_1}\cap\gamma^{x_2}_{u_2}\cap\HH\neq\emptyset\right\}$.
\end{enumerate}
In \cite{MillerWuSLEIntersection}, the authors proved the Hausdorff dimension of the intersection of flow lines which corresponds to $\kappa\in (0,4)$. The same proof works for the intersection of level lines.  
\end{remark}

The following theorem tells the reversibility of the level lines of $\GFF$.

\begin{theorem}\label{thm::sle_chordal_reversibility}
Suppose that $h$ is a $\GFF$ on $\HH$ whose boundary value is piecewise constant. Let $\gamma$ be the level line of $h$ starting from 0 targeted at $\infty$; and $\gamma'$ be the level line of $-h$ staring from $\infty$ targeted at 0. Then, on the event that the two paths do not hit the continuation thresholds before they reach the target points, the two paths $\gamma'$  and $\gamma$ are equal (viewed as sets) almost surely.

This implies the reversibility of $\SLE_4(\underline{\rho})$ process.
Suppose that $\gamma$ is an $\SLE_4(\underline{\rho})$ process in $\HH$ starting from 0 targeted at $\infty$. Then, conditioned on the event that the continuation threshold is not hit before $\gamma$ reaches $\infty$, we have that the time-reversal of $\gamma$ has the law of $\SLE_4(\underline{\rho})$ process starting from $\infty$ targeted at $0$ (with appropriate weights and force points) conditioned on the event that the continuation threshold is not hit before it reaches $0$.
In particular, fix $\rho_1>-2,\rho_2>-2$. Suppose that $\gamma$ is an $\SLE_4(\rho_1;\rho_2)$ process in $\HH$ starting from 0 targeted at $\infty$ with two force points next to the starting point. Then the time-reversal of $\gamma$ has the law of $\SLE_4(\rho_2;\rho_1)$ process in $\HH$ starting from $\infty$ targeted at 0 with two force points next to the starting point.
\end{theorem}

\begin{figure}[ht!]
\begin{center}
\includegraphics[width=0.325\textwidth]{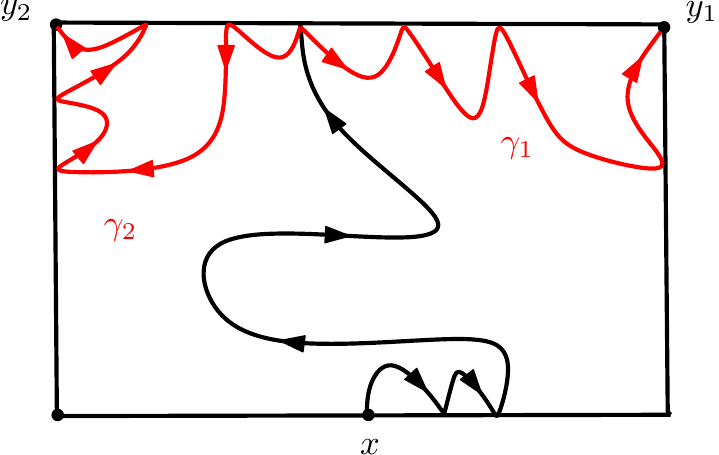}
\end{center}
\caption{\label{fig::boundary_levellines_targetindependence} Consider the two level lines starting from the same point targeted at distinct points: they coincide up to the first disconnecting time, after which they continue towards their target points independently.}
\end{figure}

Finally, we state the result which is called the ``target-independent" property of the level lines of $\GFF$.
\begin{theorem}\label{thm::boundary_levellines_targetindependence}
Suppose that $h$ is a $\GFF$ on $\HH$ whose boundary value is piecewise constant. Fix three distinct boundary points $y_2<x<y_1$. For $i=1,2$, let $\gamma_i$ be the level line of $h$ starting from $x$ targeted at $y_i$; define $T_i$ to be the first disconnecting time: $T_i$ is the inf of $t$ such that $y_1$ and $y_2$ are not on the boundary of the same connected component of $\HH\setminus\gamma_i[0,t]$. See Figure \ref{fig::boundary_levellines_targetindependence}. Then, almost surely, the paths $\gamma_1$ and $\gamma_2$ coincide up to and including the first disconnecting time (modulo time-change); given $(\gamma_1[0,T_1],\gamma_2[0,T_2])$, the two paths continue towards their target points independently.
\end{theorem}

\subsection{Couplings between $\GFF$ and $\CLE_4$}
Note that $\SLE_4$ paths can be viewed as level lines of GFF; and that $\CLE_4$ is a collection of $\SLE_4$-type loops. It is natural to expect that the collections of level loops of $\GFF$ correspond to $\CLE_4$. In this section, we will describe two couplings between $\GFF$ and $\CLE_4$. Before this, 
we first recall a standard result about Brownian motion. Consider a one-dimensional standard Brownian motion $(B_t,t\ge 0)$. We take the reflected Brownian motion $(Y_t=|B_t|, t\ge 0)$ and it is well known that , if we decompose this $Y$ at zero set of the Brownian motion $Z=\{t: B_t=0\}$, then this reflected Brownian motion $Y$ can be decomposed into countably many Brownian excursions (a Brownian excursion $(e(t), 0\le t\le \tau)$ is a Brownian path with $e(0)=0, e(\tau)=0$ and $e(t)>0$ for $0<t<\tau$). Consider the local time process $(L_t, t\ge 0)$ of the Brownian motion, it is increasing on zero set of the Brownian motion and is constant inside each excursion. If we parameterize these Brownian excursions by the local time process, then the sequence of these Brownian excursions is a Poisson point process. We can also reverse this procedure to construct a Brownian motion from a Poisson point process of Brownian excursions. Given a Poisson point process of Brownian excursions $(e_u,u\ge 0)$, there are two ways to reconstruct a Brownian motion:
\begin{enumerate}
\item [(a)] Sample i.i.d. coin tosses $\sigma_u$ for each excursion $e_u$, let the excursion to be positive or negative according to the sign $\sigma_u e_u$; then concatenate these signed excursions. The process we get is a Brownian motion.
\item [(b)] Concatenate all these excursions and get the reflected Brownian motion $(Y_t,t\ge 0)$. Define the local time process $(L_t,t\ge 0)$ of $Y$. Then the process $(Y_t-L_t, t\ge 0)$ has the same law as a Brownian motion.
\end{enumerate}

In the following of this section, we will describe somewhat analogous pair of couplings between GFF and $\CLE_4$. The first coupling between GFF and $\CLE_4$ is stated by Jason Miller and Scott Sheffield in 2011 and a proof can be found as a special case in \cite[Theorem 1.2]{MillerWatsonWilsonCLENestingfield}.

Let $\Gamma$ be a $\CLE_4$ in $\U$. For each loop $L\in\Gamma$, sample $\sigma_L$ to be $+1$ or $-1$ with equal probability $1/2$. We also call $\sigma_L$ as the orientation of $L$, i.e. $\sigma_L=+1$ (resp. $\sigma_L=-1$) corresponds to $L$ being oriented clockwise (resp. counterclockwise). All these orientations are sampled in the way that, given $\Gamma$, they are conditionally independent. The the obtained sample $((L,\sigma_L), L\in\Gamma)$ is called \textbf{$\CLE_4$ with symmetric orientations}.

\begin{theorem}\label{thm::coupling_gff_cle4_symmetric}
Suppose that $h$ is a zero-boundary $\GFF$ on $\U$ and that $((L,\sigma_L),L\in\Gamma)$ is a $\CLE_4$ with symmetric orientations in $\U$. There exists a coupling between $h$ and $((L,\sigma_L), L\in\Gamma)$ such that the following is true. Given $((L,\sigma_L), L\in\Gamma)$, for each loop $L$, the conditional law of $h$ restricted to the interior of $L$, denoted by $\inte(L)$, is the same as $\GFF$ with boundary value $2\lambda\sigma_L$; for different loops, the restrictions of the field are conditionally independent.
\end{theorem}

We refer to the coupling between GFF and CLE$_4$ with symmetric orientations in Theorem \ref{thm::coupling_gff_cle4_symmetric} as the first coupling between $\GFF$ and $\CLE_4$. As we can see, it can be viewed as the analog of the first reconstruction of Brownian motion from Brownian excursions.

In the current paper, we focus on another coupling between GFF and $\CLE_4$: the coupling between GFF and CLE$_4$ with time parameter, which can be viewed as the analog of the second reconstruction of Brownian motion. $\CLE_4$ with time parameter is a $\CLE_4$ where each loop has a time parameter (precise construction will be recalled in Section \ref{subsec::cle}). Roughly speaking, the time parameter for each loop is the counterpart of the local time for each excursion in the second reconstruction of Brownian motion. The main theorem is the following.
\begin{theorem}\label{thm::coupling_gff_cle4time}
Suppose that $h$ is a zero-boundary $\GFF$ on $\U$ and that $((L,t_L),L\in\Gamma)$ is a $\CLE_4$ with time parameter in $\U$. There exists a coupling between $h$ and $((L,t_L), L\in\Gamma)$ such that the following is true. Given $((L,t_L), L\in\Gamma)$, for each loop $L$, the conditional law of $h$ restricted to the interior of $L$, denoted by $\inte(L)$, is the same as $\GFF$ with boundary value $2\lambda(1-t_L)$; for different loops, the restrictions of the field are conditionally independent.
\end{theorem}

The second coupling between GFF and $\CLE_4$ is stated by Scott Sheffield, Samuel Watson, and Hao Wu in 2012. The authors thank Scott Sheffield and Samuel Watson for allowing us to write up the details of the proof.  
\subsection{Relation to previous works and outline}
\label{subsec::relation_previous_works}
We prove Theorems \ref{thm::boundary_levelline_gff_coupling} to \ref{thm::boundary_levellines_targetindependence} in Section \ref{sec::boundary_levellines}. Following is the outline of Section \ref{sec::boundary_levellines}.
\begin{itemize}
\item Section \ref{subsec::chordal_sle} is an introduction to chordal $\SLE$ process.
\item Section \ref{subsec::boundary_levellines_gff} is an introduction to zero-boundary GFF.
The proofs of Theorems \ref{thm::boundary_levelline_gff_coupling} and \ref{thm::boundary_levelline_gff_deterministic} can be found in previous works: \cite{DubedatSLEFreeField, SchrammSheffieldContinuumGFF, MillerSheffieldIG1}. To be self-contained, we still write the proof of Theorem \ref{thm::boundary_levelline_gff_coupling} in Section \ref{subsec::boundary_levellines_gff}.
\item Sections \ref{subsec::boundary_levellines_boundaryavoiding} to \ref{subsec::boundary_levellines_continuity_general} complete the proofs of Theorems \ref{thm::boundary_levelline_gff_deterministic} and \ref{thm::sle_chordal_continuity_transience}. The proofs in these sections are similar to those in \cite[Section 4, Section 5, Section 6]{MillerSheffieldIG1}, whereas some part of the reasoning in \cite{MillerSheffieldIG1} does not apply to $\kappa=4$ case. Thus, even though similar, we choose to rewrite and reorganize these proofs to treat $\kappa=4$ case.
\item Section \ref{subsec::boundary_levellines_interacting_general} completes the proof of Theorem \ref{thm::boundary_levelline_gff_interacting}. 
In \cite[Section 7.2]{MillerSheffieldIG1}, the authors prove a similar merging and monotonicity result for flow lines (i.e. for $\SLE_{\kappa}$ with $\kappa\in (0,4)$). In the current paper, we give a different proof for level line merging and monotonicity result based on the special property of the level lines: reversibility.
\item Section \ref{subsec::boundary_levellines_reversibility_general} completes the proofs of Theorems \ref{thm::sle_chordal_reversibility} and \ref{thm::boundary_levellines_targetindependence}. The discussion about the reversibility of level lines of $\GFF$ can be found in \cite{SchrammSheffieldContinuumGFF} for the setting that the curves are standard $\SLE_4$ paths. We generalize this result to $\SLE_4(\underline{\rho})$ process. The discussion about reversibility for SLE paths can also be found in \cite{ZhanReversibility, ZhanReversibilityMore, DubedatSLEFreeField,  MillerSheffieldIG2, MillerSheffieldIG3} for different settings. 
\end{itemize}
We prove Theorem \ref{thm::coupling_gff_cle4time} in Section \ref{sec::interior_levellines}. Following is the outline of Section \ref{sec::interior_levellines}.
\begin{itemize}
\item Section \ref{subsec::radial_sle} is an introduction to radial $\SLE$ process.
\item Section \ref{subsec::cle} is an introduction to CLE. This section is a summary of the results in \cite{SheffieldWernerCLE, WernerWuCLEExploration}.
\item In Section \ref{subsec::interiror_levellines_interior}, we introduce level lines of $\GFF$ targeted at interior points (in contrast with the level lines targeted at boundary points studied in Section \ref{sec::boundary_levellines}). We explain the monotonicity property and the target-independent property of the level lines targeted at interior points, which are analogous to Theorems \ref{thm::boundary_levelline_gff_interacting} and \ref{thm::boundary_levellines_targetindependence} for the level lines targeted at boundary points.
\item Sections \ref{subsec::interior_levellines_upward} and \ref{subsec::interior_levellines_tree} complete the proof of Theorem \ref{thm::coupling_gff_cle4time}
\end{itemize}
Our contribution in the current paper.
\begin{itemize}
\item We summarize various properties of level lines of $\GFF$ that start from boundary points and are targeted at boundary points (Section \ref{sec::boundary_levellines}). 
\item We give new proofs for Theorem \ref{thm::boundary_levelline_gff_interacting} (the interaction behaviors between level lines), Theorem \ref{thm::sle_chordal_reversibility} (the reversibility of level lines), and Theorem \ref{thm::boundary_levellines_targetindependence} (the target-independence of level lines).  
\item We describe the behavior of level lines of $\GFF$ that start from boundary points and are targeted at interior points (Section \ref{sec::interior_levellines}). We give a complete proof for the existence of the coupling between $\GFF$ and $\CLE_4$ with time parameter: Theorem \ref{thm::coupling_gff_cle4time}. 
The close study on $\CLE_4$ exploration process in Section \ref{sec::interior_levellines} is an essential ingredient in the study of the level lines of $\GFF$ emanating from an interior point in a forthcoming paper \cite{WangWuLevellinesGFFII}. All this work plays an important role in a larger program on the study of conformal invariant metric on $\CLE$ which includes \cite{WernerWuCLEExploration, SheffieldWatsonWuSimpleCLEDoubly, WangWuLevellinesGFFII, SheffieldWatsonWuConformalInvariantMetricCLE4}.  
\end{itemize}
In this paper, we focus on $\GFF$ with piecewise constant boundary conditions. It is natural to wonder the situation of $\GFF$ when the boundary condition is not piecewise constant. However, the techniques in this paper and the previous papers on the level lines of GFF do not apply directly to the general case.
In \cite{PowellWuLevellinesGFF}, the authors study the properties of level lines of $\GFF$ whose boundary condition is regulated without continuation threshold, and the important ingredient there is \cite{KemppainenSmirnovRandomCurves}. The results in Section \ref{sec::boundary_levellines} provide the platform for further studies there. As discussed in the introduction of \cite{PowellWuLevellinesGFF}, the properties of the level lines of $\GFF$ when the boundary condition is not piecewise constant and does have continuation threshold is still open. 

\noindent\textbf{Acknowledgements.} We thank Richard Kenyon, Jason Miller, Scott Sheffield, Samuel Watson, and Wendelin Werner for helpful discussions. H.\ Wu's work is funded by NSF DMS-1406411.

\section{Boundary emanating Level Lines of $\GFF$}\label{sec::boundary_levellines}
\subsection{Chordal $\SLE$}\label{subsec::chordal_sle}
%\subsubsection{Chordal $\SLE_{\kappa}$}
\begin{comment}
We call a compact subset $K$ of $\overline{\HH}$ a hull if $H=\HH\setminus K$ is simply connected. Riemann's mapping theorem asserts that there exists a conformal map $\Psi$ from $H$ onto $\HH$ that $\Psi(\infty)=\infty$. Moreover, there exists a unique conformal map $\Psi$ from $H=\HH\setminus K$ onto $\HH$ such that $\Psi(z)=z+0+O(1/z)$, as $z\to\infty$.
We call such a conformal map as the conformal map from $H=\HH\setminus K$ onto $\HH$ normalized at $\infty$, and denote it as $\Psi_K$.
In particular, there exists a real $a=a(K)$ such that
\[\Psi(z)=z+\frac{2a}{z}+o(\frac{1}{z}),\quad\text{as }z\to\infty.\]
We can check that $a(K)$ is a non-negative increasing function of the set $K$. Thus this number $a(K)$ is a way to measure the size of $K$.
We call $a(K)$ as the capacity of $K$ in $\HH$ seen from $\infty$ or \textbf{half-plane capacity}.
\end{comment}
Suppose that $(W_t,t\ge 0)$ is a continuous real function with $W_0=0$. For each $z\in\overline{\HH}$, define the function $g_t(z)$ as the solution to \textit{Chordal Loewner Equation}
\[\partial_t g_t(z)=\frac{2}{g_t(z)-W_t},\quad g_0(z)=z.\]
This is well-defined as long as $g_t(z)-W_t$ does not hit 0. Define
\[T(z)=\sup\{t>0: \min_{s\in[0,t]}|g_s(z)-W_s|>0\}.\]
This is the largest time up to which $g_t(z)$ is well-defined. Set
\[K_t=\{z\in\overline{\HH}: T(z)\le t\},\quad H_t=\HH\setminus K_t.\]
We can check that $g_t$ is a conformal map from $H_t$ onto $\HH$ normalized at $\infty$ and that, for each $t$,
$g_t(z)=z+2t/z+o(1/z),$ as $z\to\infty$.
The family $(K_t,t\ge 0)$ is called the \textbf{Loewner chain} driven by $(W_t,t\ge 0)$.
Here we collect some general results about chordal Loewner chain.
\begin{proposition}\label{prop::loewner_chordal_driving_zeroset}
Suppose that $\gamma$ is a continuous curve in $\overline{\HH}$ from $0$ to $\infty$ with a continuous Loewner driving function $W$. Then the set $\{t:\gamma(t)\in\R\}$ has Lebesgue measure zero.
\end{proposition}
\begin{proof}\cite[Lemma 2.5]{MillerSheffieldIG1}
\end{proof}
\begin{proposition}\label{prop::loewner_chordal_driving_existence_continuity}
Suppose that $T\in (0,\infty]$. Let $\gamma: [0,T)\to\overline{\HH}$ be a continuous, non-crossing curve with $\gamma(0)=0$. Assume that $\gamma$ satisfies the following:
\begin{enumerate}
\item [(1)] $\gamma(t,T)$ is contained in the closure of the unbounded connected component of $\HH\setminus\gamma(0,t)$;
\item [(2)] $\gamma^{-1}(\gamma[0,t]\cup\R)$ has empty interior in $(t,T)$.
\end{enumerate}
For each $t>0$, let $g_t$ be the conformal map from the unbounded connected component of $\HH\setminus\gamma[0,t]$ onto $\HH$ with $\lim_{z\to\infty}|g_t(z)-z|=0$. After reparameterization, $(g_t,t\ge 0)$ solves the Loewner equation
\[\partial_t g_t(z)=\frac{2}{g_t(z)-W_t},\quad g_0(z)=z,\]
with continuous driving function $W$.
\end{proposition}
\begin{proof}
\cite[Proposition 4.3]{LawlerConformallyInvariantProcesses} and \cite[Proposition 6.12]{MillerSheffieldIG1}
\end{proof}
\medbreak
\textbf{Chordal $\SLE_{\kappa}$} for $\kappa\ge 0$ is the Loewner chain driven by $W_t=\sqrt{\kappa}B_t$ where $B$ is a 1-dimensional Brownian motion. Here are several basic properties of chordal SLE:
\begin{itemize}
\item It is scale-invariant: For any $\lambda>0$, the process $(\lambda^{-1}K_{\lambda^2t},t\ge 0)$ has the same law as $K$ itself.
\item It satisfies domain Markov property: For any finite stopping time $\tau$, the process $(f_{\tau}(K_{t+\tau}),t\ge 0)$ has the same law as $K$ itself where $f_t:=g_t-W_t$.
%\item It is symmetric with respect to the imaginary axis: $(-K_t,t\ge 0)$ has the same law as $K$ itself.
\end{itemize}
\begin{proposition}
For all $\kappa\in [0,4]$, chordal $\SLE_{\kappa}$ is almost surely generated by a simple continuous curve, i.e. there exists a simple continuous curve $\gamma$ such that
$K_t=\gamma[0,t]$ for all $t\ge 0$. Moreover, the curve is almost surely transient: $\lim_{t\to\infty}\gamma(t)=\infty$.
\end{proposition}
\begin{proof}
\cite[Theorem 5.1]{RohdeSchrammSLEBasicProperty}
\end{proof}
Fix $\kappa\ge 0$ and
\[\underline{\rho}^L=(\rho^{l,L},...,\rho^{1,L}), \quad\underline{\rho}^R=(\rho^{1,R},...,\rho^{r,R});\quad \underline{x}^{L}=(x^{l,L}<\cdots<x^{1,L}\le 0),\quad  \underline{x}^R=(0\le x^{1,R}<\cdots<x^{r,R}).\] Take the convention
\[\rho^{0,R}=\rho^{0,L}=0,\quad x^{0,L}=0^-,\quad x^{0,R}=0^+,\quad x^{l+1,L}=-\infty, \quad x^{r+1,R}=\infty.\]
Define
\[\overline{\rho}^{j,L}=\sum_{i=0}^j\rho^{i,L},\quad \text{ for }0\le j\le l;\quad \overline{\rho}^{j,R}=\sum_{i=0}^j\rho^{i,R},\quad \text{ for }0\le j\le r.\]

\begin{definition}\label{def::chordal_sle_general}
Let $B_t$ be a standard Brownian motion. We will say that the process $(W_t,V^{i,q}_t)$ describe an \textbf{$\SLE_{\kappa}(\underline{\rho}^L;\underline{\rho}^R)$ process with force points $(\underline{x}^L;\underline{x}^R)$} if they are adapted to the filtration of $B$ and the following hold:
\begin{enumerate}
\item [(1)] The processes $W_t,V^{i,q}_t$ and $B_t$ satisfy the following SDE on the time intervals on which $W_t$ does not collide with any of $V^{i,q}_t$.
\[dW_t=\sqrt{\kappa}dB_t+\sum_{q\in\{L,R\}}\sum_{i}\frac{\rho^{i,q}dt}{W_t-V^{i,q}_t},\quad dV^{i,q}_t=\frac{2dt}{V^{i,q}_t-W_t},\quad \text{for } q\in\{L,R\}.\]
\item [(2)] We have instantaneous reflection of $W_t$ off of the $V^{i,q}_t$, i.e. it is almost surely the case that for Lebesgue almost all times $t$ we have that $W_t\neq V^{i,q}_t$ for each $i,q$.
\item [(3)] We also have almost surely that, for each $i,q$,
\[V^{i,q}_t=x^{i,q}+\int_0^t\frac{2ds}{V^{i,q}_s-W_s}.\]
\end{enumerate}
\end{definition}

We define the \textbf{continuation threshold} to be the infimum of the $t$ values for which
\[\text{either }\sum_{i:V^{i,L}_t=W_t}\rho^{i,L}\le -2,\quad \text{or }\sum_{i:V^{i,R}_t=W_t}\rho^{i,R}\le -2.\]

\begin{proposition}\label{prop::chordal_sle_existence_uniqueness}
Definition \ref{def::chordal_sle_general} uniquely determines a joint law for $(B_t, W_t,V^{i,q}_t)$ --each defined for all $t$ up to the continuation threshold. Under this law, the process $(B_t, W_t,V^{i,q}_t)$ is a continuous multidimensional Markovian process indexed by $t$.
\end{proposition}
\begin{proof}
\cite[Theorem 2.2]{MillerSheffieldIG1}.
\end{proof}

\begin{proposition}\label{prop::chordal_sle_continuity_nonintersecting}
Assume that
\[\overline{\rho}^{i,q}\ge \kappa/2-2, \quad \text{for all }i,q.\] Then $\SLE_{\kappa}(\underline{\rho}^L;\underline{\rho}^R)$ process is a continuous curve.
\end{proposition}
\begin{proof}
Under the assumption that $\overline{\rho}^{i,q}\ge \kappa/2-2$, for all $i,q$, the curve can not hit the boundary and thus it is absolutely continuous with respect to $\SLE_{\kappa}$, and then it is almost surely continuous.
\end{proof}

\begin{proposition}\label{prop::chordal_sle4_mart}
Suppose we are given a random continuous curve $\gamma$ in $\overline{\HH}$ from $0$ to $\infty$ whose Loewner driving function $W$ is almost surely continuous. Suppose $(K_t,t\ge 0)$ is the corresponding Loewner chain and $(g_t,t\ge 0)$ is the corresponding sequence of conformal maps. Set $f_t=g_t-W_t$. Suppose that $V^{i,q}_t$ is the image of $x^{i,q}$ under $g_t$. Let $\eta_t(z)$ be the function defined in Theorem \ref{thm::boundary_levelline_gff_coupling}.
Then $W_t$ and $V^{i,q}_t$ can be coupled with a standard Brownian motion to describe an $\SLE_4(\underline{\rho}^L;\underline{\rho}^R)$ process with force points $(\underline{x}^L;\underline{x}^R)$ up to the continuation threshold if and only if $\eta_t(z)$ evolves as a Brownian motion when parameterized by the log of the conformal radius of $\HH\setminus K_t$ seen from $z$, for each fixed $z\in\HH$, until the time $t$ that $z$ is swallowed by $K_t$.
\end{proposition}
\begin{proof}
\cite[Theorem 2.6]{MillerSheffieldIG1}.
\end{proof}

The following two propositions are results about the interacting behavior of SLE process with the boundary that we will use later in the paper.
\begin{proposition}\label{prop::dubedat_lemma15}
Fix $\kappa\in [0,4]$. Suppose that $K$ is an $\SLE_{\kappa}(\underline{\rho}^R)$ process with force points $(\underline{x}^R)$ where $0<x^{1,R}<\cdots<x^{r,R}$. Set
\[\overline{\rho}^{j,R}=\sum_{i=1}^j\rho^{i,R},\quad \text{for }1\le j\le r.\] Let $\tau^{1,R}$ be the first time that $K$ swallows $x^{1,R}$.
\begin{enumerate}
\item[(1)] Assume that, for some $k$, we have that $\overline{\rho}^{i,R}\ge\kappa/2-2$ for $i<k$; and $\overline{\rho}^{i,R}\le \kappa/2-4$ for $i\ge k$. Then almost surely, as $t\uparrow\tau^{1,R}$, $K_t$ accumulates at $x^{k,R}$ without hitting any other point in $[x^{1,R},\infty)$.
\item[(2)] Assume that, for some $k$, we have that $\overline{\rho}^{i,R}\ge\kappa/2-2$ for $i<k$; $\overline{\rho}^{k,R}\in (\kappa/2-4,\kappa/2-2)$; and $\overline{\rho}^{i,R}\le \kappa/2-4$ for $i> k$. Then almost surely, as $t\uparrow\tau^{1,R}$, $K_t$ accumulates at a point in $[x^{k,R},x^{k+1,R}]$ without hitting any other point in $[x^{1,R}, x^{k,R})\cup(x^{k+1,R},\infty)$.
\end{enumerate}
\end{proposition}
\begin{proof}
\cite[Lemma 15]{DubedatSLEDuality}.
\end{proof}

\begin{proposition}\label{prop::sle_chordal_forbidden_intervals} Fix $\kappa\in [0,4]$.
Suppose that $K$ is an $\SLE_{\kappa}(\underline{\rho}^R)$ process with force points $(\underline{x}^R)$ where $0<x^{1,R}<\cdots<x^{r,R}$. Set
\[\overline{\rho}^{j,R}=\sum_{i=1}^j\rho^{i,R},\quad \text{for }1\le j\le r.\] Assume that, for some stopping time $0<T_0<\infty$, $(K_t,0\le t\le
T_0)$ is almost surely generated by a continuous curve $(\gamma(t),0\le t\le T_0)$. Then $\gamma[0,T_0]$ almost surely does not intersect any interval $(x^{j,R},x^{j,R+1})$ such that
\[\overline{\rho}^{j,R}\ge \kappa/2-2,\quad \text{or}\quad \overline{\rho}^{j,R}\le \kappa/2-4.\]
\end{proposition}
\begin{proof}
\cite[Lemma 5.2 and Remark 5.3]{MillerSheffieldIG1}.
\end{proof}

\subsection{The Zero-Boundary $\GFF$}\label{subsec::boundary_levellines_gff}
Suppose that $D\subsetneq \C$ is a proper domain with harmonically non-trivial boundary (i.e. a Brownian motion started at a point in $D$ hits $\partial D$ almost surely.) For $f,g\in L^2(D)$, we denote by $(f,g)$ the inner product of $L^2(D)$:
\[(f,g)=\int_D f(z)g(z)d^2z,\]
where $d^2z$ is the Lebesgue area measure. Denote by $H_s(D)$ the space of real-valued smooth functions which are compactly supported in $D$. This space has a \textbf{Dirichlet inner product} defined by
\[(f,g)_{\nabla}=\frac{1}{2\pi}\int_D\nabla f(z)\cdot\nabla g(z)d^2z.\]
Denote by $H(D)$ the Hilbert space completion of $H_s(D)$.

The \textbf{zero-boundary $\GFF$} on $D$ is a random sum of the form
$h=\sum_{j=1}^{\infty}\alpha_jf_j$, 
where the $\alpha_j$ are i.i.d. one-dimensional standard Gaussians (with mean zero and variance 1) and the $f_j$ are an orthonormal basis for $H(D)$. This sum almost surely does not converge within $H(D)$; however, it does converge almost surely within the space of distributions--- that is, the limit $\sum_j \alpha_j(f_j,p)$ almost surely exists for all $p\in H_s(D)$, and the limiting values, denoted by $(h,p)$, as a function of $p$ is almost surely a continuous functional on $H_s(D)$.
For any $f\in H_s(D)$, let $p=-\Delta f\in H_s(D)$, and define
\[(h,f)_{\nabla}:=\frac{1}{2\pi}(h,p).\]
Then $(h,f)_{\nabla}$ is a mean-zero Gaussian with variance
\[\frac{1}{4\pi^2}\sum_j(f_j,p)^2=\sum_j(f_j,f)_{\nabla}^2=(f,f)_{\nabla}^2.\]
The zero-boundary $\GFF$ on $D$ is the only random distribution on $D$ with the property that, for each $f\in H_s(D)$, $(h,f)_{\nabla}$ is a mean-zero Gaussian with variance $(f,f)_{\nabla}$.

When $z\in D$ is fixed, let $\tilde{G}_z(w)$ be the harmonic extension to $w\in D$ of the function of $w$ on $\partial D$ given by $-\log|z-w|$. Then the \textbf{Green's function} in the domain $D$ is given by
\begin{equation}\label{eqn::greenfunction_definition}G_D(z,w)=-\log|z-w|-\tilde{G}_z(w).\end{equation}
For any $p\in H_s(D)$, define $\Delta^{-1}p$ on $D$ by
\[\Delta^{-1}p(\cdot):=-\frac{1}{2\pi}\int_D G_D(\cdot,y)p(y)dy.\]
This is a smooth function in $D$ whose Laplacian is $p$ and whose boundary value is zero on $\partial D$. We point out that the Green's function is conformally invariant: if $\phi$ is a conformal map on $D$, then, for any $z,w\in D$, we have
\begin{equation}\label{eqn::greenfunction_conformalinvariance}
G_D(z,w)=G_{\phi(D)}(\phi(z),\phi(w)).
\end{equation}
Note that, for any $f_1, f_2, p_1, p_2\in H_s(D)$, we have that
\[\cov((h,f_1)_{\nabla},(h,f_2)_{\nabla})=\frac{1}{2\pi}\int_D\nabla f_1(z)\cdot\nabla f_2(z) d^2z,\quad 
\cov((h,p_1),(h,p_2))=\iint_{D\times D}p_1(z)p_2(w)G_D(z,w)d^2zd^2w.\]
For any deterministic open subset $W\subsetneq D$, there is a natural inclusion $\iota$ of $H(W)$ into $H(D)$ by
$\iota(f)(z)=f(z)1_{z\in W}$.
We can see that $H(D)$ admits the $(\cdot,\cdot)_{\nabla}$-orthogonal decomposition
\begin{equation}\label{eqn::gff_space_decomposition}
H(D)=H(W)\oplus H^{\bot}(W),
\end{equation}
where $H^{\bot}(W)$ is the space of functions in $H(D)$ that are harmonic in $W$. The reason is the following. For any $f\in H_s(D)$, let $f_W^{\bot}$ be the function that equals $f$ on $D\setminus W$ and be harmonic in $W$ and let $f_W=f-f_W^{\bot}$. Then $f_W^{\bot}\in H^{\bot}(W), f_W\in H(W)$, and $(f_W^{\bot},f_W)_{\nabla}=\frac{1}{2\pi}(-\Delta f_W^{\bot},f_W)=0$.

The decomposition in Equation (\ref{eqn::gff_space_decomposition}) leads to a decomposition of the $\GFF$ on $D$:
\begin{equation}\label{eqn::gff_distribution_decomposition}
h=h_W+h_W^{\bot},
\end{equation}
where $h_W$ and $h_W^{\bot}$ are distributions on $H(D)$ such that, for any $f\in H(D)$,
\[(h_W,f)_{\nabla}=(h,f_W)_{\nabla},\quad (h_W^{\bot},f)_{\nabla}=(h,f_W^{\bot})_{\nabla}.\]
Clearly, $h_W$ and $h_W^{\bot}$ are independent.

For any distribution $h$ on $H(D)$, we define the restriction of $h$ to $W$, denoted by $h|_W$, to be $h$ restricted to the functions that are compactly supported in $W$. If $h$ is a zero-boundary $\GFF$ on $D$, then $h_W|_W$ is a zero-boundary $\GFF$ on $W$ and $h_W^{\bot}|_W$ is almost surely harmonic. Thus, the conditional law of $h|_W$ given $h|_{D\setminus W}$ is that of the zero-boundary $\GFF$ on $W$ plus the harmonic extension of $h|_{D\setminus W}$ to $W$. This is called ``domain Markov property" of the $\GFF$.

Suppose that $f$ is a piecewise continuous function on $\partial D$, and that $F$ is the harmonic extension of $f$ to $D$. We define the $\GFF$ on $D$ with mean $F$ to be the sum of a zero-boundary $\GFF$ plus $F$. Sometimes, we use the term ``the $\GFF$ with boundary value $f$" to refer to the $\GFF$ with mean $F$.

\begin{proposition}\label{prop::gff_absolutecontinuity}
Suppose that $D_1,D_2$ are simply connected domains with $D_1\cap D_2\neq\emptyset$. For $i=1,2$, let $h_i$ be a zero-boundary $\GFF$ on $D_i$ and $F_i$ be harmonic on $D_i$. Fix a simply connected open domain $U\subset D_1\cap D_2$.
\begin{enumerate}
\item [(1)] If $\dist(U,\partial D_i)>0$ for $i=1,2$, then the law of
\[(h_1+F_1)|_U,\quad\text{and}\quad (h_2+F_2)|_U\]
are mutually absolutely continuous.
\item [(2)] Suppose that there is a neighborhood $U'$ of the closure $\overline{U}$ such that $D_1\cap U'=D_2\cap U'$, and that $F_1-F_2$ tends to zero as one approaches points in the sets $\partial D_i\cap U'$. Then the laws of
    \[(h_1+F_1)|_U,\quad\text{and}\quad (h_2+F_2)|_U\]
are mutually absolutely continuous.
\end{enumerate}
\end{proposition}
\begin{proof}\cite[Proposition 3.2]{MillerSheffieldIG1}.
\end{proof}

Suppose that $D$ is a simply connected domain, and that $A$ is a random closed subset of $\overline{D}$. For $\delta>0$, let $A_{\delta}$ denote the closed set containing all points in $D$ whose distance from $A$ is at most $\delta$. Let $\LA_{\delta}$ be the smallest $\sigma$-algebra in which $A$ and the restriction of $h$ to the interior of $A_{\delta}$ are measurable. Let $\LA=\cap_{\delta\in\QQ,\delta>0}\LA_{\delta}$. Intuitively, this is the smallest $\sigma$-algebra in which $A$ and the values of $h$ in an infinitesimal neighborhood of $A$ are measurable.

\begin{proposition}\label{prop::gff_localset_definition}
Suppose that $(h,A)$ is a random variable which is a coupling of an instance of the $\GFF$ and a random closed subset $A\subset \overline{D}$. Then the following are equivalent:
\begin{enumerate}
\item [(1)] For any deterministic open set $U\subseteq D$, we have that, given the orthogonal projection of $h$ onto $H^{\bot}(U)$, the event $\{A\cap U=\emptyset\}$ is independent of the orthogonal projection of $h$ onto $H(U)$. In other words, the conditional probability of $\{A\cap U=\emptyset\}$ given $h$ is a measurable function of the orthogonal projection of $h$ onto $H^{\bot}(U)$.
\item [(2)] Given $\LA$, the conditional law of $h$ is that of $h_1+h_2$ where $h_2$ is a zero-boundary $\GFF$ on $D\setminus A$ and $h_1$ is an $\LA$-measurable random distribution which is almost surely harmonic on $D\setminus A$.
\end{enumerate}
\end{proposition}
\begin{proof}
\cite[Lemma 3.9]{SchrammSheffieldContinuumGFF}.
\end{proof}

We say a random closed set $A$ coupled with an instance $h$ of $\GFF$ is a \textbf{local set} for $h$ if one of the equivalent items in Proposition \ref{prop::gff_localset_definition} holds. For any coupling of $A$ and $h$, we use the notation $\LC_A$ to describe the conditional expectation of $h$ given $\LA$. When $A$ is local, $\LC_A$ is the $h_1$ described in Item (2) in Proposition \ref{prop::gff_localset_definition}. We use the notation $(A,h|_A)$ to refer to $\LA$ and also say that $\LC_A$ is the conditional expectation of $h$ given $A$ and $h|_A$. By convention, we write $\LC$ the mean of $h$.

\begin{proposition}\label{prop::gff_localsets_union}
Suppose that $h$ is a $\GFF$ and $A_1,A_2$ are random closed subsets of $D$ and that $(h,A_1)$ and $(h,A_2)$ are couplings for which $A_1,A_2$ are local. Let $A=A_1\tilde{\cup}A_2$ denote the random closed subset of $D$ which is given by first sampling $h$, then sampling $A_1,A_2$ conditionally independent given $h$, and then taking the union of $A_1$ and $A_2$. Then $A$ is also local for $h$. Moreover, given $(A_1, A_2, A, h|_A)$, the conditional law of $h$ is given by $\LC_A$ plus an instance of zero-boundary $\GFF$ on $D\setminus A$.
\end{proposition}
\begin{proof}
\cite[Lemma 3.10]{SchrammSheffieldContinuumGFF}.
\end{proof}

\begin{proposition}\label{prop::gff_localsets_boundarybehavior}
Let $A_1,A_2$ be connected local sets which are conditionally independent and $A=A_1\tilde{\cup}A_2$. Then $\LC_A-\LC_{A_2}$ is almost surely a harmonic function in $D\setminus A$ that tends to zero along all sequences of points in $D\setminus A$ that tend to a limit in a connected component of $A_2\setminus A_1$ (which consists of more than a single point) or that tend to a limit on a connected component of $A_1\cap A_2$ (which consists of more than a single point) at a point that is a positive distance from either $A_2\setminus A_1$ or $A_1\setminus A_2$.
\end{proposition}
\begin{proof}
\cite[Lemma 3.11]{SchrammSheffieldContinuumGFF} and \cite[Proposition 3.6]{MillerSheffieldIG1}.
\end{proof}

\begin{proposition}\label{prop::gff_localsets_interacting}
Let $A_1,A_2$ be connected local sets which are conditionally independent and $A=A_1\tilde{\cup}A_2$. Suppose that $C$ is a $\sigma(A_1)$-measurable connected component of $D\setminus A_1$ such that $\{C\cap A_2=\emptyset\}$ almost surely. Then $\LC_A|_C=\LC_{A_1}|_C$ almost surely, given $A_1$. In particular, $h|_C$ is independent of the pair $(h|_{D\setminus C},A_2)$ given $A_1$.
\end{proposition}
\begin{proof}
\cite[Proposition 3.7]{MillerSheffieldIG1}.
\end{proof}
\begin{proposition}\label{prop::gff_localsets_bm}
Let $h$ be a $\GFF$ on $D$ and suppose that $(Z(t),t\ge 0)$ is an increasing family of closed sets such that $Z(\tau)$ is local for $h$ for every $Z$-stopping time $\tau$; and, for a fixed $z\in D$, that $\CR(D\setminus Z(t);z)$ is almost surely continuous and monotonic in $t$. Then $\LC_{Z(t)}(z)-\LC_{Z(0)}(z)$ has a modification which is a Brownian motion when parameterized by \[\log\CR(D\setminus Z(0);z)-\log\CR(D\setminus Z(t);z)\] up until the first time that $Z(t)$ accumulates at $z$. In particular, $\LC_{Z(t)}(z)$ has a modification which is almost surely continuous in $t$.
\end{proposition}
\begin{proof}
\cite[Proposition 6.5]{MillerSheffieldIG1}.
\end{proof}
\begin{comment}
\begin{proof}
For each $s>0$, set
\[\tau(s):=\inf\{t\ge 0: \log\CR(D\setminus Z(0);z)-\log\CR(D\setminus Z(t);z)=s\}.\]
We only need to show that, for any fixed $s<t$, we have
\begin{enumerate}
\item [(a)] $\LC_{Z(\tau(t))}-\LC_{Z(\tau(s))}$ is independent of $\LC_{Z(\tau(s))}$;
\item [(b)] $\LC_{Z(\tau(t))}-\LC_{Z(\tau(s))}$ has the law of a Gaussian with mean zero and variance $(t-s)$.
\end{enumerate}
From Lemma \ref{lem::gff_localsets_cr}, we know that, given $(Z(\tau(s)), h|_{Z(\tau(s))})$, the conditional law of $\LC_{Z(\tau(t))}-\LC_{Z(\tau(s))}$ is the same as a Gaussian with mean zero and variance
\[\log\CR(D\setminus Z(\tau(s));z)-\log\CR(D\setminus Z(\tau(t));z)=t-s.\]
Thus, $\LC_{Z(\tau(t))}-\LC_{Z(\tau(s))}$ has the law of a Gaussian with mean zero and variance $(t-s)$ and is independent of $(Z(\tau(s)), h|_{Z(\tau(s))})$; in particular, it is independent of $\LC_{Z(\tau(s))}$. This completes the proof.
\end{proof}
\end{comment}

We close this section by the proof of Theorem \ref{thm::boundary_levelline_gff_coupling}. To simplify the notations, we only prove the conclusion for the case when there is only one right force point and one left force point. The general case can be proved similarly. We state it as a proposition.

\begin{figure}[ht!]
\begin{subfigure}[b]{0.47\textwidth}
\begin{center}
\includegraphics[width=\textwidth]{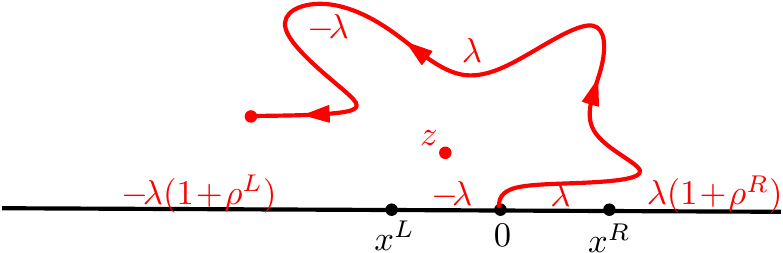}
\end{center}
\caption{Before the time that $z$ is swallowed, $\eta_t(z)$ is the harmonic extension to $z$ of the boundary value along the boundary of $\HH\setminus K_t$ as indicated in the figure.}
\end{subfigure}
$\quad$
\begin{subfigure}[b]{0.47\textwidth}
\begin{center}\includegraphics[width=\textwidth]{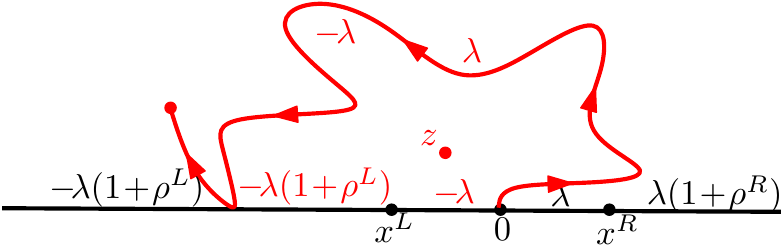}
\end{center}
\caption{After the time that $z$ is swallowed, $\eta_t(z)$ is the harmonic extension to $z$ of the boundary value along the connected component containing $z$.}
\end{subfigure}
\caption{\label{fig::boundary_levellines_coupling_eta} For fixed $t$, the function $\eta_t(\cdot)$ is harmonic in each connected component of $\HH\setminus \partial_t$. }
\end{figure}

\begin{proposition}\label{prop::boundary_levellines_gff_coupling}
Fix $\rho^L,\rho^R\in\R$ and $x^L\le 0\le x^R$. Suppose that the process $(W_t,V^L_t,V^R_t)$ describe an $\SLE_4(\rho^L;\rho^R)$ process with force points $(x^L;x^R)$. Let $(K_t,t\ge 0)$ be the corresponding Loewner chain and let $(g_t,t\ge 0)$ be the corresponding sequence of conformal maps, and set $f_t=g_t-W_t$. There exists a coupling $(K,h)$ where $h$ is a zero-boundary $\GFF$ on $\HH$ such that the following domain Markov property is true. Suppose that $\tau$ is any finite stopping time less than the continuation threshold for $K$. Let $\eta_t^0$ be the function which is harmonic in $\HH$ with boundary values
\[\begin{cases}
\lambda(1+\rho^R)& \text{if}\quad x\ge V^R_t-W_t,\\
\lambda& \text{if}\quad0\le x<V^R_t-W_t,\\
-\lambda&\text{if}\quad V^L_t-W_t\le x<0,\\
-\lambda(1+\rho^L)&\text{if}\quad x<V^L_t-W_t.
\end{cases}\]
Define, for $z\in\HH\setminus K_t$,
\[\eta_t(z)=\eta_t^0(f_t(z)).\] Then, given $K_{\tau}$, the conditional law of $(h+\eta_0)|_{\HH\setminus K_{\tau}}$ is equal to the law of $h\circ f_{\tau}+\eta_{\tau}$.
\end{proposition}

\begin{proof}
For $t\ge 0$, define
\[\partial_t=\cup_{0\le s\le t}\partial K_s, \quad \partial_{\infty}=\cup_{0\le s\le T}\partial K_s,\]
where $T$ is the continuation threshold. From Proposition \ref{prop::chordal_sle_continuity_nonintersecting}, when $\rho^L\ge 0,\rho^R\ge 0$, we know that $\SLE_4(\rho^L;\rho^R)$ is generated by a continuous curve, and $\partial_{\infty}$ is exactly this continuous curve. We will prove that, for general $\rho^L$ and $\rho^R$, the process $\SLE_4(\rho^L;\rho^R)$ is also generated by a continuous curve up to and including the continuation threshold. But we do not assume the continuity of SLE$_4(\underline{\rho})$ process in the proof of Theorem \ref{thm::boundary_levelline_gff_coupling}. In fact, the proof of the continuity for general $\underline{\rho}$, that we will show later, based on the coupling between $\GFF$ and the Loewner chain of $\SLE_4(\underline{\rho})$ in Theorem \ref{thm::boundary_levelline_gff_coupling}.

\textit{First}, we analyze the function $\eta_t(z)$. It is well-defined when $z\in\HH\setminus K_t$, and when $z$ is swallowed by $K_t$, we define $\eta_t(z)$ to be the limiting value of $\eta_s(z)$ as $s$ approaches the first time at which $z$ is swallowed by $K$. Note that, for fixed $t$, the function $\eta_t(\cdot)$ is harmonic in $\HH\setminus K_t$, and it is also harmonic in the finite connected component of $\HH\setminus \partial_t$ with certain boundary value, see Figure \ref{fig::boundary_levellines_coupling_eta}. From Proposition \ref{prop::chordal_sle4_mart}, we know that $\eta_t(z)$ is a continuous martingale up to the first time that $z$ is swallowed by $K$. Suppose that $m_t(\cdot)$ is a M\"{o}bius transformation of $\HH$ such that $m_t\circ f_t$ is a conformal map from $\HH\setminus K_t$ onto $\HH$ that preserves $z$. Define
\[C_t(z)=\log 2-\Re \log m_t'(f_t(z))-\Re \log g_t'(z),\]
which is the log of the conformal radius of $\HH\setminus K_t$ seen from $z$. We have that
\[d\la\eta_t(z)\ra=\left(\Im{\frac{2}{f_t(z)}}\right)^2dt=-dC_t(z).\]
\smallbreak
\textit{Second}, we analyze the product $\eta_t(z)\eta_t(w)$ for $z,w\in\HH$. Recall that the Green's function of the upper-half plane is given by
\[G(z,w)=\log|\frac{z-\bar{w}}{z-w}|,\quad z,w\in\HH.\]
Fix $z,w\in\HH$, define
\[G_t(z,w)=G(f_t(z),f_t(w)),\quad \text{if }z,w\in\HH\setminus K_t;\]
and when at least of one of $z,w$ is swallowed by $K_t$, we define $G_t(z,w)$ to be the limiting value of $G_s(z,w)$ when $s$ approaches the first time at which at least one of $z,w$ is swallowed. Note that, when $z,w\in\HH\setminus K_t$, $G_t(z,w)$ is the Green's function of the domain $\HH\setminus K_t$; when $z,w$ are not in the same connected component of $\HH\setminus\partial_t$, $G_t(z,w)$ becomes zero; and when $z,w$ are in the same connected component of $\HH\setminus\partial_t$, $G_t(z,w)$ is just the Green's function of that connected component.

We will show that $\eta_t(z)\eta_t(w)+G_t(z,w)$ is a continuous martingale up to the first time that at least one of $z,w$ is swallowed. Note that
\[d\la\eta_t(z),\eta_t(w)\ra=\Im{\frac{2}{f_t(z)}}\Im{\frac{2}{f_t(w)}}dt.\]
By It\^{o}'s formula,
\begin{align*}
dG_t(z,w)&=d\Re{\log\left(g_t(z)-\overline{g_t(w)}\right)}-d\Re{\log\left(g_t(z)-g_t(w)\right)}\\
&=\Re{\frac{-2dt}{f_t(z)\overline{f_t(w)}}}+\Re{\frac{2dt}{f_t(z)f_t(w)}}\\
&=-\Im{\frac{2}{f_t(z)}}\Im{\frac{2}{f_t(w)}}dt.
\end{align*}
Thus $d\la\eta_t(z),\eta_t(w)\ra=-dG_t(z,w)$ and $\eta_t(z)\eta_t(w)+G_t(z,w)$ is a local martingale. Note that $\eta_t(z)$ and $\eta_t(w)$ are continuous and bounded, and that $G_t(z,w)$ is continuous and non-increasing in $t$. These imply that $\eta_t(z)\eta_t(w)+G_t(z,w)$ is a continuous martingale.
\smallbreak
\textit{Third,} for any test function $p\in H_s(\HH)$, define
\[E_t(p)=\iint p(z)p(w)G_t(z,w)d^2zd^2w,\]
and we will explain that $(\eta_t,p)$ is a continuous martingale and
\begin{equation}\label{eqn::boundary_levelines_coupling_test_variation_1}d\la(\eta_t,p)\ra=-dE_t(p).\end{equation}
Since $\eta_t(z)$ is a continuous martingale and is bounded uniformly over $z$, by Fubini's theorem, the integral $(\eta_t,p)$ is also a bounded continuous martingale. To show Equation (\ref{eqn::boundary_levelines_coupling_test_variation_1}), it suffices to show that $(\eta_t,p)^2+E_t(p)$ is a martingale. Note that
\begin{equation}\label{eqn::boundary_levelines_coupling_test_variation_2}(\eta_t,p)^2+E_t(p)=\iint p(z)p(w)\left(\eta_t(z)\eta_t(w)+G_t(z,w)\right)d^2zd^2w.\end{equation}
We know that $\eta_t(z)\eta_t(w)+G_t(z,w)$ is a continuous martingale; and that $\eta_t(z),\eta_t(w)$ are bounded (uniformly over $z,w$); and that $G_t(z,w)$ is non-increasing in $t$. Thus, by Fubini's theorem again, the right-hand side of Equation (\ref{eqn::boundary_levelines_coupling_test_variation_2}) is a continuous martingale.
\smallbreak
\textit{Finally,} we will explain how to construct the coupling that satisfies the domain Markov property. Define, for $z\in\HH$,
\[\eta_{\infty}(z)=\lim_{t\to\infty}\eta_t(z).\]
The limit exists almost surely for fixed $z$ since $\eta_t(z)$ is a bounded martingale. Define, for $z,w\in\HH$ and $p\in H_s(\HH)$ which is non-negative,
\[G_{\infty}(z,w)=\lim_{t\to\infty}G_t(z,w),\quad E_{\infty}(p)=\lim_{t\to\infty}E_t(p).\]
The limits exist because that $G_t(z,w)$ and $E_t(p)$ are non-increasing functions in $t$.
Let $\tilde{h}$ equal to $\eta_{\infty}-\eta_0$ plus a sum of independent zero-boundary $\GFF$'s, one in each connected component of $\HH\setminus\partial_{\infty}$. The marginal law of $\tilde{h}$ is the same as a zero-boundary $\GFF$ and the reason is the following. For any test function $p\in H_s(\HH)$ which is non-negative, and any $\mu\ge 0$, we have that
\begin{align*}
\E[\exp(-\mu(\tilde{h},p))]&=\E[\E[\exp(-\mu(\tilde{h},p))\cond K]]\\
&=\E\left[\exp\left(-\mu(\eta_{\infty}-\eta_0,p)-\frac{\mu^2}{2}E_{\infty}(p)\right)\right]\\
&=\E\left[\exp\left(-\mu(\eta_{\infty}-\eta_0,p)+\frac{\mu^2}{2}(E_0(p)-E_{\infty}(p))\right)\right]\exp\left(-\frac{\mu^2}{2}E_0(p)\right)\\
&=\exp\left(-\frac{\mu^2}{2}E_0(p)\right),
\end{align*}
where the last equality is due to the fact that $(\eta_t,p)$ is a continuous bounded martingale with mean $(\eta_0,p)$ and quadratic variation $\la(\eta_t,p)\ra=E_0(p)-E_t(p)$.
To complete the proof, we need to explain that the coupling $(K,\tilde{h})$ satisfies the domain Markov property. In fact, for any test function $p\in H_s(\HH)$, the conditional law of $((\tilde{h}+\eta_0)|_{\HH\setminus K_{\tau}},p)$ given $K_{\tau}$ is the same as a Gaussian with mean $(\eta_{\tau},p)$ and variance $E_{\tau}(p)$.
\end{proof}

If $K$ and a $\GFF$ $h$ are coupled as in Theorem \ref{thm::boundary_levelline_gff_coupling}, we say that the path $\gamma=\cup_{0\le s\le T}\partial K_s$, where $T$ is the continuation threshold of $K$, is the \textbf{level line} of $h+\eta_0$. Generally, for any simply connected domain $D$ with two distinct boundary points $x$ and $y$ and a fixed number $u\in\R$, we say that the path $\gamma$ is the level line of a $\GFF$ $h$ with height $u$ in $D$ starting from $x$ targeted at $y$ if $\Phi(\gamma)$ is the level line of $h\circ\Phi^{-1}+u$ where $\Phi$ is any conformal map from $D$ onto $\HH$ that sends $x$ to $0$ and $y$ to $\infty$.

\subsection{Proof of Theorems \ref{thm::boundary_levelline_gff_deterministic}-\ref{thm::sle_chordal_reversibility}---non-boundary-intersecting case}\label{subsec::boundary_levellines_boundaryavoiding}
In this section, we mostly work in the horizontal strip:
\[\T:=\R\times(0,1).\]
Write the upper-boundary and the lower-boundary of $\T$ in the following way:
\[\partial_U\T:=\R\times\{1\},\quad \partial_L\T:=\R\times\{0\}.\]
\begin{lemma}\label{lem::boundary_levellines_deterministic_case1}
Suppose that $h$ is a $\GFF$ on the strip $\T$ whose boundary value is as in Figure \ref{fig::boundary_levellines_deterministic_nonintersecting_cases}(a) and let $\gamma$ be the level line of $h$ starting from $0$. If $a\ge \lambda$, then $\gamma$ almost surely accumulates at $-\infty$; if $a\le-\lambda$, then $\gamma$ almost surely accumulates at $+\infty$. In both cases, $\gamma$ almost surely does not hit $\partial_U\T$. If $a\in (-\lambda,\lambda)$, then $\gamma$ almost surely accumulates in $\partial_U\T$; and after it accumulates in $\partial_U\T$, $\gamma$ can be continued when it is targeted to $-\infty$ or $+\infty$---i.e. the continuation threshold is not hit when $\gamma$ first accumulates in $\partial_U\T$.
\end{lemma}
\begin{proof}
\textit{Case 1:} $a\ge \lambda$. Let $\psi$ be the conformal map from $\T$ onto $\HH$ which sends 0 to 0, $+\infty$ to $+1$, and $-\infty$ to $\infty$. Then $\psi(\gamma)$ has the law of $\SLE_4(\rho^R)$ with force point at $1$ where $\rho^R=a/\lambda-1\ge 0$. From Proposition \ref{prop::dubedat_lemma15}, we know that $\psi(\gamma)$ accumulates at $\infty$ without hitting the boundary.
\smallbreak
\textit{Case 2:} $a\le -\lambda$. This case can be proved similarly.
\smallbreak
\textit{Case 3:} $a\in (-\lambda,\lambda)$. We have the following two observations:
\begin{enumerate}
\item [(a)] Let $\psi_+$ be the conformal map from $\T$ onto $\HH$ that sends 0 to 0, $+\infty$ to 1, $-\infty$ to $\infty$. Then $\psi_+(\gamma)$ has the law of $\SLE_4(\rho^R)$ process with force point 1 where $\rho^R=a/\lambda-1\in (-2,0)$. From Proposition \ref{prop::dubedat_lemma15}, we know that $\psi_+(\gamma)$ accumulates in $[1,\infty)$. This implies that $\gamma$ accumulates in $\partial_U\T$ or $+\infty$ before reaches $-\infty$.
\item [(b)] Let $\psi_-$ be the conformal map from $\T$ onto $\HH$ that sends 0 to 0, $+\infty$ to $\infty$, $-\infty$ to $-1$. Then $\psi_-(\gamma)$ has the law of $\SLE_4(\rho^L)$ process with force point $-1$ where $\rho^L=-a/\lambda-1\in (-2,0)$. From Proposition \ref{prop::dubedat_lemma15}, we know that $\psi_-(\gamma)$ accumulates in $(-\infty,-1]$. This implies that $\gamma$ accumulates in $\partial_U\T$ or $-\infty$ before reaches $+\infty$.
\end{enumerate}
Combining these two facts, we know that $\gamma$ almost surely accumulates in $\partial_U\T$ before reaches $\pm\infty$.
\end{proof}
\begin{figure}[ht!]
\begin{subfigure}[b]{0.3\textwidth}
\begin{center}
\includegraphics[width=\textwidth]{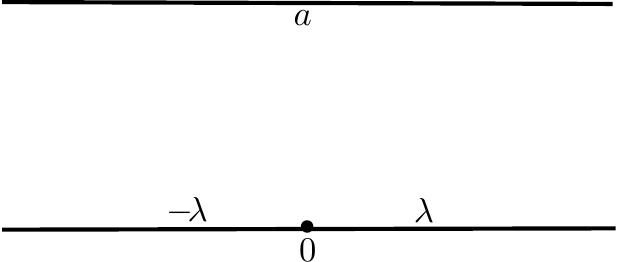}
\end{center}
\caption{The boundary values of $h$ in Lemma \ref{lem::boundary_levellines_deterministic_case1}.}
\end{subfigure}
$\quad$
\begin{subfigure}[b]{0.3\textwidth}
\begin{center}\includegraphics[width=\textwidth]{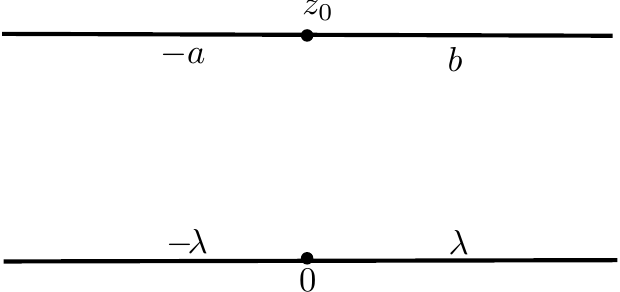}
\end{center}
\caption{The boundary values of $h$ in Lemma \ref{lem::boundary_levellines_deterministic_case2}.}
\end{subfigure}
$\quad$
\begin{subfigure}[b]{0.3\textwidth}
\begin{center}\includegraphics[width=\textwidth]{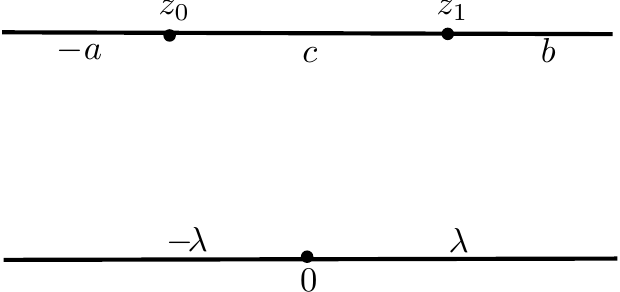}
\end{center}
\caption{The boundary values of $h$ in Lemma \ref{lem::boundary_levellines_deterministic_case3}.}
\end{subfigure}
\caption{\label{fig::boundary_levellines_deterministic_nonintersecting_cases} The boundary values of the $\GFF$ on $\T$.}
\end{figure}

\begin{remark}\label{rem::boundary_levellines_deterministic_case1}
The conclusions in Lemma \ref{lem::boundary_levellines_deterministic_case1} hold more generally when the boundary data of $h$ on $\partial_L\T$ is piecewise constant, and is at least $\lambda$ to the right of 0 and is at most $-\lambda$ to the left of 0. Furthermore, the level line $\gamma$ is almost surely a continuous curve until it first accumulates in $\partial_U\T$.
\end{remark}
\begin{proof}
From Proposition \ref{prop::dubedat_lemma15}, we know that $\gamma$ almost surely does not hit $\partial_L\T$ after time 0.
Thus, up to the first time $T$ that $\gamma$ accumulates in $\partial_U\T$, the law of $\gamma$ is absolutely continuous with respect to the law of $\SLE_4$. Therefore $\gamma$ is continuous up to $T$.
\end{proof}
\begin{lemma}\label{lem::boundary_levellines_deterministic_case2}
Suppose that $h$ is a $\GFF$ on the strip $\T$ whose boundary value is as in Figure \ref{fig::boundary_levellines_deterministic_nonintersecting_cases}(b) and let $\gamma$ be the level line of $h$ starting from $0$. If $a,b\ge\lambda$, then $\gamma$ almost surely exits $\T$ at $z_0$ without otherwise hitting $\partial_U\T$.
\end{lemma}
\begin{proof}
Let $\psi$ be the conformal map from $\T$ onto $\HH$ that sends 0 to 0, $+\infty$ to 1, and $-\infty$ to $\infty$. Put $x^{2,R}=\psi(z_0)$. Then $\psi(\gamma)$ has the law of $\SLE_4(\rho^{1,R},\rho^{2,R})$ with force points $(1,x^{2,R})$ where
\[\rho^{1,R}=b/\lambda-1\ge 0,\quad \rho^{1,R}+\rho^{2,R}=-a/\lambda-1\le -2.\]
From Proposition \ref{prop::dubedat_lemma15}, we know that $\psi(\gamma)$ almost surely accumulates at $x^{2,R}$ without otherwise hitting the boundary.
\end{proof}
\begin{remark}\label{rem::boundary_levellines_deterministic_case2}
The conclusion in Lemma \ref{lem::boundary_levellines_deterministic_case2} holds more generally when the boundary data of $h$ is piecewise constant, and is \[\begin{array}{ll}
\text{at most } -\lambda \text{ to the left of } z_0 \text{ on } \partial_U\T, & \text{at least } \lambda \text{ to the right of } z_0 \text{ on } \partial_U\T, \\
\text{at most } -\lambda \text{ to the left of } 0 \text{ on } \partial_L\T, & \text{at least } \lambda \text{ to the right of } 0 \text{ on } \partial_L\T.
\end{array}\]
Furthermore, the level line $\gamma$ is almost surely a continuous curve from 0 to $z_0$.
\end{remark}
\begin{lemma}\label{lem::boundary_levellines_deterministic_case3}
Suppose that $h$ is a $\GFF$ on the strip $\T$ whose boundary value is as in Figure \ref{fig::boundary_levellines_deterministic_nonintersecting_cases}(c) and let $\gamma$ be the level line of $h$ starting from $0$. If $a,b\ge \lambda$ and $c\in (-\lambda,\lambda)$, then $\gamma$ almost surely exits $\T$ in $[z_0,z_1]$ without otherwise hitting $\partial_U\T$.
\end{lemma}
\begin{proof}
Let $\psi$ be the conformal map from $\T$ onto $\HH$ that sends 0 to 0, $+\infty$ to 1, and $-\infty$ to $\infty$. Put
\[x^{2,R}=\psi(z_1)>0,\quad x^{3,R}=\psi(z_0)>x^{2,R}.\] Then $\psi(\gamma)$ has the law of $\SLE_4(\rho^{1,R},\rho^{2,R},\rho^{3,R})$ with force points $(1, x^{2,R},x^{3,R})$ where
\[\rho^{1,R}=b/\lambda-1\ge0,\quad \rho^{1,R}+\rho^{2,R}=c/\lambda-1\in(-2,0),\quad \rho^{1,R}+\rho^{2,R}+\rho^{3,R}=-a/\lambda-1\le -2.\]
Thus $\psi(\gamma)$ will first accumulates in $[x^{2,R},x^{3,R}]$ without hitting other boundary points by Proposition \ref{prop::dubedat_lemma15}.
\end{proof}
\begin{remark}\label{rem::boundary_levellines_deterministic_case3}
The conclusion in Lemma \ref{lem::boundary_levellines_deterministic_case3} holds more generally when the boundary data of $h$ is piecewise constant, and is \[\begin{array}{ll}
\text{at most } -\lambda \text{ to the left of } z_0 \text{ on } \partial_U\T, & \text{at least } \lambda \text{ to the right of } z_1 \text{ on } \partial_U\T, \\
\text{at most } -\lambda \text{ to the left of } 0 \text{ on } \partial_L\T, & \text{at least } \lambda \text{ to the right of } 0 \text{ on } \partial_L\T.
\end{array}\]
Furthermore, the level line $\gamma$ is almost surely a continuous curve until it first accumulates in $[z_0,z_1]$.
\end{remark}

\begin{proposition}\label{prop::boundary_levellines_deterministic_nonintersecting_coincide}
Suppose that $h$ is a $\GFF$ on the strip $\T$ whose boundary value is as in Figure \ref{fig::boundary_levellines_deterministic_nonintersecting_coincide}(a). Let $\gamma$ be the level line of $h$ starting from 0 and $\gamma'$ be the level line of $-h$ starting from $z_0$, and assume that the triple $(h,\gamma,\gamma')$ are coupled so that $\gamma$ and $\gamma'$ are conditionally independent given $h$. Then almost surely $\gamma$ and $\gamma'$ (viewed as sets) are equal. In particular, the level line $\gamma$ is almost surely determined by $h$.
\end{proposition}

\begin{figure}[ht!]
\begin{subfigure}[b]{0.3\textwidth}
\begin{center}
\includegraphics[width=\textwidth]{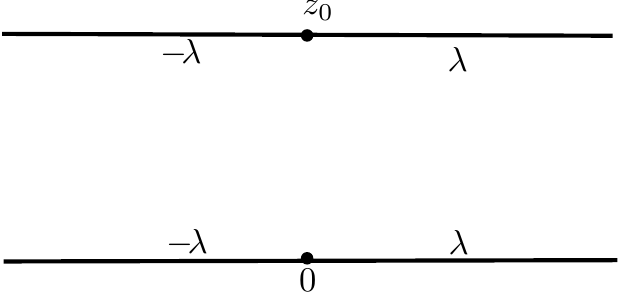}
\end{center}
\caption{}
\end{subfigure}
$\quad$
\begin{subfigure}[b]{0.3\textwidth}
\begin{center}\includegraphics[width=\textwidth]{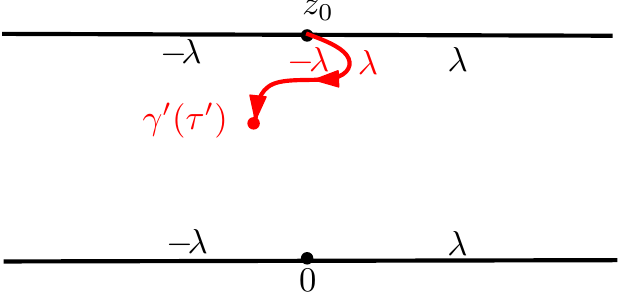}
\end{center}
\caption{}
\end{subfigure}
$\quad$
\begin{subfigure}[b]{0.3\textwidth}
\begin{center}\includegraphics[width=\textwidth]{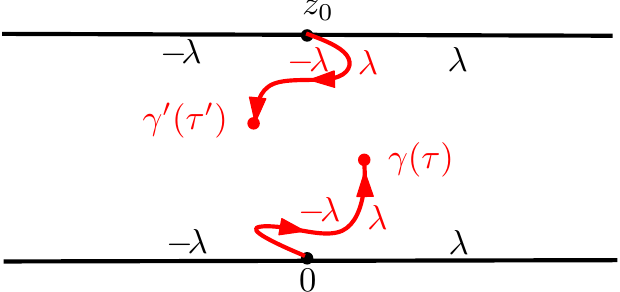}
\end{center}
\caption{}
\end{subfigure}
\caption{\label{fig::boundary_levellines_deterministic_nonintersecting_coincide} The boundary values of the fields in the proof of Lemma \ref{lem::boundary_levellines_deterministic_nonintersecting_coincide}.}
\end{figure}

To prove Proposition \ref{prop::boundary_levellines_deterministic_nonintersecting_coincide}, we first prove the following lemma.

\begin{lemma}\label{lem::boundary_levellines_deterministic_nonintersecting_coincide}
Suppose the same assumption as in Proposition \ref{prop::boundary_levellines_deterministic_nonintersecting_coincide}. Let $\tau'$ be any $\gamma'$-stopping time. Then, given $\gamma'[0,\tau']$, the level line $\gamma$ almost surely first exits $\T\setminus\gamma'[0,\tau']$ at $\gamma'(\tau')$.
\end{lemma}
\begin{proof}
Given $\gamma'[0,\tau']$, denote by $\tilde{h}$ the restriction of $h$ to $\T\setminus\gamma'[0,\tau']$. Since $\gamma'$ and $-h$ are coupled as in Theorem \ref{thm::boundary_levelline_gff_coupling}, we know that the conditional law of $\tilde{h}$ given $\gamma'[0,\tau']$ is the same as a $\GFF$ on $\T\setminus\gamma'[0,\tau']$ with boundary data as in Figure \ref{fig::boundary_levellines_deterministic_nonintersecting_coincide}(b).

We argue that, given $\gamma'[0,\tau']$, the path $\gamma$ and the field $\tilde{h}$ are coupled so that $\gamma$ is the level line of $\tilde{h}$. Assuming this is true, then, from Lemma \ref{lem::boundary_levellines_deterministic_case2}, we know that $\gamma$ almost surely exits $\T\setminus\gamma'[0,\tau']$ at $\gamma'(\tau')$ which implies the conclusion. Thus we only need to show that, given $\gamma'[0,\tau']$, the path $\gamma$ and the field $\tilde{h}$ are coupled so that $\gamma$ is the level line of $\tilde{h}$.

Suppose $\tau$ is any $\gamma$-stopping time. We know that $\gamma[0,\tau]$ is a local set for $h$, that $\gamma'[0,\tau']$ is a local set for $h$, and that $\gamma[0,\tau]$ and $\gamma'[0,\tau']$ are conditionally independent given $h$. From Proposition \ref{prop::gff_localsets_union}, we know that the union $\gamma[0,\tau]\cup\gamma'[0,\tau']$ is also a local set for $h$; furthermore, given $\gamma[0,\tau]$ and $\gamma'[0,\tau']$, and on the event $\{\gamma[0,\tau]\cap\gamma'[0,\tau']=\emptyset\}$, the conditional law of $h|_{\T\setminus (\gamma[0,\tau]\cup\gamma'[0,\tau'])}=\tilde{h}|_{(\T\setminus\gamma'[0,\tau'])\setminus\gamma[0,\tau]}$ is the same as a $\GFF$ in $\T\setminus(\gamma[0,\tau]\cup\gamma'[0,\tau'])$ with boundary data as in Figure \ref{fig::boundary_levellines_deterministic_nonintersecting_coincide}(c). This implies that, given $\gamma'[0,\tau']$, the path $\gamma$ and the field $\tilde{h}$ are coupled so that $\gamma$ is the level line of $\tilde{h}$ up until the first time that $\gamma$ hits $\gamma'[0,\tau']$. This completes the proof.
\end{proof}

\begin{proof}[Proof of Proposition \ref{prop::boundary_levellines_deterministic_nonintersecting_coincide}]
From Lemma \ref{lem::boundary_levellines_deterministic_nonintersecting_coincide}, we have almost surely that $\gamma$ hits $\gamma'[0,\tau']$ for the first time at $\gamma'(\tau')$. Since this holds for any $\gamma'$-stopping time $\tau'$, we know that $\gamma$ hits a dense countable set of points along $\gamma'$ (in reverse chronological order). By symmetry, $\gamma'$ hits a dense countable set of points along $\gamma$. Since both $\gamma$ and $\gamma'$ are continuous simple curves, the two paths (viewed as sets) are equal.
\end{proof}

\begin{remark}\label{rem::boundary_levellines_deterministic_nonintersecting_coincide}
The conclusions in Lemma \ref{lem::boundary_levellines_deterministic_nonintersecting_coincide} and Proposition \ref{prop::boundary_levellines_deterministic_nonintersecting_coincide} hold more generally when the boundary data of $h$ is piecewise constant, and is
\[\begin{array}{ll}
\text{at most } -\lambda \text{ to the left of } z_0 \text{ on } \partial_U\T, & \text{at least } \lambda \text{ to the right of } z_0 \text{ on } \partial_U\T, \\
\text{at most } -\lambda \text{ to the left of } 0 \text{ on } \partial_L\T, & \text{at least } \lambda \text{ to the right of } 0 \text{ on } \partial_L\T.
\end{array}\]
\end{remark}

From Proposition \ref{prop::boundary_levellines_deterministic_nonintersecting_coincide} and Remark \ref{rem::boundary_levellines_deterministic_nonintersecting_coincide}, we finish the proof of Theorems \ref{thm::boundary_levelline_gff_deterministic} to \ref{thm::sle_chordal_reversibility} for the case that the level lines are non-boundary-intersecting. We record these results in the following proposition.

\begin{proposition}\label{prop::boundary_levellines_nonintersecting_all}
Suppose that $h$ is a $\GFF$ on $\HH$ whose boundary value is piecewise constant, and is
\[\begin{array}{ll}
\text{at most } -\lambda \text{ to the left of } 0 \text{ on } \partial\HH, & \text{at least } \lambda \text{ to the right of } 0 \text{ on } \partial\HH.
\end{array}\]
Let $\gamma$ be the level line of $h$ starting from 0 targeted at $\infty$ and $\gamma'$ be the level line of $-h$ starting from $\infty$ targeted at 0. Then we have the following conclusions.
\begin{enumerate}
\item [(1)] The level line $\gamma$ is almost surely determined by $h$.
\item [(2)] The level line $\gamma$ is almost surely continuous and transient.
\item [(3)] The level lines $\gamma'$ and $\gamma$ are equal.
\end{enumerate}
\end{proposition}
\begin{proof} We only need to explain the transience. Let $\tau'$ be any $\gamma'$-stopping time that is positive and finite. From Remark \ref{rem::boundary_levellines_deterministic_nonintersecting_coincide}, we know that, given $\gamma'[0,\tau']$, the level line $\gamma$ first exits $\HH\setminus\gamma'[0,\tau']$ at $\gamma'(\tau')$ and then merges with $\gamma'$ afterwards. Therefore $\gamma$ is transient by the continuity of $\gamma'$ at time $0$.
\end{proof}

In this section, we will consider the relation between two level lines of the same $\GFF$. Suppose that $h$ is a $\GFF$, for any $u\in\R$, we define the level line of $h$ with height $u$ to be the level line of $h+u$. We will show that the level lines of $h$ enjoy the same monotonicity property as if $h$ were a smooth function. Namely, if $u_1<u_2$ and $\gamma_{u_i}$ is the level line of $h$ with height $u_i$ for $i=1,2$. Then almost surely $\gamma_{u_1}$ lies to the right of $\gamma_{u_2}$.

\begin{proposition}\label{prop::boundary_levellines_nonintersecting_monotonicity_reverse}
Suppose that $h$ is a $\GFF$ on $\T$ with boundary data as in Figure \ref{fig::boundary_levellines_nonintersecting_monotonicity_reverse}(a). Assume that $a,b,a',b'\ge\lambda$ and let $\gamma'$ be the level line of $-h$ starting from $z_0$. Fix $u$ such that
\begin{equation}\label{eqn::boundary_levellines_nonintersecting_monotonicity_reverse}
\lambda-b\le u\le a-\lambda,
\end{equation}
and let $\gamma_u$ be the level line of $h$ with height $u$ starting from 0 and stopped at the first time that it accumulates in $\partial_U\T$. If $u>0$, then $\gamma_u$ almost surely passes to the left of $\gamma'$; if $u<0$, then $\gamma_u$ almost surely passes to the right of $\gamma'$.
\end{proposition}

To prove the proposition, we need the following lemma.
\begin{lemma}\label{lem::levelline_forbidden_intervals}
Suppose that $h$ is a $\GFF$ on $\T$ whose boundary data is as in Figure \ref{fig::levelline_forbidden_intervals}. Let $\gamma$ be the level line of $h$ starting from 0. If $(\gamma(t), 0\le t\le T_0)$ is almost surely continuous for some $\gamma$-stopping time $0<T_0<\infty$, then $[\gamma[0,T_0]\cap J=\emptyset]$ almost surely.
\end{lemma}
\begin{proof}
This is a direct consequence of Proposition \ref{prop::sle_chordal_forbidden_intervals}.
\end{proof}
\begin{figure}[ht!]
\begin{center}
\includegraphics[width=0.3\textwidth]{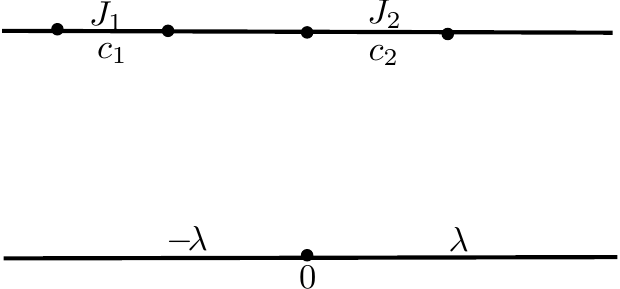}
\end{center}
\caption{\label{fig::levelline_forbidden_intervals} Suppose that $h$ is a $\GFF$ on $\T$ and let $J\subset\partial_U\T$ be open. Write $J=\cup_k J_k$ where the $J_k$ are disjoint open intervals and assume that $h|_{J_k}\equiv c_k$ for given constant $c_k\not\in (-\lambda,\lambda)$.}
\end{figure}

\begin{figure}[ht!]
\begin{subfigure}[b]{0.3\textwidth}
\begin{center}
\includegraphics[width=\textwidth]{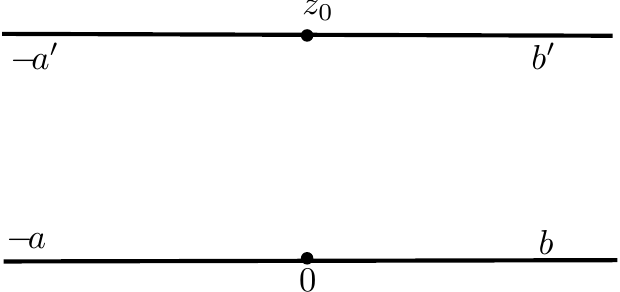}
\end{center}
\caption{The boundary value of the field $h$ in the horizontal strip $\T$.}
\end{subfigure}
$\quad$
\begin{subfigure}[b]{0.3\textwidth}
\begin{center}\includegraphics[width=\textwidth]{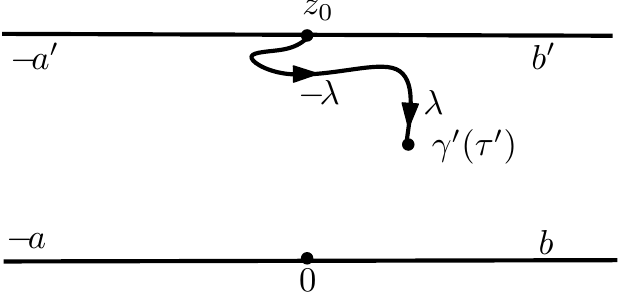}
\end{center}
\caption{The boundary value of $h$ restricted to $\T\setminus\gamma'[0,\tau']$.}
\end{subfigure}
$\quad$
\begin{subfigure}[b]{0.3\textwidth}
\begin{center}\includegraphics[width=\textwidth]{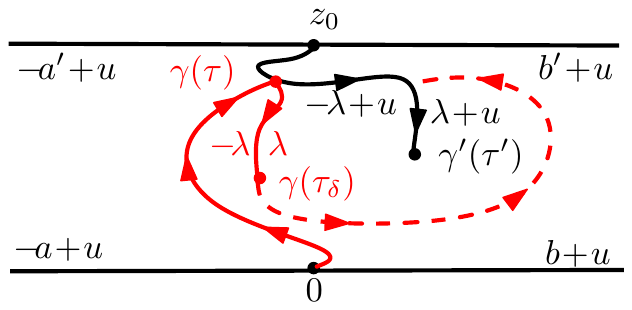}
\end{center}
\caption{The boundary value of $h\!\!+\!\!u$ restricted to $\T\!\setminus\!(\!\gamma'[\!0,\tau'\!]\!\cup\!\gamma[\!0,\tau_{\delta}\!]\!)$.}
\end{subfigure}
\caption{\label{fig::boundary_levellines_nonintersecting_monotonicity_reverse}The boundary values of the fields in the proof of Proposition \ref{prop::boundary_levellines_nonintersecting_monotonicity_reverse}. }
\end{figure}

\begin{proof}[Proof of Proposition \ref{prop::boundary_levellines_nonintersecting_monotonicity_reverse}]
We only need to show the result for $u>0$. Let $\gamma'$ be the level line of $-h$ starting from $z_0$, and $\gamma_u$ be the level line of $h$ with height $u$ starting from 0. The hypothesis that $a,b,a',b'\ge \lambda$ implies that $\gamma'$ almost surely first exits $\T$ at 0 without hitting any other boundary point (except $0$ and $z_0$) and $\gamma'$ is almost surely continuous. The hypothesis in Equation (\ref{eqn::boundary_levellines_nonintersecting_monotonicity_reverse}) implies that $\gamma_u$ almost surely accumulates in $\partial_U\T$ or tends to $\pm\infty$ before hitting $\partial_L\T$ after time 0 and $\gamma_u$ is almost surely continuous up to the first time that it hits $\partial_U\T$. Let $\tau'$ be any $\gamma'$-stopping time before it hits 0. Let $\tilde{h}$ be $h$ restricted to $\T\setminus\gamma'[0,\tau']$. Then, given $\gamma'[0,\tau']$, the conditional law of $\tilde{h}$ is the same as a $\GFF$ with boundary data as in Figure \ref{fig::boundary_levellines_nonintersecting_monotonicity_reverse}(b). Furthermore, $\gamma_u$ is the level line of $\tilde{h}+u$.

We say the left (resp. right) side of $\partial_U\T\cup\gamma'[0,\tau']$ the union of the left (resp. right) side of $\gamma'[0,\tau']$ and the part of $\partial_U\T$ that is to the left (resp. right) of $z_0$. From Lemma \ref{lem::boundary_levellines_deterministic_case3} and Remark \ref{rem::boundary_levellines_deterministic_case3}, we know that $\gamma_u$ almost surely exits $\T\setminus\gamma'[0,\tau']$ on the left side of $\partial_U\T\cup\gamma'[0,\tau']$, say at time $\tau$, or does not hit $\partial_U\T$.

We will argue that $\gamma'(\tau')$ is almost surely to the right of $\gamma_u$. If $\gamma_u(\tau)$ is in the part of $\partial_U\T$ that is to the left of $z_0$, then we are done. If this does not hold, then $\gamma'(\tau')$ is to the left of $\gamma_u$, and therefore $\gamma_u$ hits the left side of $\gamma'[0,\tau']$ at time $\tau$ and, after time $\tau$, the path $\gamma_u$ wraps around $\gamma'(\tau')$ and then hits the right side of $\gamma'[0,\tau']$. See Figure \ref{fig::boundary_levellines_nonintersecting_monotonicity_reverse}(c). Let $\tau_{\delta}$ be the first time after $\tau$ that $\gamma_u(t)$ is in the right connected component of $\T\setminus(\gamma'[0,\tau']\cup\gamma_u[0,\tau])$ and $\dist(\gamma_u(t),\gamma'[0,\tau'])\ge \delta$ (set $\tau_{\delta}=\infty$ if this never happens). Then for $\delta$ small enough, the probability of the event $[\tau_{\delta}<\infty]$ is positive. Given $\gamma'[0,\tau']\cup\gamma_u[0,\tau_{\delta}]$, the conditional law of $h+u$ restricted to the right connected component of $\T\setminus(\gamma'[0,\tau']\cup\gamma_u[0,\tau_{\delta}])$ is the same as a $\GFF$ with boundary value as in Figure \ref{fig::boundary_levellines_nonintersecting_monotonicity_reverse}(c). Furthermore, $(\gamma_u(t),t\ge \tau_{\delta})$ is the level line of this field. From Lemma \ref{lem::levelline_forbidden_intervals}, $(\gamma_u(t),t\ge \tau_{\delta})$ will never hits the right side of $\gamma'[0,\tau']$, contradiction.
\end{proof}

\begin{proposition}\label{prop::boundary_levellines_nonintersecting_monotonicity}
Suppose that $h$ is a $\GFF$ on $\T$ with boundary data as in Figure \ref{fig::boundary_levellines_nonintersecting_monotonicity_reverse}(a). Assume that $a,b\ge \lambda$. Fix $u_1,u_2$ such that
\[\lambda-b\le u_1<u_2\le a-\lambda.\]
For $i=1,2$, let $\gamma_{u_i}$ be the level line of $h$ with height $u_i$ starting from 0 and let $\tau_i$ be the first time that $\gamma_{u_i}$ accumulates in $\partial_U\T$. Then, almost surely, $\gamma_{u_2}[0,\tau_2]$ lies to the left of $\gamma_{u_1}[0,\tau_1]$. We emphasize that there is no restriction for the boundary data of $h$ on $\partial_U\T$.
\end{proposition}
\begin{proof}
We first assume that $a'\ge \lambda+u_1$ and $b'\ge \lambda-u_1$. By replacing $h$ with $h+u_1$, we may assume that $u_1=0$. Let $\gamma'$ be the level line of $-h$ starting from $z_0$. From Proposition \ref{prop::boundary_levellines_nonintersecting_monotonicity_reverse}, we know that $\gamma_{u_2}$ almost surely stays to the left of $\gamma'$. We also know that the range of $\gamma'$ is the same as the range of $\gamma_{u_1}$. These imply the conclusion.

Now we treat the case when the boundary data of $h$ on $\partial_U\T$ is general. Fix $\eps>0$, for $i=1,2$, let $\tau_i^{\eps}$ be the first time $t$ that $\gamma_{u_i}$ gets within distance $\eps$ of $\partial_U\T$. It suffices to show that $\gamma_{u_2}[0,\tau_2^{\eps}]$ almost surely lies to the left of $\gamma_{u_1}[0,\tau_1^{\eps}]$ for every $\eps>0$.

Let $\tilde{h}$ be a $\GFF$ on $\T$ whose boundary data is the same as $h$ on $\partial_L\T$ and, on $\partial_U\T$, is at most $-\lambda-u_1$ to the left of $z_0$ and is at least $\lambda-u_1$ to the right of $z_0$. For $i=1,2$, let $\tilde{\gamma}_{u_i}$ be the level line of $\tilde{h}$ with height $u_i$ starting from 0. On the one hand, from the above analysis, $\tilde{\gamma}_{u_2}$ almost surely lies to the left of $\tilde{\gamma}_{u_1}$. On the other hand, the laws of $\tilde{h}|_{\R\times (0,1-\eps)}$ and $h|_{\R\times (0,1-\eps)}$ are mutually absolutely continuous, see Proposition \ref{prop::gff_absolutecontinuity}. Combining these two facts, we know that $\gamma_{u_2}[0,\tau_2^{\eps}]$ almost surely lies to the left of $\gamma_{u_1}[0,\tau_1^{\eps}]$.
\end{proof}

\begin{corollary}\label{cor::boundary_levellines_nonintersecting_monotonicity}
Suppose that $h$ is a $\GFF$ on $\HH$ whose boundary data is $b$ on $\R_+$ and $-a$ on $\R_-$. Assume that $a,b\ge \lambda$. Fix $u_1,u_2$ such that
\[\lambda-b\le u_1<u_2\le a-\lambda.\]
For $i=1,2$, let $\gamma_{u_i}$ be the level line of $h$ with height $u_i$ starting from 0. Then almost surely $\gamma_{u_2}$ lies to the left of $\gamma_{u_1}$.
\end{corollary}

Suppose that $h$ is a $\GFF$ on $\HH$ whose boundary value is as in Figure \ref{fig::boundary_levellines_nonintersecting_interacting}(a). For each $u\in\R$, let $\gamma_u$ be the level line of $h$ with height $u$ starting from 0. Fix $u_1<u_2$ and assume that $a,b$ are large enough so that Corollary \ref{cor::boundary_levellines_nonintersecting_monotonicity} is applicable to $\gamma_{u_1}$ and $\gamma_{u_2}$. We know from Corollary \ref{cor::boundary_levellines_nonintersecting_monotonicity} that $\gamma_{u_1}$ almost surely lies to the right of $\gamma_{u_2}$. The purpose of the rest of this section is to calculate the conditional mean of $h$ given both $\gamma_{u_1}$ and $\gamma_{u_2}$, and to show that the Loewner driving function of $\gamma_{u_1}$, viewed as a path in the right connected component of $\HH\setminus\gamma_{u_2}$, exists and is continuous, and likewise when the roles of $\gamma_{u_1}$ and $\gamma_{u_2}$ are swapped. We emphasize that, in this section, the results will be for paths which do not intersect the boundary where we have the almost sure continuity of the level lines at this point.

\begin{lemma}\label{lem::boundary_levellines_locatsets}
Suppose that $\gamma_1,...,\gamma_k$ are continuous paths such that, for each $1\le i\le k$, we have that
\begin{enumerate}
\item [(1)] $\gamma_i[0,\tau]$ is a local set for $h$ for any $\gamma_i$-stopping time $\tau$;
\item [(2)] $\gamma_i$ is almost surely determined by $h$.
\end{enumerate}
Suppose that $\tau_1$ is a stopping time for $\gamma_1$ and, for each $2\le j\le k$, inductively, let $\tau_j$ be a stopping time for the filtration $\LF^j_t$ generated by $\gamma_1|_{[0,\tau_1]},...,\gamma_{j-1}|_{[0,\tau_{j-1}]}$ and $\gamma_j|_{[0,t]}$. Then $\cup_{1\le i\le k} \gamma_i[0,\tau_i]$ is a local set for $h$ and is almost surely determined by $h$.
\end{lemma}
\begin{proof}
For $1\le j\le k$, set $A_j=\cup_{1\le i\le j}\gamma_i[0,\tau_i]$. Fix $U\subseteq\HH$ open. We are going to prove that the event $[A_j\cap U=\emptyset]$ is almost surely determined by $h|_{U^c}$ and that, on the event $[A_j\cap U=\emptyset]$, the set $A_j$ is almost surely determined by $h|_{U^c}$. We will prove this by induction on the number of the paths.
The hypotheses for $\gamma_1$ imply that this is true for $j=1$.
Suppose the result holds for $j-1$ paths for $j\ge 2$ fixed. We will show that it holds for $j$ paths.

Let $\tau_j^U$ be the first time that $\gamma_j$ hits $\overline{U}$. The hypotheses of $\gamma_j$ imply that $\gamma_j[0,\tau_j^U]$ is almost surely determined by $h|_{U^c}$. Note that
\begin{enumerate}
\item [(a)] $\{A_j\cap U=\emptyset\}=\{A_{j-1}\cap U=\emptyset\}\cap\{\tau_j\le \tau_j^U\}$;
\item [(b)] the event $\{\tau_j\le \tau_j^U\}$ is almost surely determined by $A_{j-1}$ and $h|_{U^c}$ (since $\tau_j$ is a $\LF^j_t$-stopping time);
\item [(c)] the event $\{A_{j-1}\cap U=\emptyset\}$ is determined by $h|_{U^c}$, and on the event $\{A_{j-1}\cap U=\emptyset\}$, the set $A_{j-1}$ is almost surely determined by $h|_{U^c}$.
\end{enumerate}
Combining these three facts, we have that, on the event $\{A_{j-1}\cap U=\emptyset\}$, the event $\{\tau_j\le \tau_j^U\}$ is almost surely determined by $h|_{U^c}$. Therefore, the event $\{A_j\cap U=\emptyset\}$ is almost surely determined by $h|_{U^c}$; moreover, on the event $\{A_j\cap U=\emptyset\}$, since $A_{j-1}$ and $\gamma_j[0,\tau_j]$ are almost surely determined by $h|_{U^c}$, $A_j$ is also almost surely determined by $h|_{U^c}$. This completes the proof of the induction step.
\end{proof}

In the rest of this section, we set
\[A(t)=\gamma_{u_2}\cup\gamma_{u_1}[0,t],\quad \LF_t=\sigma(\gamma_{u_2},\gamma_{u_1}|_{[0,t]}).\]
\begin{proposition}\label{prop::boundary_levellines_nonintersecting_conditionalmean}
Suppose that $h$ is a $\GFF$ on $\HH$ whose boundary value is as in Figure \ref{fig::boundary_levellines_nonintersecting_interacting}(a). Fix an $\LF_t$-stopping time $\tau$. Let $C$ be any connected component of $\HH\setminus A(\tau)$. Then, given $A(\tau)$, the conditional law of $h|_C$ is the same as a $\GFF$ with mean $\eta_C$ which is harmonic in $C$ with certain boundary value that will be described in the following. There are three types of $C$ and we will describe the boundary value of $\eta_C$ one by one, see Figure \ref{fig::boundary_levellines_nonintersecting_interacting}(b).
\begin{enumerate}
\item [(1)] $C$ is the connected component that stays to the left of $\gamma_{u_2}$. Then $\eta_C$ is $-a$ on $\R_-$ and $-\lambda-u_2$ to the left of $\gamma_{u_2}$.
\item [(2)] $C$ is any connected component between $\gamma_{u_2}$ and $\gamma_{u_1}[0,\tau]$. Then $\eta_C$ is $\lambda-u_2$ to the right of $\gamma_{u_2}$ and $-\lambda-u_1$ to the left of $\gamma_{u_1}$.
\item [(3)] $C$ is the connected component whose boundary contains $\R_+$. Then $\eta_C$ is $\lambda-u_2$ to the right of $\gamma_{u_2}$, $-\lambda-u_1$ to the left of $\gamma_{u_1}$, $\lambda-u_1$ to the right of $\gamma_{u_1}$, and $b$ on $\R_+$.
\end{enumerate}
\end{proposition}

\begin{figure}[ht!]
\begin{subfigure}[b]{0.3\textwidth}
\begin{center}
\includegraphics[width=\textwidth]{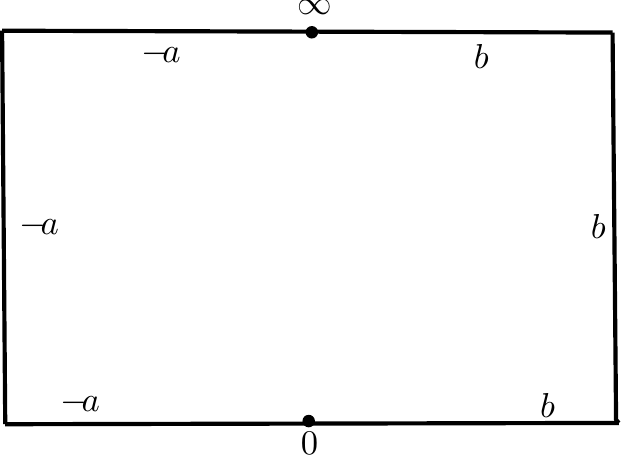}
\end{center}
\caption{The boundary value for the $\GFF$ $h$ in $\HH$.}
\end{subfigure}
$\quad$
\begin{subfigure}[b]{0.3\textwidth}
\begin{center}\includegraphics[width=\textwidth]{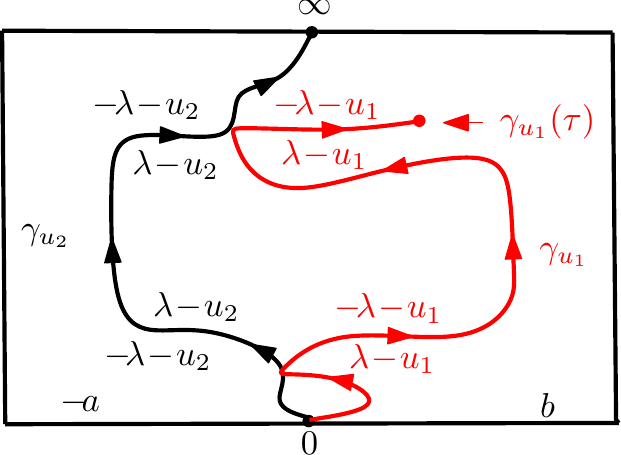}
\end{center}
\caption{The boundary value for connected components of $\HH\setminus\! A\!(\!\tau\!)$.}
\end{subfigure}
$\quad$
\begin{subfigure}[b]{0.3\textwidth}
\begin{center}\includegraphics[width=\textwidth]{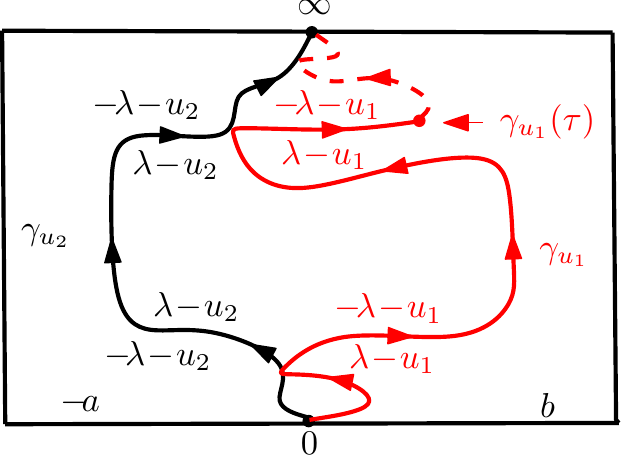}
\end{center}
\caption{The dashed red path is \\$A\setminus A(\tau)$.}
\end{subfigure}
\caption{\label{fig::boundary_levellines_nonintersecting_interacting}The boundary values of the fields in the proof of Proposition \ref{prop::boundary_levellines_nonintersecting_conditionalmean}.}
\end{figure}

\begin{proof}
From Lemma \ref{lem::boundary_levellines_locatsets}, we know that $A(\tau)$ is a local set for $h$. Since $\gamma_{u_1}$ and $\gamma_{u_2}$ are continuous, the connected components of $\gamma_{u_1}\setminus\gamma_{u_2}$ and of $\gamma_{u_2}\setminus\gamma_{u_1}$ consist of more than a single point. For $i=1,2$, let $A_i$ be the range of $\gamma_{u_i}$, and set $A=A_1\cup A_2$.
\smallbreak
\textit{First,} assume that $\tau=\infty$ and $C$ is the connected component to the left of $\gamma_{u_2}$ or the connected component to the right of $\gamma_{u_1}$. We only need to explain the result for $C$ that is the connected component to the left of $\gamma_{u_2}$. Note that $A$ and $A_2$ are local sets determined by $h$. Apply Proposition \ref{prop::gff_localsets_interacting} to $A$ and $A_2$, we have that $\LC_A|_C=\LC_{A_2}|_C$ almost surely given $A_2$. Therefore, given $A$, the conditional mean $\LC_A|_C$ agrees with $\eta_C$ almost surely.
\smallbreak
\textit{Second,} assume that $\tau=\infty$ and $C$ is any connected component between $\gamma_{u_1}$ and $\gamma_{u_2}$. Then $\partial C$ has two special points, say $x_0$ and $y_0$, which are contained in $A_1\cap A_2$. For any fixed point $z\in\partial C\cap A_2$ (resp. $z\in \partial C\cap A_1$) other than $x_0,y_0$, from Proposition \ref{prop::gff_localsets_boundarybehavior}, we know that $\LC_A-\LC_{A_2}$ (resp. $\LC_A-\LC_{A_1}$) tends to zero along any sequence in $C$ which converges to $z$. Thus $\LC_A|_C$ agrees with $\eta_C$ on $\partial C\setminus\{x_0,y_0\}$. Then we need to show that $\LC_A|_C$ also agrees with $\eta_C$ at $x_0$ and $y_0$.

Assume $x_0=\gamma_{u_1}(s_0)$ and $y_0=\gamma_{u_1}(t_0)$ with $s_0<t_0$. Proposition \ref{prop::gff_localsets_bm} implies that $\LC_{A(t)}$ has a continuous modification in $t$ since that $A(\sigma)$ is local for any $\LF_t$-stopping time $\sigma$ and that $\gamma_{u_1}$ is continuous. The continuity of $\LC_{A(t)}$ in $t$ implies that $\LC_A$ has the same boundary behavior as $\eta_C$ near $y_0$ since that the boundary data of $\LC_{A(t)}$ agrees with $\eta_C$ as $t\uparrow t_0$. This leaves us to deal with the boundary behavior near $x_0$.

Let $\gamma'$ be the level line of $-h$ with height $-u_1$ starting from $\infty$, and let $A_1'$ be the range of $\gamma'$. Then almost surely $A_1'=A_1$. Apply Proposition \ref{prop::gff_localsets_interacting} to the sets $A$ and $A_1'\cup A_2$, we have that $\LC_{A_1'\cup A_2}|_C=\LC_A|_C$ almost surely given $A$. An analogous continuity argument implies that $\LC_{A_1'\cup A_2}|_C$ has the same boundary behavior as $\eta_C$ near $x_0$. Consequently, $\LC_A|_C$ also has the same boundary behavior as $\eta_C$ near $x_0$.
\smallbreak
\textit{Third,} assume that $\tau<\infty$ and $C$ is the connected component to the left of $\gamma_{u_2}$. Apply Proposition \ref{prop::gff_localsets_interacting} to $A_2$ and $A(\tau)$, which are local and are determined by $h$, we have that $\LC_{A(\tau)}|_C=\LC_{A_2}|_C$ given $A_2$. This implies that $\LC_{A(\tau)}|_C$ agrees with $\eta_C$.
\smallbreak
\textit{Fourth,} assume that $\tau<\infty$ and $C$ is any connected component between $\gamma_{u_2}$ and $\gamma_{u_1}[0,\tau]$. Apply Proposition \ref{prop::gff_localsets_interacting} to $A(\tau)$ and $A$, we have that $\LC_{A(\tau)}|_C=\LC_A|_C$ given $A(\tau)$. Since $\LC_A|_C$ agrees with $\eta_C$, $\LC_{A(\tau)}|_C$ also agrees with $\eta_C$.
\smallbreak
\textit{Finally,} assume that $\tau<\infty$ and $C$ is the connected component whose boundary contains $\R_+$, see Figure \ref{fig::boundary_levellines_nonintersecting_interacting}(c). Take a point $z$ on $\partial C$. If $z$ is at positive distance from $A(\tau)\setminus A_2$, by applying Proposition \ref{prop::gff_localsets_boundarybehavior} to the sets $A_2$ and $A(\tau)$, we know that $\LC_{A(\tau)}-\LC_{A_2}$ tends to zero along any sequence in $C$ that converges to $z$, thus $\LC_{A(\tau)}|_C$ agrees with $\eta_C$ at $z$. If $z$ is at positive distance from $A\setminus A(\tau)$, by applying Proposition \ref{prop::gff_localsets_boundarybehavior} to the sets $A$ and $A(\tau)$, we know that $\LC_{A(\tau)}-\LC_{A}$ tends to zero along any sequence in $C$ that converges to $z$, thus $\LC_{A(\tau)}|_C$ agrees with $\eta_C$ at $z$. If $z$ is not in the previous two cases, then a similar continuity argument as above will imply that $\LC_{A(\tau)}|_C$ agrees with $\eta_C$ at $z$.
\end{proof}

\begin{figure}[ht!]
\begin{center}
\includegraphics[width=0.3\textwidth]{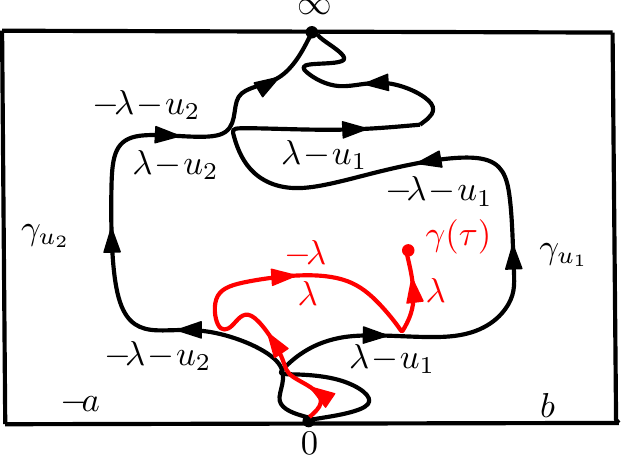}
\end{center}
\caption{\label{fig::boundary_levellines_nonintersecting_interacting_three} The boundary values of the fields in Remark \ref{rem::boundary_levellines_nonintersecting_threelevellines}.}
\end{figure}

\begin{remark}\label{rem::boundary_levellines_nonintersecting_threelevellines}[Three level lines]
Suppose that $u_1<0<u_2$ and $\gamma=\gamma_0$. A statement analogous to Proposition \ref{prop::boundary_levellines_nonintersecting_conditionalmean} also holds for the conditional mean of $h$ given $\gamma_{u_1}$, $\gamma_{u_2}$ and $\gamma[0,\tau]$ where $\tau$ is any stopping time for the filtration $\LF_t=\sigma(\gamma_{u_1},\gamma_{u_2},\gamma|_{[0,t]})$. The boundary data are depicted in Figure \ref{fig::boundary_levellines_nonintersecting_interacting_three}.
\end{remark}

The rest of this section is to establish the existence and continuity of the Loewner driving function for $\gamma_{u_1}$ viewed as a path in the right connected component of $\HH\setminus\gamma_{u_2}$. We will use Proposition \ref{prop::loewner_chordal_driving_existence_continuity}.

\begin{proposition}\label{prop::boundary_levellines_conditionallaw_loewner_continuity}
Let $\psi$ be a conformal map from the right connected component of $\HH\setminus\gamma_{u_2}$ onto $\HH$ with $\psi(0)=0$ and $\psi(\infty)=\infty$. Then $\psi(\gamma_{u_1})$ has a continuous Loewner driving function as a path in $\overline{\HH}$ from $0$ to $\infty$.
\end{proposition}
\begin{proof}
Since $\gamma_{u_2}$ is continuous, the right connected component of $\HH\setminus\gamma_{u_2}$, denoted by $C$, is almost surely a Jordan domain. Thus $\psi$ extends as a homeomorphism from $\overline{C}$ onto $\overline{\HH}$, and $\psi(\gamma_{u_1})$ is almost surely a continuous curve in $\overline{\HH}$ from 0 to $\infty$. We will check the two criteria in Proposition \ref{prop::loewner_chordal_driving_existence_continuity}.
\smallbreak
\textit{Proof of Condition (1).} The only way this could fail is if the following occurs. After intersecting $\gamma_{u_2}$, say at time $t_0$, $\gamma_{u_1}$ enters a bounded connected component of $\HH\setminus(\gamma_{u_2}\cup\gamma_{u_1}[0,t_0])$, denoted by $C_0$. Since $\gamma_{u_1}$ lies to the right of $\gamma_{u_2}$, this would force $\gamma_{u_1}$ to have a self-intersecting upon exiting $C_0$. This contradicts with the fact that $\gamma_{u_1}$ is simple.
\smallbreak
\textit{Proof of Condition (2).} It suffices to show that the set $I$ of times $t\in [0,\infty)$ such that $\gamma_{u_1}(t)$ is contained in the range of $\gamma_{u_2}$ is nowhere dense in $[0,\infty)$ almost surely. Since $I$ is closed, we only need to show that the event $E=[I \text{contains an open interval}]$ has probability zero. We prove by contradiction. Suppose that $\PP[E]>0$. Let $\LF=\sigma(\gamma_{u_1},\gamma_{u_2})$, and let $T_0$ be an $\LF$-measurable random variable taking values in $[0,\infty)$ such that, on the event $E$, $T_0$ is almost surely contained in an open interval $I_0$ of $I$. On the event $E$, since $\gamma_{u_1}$ and $\gamma_{u_2}$ are simple, we can find a sequence $(z_k)$ in the connected component that is to the left of $\gamma_{u_2}$ converging to $\gamma_{u_1}(T_0)$. Since that $\gamma_{u_1}(I_0)$ is connected and contains more than a single point, and that $\gamma_{u_1}(T_0)$ is at positive distance from $\gamma_{u_2}\setminus\gamma_{u_1}$, from Proposition \ref{prop::gff_localsets_boundarybehavior}, we know that $\LC_{\gamma_{u_1}\cup\gamma_{u_2}}(z_k)$ converges to both $-\lambda-u_1$ and $-\lambda-u_2$, contradiction.
\end{proof}

\begin{remark}\label{rem::boundary_levellines_conditionallaw_loewner_continuity_threelevellines}[Three level lines]
Suppose that $u_1<0<u_2$, and $\gamma=\gamma_0$. A statement analogous to Proposition \ref{prop::boundary_levellines_conditionallaw_loewner_continuity} also holds for the path $\gamma$ given $\gamma_{u_1}$ and $\gamma_{u_2}$. Let $C$ be any connected component of $\HH\setminus(\gamma_{u_1}\cup\gamma_{u_2})$ which lies between $\gamma_{u_1}$ and $\gamma_{u_2}$ and let $x_0,y_0$ be the first and the last points on $\partial C$ traced by $\gamma_{u_1}$. Let $\psi$ be a conformal map from $C$ onto $\HH$ with $\psi(x_0)=0$ and $\psi(y_0)=\infty$. Then almost surely $\psi(\gamma)$ has a continuous Loewner driving function as a curve in $\overline{\HH}$.
\end{remark}

\subsection{Proof of Theorem \ref{thm::boundary_levelline_gff_deterministic}---general case}\label{subsec::boundary_levellines_deterministic_general}
In this section, we will first prove Theorems \ref{thm::boundary_levelline_gff_deterministic} to \ref{thm::sle_chordal_reversibility} in the special case of two force points $x^L=0^-$ and $x^R=0^+$ with weights $\rho^L>-2,\rho^R>-2$. Then by an induction argument, we will complete the proof of Theorem \ref{thm::boundary_levelline_gff_deterministic} for multiple force points case.

\begin{figure}[ht!]
\begin{subfigure}[b]{0.28\textwidth}
\begin{center}
\includegraphics[width=\textwidth]{figures/boundary_levellines_nonintersecting_interacting_1}
\end{center}
\caption{The boundary value for the $\GFF$ $h$ in $\HH$.}
\end{subfigure}
$\quad$
\begin{subfigure}[b]{0.68\textwidth}
\begin{center}\includegraphics[width=\textwidth]{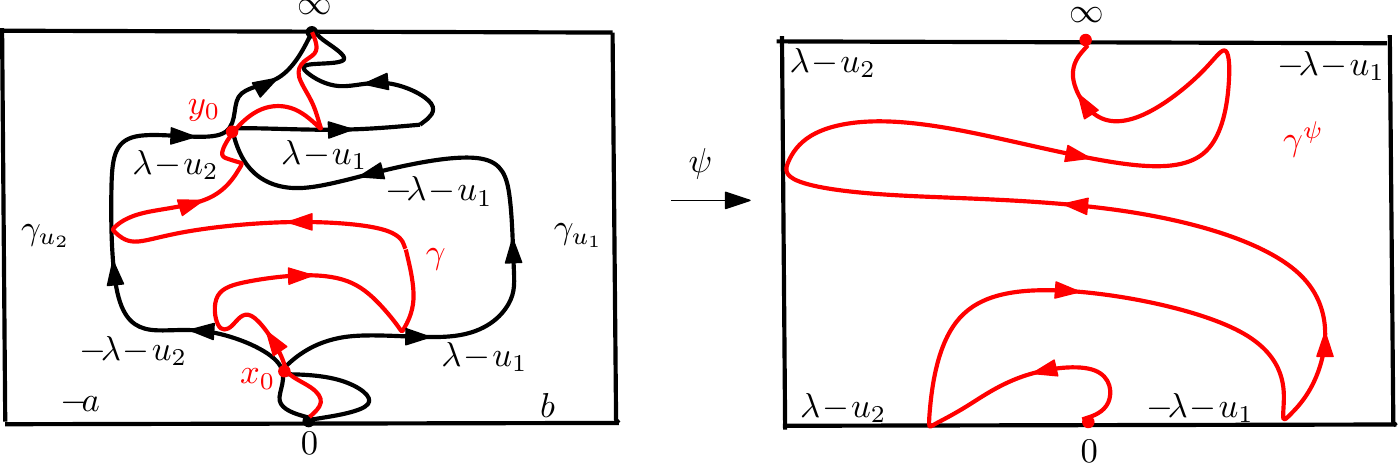}
\end{center}
\caption{Let $\psi$ be any conformal map from $C$ onto $\HH$ that sends $x_0$ to 0 and $y_0$ to $\infty$. The curve $\psi(\gamma|_C)$ is continuous with continuous Loewner driving function.}
\end{subfigure}
\caption{\label{fig::boundary_levellines_intersecting_interacting} The boundary values of the fields in the proof of Lemma \ref{lem::boundary_levellines_deterministic_intersecting_1}.}
\end{figure}

We will work in the setting of Remark \ref{rem::boundary_levellines_nonintersecting_threelevellines} and Remark \ref{rem::boundary_levellines_conditionallaw_loewner_continuity_threelevellines}. Suppose that $h$ is a $\GFF$ on $\HH$ whose boundary data is as depicted in Figure \ref{fig::boundary_levellines_intersecting_interacting}(a) and assume that $a,b$ are large enough so that all the level lines we will consider do not intersect the boundary. Fix $u_1<0<u_2$. For $i=1,2$, let $\gamma_{u_i}$ be the level line of $h$ with height $u_i$ starting from 0 and let $\gamma$ be the level line of $h$ starting from 0. From Proposition \ref{prop::boundary_levellines_deterministic_nonintersecting_coincide} and Corollary \ref{cor::boundary_levellines_nonintersecting_monotonicity}, we know that $\gamma_{u_1}$, $\gamma$, $\gamma_{u_2}$ are all almost surely continuous and are determined by $h$, that $\gamma_{u_1}$ stays to the right of $\gamma$, and that $\gamma$ stays to the right of $\gamma_{u_2}$. Fix a connected component $C$ of $\HH\setminus(\gamma_{u_1}\cup\gamma_{u_2})$ which lies between $\gamma_{u_1}$ and $\gamma_{u_2}$. Let $x_0$ be the first point in $\partial C$ traced by $\gamma_{u_1}$ and $y_0$ be the last point. Let $h|_C$ be the restriction of $h$ to $C$ and let $\gamma|_C$ be the restriction of $\gamma$ to the time interval in which it takes values in $\overline{C}$. Let $\psi:C\to\HH$ be any conformal map from $C$ onto $\HH$ that sends $x_0$ to 0, $y_0$ to $\infty$. Note that, by the continuity of $\gamma_{u_1}$ and $\gamma_{u_2}$, the map $\psi$ can be extended as a homeomorphism from $\overline{C}$ onto $\overline{\HH}$, and $\psi(\gamma|_C)$ is almost surely a continuous curve in $\overline{\HH}$ from $0$ to $\infty$ with continuous Loewner driving function, see Figure \ref{fig::boundary_levellines_intersecting_interacting}(b). Define
\[h^{\psi}=(h|_C)\circ\psi^{-1},\quad \gamma^{\psi}=\psi(\gamma|_C).\]

\begin{lemma}\label{lem::boundary_levellines_deterministic_intersecting_1}
We have that $\gamma^{\psi}$ has the law of $\SLE_4(\rho^L;\rho^R)$ process with force points $(0^-;0^+)$ where
\[\rho^L=u_2/\lambda-2,\quad \rho^R=-u_1/\lambda-2.\]
Moreover, $\gamma^{\psi}$ is almost surely continuous with $\lim_{t\to\infty}\gamma^{\psi}(t)=\infty$ and $(\gamma^{\psi},h^{\psi})$ are coupled as in Theorem \ref{thm::boundary_levelline_gff_coupling}.
\end{lemma}
\begin{proof}
\textit{First}, we show that $\gamma^{\psi}[0,\tau^{\psi}]$ is a local set for $h^{\psi}$ for every $\gamma^{\psi}$-stopping time $\tau^{\psi}$. Let $\tau$ be the time that $\gamma(\tau)=\psi^{-1}(\gamma^{\psi}(\tau^{\psi}))$, and define
\[A(\tau)=\gamma_{u_1}\cup\gamma[0,\tau]\cup\gamma_{u_2}.\]
From Remark \ref{rem::boundary_levellines_nonintersecting_threelevellines}, we know that $A(\tau)$ is a local set for $h$ and is determined by $h$; and that, given $A(\tau)$, the conditional law of $h|_{C\setminus A(\tau)}$ is the same as a $\GFF$ on $C\setminus A(\tau)$ whose boundary data is as depicted in Figure \ref{fig::boundary_levellines_nonintersecting_interacting_three}. Thus, the conditional expectation $\LC_{A(\tau)}$ restricted to $C\setminus A(\tau)$ given $(A(\tau),h|_{A(\tau)})$ is harmonic with boundary data as depicted in Figure \ref{fig::boundary_levellines_nonintersecting_interacting_three}. Define
\[\LC^{\psi}_{\gamma^{\psi}[0,\tau^{\psi}]}=(\LC_{A(\tau)}|_{C\setminus A(\tau)})\circ\psi^{-1}.\]
The above analysis implies that $\gamma^{\psi}$ and $h^{\psi}$ are coupled so that, given $(\gamma^{\psi}[0,\tau^{\psi}], h^{\psi}|_{\gamma^{\psi}[0,\tau^{\psi}]})$, the conditional law of $h^{\psi}|_{\HH\setminus\gamma^{\psi}[0,\tau^{\psi}]}$ is the same as a $\GFF$ on $\HH\setminus\gamma^{\psi}[0,\tau^{\psi}]$ with mean $\LC^{\psi}_{\gamma^{\psi}[0,\tau^{\psi}]}$ which is harmonic in $\HH\setminus\gamma^{\psi}[0,\tau^{\psi}]$. This implies that $\gamma^{\psi}[0,\tau^{\psi}]$ is a local set for $h^{\psi}$. This also implies that $(\gamma^{\psi},h^{\psi})$ are coupled as in Theorem \ref{thm::boundary_levelline_gff_coupling}.
\smallbreak
\textit{Second}, we show the law of $\gamma^{\psi}$. Define, for $z\in\HH$ and $t\ge 0$,
\[\eta_t(z)=\LC^{\psi}_{\gamma^{\psi}[0,t]}(z).\]
By the continuity of $\gamma^{\psi}$, we know that the conformal radius $\CR(z;\HH\setminus\gamma^{\psi}[0,t])$ is almost surely continuous. From Proposition \ref{prop::gff_localsets_bm}, $(\eta_t(z),t\ge 0)$ has a modification which is a Brownian motion when parameterized by minus the log of  the conformal radius. Then Proposition \ref{prop::chordal_sle4_mart} implies the law of $\gamma^{\psi}$.
\end{proof}

\begin{lemma}\label{lem::boundary_levellines_deterministic_intersecting_2}
Almost surely, $\gamma^{\psi}$ is determined by $h^{\psi}$.
\end{lemma}
\begin{proof}
Let $h_1'$ (resp. $h_2'$) be the restriction of $h$ to the connected component of $\HH\setminus(\gamma_{u_1}\cup\gamma_{u_2})$ that lies to the right of $\gamma_{u_1}$ (resp. lies to the left of $\gamma_{u_2}$).
For the connected components of $\HH\setminus (\gamma_{u_1}\cup\gamma_{u_2})$ that lie between $\gamma_{u_1}$ and $\gamma_{u_2}$, we can put an ordering by saying that $A<B$ for two connected components $A,B$ if and only if $\gamma$ intersects $A$ before $B$. Let $h_3'$ (resp. $h_4'$) be the restriction of $h$ to the connected components which come strictly before (resp. strictly after) $C$ in this ordering. We summarize the facts that we know in the following.
\begin{enumerate}
\item [(a)] Given $(\gamma_{u_1},\gamma_{u_2})$, the field $h$ is determined by $(h_1',h_2',h_3',h_4',h|_C)$.
\item [(b)] From Proposition \ref{prop::boundary_levellines_deterministic_nonintersecting_coincide} and Remark \ref{rem::boundary_levellines_deterministic_nonintersecting_coincide}, we know that $(\gamma_{u_1},\gamma,\gamma_{u_2})$ is almost surely determined by $h$.
\item [(c)] From Lemma \ref{lem::boundary_levellines_deterministic_intersecting_1}, we know that $(\gamma^{\psi},h^{\psi})$ is independent of $(\gamma_{u_1},\gamma_{u_2})$.
\end{enumerate}
Combining these three facts, to show the conclusion, we only need to show that, given $(\gamma_{u_1},\gamma_{u_2})$, the couple $(\gamma|_C,h|_C)$ is independent of $(h_1',h_2',h_3',h_4')$. Assume $x_0=\gamma(s_0)$ and $y_0=\gamma(t_0)$ for $s_0<t_0$.
\smallbreak
\textit{First}, we show that, given $(\gamma_{u_1},\gamma_{u_2})$, the multiple $(\gamma, h_3', h|_C, h_4')$ is independent of $(h_1',h_2')$. This can be obtained by applying Proposition \ref{prop::gff_localsets_interacting} to the sets $\gamma$ and $\gamma_{u_1}\cup\gamma_{u_2}$.
\smallbreak
\textit{Second}, we show that, given $(\gamma_{u_1},\gamma_{u_2})$, the triple $(\gamma[0,t_0], h_3', h|_C)$ is independent of $h_4'$. Applying Proposition \ref{prop::gff_localsets_interacting} to the sets $\gamma_{u_1}\cup\gamma_{u_2}$ and $\gamma_{u_1}\cup\gamma[0,t_0]\cup\gamma_{u_2}$, we know that, given $(\gamma_{u_1},\gamma_{u_2})$, the triple $(\gamma_{u_1}\cup\gamma[0,t_0]\cup\gamma_{u_2}, h_3', h|_C)$ is independent of $h_4'$. In particular, given $(\gamma_{u_1},\gamma_{u_2})$, the triple $(\gamma[0,t_0], h_3', h|_C)$ is independent of $h_4'$.
\smallbreak
\textit{Finally}, we show that, given $(\gamma_{u_1},\gamma_{u_2})$, the couple $(\gamma[s_0,t_0], h|_C)$ is independent of $h_3'$. By a similar analysis as in the second step, we know that, given $(\gamma_{u_1},\gamma_{u_2})$, the triple $(\gamma[s_0,\infty], h_4', h|_C)$ is independent of $h_3'$ (by considering the level line of $-h$ starting from $\infty$ which merges with $\gamma$ almost surely.) In particular, the couple $(\gamma[s_0,t_0],h|_C)$ is independent of $h_3'$. This completes the proof.
\end{proof}

By combining Lemma \ref{lem::boundary_levellines_deterministic_intersecting_1} and Lemma \ref{lem::boundary_levellines_deterministic_intersecting_2}, we have obtained Theorems \ref{thm::boundary_levelline_gff_deterministic} to \ref{thm::sle_chordal_reversibility} in the special case of $\SLE_4(\rho^L;\rho^R)$ process with force points $(0^-;0^+)$. We record it in the following proposition.

\begin{proposition}\label{prop::boundary_levellines_deterministic_intersecting}
Suppose that $h$ is a $\GFF$ on $\HH$ whose boundary value is $-a$ on $\R_-$ and is $b$ on $\R_+$. Assume that
\[a>-\lambda,\quad b>-\lambda.\]
Let $\gamma$ be the level line of $h$ starting from 0 targeted at $\infty$ and let $\gamma'$ be the level line of $-h$ starting from $\infty$ targeted at $0$. Then we have the following conclusions.
\begin{enumerate}
\item [(1)] The level line $\gamma$ is almost surely determined by $h$.
\item [(2)] The level line $\gamma$ is continuous and transient.
\item [(3)] The level lines $\gamma'$ and $\gamma$ are equal.
\end{enumerate}
\end{proposition}
\begin{proof} We only need to show that $\gamma'$ and $\gamma$ are equal. Suppose that $\tilde{h}$ is a $\GFF$ on $\HH$ whose boundary data is $\tilde{b}$ on $\R_+$ and $-\tilde{a}$ on $\R_-$. Assume that $\tilde{a}$ and $\tilde{b}$ satisfy
\[\tilde{a}\ge2\lambda+a,\quad \tilde{b}\ge 2\lambda+b.\]
Set
\[u_1=-\lambda-b\in (\lambda-\tilde{b},0),\quad u_2=\lambda+a\in (0,\tilde{a}-\lambda).\]
For $i=1,2$, let $\tilde{\gamma}_{u_i}$ be the level line of $\tilde{h}$ with height $u_i$ starting from 0. Let $\tilde{\gamma}$ be the level line of $\tilde{h}$ starting from 0 and $\tilde{\gamma}'$ be the level line of $-\tilde{h}$ starting from $\infty$. We know that $\tilde{\gamma}_{u_1}$ stays to the right of $\tilde{\gamma}$ and that $\tilde{\gamma}$ stays to the right of $\tilde{\gamma}_{u_2}$. From Proposition \ref{prop::boundary_levellines_deterministic_nonintersecting_coincide}, $\tilde{\gamma}'$ almost surely merges with $\tilde{\gamma}$.

Given $(\tilde{\gamma}_{u_1},\tilde{\gamma}_{u_2})$, let $C$ be any connected component of $\HH\setminus(\tilde{\gamma}_{u_1}\cup\tilde{\gamma}_{u_2})$ that lies between $\tilde{\gamma}_{u_1}$ and $\tilde{\gamma}_{u_2}$ and let $x_0$ be the first point on $\partial C$ traced by $\gamma_{u_1}$ and $y_0$ be the last point. Let $\psi$ be any conformal map from $C$ onto $\HH$ that sends $x_0$ to $0$ and $y_0$ to $\infty$. Define
\[\gamma=\psi(\tilde{\gamma}|_C),\quad \gamma'=\psi(\tilde{\gamma}'|_C),\quad h=\psi(\tilde{h}|_C).\]
From Lemma \ref{lem::boundary_levellines_deterministic_intersecting_1}, we know that $\gamma$ is the level line of $h$ starting from 0 and that $\gamma'$ is the level line of $-h$ starting from $\infty$. From the above analysis, $\gamma'$ merges with $\gamma$. This completes the proof.
\end{proof}

\begin{remark}\label{rem::boundary_levellines_deterministic_intersecting}
The conclusions in Proposition \ref{prop::boundary_levellines_deterministic_intersecting} also hold when the boundary value of the $\GFF$ is $b>-\lambda$ on $\R_+$ and is piecewise constant, and is at most $-\lambda$ on $\R_-$.
\end{remark}
The technique we use to prove Proposition \ref{prop::boundary_levellines_deterministic_intersecting} can be applied to multiple level lines, we obtain as a consequence the following proposition.
\begin{proposition}\label{prop::boundary_levellines_deterministic_intersecting_conditionallaw}
Suppose that $h$ is a $\GFF$ on $\HH$ whose boundary data is as depicted in Figure \ref{fig::boundary_levellines_intersecting_interacting}(a). Fix heights $u_1<u_2$ and assume that
\[-\lambda-b<u_1<u_2<\lambda+a.\]
For $i=1,2$, let $\gamma_{u_i}$ be the level line of $h$ with height $u_i$ starting from 0. Then almost surely $\gamma_{u_1}$ lies to the right of $\gamma_{u_2}$. Moreover, given $\gamma_{u_2}$, the curve $\gamma_{u_1}$ has the law of $\SLE_4((u_2-u_1)/\lambda-2; (b+u_1)/\lambda-1)$ independently in each connected component of $\HH\setminus\gamma_{u_2}$ that lie to the right of $\gamma_{u_2}$. Similarly, given $\gamma_{u_1}$, the curve $\gamma_{u_2}$ has the law of $\SLE_4((a-u_2)/\lambda-1; (u_2-u_1)/\lambda-2)$ independently in each connected component of $\HH\setminus\gamma_{u_1}$ that lie to the left of $\gamma_{u_1}$.
\end{proposition}
\begin{remark}\label{rem::boundary_levellines_deterministic_monotonicity_generalization}
The conclusion in Proposition \ref{prop::boundary_levellines_nonintersecting_monotonicity} also holds when we replace the restriction on $a,b,u_1,u_2$ by the following:
\[-\lambda-b<u_1<u_2<\lambda+a.\]
\end{remark}
\begin{proof}[Proof of Theorem \ref{thm::boundary_levelline_gff_deterministic}]
We will now complete the proof of Theorem \ref{thm::boundary_levelline_gff_deterministic}. Namely, we will show that level line $\gamma$, whose law is $\SLE_4(\underline{\rho}^L;\underline{\rho}^R)$, of the $\GFF$ $h$ is almost surely determined by $h$.

Write
\[\underline{\rho}^L=(\rho^{l,L},...,\rho^{1,L}),\quad \underline{\rho}^R=(\rho^{1,R},...,\rho^{r,R}).\]
We are going to prove the result by induction on $l$ and $r$. We may assume $x^{1,L}=0^-, x^{1,R}=0^+$ by possibly adding zero weight force points. By Proposition \ref{prop::boundary_levellines_deterministic_intersecting}, the conclusion holds when $l,r\le 1$.
Let $(K_t,t\ge 0)$ be the Loewner chain of $\gamma$ and $(f_t,t\ge 0)$ be the sequence of centered conformal maps $f_t:\HH\setminus K_t\to \HH$. Assume that the conclusion holds for some $l,r\ge 1$. We are going to prove that the conclusion holds for $l+1$ left force points and $r$ right force points.
Let $\tau$ be the first time $t$ that $K_t$ accumulates in $(-\infty,x^{l+1,L}]$ (set $\tau=\infty$ if this never happens).
\smallbreak
\textit{First}, we explain that $K|_{[0,\tau]}$ is almost surely determined by $h$. Let $\tilde{h}$ be the $\GFF$ on $\HH$ whose boundary data is  the same as $h$ on $(x^{l,L},\infty)$ and is $-\lambda(1+\rho^{1,L}+\cdots+\rho^{l,L})$ on $(-\infty,x^{l,L})$; and let $\tilde{K}$ be the Loewner chain of the level line of $\tilde{h}$. Note that $\tilde{K}$ has $l$ left force points and $r$ right force points. For $\eps>0$, let $\tau_{\eps}$ (resp. $\tilde{\tau}_{\eps}$) be the first time $t$ that $K_t$ (resp. $\tilde{K}_t$) gets within distance $\eps$ of $(-\infty,x^{l+1,L})$. Define $H_{\eps}$ to be the open set obtained by removing from $\HH$ the points that are within distance $\eps$ of $(-\infty,x^{l+1,L})$. It suffices to show that $K|_{[0,\tau_{\eps}]}$ is almost surely determined by $h$. Note that
\begin{enumerate}
\item [(a)] From the induction hypothesis, we have that $\tilde{K}|_{[0,\tilde{\tau}_{\eps}]}$ is almost surely determined by $\tilde{h}|_{H_{\eps}}$.
\item [(b)] From Proposition \ref{prop::gff_absolutecontinuity}, we have that $\tilde{h}|_{H_{\eps}}$ and $h|_{H_{\eps}}$ are mutually absolutely continuous.
\end{enumerate}
Combining these two facts, the set $K|_{[0,\tau_{\eps}]}$ is almost surely determined by $h|_{H_{\eps}}$ as desired.

Note that if $\tau$ is the continuation threshold, then we are done.
\smallbreak
\textit{Next}, we assume that $\tau$ is not the continuation threshold and we explain that $K|_{(\tau,\infty)}$ is almost surely determined by $h$. Suppose that the rightmost point of $K_{\tau}\cap\R$ is contained in $[x^{j_0,R},x^{j_0+1,R})$. Then the conditional law of $(f_{\tau}(K_t),t\ge \tau)$ given $K_{\tau}$ is an $\SLE_4(\overline{\rho}^L;\tilde{\underline{\rho}}^R)$ process in $\overline{\HH}$ from 0 to $\infty$ where
\[\overline{\rho}^L=\sum_{i=1}^{l+1}\rho^{i,L},\quad \tilde{\rho}^{1,R}=\sum_{i=1}^{j_0}\rho^{i,R},\quad \tilde{\rho}^{2,R}=\rho^{j_0+1,R},\quad\dots,\quad \tilde{\rho}^{r-j_0+1,R}=\rho^{r,R}.\]
By the induction hypothesis, we know that $(f_{\tau}(K_t),t\ge \tau)$ is almost surely determined by $h\circ f_{\tau}^{-1}$, hence it is determined by $h|_{\HH\setminus K_{\tau}}$ given $K_{\tau}$. This implies the conclusion.
\end{proof}

\subsection{Proof of Theorem \ref{thm::sle_chordal_continuity_transience}---general case}\label{subsec::boundary_levellines_continuity_general}
We will complete the proof of Theorem \ref{thm::sle_chordal_continuity_transience}---the continuity of $\SLE_4(\underline{\rho}^L;\underline{\rho}^R)$ process---by extending the special case proved in Proposition \ref{prop::boundary_levellines_deterministic_intersecting}.

\begin{remark}\label{rem::boundary_levellines_continuity_exceptforcepoints}
Suppose that $\gamma$ is an $\SLE_4(\underline{\rho}^L;\underline{\rho}^R)$ process, we have the following observations.
\begin{enumerate}
\item [(1)] $\gamma$ is almost surely continuous when it is away from the boundary $\partial\HH$.
\item [(2)] When $\gamma$ hits $\partial\HH$, say at time $\tau$, between force points before the continuation threshold is hit, from the absolute continuity in Proposition \ref{prop::gff_absolutecontinuity}, we know that $(\gamma(\tau+s),0\le s\le \eps)$ locally evolves like an $\SLE_4(\rho)$ process with one force point of weight $\rho>-2$ (since $\tau$ is not the continuation threshold). Therefore, $\gamma$ is continuous at time $\tau$.
\end{enumerate}
Combining these two facts, to get the continuity of $\gamma$, we need to rule out pathological behavior when $\gamma$ interacts with a force point or hits the continuation threshold.
\end{remark}

\begin{figure}[ht!]
\begin{subfigure}[b]{0.48\textwidth}
\begin{center}
\includegraphics[width=0.6\textwidth]{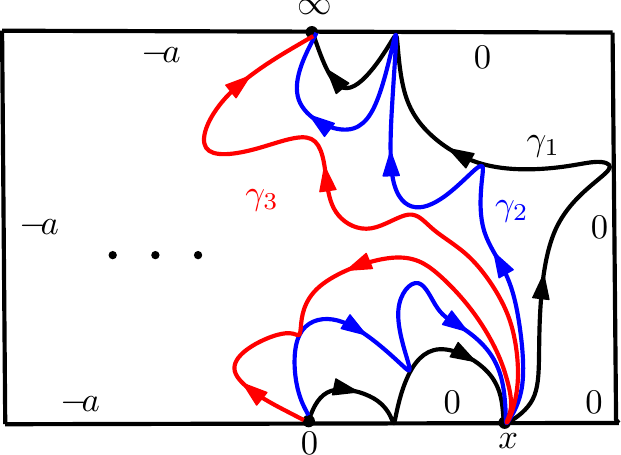}
\end{center}
\caption{If $\gamma_n$ hits $x$, then all $\gamma_1,...,\gamma_{n-1}$ have to hit $x$.}
\end{subfigure}
\begin{subfigure}[b]{0.48\textwidth}
\begin{center}\includegraphics[width=0.6\textwidth]{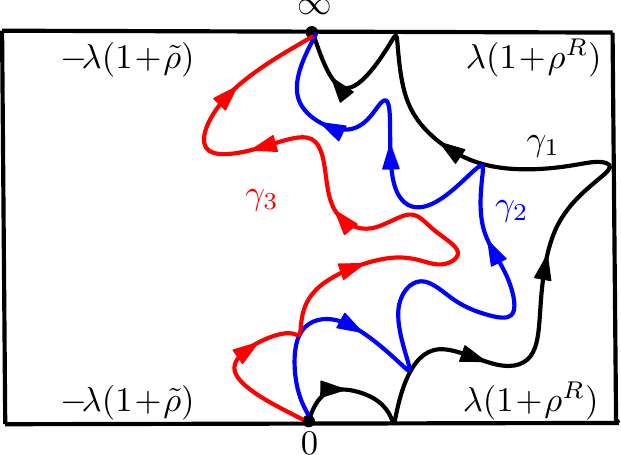}
\end{center}
\caption{Given $\gamma_2$, the conditional law of $\gamma_1$ is $\SLE_4(\rho^L;\rho^R)$.}
\end{subfigure}
\caption{\label{fig::boundary_levellines_zerohittingprobability} The explanation of the behaviour of paths in the the proof of Lemma \ref{lem::boundary_levellines_zerohittingprobability}.}
\end{figure}

\begin{lemma}\label{lem::boundary_levellines_zerohittingprobability}
Suppose that $\gamma$ is an $\SLE_4(\rho^L;\rho^R)$ process with force points $(0^-;0^+)$ where $\rho^L>-2,\rho^R\in (-2,0)$. Then the Lebesgue measure of $\gamma\cap\partial\HH$ is almost surely zero. In particular, for any $x\in\partial\HH\setminus\{0\}$, the probability that $\gamma$ hits $x$ is zero.
\end{lemma}
\begin{proof}
Fix $x\in\R_+$. We only need to show that the probability that $\gamma$ hits $x$ is zero.
\smallbreak
\textit{First}, we show that there exists some $\tilde{\rho}^L\ge 0$ such that the probability that $\SLE_4(\tilde{\rho}^L;\rho^R)$ hits $x$ is zero. Suppose that $h$ is a $\GFF$ on $\HH$ whose boundary data is $0$ on $\R_+$ and $-a$ on $\R_-$, see Figure \ref{fig::boundary_levellines_zerohittingprobability}(a). Since $\rho^R\in (-2,0)$, we can pick $n\in\N$ so that
\[2n-2+n\rho^R\ge 0.\]
Fix $a\ge \lambda(2n+n\rho^R)$. For $1\le j\le n$, set $u_j=\lambda(2j-1+j\rho^R)$ and let $\gamma_j$ be the level line of $h$ with height $u_j$ starting from 0. For $1\le j\le n$, set
\[\rho^L_j=a/\lambda-2j-j\rho^R\ge 0,\quad \rho^R_j=2j-2+j\rho^R.\]
From Proposition \ref{prop::boundary_levellines_deterministic_intersecting_conditionallaw}, we have the following facts.
\begin{enumerate}
\item [(a)] For $1\le j\le n$, the marginal law of $\gamma_j$ is $\SLE_4(\rho_j^L;\rho_j^R)$.
\item [(b)] For $2\le j\le n$, given $\gamma_{j-1}$, the conditional law of $\gamma_j$ is $\SLE_4(\rho^L_j;\rho^R)$.
\end{enumerate}
From the scale-invariance of $\SLE_4(\rho^L_j;\rho^R)$ process, we know that, for any fixed $y$ on the right part of the boundary, the probability that $\SLE_4(\rho^L_j;\rho^R)$ process hits $y$ is $p(\rho^L_j)$ which is independent of $y$. Thus
\[\PP[\gamma_n\text{ hits }x]=p(\rho^L_1)\cdots p(\rho^L_n).\]
By the choice of $n$, we know that $\rho_n^R\ge 0$, therefore $\PP[\gamma_n\text{ hits }x]=0$. Thus there exists some $k_0$ such that $p(\tilde{\rho}^L)=0$ where $\tilde{\rho}^L:=\rho^L_{k_0}$.
\smallbreak
\textit{Next}, we show the conclusion for general $\rho^L>-2$. Fix $\tilde{\rho}>\rho^L\vee\tilde{\rho}^L+2$. Suppose that $h$ is a $\GFF$ on $\HH$ whose boundary value is $\lambda(1+\rho^R)$ on $\R_+$ and is $-\lambda(1+\tilde{\rho})$ on $\R_-$, see Figure \ref{fig::boundary_levellines_zerohittingprobability}(b). Set
\[u_1=0,\quad u_2=\lambda(2+\rho^L)>0,\quad u_3=\lambda(2+\tilde{\rho}^L)\ge 2\lambda.\]
For $i=1,2,3$, let $\gamma_i$ be the level line of $h$ with height $u_i$ starting from 0. From Proposition \ref{prop::boundary_levellines_deterministic_intersecting_conditionallaw}, we have the following facts.
\begin{enumerate}
\item [(a)] Given $\gamma_2$, the conditional law of $\gamma_1$ is $\SLE_4(\rho^L;\rho^R)$.
\item [(b)] Given $\gamma_3$, the conditional law of $\gamma_1$ is $\SLE_4(\tilde{\rho}^L;\rho^R)$.
\item [(c)] The marginal law of $\gamma_3$ is $\SLE_4(\tilde{\rho}-2-\tilde{\rho}^L; 2+\rho^R+\tilde{\rho}^L)$.
\end{enumerate}
Since $2+\rho^R+\tilde{\rho}^L\ge 0$, we know that the probability that $\gamma_3$ hits $x$ is zero. From the choice of $\tilde{\rho}^L$, we know that, given $\gamma_3$, the probability that $\gamma_1$ hits $x$ is zero. Combining these two facts, we know that the probability that $\gamma_1$ hits $x$ is zero. Consequently, given $\gamma_2$, the probability that $\gamma_1$ hits $x$ is also zero. This completes the proof.
\end{proof}

Suppose that $h$ is a $\GFF$ on $\HH$ whose boundary value is $b$ on $\R_+$ and $-a$ on $\R_-$. Fix heights $u_1,...,u_k$ and assume that
\begin{align}
\label{eqn::boundary_levellines_heigthvarying_boundaryrestriction}
&\text{for }1\le i\le k, \quad b>-u_i-\lambda,\quad a>u_i-\lambda;\\
\label{eqn::boundary_levellines_heigthvarying_differencerestriction}
&\text{for }1\le i<j\le k, \quad |u_i-u_j|<2\lambda.
\end{align}
Let $\gamma_{u_1}$ be the level line of $h$ with height $u_1$ starting from 0 and let $\tau_1$ be a $\gamma_{u_1}$-stopping time. For each $2\le j \le k$, we inductively let $\gamma_{u_1\cdots u_j}$ be the level line of $h$ given $\gamma_{u_1\cdots u_{j-1}}[0,\tau_{j-1}]$ with height $u_j$ starting from $\gamma_{u_1\cdots u_{j-1}}(\tau_{j-1})$ and let $\tau_j$ be a $\gamma_{u_1\cdots u_j}$-stopping time. The restriction in Equation (\ref{eqn::boundary_levellines_heigthvarying_boundaryrestriction}) guarantees that $\gamma_{u_1\cdots u_k}$ does not hit the continuation threshold when it hits $\partial\HH$ and the restriction in Equation (\ref{eqn::boundary_levellines_heigthvarying_differencerestriction}) guarantees that $\gamma_{u_1\cdots u_k}$ does not hit the continuation threshold when it hits itself. We call $\gamma_{u_1\cdots u_k}$ a \textbf{height-varying level line} of $h$ starting from 0 with heights $u_1,...,u_k$ with respect to the height change times $\tau_1,...,\tau_{k-1}$. By Theorem \ref{thm::boundary_levelline_gff_deterministic} and an induction argument, we know that $\gamma_{u_1\cdots u_k}$ is almost surely determined by $h$.

\begin{lemma}\label{lem::boundary_levellines_heightvarying_monotoncity}
Suppose that $h$ is a $\GFF$ on $\HH$ whose boundary data is $b$ on $\R_+$ and $-a$ on $\R_-$. Fix heights $u_1,u_2,c_1,c_2$ and assume that
\[c_1<u_1,u_2<c_2, \quad |u_1-u_2|<2\lambda, \quad a>c_2-\lambda,\quad b>-c_1-\lambda.\]
For $i=1,2$, let $\gamma_{c_i}$ be the level line of $h$ with height $c_i$ starting from 0. Let $\gamma_{u_1u_2}$ be the height-varying level line of $h$ with height change time $\tau_1>0$. Then we have the following conclusions.
\begin{enumerate}
\item [(1)] $\gamma_{u_1u_2}$ is almost surely continuous and $\lim_{t\to\infty}\gamma_{u_1u_2}(t)=\infty$.
\item [(2)] $\gamma_{c_1}$ almost surely passes to the right of $\gamma_{u_1u_2}$ and $\gamma_{c_2}$ almost surely passes to the left of $\gamma_{u_1u_2}$.
\end{enumerate}
\end{lemma}

\begin{figure}[ht!]
\begin{subfigure}[b]{0.3\textwidth}
\begin{center}
\includegraphics[width=\textwidth]{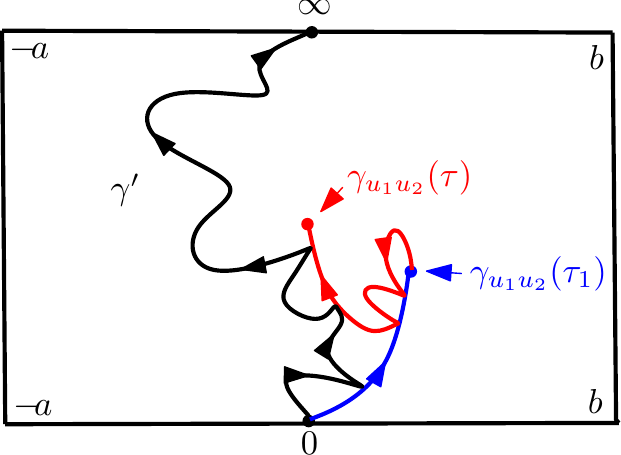}
\end{center}
\caption{$\gamma'$ can not hit the right side of $\gamma_{u_1u_2}[0,\tau]$.}
\end{subfigure}
$\quad$
\begin{subfigure}[b]{0.3\textwidth}
\begin{center}\includegraphics[width=\textwidth]{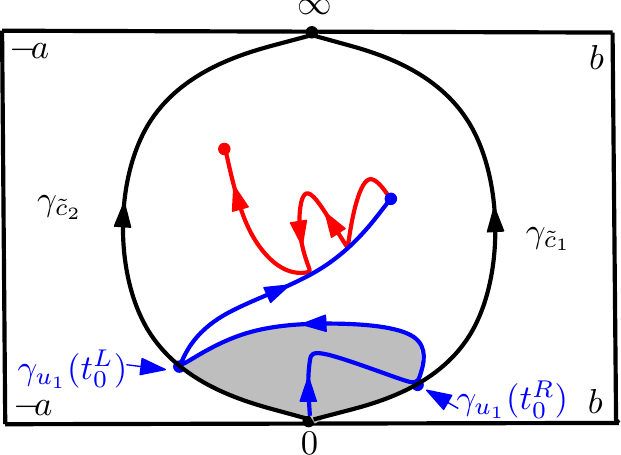}
\end{center}
\caption{$(\gamma_{u_1u_2}(\tau_1+s),s\ge 0)$ can not enter the grey region.}
\end{subfigure}
$\quad$
\begin{subfigure}[b]{0.3\textwidth}
\begin{center}\includegraphics[width=\textwidth]{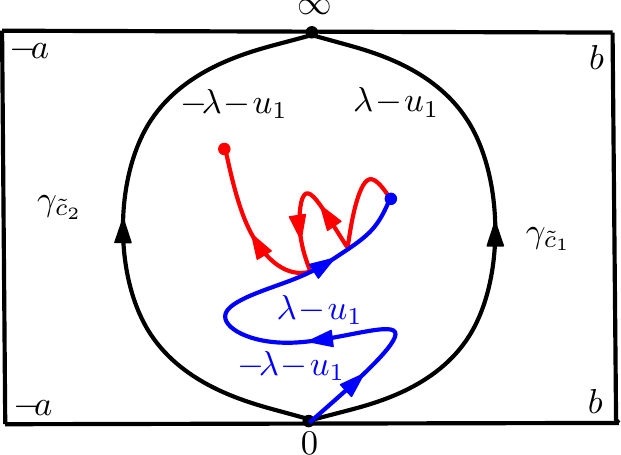}
\end{center}
\caption{The conditional law of the red path is $\SLE_{\!4}\!(\!\rho^L\!;\!\rho^R\!)$.}
\end{subfigure}
\caption{\label{fig::boundary_levellines_heightvarying_monotoncity} The explanation of the behaviour of paths in the proof of Lemma \ref{lem::boundary_levellines_heightvarying_monotoncity}.}
\end{figure}

\begin{proof}
We start by explaining that we only need to show the conclusions for large $a,b$. Suppose that the conclusions hold for large $a,b$, then we could apply the same argument used to prove Proposition \ref{prop::boundary_levellines_deterministic_intersecting} by conditioning on level lines with appropriately chosen heights and use the results of $\gamma_{u_1u_2}$ for large $a,b$.

In the rest of the proof, we suppose that $a,b$ are large enough so that none of level lines that we use later hits the boundary $\partial\HH$ (except at 0 and $\infty$). For $\eps>0$, let $T_{\eps}$ be the first time after $\tau_1$ that $\gamma_{u_1u_2}$ gets within distance $\eps$ of the origin.
\smallbreak
\textit{First}, we explain that $\gamma_{u_1u_2}|_{[0,T_{\eps}]}$  is continuous. Since that $(\gamma_{u_1u_2}(s),0\le s\le \tau_1)$ is the level line of $h$ with height $u_1$, it is almost surely continuous. Given $\gamma_{u_1u_2}[0,\tau_1]$, let $\tilde{h}$ be the restriction of $h$ to $\HH\setminus\gamma_{u_1u_2}[0,\tau_1]$, then the conditional law of $\tilde{h}$ is the same as a $\GFF$ on $\HH\setminus\gamma_{u_1u_2}[0,\tau_1]$ whose boundary value is consistent with $h$ on $\partial\HH$, is $\lambda-u_1$ to the right of $\gamma_{u_1u_2}[0,\tau_1]$ and is $-\lambda-u_1$ to the left of $\gamma_{u_1u_2}[0,\tau_1]$. Furthermore, given $\gamma_{u_1u_2}[0,\tau_1]$, $(\gamma_{u_1u_2}(t+\tau_1),t\ge 0)$ is the level line of $\tilde{h}$ with height $u_2$. From Remark \ref{rem::boundary_levellines_continuity_exceptforcepoints}, it is continuous up to $T_{\eps}$.
\smallbreak
\textit{Second}, we explain that $\gamma_{u_1u_2}[0,T_{\eps}]$ almost surely passes to the right of $\gamma_{c_2}$. Symmetrically, we will have that $\gamma_{u_1u_2}[0,T_{\eps}]$ almost surely passes to the left of $\gamma_{c_1}$. Suppose that $\tau$ is any $\gamma_{u_1u_2}$-stopping time such that $\tau\le T_{\eps}$.
Let $\gamma'$ be the level line of $-h$ with height $-c_2$ starting from $\infty$.
From Lemma \ref{lem::boundary_levellines_deterministic_case3} and Remark \ref{rem::boundary_levellines_deterministic_case3}, we know that, given $\gamma_{u_1u_2}[0,\tau]$, the level line $\gamma'$ first exits $\HH\setminus\gamma_{u_1u_2}[0,\tau]$ through the left side of $\gamma_{u_1u_2}[0,\tau]$. Since that $(\gamma_{u_1u_2}(s),0\le s\le \tau)$ and $\gamma'$ are continuous, we can use the similar proof as the proof of Proposition \ref{prop::boundary_levellines_nonintersecting_monotonicity_reverse} to show that $\gamma_{u_1u_2}[0,\tau]$ stays to the right of $\gamma'$, see Figure \ref{fig::boundary_levellines_heightvarying_monotoncity}(a). Since the range of $\gamma'$ coincides with the range of $\gamma_{c_2}$, this implies the conclusion.
\smallbreak
\textit{Third}, we explain that $\gamma_{u_1u_2}$ does not hit the origin after time 0 and hence it is continuous by Remark \ref{rem::boundary_levellines_continuity_exceptforcepoints}. Choose constants $\tilde{c}_1,\tilde{c}_2$ satisfying
\[u_1-2\lambda<\tilde{c}_1<u_1,u_2<\tilde{c}_2<2\lambda+u_1.\]
The restriction that $|u_1-u_2|<2\lambda$ guarantees that we are able to choose such $\tilde{c}_1,\tilde{c}_2$. For $i=1,2$, let $\gamma_{\tilde{c}_i}$ be the level line of $h$ with height $\tilde{c}_i$ starting from 0. Note that $\gamma_{u_1}$ stays between $\gamma_{\tilde{c}_1}$ and $\gamma_{\tilde{c}_2}$, and that, given $\gamma_{\tilde{c}_1}$ and $\gamma_{\tilde{c}_2}$, the conditional law of $\gamma_{u_1}$ is $\SLE_4(\tilde{\rho}^L;\tilde{\rho}^R)$ with force points next to the starting point where
\[\tilde{\rho}^L=(\tilde{c}_2-u_1)/\lambda-2\in (-2,0),\quad \tilde{\rho}^R=(u_1-\tilde{c}_1)/\lambda-2\in (-2,0).\]
Define $t_0^L$ (resp. $t_0^R$) to be the sup of the times $t\le\tau_1$ such that $\gamma_{u_1}(t)\in\gamma_{\tilde{c}_2}$ (resp. $\gamma_{u_1}(t)\in\gamma_{\tilde{c}_1}$) with the conventions that $\sup\emptyset=0$. By the choice of $\tilde{c}_1,\tilde{c}_2$, we have that $t_0^L>0,t_0^R>0$, see Figure \ref{fig::boundary_levellines_heightvarying_monotoncity}(b). From the above analysis, we know that, $\gamma_{u_1u_2}[0,T_{\eps}]$ stays between $\gamma_{\tilde{c}_1}$ and $\gamma_{\tilde{c}_2}$ for any $\eps>0$. Thus, the path $\gamma_{u_1u_2}|_{[\tau_1,T_{\eps}]}$ can not enter the bounded domain which is bounded by $\gamma_{\tilde{c}_1}$, $\gamma_{\tilde{c}_2}$ and $\gamma_{u_1}[t_0^L\wedge t_0^R, t_0^L\vee t_0^R]$. This implies that the distance $\dist(0,\gamma_{u_1u_2}[\tau_1,\infty))$ is strictly positive. Thus $\gamma_{u_1u_2}$ never hits the origin after time 0.
\smallbreak
\textit{Fourth}, we explain that $\gamma_{u_1u_2}$ almost surely stays between $\gamma_{c_1}$ and $\gamma_{c_2}$. We have the following facts.
\begin{enumerate}
\item [(a)] $\gamma_{u_1u_2}[0,T_{\eps}]$ almost surely stays between $\gamma_{c_1}$ and $\gamma_{c_2}$.
\item [(b)] The distance $\dist(0,\gamma_{u_1u_2}[\tau_1,\infty))$ is strictly positive.
\end{enumerate}
Combining these two facts, we have that $\gamma_{u_1u_2}$ almost surely stays between $\gamma_{c_1}$ and $\gamma_{c_2}$.
\smallbreak
\textit{Finally}, we explain the transience of $\gamma_{u_1u_2}$, namely $\lim_{t\to\infty}\gamma_{u_1u_2}(t)=\infty$. Choose
\[\tilde{c}_1=u_1-2\lambda,\quad \tilde{c}_2=u_1+2\lambda.\]
For $i=1,2$, let $\gamma_{\tilde{c}_i}$ be the level line of $h$ with height $\tilde{c}_i$ starting from 0. We know that $\gamma_{u_1u_2}$ stays between $\gamma_{\tilde{c}_1}$ and $\gamma_{\tilde{c}_2}$, see Figure \ref{fig::boundary_levellines_heightvarying_monotoncity}(c); furthermore, given $(\gamma_{\tilde{c}_1}, \gamma_{\tilde{c}_2}, \gamma_{u_1}[0,\tau_1])$, the conditional law of $(\gamma_{u_1u_2}(s+\tau_1),s\ge 0)$ is $\SLE_4(\rho^L;\rho^R)$ with force points next to the starting point where
\[\rho^L=(u_1-u_2)/\lambda>-2,\quad \rho^R=(u_2-u_1)/\lambda>-2.\]
From Proposition \ref{prop::boundary_levellines_deterministic_intersecting}, we have the desired transience.
\end{proof}

\begin{figure}[ht!]
\begin{center}
\includegraphics[width=0.3\textwidth]{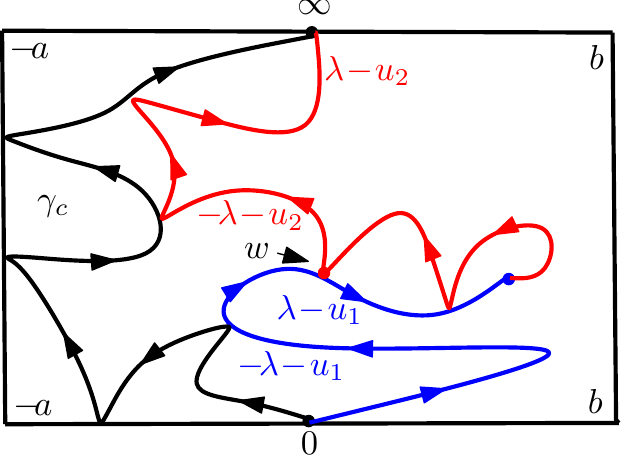}
\end{center}
\caption{\label{fig::boundary_levellines_heightvarying_conditionallaw} The explanation of the behaviour of paths in the proof of Lemma \ref{lem::boundary_levellines_heightvarying_conditionallaw}.}
\end{figure}

\begin{lemma}\label{lem::boundary_levellines_heightvarying_conditionallaw}
Suppose that $h$ is a $\GFF$ on $\HH$ whose boundary data is $b$ on $\R_+$ and $-a$ on $\R_-$. Fix heights $u_1,u_2,c$ and assume that
\[u_1,u_2<c,\quad |u_1-u_2|<2\lambda,\quad a>c-\lambda,\quad a\ge u_1\vee u_2+\lambda,\quad b\ge -(u_1\wedge u_2)-\lambda.\]
Let $\gamma_c$ be the level line of $h$ with height $c$ starting from 0. Let $\gamma_{u_1u_2}$ be the height-varying level line of $h$ starting from 0 with heights $u_1,u_2$ and height change time $\tau_1>0$. Then, given $\gamma_{u_1u_2}$, the conditional law of $\gamma_c$ is $\SLE_4(\rho^{1,L}; \rho^{1,R},\rho^{2,R})$ process in the left connected component of $\HH\setminus\gamma_{u_1u_2}$ where
\[\rho^{1,L}=(a-c)/\lambda-1,\quad \rho^{1,R}=(c-u_1)/\lambda-2,\quad \rho^{1,R}+\rho^{2,R}=(c-u_2)/\lambda-2.\]
\end{lemma}
\begin{proof}
From Lemma \ref{lem::boundary_levellines_heightvarying_monotoncity}, we know that $\gamma_{u_1u_2}$ is almost surely continuous and that $\gamma_c$ almost surely stays to the left of $\gamma_{u_1u_2}$. Define $w$ to be $\gamma_{u_1u_2}(\tau_1)$ if $u_2<u_1$ and to be the last point in $\gamma_{u_1u_2}[\tau_1,\infty)$ that is contained in $\gamma_{u_1u_2}[0,\tau_1]$ if $u_2>u_1$. Let $\psi$ be the conformal map from the left connected component of $\HH\setminus\gamma_{u_1u_2}$, denoted by $C$, onto $\HH$ such that sends 0 to 0, $w$ to 1, and $\infty$ to $\infty$. By the continuity of $\gamma_{u_1u_2}$, the map $\psi$ can be extended as a homeomorphism from $\overline{C}$ onto $\overline{\HH}$, and $\psi(\gamma_c)$ is almost surely a continuous curve in $\overline{\HH}$ from 0 to $\infty$ with continuous Loewner driving function (by a similar argument as in the proof of Proposition \ref{prop::boundary_levellines_conditionallaw_loewner_continuity}). We can then use a similar argument as in the proof of Lemma \ref{lem::boundary_levellines_deterministic_intersecting_1} to show that the law of $\psi(\gamma_c)$ is $\SLE_4(\rho^{1,L};\rho^{1,R},\rho^{2,R})$ in $\overline{\HH}$ from 0 to $\infty$ with force points $(0^-;0^+,1)$.
\end{proof}
\begin{lemma}\label{lem::boundary_levellines_continuity_tworight_forcepoints_1}
Suppose that $h$ is a $\GFF$ on $\HH$ whose boundary value is $-\lambda$ on $\R_-$, is $b$ on $(0,1)$, and is $c$ on $(1,\infty)$. Assume that
\[b>-\lambda,\quad c>-\lambda,\quad |b-c|<2\lambda.\]
Let $\gamma$ be the level line of $h$ starting from 0 targeted at $\infty$ and $\gamma'$ be the level line of $-h$ starting from $\infty$ targeted at 0. Then we have the following conclusions.
\begin{enumerate}
\item [(1)] The level line $\gamma$ is almost surely continuous and transient; moreover, $\gamma$ does not hit 1.
\item [(2)] The level lines $\gamma'$ and $\gamma$ are equal.
\end{enumerate}
\end{lemma}
\begin{proof}
Suppose that $\tilde{h}$ is a GFF on $\HH$ whose boundary value is $-\lambda$ on $\R_-$ and is $\tilde{b}$ on $\R_+$. Assume that $\tilde{b}$ is large enough so that the level lines we will use later are non-boundary-intersecting. Set
\[u_1=-\lambda-b<0,\quad u_2=-\lambda-c<0.\]
Let $\tilde{\gamma}$ be the level line of $\tilde{h}$ starting from 0 and $\tilde{\gamma}'$ be the level line of $-\tilde{h}$ starting from $\infty$. Let $\tilde{\gamma}_{u_1u_2}$ be the height-varying level line of $\tilde{h}$ starting from 0 with heights $u_1,u_2$ and height change time 1. Define $w$ to be $\tilde{\gamma}_{u_1u_2}(1)$ if $u_2\le u_1$ and to be the last point in $\tilde{\gamma}_{u_1u_2}[1,\infty)$ that is contained in $\tilde{\gamma}_{u_1u_2}[0,1]$ if $u_2>u_1$. We summarize the relations between $\tilde{\gamma}$, $\tilde{\gamma}'$ and $\tilde{\gamma}_{u_1u_2}$ as follows.
\begin{enumerate}
\item [(a)] By Proposition \ref{prop::boundary_levellines_nonintersecting_all}, the level line $\tilde{\gamma}$ is continuous and transient; moreover, the paths $\tilde{\gamma}'$ and $\tilde{\gamma}$ are equal.
\item [(b)] By Lemma \ref{lem::boundary_levellines_heightvarying_monotoncity}, the level line $\tilde{\gamma}$ stays to the left of $\tilde{\gamma}_{u_1u_2}$. Define $C$ to be the connected component of $\HH\setminus\tilde{\gamma}_{u_1u_2}$ that has $\R_-$ on the boundary. Let $\psi$ be the conformal map from $C$ onto $\HH$ that sends $0$ to $0$, $w$ to $1$, and $\infty$ to $\infty$. Define $h=\tilde{h}|_C\circ\psi^{-1}$ which has the same boundary value as the $\GFF$ in the statement of this lemma.
\item [(c)] By Lemma \ref{lem::boundary_levellines_heightvarying_conditionallaw}, given $\tilde{\gamma}_{u_1u_2}$, the path $\psi(\tilde{\gamma})$ is the level line of $h$ and the path $\psi(\tilde{\gamma}')$ is the level line of $-h$.
\end{enumerate}
Combining these three facts, given $\tilde{\gamma}_{u_1u_2}$, we have that the level line $\psi(\tilde{\gamma})$ is continuous and transient; moreover, the paths $\psi(\tilde{\gamma}')$ and $\psi(\tilde{\gamma})$ are equal. Finally, we only need to show that $\tilde{\gamma}$ does not hit $w$. Take another level line $\tilde{\gamma}_u$ of $\tilde{h}$ with height $u\in (u_1\vee u_2,0)$ starting from 0. Then almost surely $\tilde{\gamma}_u$ stays between $\tilde{\gamma}$ and $\tilde{\gamma}_{u_1u_2}$. If $\tilde{\gamma}_u$ does not hit $w$, neither does $\tilde{\gamma}$, then we are done. If $\tilde{\gamma}_u$ hits $w$, since the conditional law of $\tilde{\gamma}$ given $\tilde{\gamma}_u$ is an $\SLE_4(\rho)$ process with $\rho=-u/\lambda-2>-2$, we can use Lemma \ref{lem::boundary_levellines_zerohittingprobability} to explain that $\tilde{\gamma}$ almost surely does not hit $w$.
\end{proof}
\begin{figure}[ht!]
\begin{center}
\includegraphics[width=0.3\textwidth]{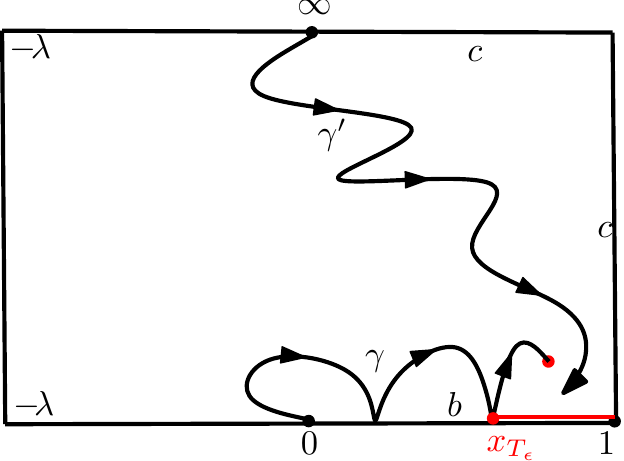}
\end{center}
\caption{\label{fig::boundary_levellines_continuity_tworight_forcepoints} The explanation of the behaviour of paths in the proof of Lemma \ref{lem::boundary_levellines_continuity_tworight_forcepoints_3}. Given $\gamma[0,T_{\eps}]$, the level line $\gamma'$ has to accumulates at $\gamma(T_{\eps})$ or accumulates in the red interval $[x_{T_{\eps}},1]$.}
\end{figure}

\begin{lemma}\label{lem::boundary_levellines_continuity_tworight_forcepoints_2}
Suppose that $h$ is a $\GFF$ on $\HH$ whose boundary value is $-\lambda$ on $\R_-$, is $b$ on $(0,1)$, and is $c$ on $(1,\infty)$. Assume that
\[b\ge \lambda, \quad c>-\lambda.\]
Let $\gamma$ be the level line of $h$ starting from 0 targeted at $\infty$. Then, almost surely, $\gamma$ does not hit 1 and is continuous and transient.
\end{lemma}
\begin{proof}If $c\ge\lambda$, the curve $\gamma$ never hits the boundary after time 0, and we are done. In the following, we assume $c\in (-\lambda,\lambda)$, then $\gamma$ accumulates in $[1,\infty)$ by Proposition \ref{prop::dubedat_lemma15}.
\smallbreak
\textit{First,} we show that $\gamma$ does not hit 1 which implies the continuity by Remark \ref{rem::boundary_levellines_continuity_exceptforcepoints}. Let $\psi$ be the conformal map from $\HH$ onto $\T$ that sends 0 to 0, $1$ to $+\infty$ and $\infty$ to $-\infty$. Then the boundary value on $\partial_U\T$ is $c\in (-\lambda,\lambda)$. By Lemma \ref{lem::boundary_levellines_deterministic_case1} and Remark \ref{rem::boundary_levellines_deterministic_case1}, we know that $\psi(\gamma)$ accumulates in $\partial_U\T$ before reaches $\pm\infty$. This implies that $\gamma$ first accumulates in the open interval $(1,\infty)$.
\smallbreak
\textit{Next,} we show that $\gamma$ is transient. Define $\tau$ to be the first time that $\gamma$ accumulates in $(1,\infty)$. Then, given $\gamma[0,\tau]$, the conditional law of $(\gamma(t+\tau),t\ge 0)$ is $\SLE_4(\tilde{\rho}^R)$ with $\tilde{\rho}^R=c/\lambda-1$, which is transient by Proposition \ref{prop::boundary_levellines_deterministic_intersecting}. This completes the proof.
\end{proof}
\begin{lemma}\label{lem::boundary_levellines_continuity_tworight_forcepoints_3}
Suppose that $h$ is a $\GFF$ on $\HH$ whose boundary value is $-\lambda$ on $\R_-$, is $b$ on $(0,1)$, and is $c$ on $(1,\infty)$. Assume that
\[b\in (-\lambda,\lambda), \quad c\ge\lambda.\]
Let $\gamma$ be the level line of $h$ starting from 0 targeted at $\infty$. Then, almost surely, $\gamma$ does not hit 1 and is continuous.
\end{lemma}
\begin{proof}
We only need to show that $\gamma$ does not hit 1.
We prove by contradiction. Assume that $\gamma$ hits $1$ with positive probability, and, on this event, define $T_{\eps}$ to be the first time that $\gamma$ gets within distance $\eps$ of 1. We have that $\gamma$ is continuous up to $T_{\eps}$. Let $x_{T_{\eps}}$ be the rightmost point of $\gamma[0,T_{\eps}]\cap\R_+$. Note that $x_{T_{\eps}}\in (0,1)$.
Let $\gamma'$ be the level line of $-h$ starting from $\infty$. From Lemma \ref{lem::boundary_levellines_continuity_tworight_forcepoints_2}, we know that $\gamma'$ is almost surely continuous and does not hit 1. By Lemma \ref{lem::levelline_forbidden_intervals}, we know that, given $\gamma[0,T_{\eps}]$, the curve $\gamma'$ has to accumulate at $\gamma(T_{\eps})$ or accumulate in $[x_{T_{\eps}},1]$, see Figure \ref{fig::boundary_levellines_continuity_tworight_forcepoints}. This implies that $\gamma'$ has to get within distance $\eps$ of 1. Since that this holds for any $\eps>0$ and that $\gamma'$ is continuous, the curve $\gamma'$ will hits 1 with positive probability, contradiction.
\end{proof}

\begin{figure}[ht!]
\begin{subfigure}[b]{0.48\textwidth}
\begin{center}
\includegraphics[width=0.6\textwidth]{figures/boundary_levellines_nonintersecting_monotonicity_reverse_1}
\end{center}
\caption{The boundary value for the field in Lemma \ref{lem::boundary_levellines_continuity_tworight_forcepoints_generalization_1}.}
\end{subfigure}
\begin{subfigure}[b]{0.48\textwidth}
\begin{center}
\includegraphics[width=0.6\textwidth]{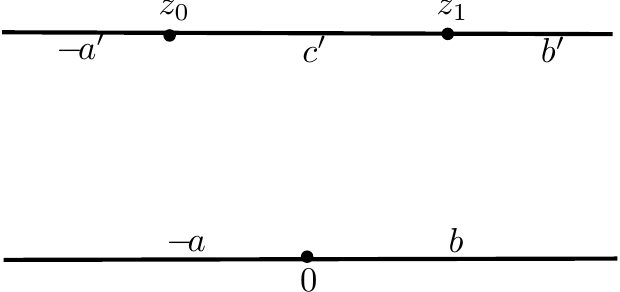}
\end{center}
\caption{The boundary value for the field in Lemma \ref{lem::boundary_levellines_continuity_tworight_forcepoints_generalization_2}.}
\end{subfigure}
\caption{\label{fig::boundary_levellines_continuity_tworight_forcepoints_generalization} The boundary values of the fields in Lemmas \ref{lem::boundary_levellines_continuity_tworight_forcepoints_generalization_1} and \ref{lem::boundary_levellines_continuity_tworight_forcepoints_generalization_2}.}
\end{figure}

The following two lemmas are generalization of Lemma \ref{lem::boundary_levellines_deterministic_case2} and Lemma \ref{lem::boundary_levellines_deterministic_case3}.

\begin{lemma}\label{lem::boundary_levellines_continuity_tworight_forcepoints_generalization_1}
Suppose that $h$ is a $\GFF$ on $\T$ whose boundary value is as depicted in Figure \ref{fig::boundary_levellines_continuity_tworight_forcepoints_generalization}(a). Assume that
\[a, a', b'\ge \lambda,\quad b>-\lambda.\]
Let $\gamma$ be the level line of $h$ starting from 0 targeted at $z_0$. Then, almost surely, $\gamma$ is continuous, $\gamma$ accumulates at $z_0$ before reaches $\pm\infty$, and $\gamma$ does not hit other points in $\partial_U\T$.
\end{lemma}
\begin{proof}
Let $\psi$ be the conformal map from $\T$ onto $\HH$ that sends 0 to 0, $+\infty$ to 1, and $z_0$ to $\infty$. Then $\psi(\gamma)$ is absolutely continuous with respect to $\SLE_4(\rho^{1,R},\rho^{2,R})$ process with force points $(0^+,1)$ where
\[\rho^{1,R}=b/\lambda-1>-2,\quad \rho^{1,R}+\rho^{2,R}=b'/\lambda-1\ge 0.\]
Thus, by Lemma \ref{lem::boundary_levellines_continuity_tworight_forcepoints_3}, we know that $\psi(\gamma)$ is continuous and does not hit $[1,\infty)$ which implies the conclusion.
\end{proof}

\begin{lemma}\label{lem::boundary_levellines_continuity_tworight_forcepoints_generalization_2}
Suppose that $h$ is a $\GFF$ on $\T$ whose boundary value is as depicted in Figure \ref{fig::boundary_levellines_continuity_tworight_forcepoints_generalization}(b). Assume that
\[a, a', b'\ge \lambda,\quad b>-\lambda,\quad c'\in (-\lambda,\lambda).\]
Let $\gamma$ be the level line of $h$ starting from 0 targeted at $z_0$. Then, almost surely, $\gamma$ accumulates in $[z_0,z_1]$ before reaches $\pm\infty$, $\gamma$ is continuous up to the first time that it hits $[z_0,z_1]$, and $\gamma$ does not hit points in $\partial_U\T\setminus [z_0,z_1]$.
\end{lemma}
\begin{proof}
Let $T_{\eps}$ be the first time that $\gamma$ gets within distance $\eps$ of $[z_0,z_1]$. Let $\tilde{h}$ be a $\GFF$ on $\T$ whose boundary value is the same as $h$ except that it is $b'$ on $[z_0,z_1]$. Let $\T_{\eps}$ be the open set obtained by removing from $\T$ all points that are within distance $\eps$ of $[z_0,z_1]$. We have the following two observations.
\begin{enumerate}
\item [(a)] The restriction of $\tilde{h}$ to $\T_{\eps}$ is absolutely continuous with respect to the restriction of $h$ to $\T_{\eps}$ by Proposition \ref{prop::gff_absolutecontinuity}.
\item [(b)] Let $\tilde{\gamma}$ be the level line of $\tilde{h}$ starting from 0. From Lemma \ref{lem::boundary_levellines_continuity_tworight_forcepoints_generalization_1}, we know that $\tilde{\gamma}$ is continuous, it accumulates at $z_0$, and it does not hit other points on $\partial_U\T$.
\end{enumerate}
Combining these two facts, we know that $\gamma$ is continuous up to $T_{\eps}$, $\gamma$ does not reach $\pm\infty$ before $T_{\eps}$, and $\gamma$ does not hit $\partial_U\T$ before $T_{\eps}$. This is true for any $\eps>0$ and we could complete the proof by letting $\eps\to 0$.
\end{proof}

\begin{figure}[ht!]
\begin{subfigure}[b]{0.48\textwidth}
\begin{center}
\includegraphics[width=0.6\textwidth]{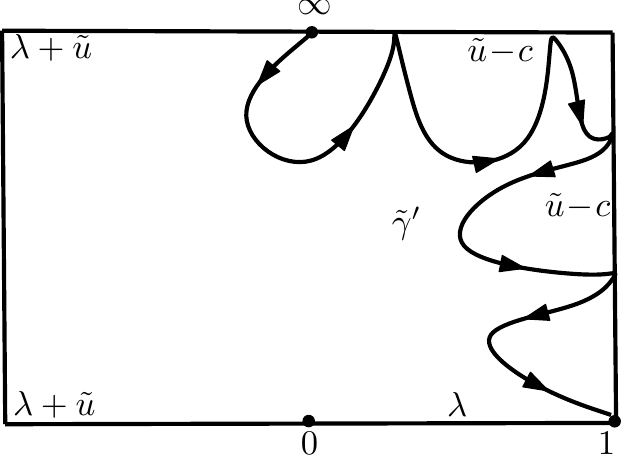}
\end{center}
\caption{Let $\tilde{\gamma}'$ be the level line of $-h$ with height $\tilde{u}$. }
\end{subfigure}
\begin{subfigure}[b]{0.48\textwidth}
\begin{center}
\includegraphics[width=0.6\textwidth]{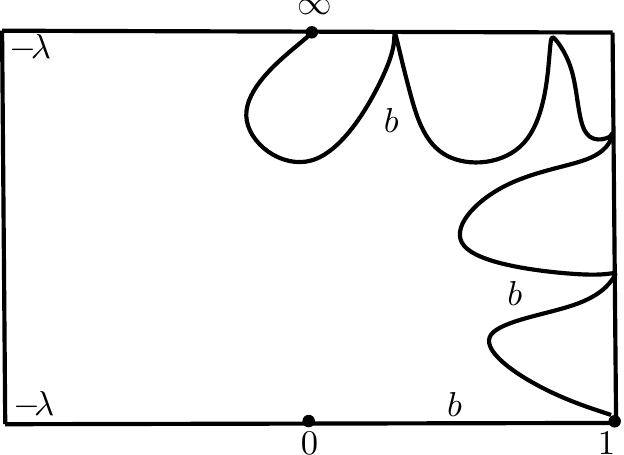}
\end{center}
\caption{Given $\tilde{\gamma}'$, the conditional law of $\gamma$ is $\SLE_4(b/\lambda-1)$.}
\end{subfigure}
\caption{\label{fig::boundary_levellines_continuity_tworight_forcepoints_3_reversibility} The boundary values of the fields in the proof of Lemma \ref{lem::boundary_levellines_continuity_tworight_forcepoints_3_reversibility}.}
\end{figure}

\begin{lemma}\label{lem::boundary_levellines_continuity_tworight_forcepoints_3_reversibility}
Suppose that $h$ is a $\GFF$ on $\HH$ whose boundary value is $-\lambda$ on $\R_-$, is $b$ on $(0,1)$, and is $c$ on $(1,\infty)$. Assume that
\[b\in (-\lambda,\lambda),\quad c\ge \lambda.\]
Let $\gamma$ be the level line of $h$ starting from 0 targeted at $\infty$ and $\gamma'$ be the level line of $-h$ starting from $\infty$ targeted at 0. Then, almost surely, the curves $\gamma'$ and $\gamma$ are equal. In particular, $\gamma$ is transient.
\begin{comment}In particular, the time-reversal of $\SLE_4(\rho^{1,R},\rho^{2,R})$ has the law of $\SLE_4(\tilde{\rho}^{2,L},\tilde{\rho}^{1,L})$ where
\[\rho^{1,R}=b/\lambda-1,\quad \rho^{1,R}+\rho^{2,R}=c/\lambda-1,\quad \tilde{\rho}^{1,L}=\rho^{1,R}+\rho^{2,R},\quad \tilde{\rho}^{1,L}+\tilde{\rho}^{2,L}=\rho^{1,R}.\]
Furthermore, the process $\SLE_4(\rho^{1,R},\rho^{2,R})$ in $\HH$ from 0 to $\infty$ is transient.
\end{comment}
\end{lemma}
\begin{proof}
Set $\tilde{u}=\lambda+b\in (0,2\lambda)$. Let $\tilde{\gamma}'$ be the level line of $-h$ with height $\tilde{u}$  starting from $\infty$ targeted at 1. We list the properties of $\gamma$ and $\tilde{\gamma}'$ as follows.
\begin{enumerate}
\item [(a)] The boundary value for $\tilde{\gamma}'$ is shown in Figure \ref{fig::boundary_levellines_continuity_tworight_forcepoints_3_reversibility}(a). Thus, from Proposition \ref{prop::boundary_levellines_deterministic_intersecting} and Remark \ref{rem::boundary_levellines_deterministic_intersecting}, we know that $\tilde{\gamma}'$ is continuous up to and including the time when it hits 1.
\item [(b)] From Lemma \ref{lem::boundary_levellines_continuity_tworight_forcepoints_3}, we know that $\gamma$ is continuous and does not hit 1.
\end{enumerate}
Combining these two facts and Lemma \ref{lem::boundary_levellines_continuity_tworight_forcepoints_generalization_2}, we could apply a similar proof as the proof of Proposition \ref{prop::boundary_levellines_nonintersecting_monotonicity_reverse} to show that $\gamma$ almost surely stays to the left of $\tilde{\gamma}'$. Furthermore, a similar proof as the proofs of Proposition \ref{prop::boundary_levellines_conditionallaw_loewner_continuity} and Proposition \ref{prop::boundary_levellines_deterministic_intersecting_conditionallaw} would show that the conditional law of $\gamma$ given $\tilde{\gamma}'$ is $\SLE_4(\rho^{R})$ where $\rho^R=b/\lambda-1$, see Figure \ref{fig::boundary_levellines_continuity_tworight_forcepoints_3_reversibility}(b). Define $C$ to be the left connected component of $\HH\setminus\tilde{\gamma}'$. We list the relations between $\gamma$, $\gamma'$ and $\tilde{\gamma}'$ as follows.
\begin{enumerate}
\item [(a)] The curve $\gamma$ stays to the left of $\tilde{\gamma}'$ and $\gamma$ is the level line of $h|_C$ given $\tilde{\gamma}'$.
\item [(b)] From Remark \ref{rem::boundary_levellines_deterministic_monotonicity_generalization}, we know that $\gamma'$ stays to the left of $\tilde{\gamma}'$ and $\gamma'$ is the level line of $-h|_C$ given $\tilde{\gamma}'$.
\end{enumerate}
Combining these two facts and Proposition \ref{prop::boundary_levellines_deterministic_intersecting}, we know that, given $\tilde{\gamma}'$, the curves $\gamma'$ and $\gamma$ are equal. This implies the conclusion.
\end{proof}

We summarize several consequences of Lemmas \ref{lem::boundary_levellines_continuity_tworight_forcepoints_1} to \ref{lem::boundary_levellines_continuity_tworight_forcepoints_3_reversibility}. Suppose that $h$ is a $\GFF$ on $\HH$ whose boundary data is $-\lambda$ on $\R_-$, is $b$ on $(0,1)$, and is $c$ on $(1,\infty)$. Assume that
\[b>-\lambda, \quad c>-\lambda.\]
Let $\gamma$ be the level line of $h$ starting from 0, then the law of $\gamma$ is $\SLE_4(\rho^{1,R},\rho^{2,R})$ with force points $(0^+,1)$ where
\[\rho^{1,R}=b/\lambda-1,\quad \rho^{1,R}+\rho^{2,R}=c/\lambda-1.\]
If $b\ge \lambda$ and $c>-\lambda$, the curve $\gamma$ is continuous by Lemma \ref{lem::boundary_levellines_continuity_tworight_forcepoints_2}; if $b\in (-\lambda,\lambda)$, and $c\ge \lambda$, the curve $\gamma$ is continuous by Lemma \ref{lem::boundary_levellines_continuity_tworight_forcepoints_3}; if $b\in (-\lambda,\lambda)$, and $c\in (-\lambda,\lambda)$, then $|c-b|<2\lambda$, thus $\gamma$ is continuous by Lemma \ref{lem::boundary_levellines_continuity_tworight_forcepoints_1}. In conclusion, we have the continuity of $\gamma$ for all cases where $b>-\lambda,c>-\lambda$. In fact, we have completed the proof of Theorems \ref{thm::sle_chordal_continuity_transience} and \ref{thm::sle_chordal_reversibility} for the case that there are two right force points with weights $\rho^{1,R}$ and $\rho^{2,R}$ satisfying
\[\rho^{1,R}>-2,\quad \rho^{1,R}+\rho^{2,R}>-2.\] We record these conclusions in the following proposition.
\begin{comment}
\begin{proposition}\label{prop::boundary_levellines_continuity_tworight_oneleft}
Suppose that $\gamma$ is an $\SLE_4(\rho^{L};\rho^{1,R},\rho^{2,R})$ process in $\HH$ from 0 to $\infty$ with force points $(0^-;0^+,1)$ where
\[\rho^{L}>-2,\quad \rho^{1,R}>-2,\quad \rho^{1,R}+\rho^{2,R}>-2.\]
Then, almost surely, $\gamma$ does not hit 1 and is continuous and transient.
\end{proposition}
\begin{proposition}\label{prop::boundary_levellines_reversibility_tworight_oneleft}
Suppose that $h$ is a $\GFF$ on $\HH$ whose boundary value is $-a$ on $\R_-$, is $b$ on $(0,1)$, and is $c$ on $(1,\infty)$. Assume that
\[a>-\lambda,\quad b>-\lambda, \quad c>-\lambda.\]
Let $\gamma$ be the level line of $h$ starting from 0 targeted at $\infty$ and $\gamma'$ be the level line of $-h$ starting from $\infty$ targeted at 0. Then, almost surely, the curves $\gamma'$ and $\gamma$ are equal.
\begin{comment}In particular, the time-reversal of $\SLE_4(\rho^L;\rho^{1,R},\rho^{2,R})$ process is $\SLE_4(\tilde{\rho}^{2,L},\tilde{\rho}^{1,L};\tilde{\rho}^R)$ where
\[\rho^L>-2,\quad \rho^{1,R}>-2,\quad \rho^{1,R}+\rho^{2,R}>-2;\]
\[\tilde{\rho}^R=\rho^L,\quad \tilde{\rho}^{1,L}=\rho^{1,R}+\rho^{2,R},\quad \tilde{\rho}^{1,L}+\tilde{\rho}^{2,L}=\rho^{1,R}.\]

\end{proposition}
\end{comment}
\begin{proposition}\label{prop::boundary_levellines_tworight}
Suppose that $h$ is a $\GFF$ on $\HH$ whose boundary value is $-\lambda$ on $\R_-$, is $b$ on $(0,1)$, and is $c$ on $(1,\infty)$. Assume that
\[b>-\lambda, \quad c>-\lambda.\]
Let $\gamma$ be the level line of $h$ starting from 0 targeted at $\infty$ and $\gamma'$ be the level line of $-h$ starting from $\infty$ targeted at 0. The we have the following conclusions.
\begin{enumerate}
\item [(1)] The level line $\gamma$ is almost surely continuous and transient; moreover, it does not hit $1$.
\item [(2)] The level lines $\gamma'$ and $\gamma$ are equal.
\end{enumerate}
\end{proposition}
\begin{remark}\label{rem::boundary_levellines_tworight}
The conclusions in Proposition \ref{prop::boundary_levellines_tworight} hold more generally when the boundary value of $h$ on $\R_-$ is  piecewise constant, and is at most $-\lambda$.
\end{remark}

\begin{figure}[ht!]
\begin{subfigure}[b]{0.3\textwidth}
\begin{center}
\includegraphics[width=\textwidth]{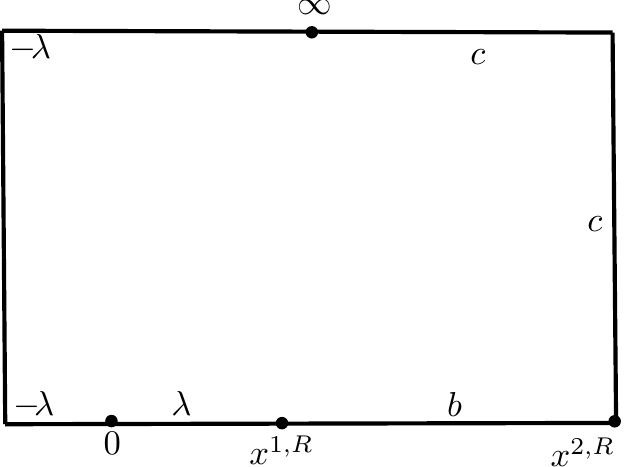}
\end{center}
\caption{The boundary value for the field in the upper-half plane $\HH$.}
\end{subfigure}
$\quad$
\begin{subfigure}[b]{0.3\textwidth}
\begin{center}\includegraphics[width=\textwidth]{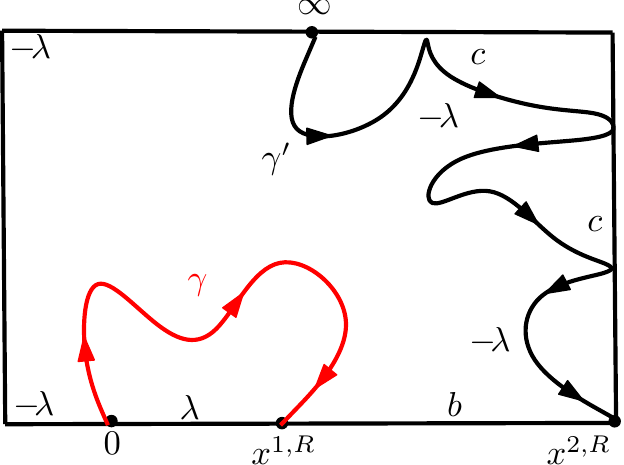}
\end{center}
\caption{Given $\tilde{\gamma}$, the level line $\gamma$ accumulates at $x^{1,R}$.}
\end{subfigure}
$\quad$
\begin{subfigure}[b]{0.3\textwidth}
\begin{center}\includegraphics[width=\textwidth]{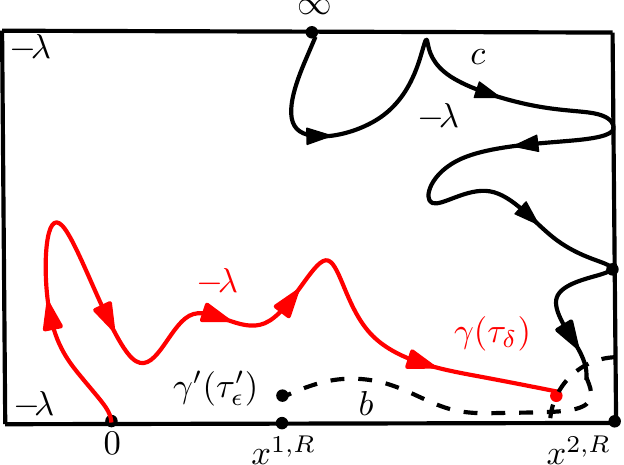}
\end{center}
\caption{Given $\tilde{\gamma}[0,\tilde{\tau}_{\eps}]$, the level line $\gamma$ will gets close to $x^{1,R}$.}
\end{subfigure}
\caption{\label{fig::boundary_levellines_continuity_tworight_forcepoints_4} The explanation of the behaviour of paths in the proof of Lemma \ref{lem::boundary_levellines_continuity_tworight_forcepoints_4}.}
\end{figure}

\begin{lemma}\label{lem::boundary_levellines_continuity_tworight_forcepoints_4}
Suppose that $\rho^{1,R}\le -2$ or $\rho^{1,R}+\rho^{2,R}\le -2$. Let $\gamma$ be an $\SLE_4(\rho^{1,R},\rho^{2,R})$ process in $\HH$ from 0 to $\infty$ with force points $(x^{1,R},x^{2,R})$ where $0<x^{1,R}<x^{2,R}$. Then $\gamma$ is almost surely continuous up to and including the continuation threshold.
\end{lemma}
\begin{proof}
Suppose that $h$ is a $\GFF$ on $\HH$ whose boundary value is $-\lambda$ on $\R_-$, is $\lambda$ on $(0,x^{1,R})$, is $b$ on $(x^{1,R},x^{2,R})$, and is $c$ on $(x^{2,R},\infty)$, see Figure \ref{fig::boundary_levellines_continuity_tworight_forcepoints_4}(a). Let $\gamma$ be the level line of $h$ starting from 0, then its law is $\SLE_4(\rho^{1,R},\rho^{2,R})$ where
\[\rho^{1,R}=b/\lambda-1,\quad \rho^{1,R}+\rho^{2,R}=c/\lambda-1.\]
\smallbreak
\textit{First}, assume that $b\le -\lambda,c\le -\lambda$. Let $\psi$ be the M\"{o}bius transformation of $\HH$ that sends 0 to 0, $x^{1,R}$ to $\infty$, and $\infty$ to $-1$. Then $\psi(\gamma)$ has the law of $\SLE_4(\rho^{2,L},\rho^{1,L})$ with force points $(\psi(x^{2,R}), -1)$ where
\[\rho^{1,L}=-c/\lambda-1\ge 0,\quad \rho^{1,L}+\rho^{2,L}=-b/\lambda-1\ge0.\]
Thus $\psi(\gamma)$ is continuous and is transient by Proposition \ref{prop::boundary_levellines_nonintersecting_all}. This implies that $\gamma$ is continuous up to and including the continuation threshold when it accumulates at $x^{1,R}$.
\smallbreak
\textit{Second}, assume that $b>-\lambda, c\le -\lambda$. Let $\psi$ be the M\"{o}bius transformation of $\HH$ that sends 0 to 0, $x^{2,R}$ to $\infty$, and $\infty$ to $-1$. Then $\psi(\gamma)$ has the law of $\SLE_4(\rho^L;\rho^R)$ with force points $(-1;\psi(x^{1,R}))$ where
\[\rho^L=-c/\lambda-1\ge 0,\quad \rho^R=b/\lambda-1>-2.\]
From Proposition \ref{prop::boundary_levellines_tworight} and Remark \ref{rem::boundary_levellines_tworight}, we have that $\psi(\gamma)$ is continuous and transient. This implies that $\gamma$ is continuous up to and including the continuation threshold when it accumulates at $x^{2,R}$.
\smallbreak
\textit{Finally}, assume that $b\le -\lambda, c>-\lambda$. Define $\tau$ to be the first time that $\gamma$ hits $[x^{1,R}, \infty)$. By Remark \ref{rem::boundary_levellines_continuity_exceptforcepoints}, we know that $(\gamma(t), 0\le t<\tau)$ is continuous.
There are three cases for $\gamma(\tau)$: case 1. $\gamma(\tau)\in (x^{2,R},\infty)$; case 2. $\gamma(\tau)\in [x,^{1,R},x^{2,R})$; case 3. $\gamma(\tau)=x^{2,R}$. We analyze case by case.

Case 1. Suppose that $\gamma(\tau)\in (x^{2,R},\infty)$. By Remark \ref{rem::boundary_levellines_continuity_exceptforcepoints}, we know that $\gamma$ is continuous up to and including $\tau$. Moreover, $\gamma$ continues towards $\infty$ after $\tau$. Thus the distance between $\gamma$ and $[x^{1,R},x^{2,R}]$ is positive and therefore the law of $\gamma$ is absolutely continuous with respect to $\SLE_4(\tilde{\rho}^R)$ with force point $x^{2,R}$ where $\tilde{\rho}^R=c/\lambda-1>-2$. By Proposition \ref{prop::boundary_levellines_tworight}, we know that $\gamma$ is continuous and transient.

Case 2. Suppose that $\gamma(\tau)\in [x,^{1,R},x^{2,R})$. In this case, $\gamma$ arrives its continuation threshold at $\tau$ and stops. We can prove that the distance between $\gamma$ and $[x^{2,R},\infty)$ is positive and the law of $\gamma$ is then absolutely continuous with respect to the law of $\SLE_4(\tilde{\rho}^R)$ with force point $x^{1,R}$ where $\tilde{\rho}^R=b/\lambda-1\le -2$. By the first step, we know that $\gamma$ is continuous up to and including the continuation threshold when it accumulates at $x^{1,R}$.

Case 3. Suppose that $\gamma(\tau)=x^{2,R}$ and, in fact, we will show that this is impossible (or this happens with zero probability). Let $\gamma'$ be the level line of $-h$ starting from $\infty$ targeted at $x^{2,R}$. There are two possibilities: either $\gamma'$ reaches $x^{2,R}$ without hitting $[0,x^{1,R}]$ or it hits its continuation threshold when it accumulates in $[0,x^{1,R}]$. In the former case, we can show that the distance between $\gamma'$ and $[0,x^{1,R}]$ is positive, therefore the law of $\gamma'$ is absolutely continuous with respect to SLE$_4(\tilde{\rho}^L)$ where $\tilde{\rho}^L=c/\lambda-1>-2$. Thus $\gamma'$ is continuous and transient. Moreover, the path $\gamma|_{[0,\tau]}$ stays to the left of $\gamma'$ and the conditional law of $\gamma|_{[0,\tau]}$ given $\gamma'$ is SLE$_4(\rho^{1,R},\tilde{\rho}^{2,R})$ with force points $(x^{1,R},x^{2,R})$ where $\rho^{1,R}+\tilde{\rho}^{2,R}=-2$, see Figure \ref{fig::boundary_levellines_continuity_tworight_forcepoints_4}(b). By the first step, we know that, given $\gamma'$, the level line $\gamma$ is continuous up to and including the continuation threshold when it accumulates at $x^{1,R}$. This contradicts with $\gamma(\tau)=x^{2,R}$. In the latter case, for $\delta>0,\eps>0$,
define $\tau_{\delta}$ to be the first time that $\gamma$ gets within distance $\delta$ of $x^{2,R}$ and define $\tau_{\eps}'$ to be the first time that $\gamma'$ gets within distance $\eps$ of $[0,x^{1,R}]$. We have the following observations.
\begin{enumerate}
\item [(a)] The path $\gamma$ is continuous up to $\tau_{\delta}$ for any $\delta>0$; the path $\gamma'$ is continuous up to $\tau'_{\eps}$ for any $\eps>0$.
\item [(b)] Given $\gamma[0,\tau_{\delta}]$, the conditional law of $\gamma'$ is the level line of $h$ restricted to $\HH\setminus\gamma[0,\tau_{\delta}]$; therefore $\gamma'$ can not hit the union of $\R_-$ and the left side of $\gamma[0,\tau_{\delta}]$. See Figure \ref{fig::boundary_levellines_continuity_tworight_forcepoints_4}(c). Thus, given $\gamma[0,\tau_{\delta}]$, the path $\gamma'$  has to get within distance $\delta$ of $x^{2,R}$ in order to get close to the interval $[0,x^{1,R}]$.
\end{enumerate}
Combining these two facts and letting $\delta$ go to zero, we have that $\gamma'$ accumulates at $x^{2,R}$ before hits $[0,x^{1,R}]$, contradiction.
\end{proof}

\begin{lemma}\label{lem::boundary_levellines_continuity_allright_forcepoints}
Let $\gamma$ be an $\SLE_4(\underline{\rho}^R)$ process in $\HH$ from 0 to $\infty$ with force points $(\underline{x}^R)$ where
\[0\le x^{1,R}<x^{2,R}<\cdots<x^{r,R}.\]
Then $\gamma$ is almost surely continuous up to and including the continuation threshold.
\end{lemma}
\begin{proof}
From Proposition \ref{prop::boundary_levellines_tworight} and Lemma \ref{lem::boundary_levellines_continuity_tworight_forcepoints_4}, we know that the conclusion holds when there are at most two force points. We prove the conclusion by induction. Suppose that the conclusion holds when there are at most $n\ge 2$ force points, and we prove the continuity when there are $n+1$ force points:
\[0\le x^{1,R}<\cdots<x^{n+1,R}< x^{n+2,R}=\infty.\]
Suppose that $h$ is a GFF on $\HH$ whose boundary value is $-\lambda$ on $\R_-$, is $\lambda$ on $(0,x^{1,R})$, and is $c_j$ on $(x^{j,R},x^{j+1,R})$ for $j=1,...,n+1$ where
\[c_j=\lambda(1+\rho^{1,R}+\cdots+\rho^{j,R}).\]
Let $\gamma$ be the level line of $h$ starting from 0 targeted at $\infty$, then the law of $\gamma$ is $\SLE_4(\underline{\rho}^R)$ with $n+1$ force points. Let $\tau$ be the first time that $\gamma$ hits $[x^{2,R},\infty)$; and set $\tau=\infty$ if this never happens. If $\tau=\infty$, then the law of $\gamma$ is absolutely continuous with respect to $\SLE_4(\rho^{1,R},\rho^{2,R})$ process with force points $(x^{1,R},x^{2,R})$ which is continuous up to and including the continuation threshold. Thus $\gamma$ is continuous up to and including the continuation threshold.

In the following we suppose $\tau<\infty$. For $2\le j\le n+1, k\ge 1$, define $T^j_k$ to be the first time that $\gamma$ gets within distance $2^{-k}$ of the interval $[x^{j,R},x^{j+1,R}]$.
\smallbreak
\textit{First}, we explain that $\gamma$ is continuous up to $T^j_k$. Let $H^j_k$ be the open set obtained by removing from $\HH$ all points that are within distance $2^{-k}$ of $[x^{j,R},x^{j+1,R}]$. Let $\tilde{h}$ be the GFF on $\HH$ whose boundary value is consistent with $h$ outside $[x^{j,R},x^{j+1,R}]$ and is $c_{j-1}$ on $[x^{j,R},x^{j+1,R}]$. Let $\tilde{\gamma}$ be the level line of $\tilde{h}$ starting from 0 targeted at $\infty$. Then $\tilde{\gamma}$ has at most $n$ force points. We have the following two observations.
\begin{enumerate}
\item [(a)] From Proposition \ref{prop::gff_absolutecontinuity}, we know that the law of $h$ restricted to $H^j_k$ is absolutely continuous with respect to the law of $\tilde{h}$ restricted to $H^j_k$.
\item [(b)] From induction hypothesis, we know that $\tilde{\gamma}$ is continuous.
\end{enumerate}
Combining these two facts, we know that $\gamma$ is continuous up to $T^j_k$.
\smallbreak
\textit{Second}, we explain that, given $\gamma[0,\tau]$, the conditional law of $(\gamma(t+\tau),t\ge 0)$ is continuous. Note that, given $\gamma[0,\tau]$, the conditional law of $(\gamma(t+\tau),t\ge 0)$ is $\SLE_4$ process with at most $n$ force points. Thus, given $\gamma[0,\tau]$, the curve $(\gamma(t+\tau),t\ge 0)$ is continuous by induction hypothesis.
\smallbreak
\textit{Finally}, we explain the continuity of $\gamma$ up to and including $\tau$.
Define
\[E=\bigcup_{k\ge 1}\bigcup_{j=2}^{n+1}\left\{\tau\le T^j_k\right\}.\]
From the above analysis, we know that $\gamma$ is continuous on the event $E$. We only need to analyze the behavior of $\gamma$ on the event $E^c$. Note that, for $2\le j\le n+1$, the event $\cap_{k\ge 1} \{\tau>T_k^j\}$ implies that $\gamma$ has to accumulate in $[x^{j,R},x^{j+1,R}]$ as $t\uparrow\tau$. This holds for all $j=2,...,n+1$ on $E^c$. The only possibility is that $n=2$ and $\gamma(t)$ accumulates at $x^{3,R}$ as $t\uparrow \tau$. Let $\psi$ be the M\"{o}bius transformation of $\HH$ that sends 0 to 0, $x^{2,R}$ to $x^{2,R}$, and $x^{3,R}$ to $\infty$. Then $\psi(\gamma)$ has the law of SLE$_4(\tilde{\rho}^L;\rho^{1,R},\rho^{2,R})$ process with force points $(\psi(\infty);\psi(x^{1,R}),x^{2,R})$. On the event $E^c$, from the above analysis, we have that $\psi(\gamma)$ does not hit the interval $(-\infty,\psi(\infty))$; therefore it is absolutely continuous with respect to SLE$_4(\rho^{1,R},\rho^{2,R})$ which is continuous and transient. This implies that, on $E^c$, the curve $\gamma$ is continuous up to and including the time $\tau$. This completes the proof.
\end{proof}
\begin{proof}[Proof of Theorem \ref{thm::sle_chordal_continuity_transience}]
We could apply a similar proof as the proof of Lemma \ref{lem::boundary_levellines_continuity_allright_forcepoints} to show the conclusion.
\end{proof}
\begin{remark}\label{rem::boundary_levellines_values_allowedhit}
Suppose that $h$ is a $\GFF$ on $\HH$ with piecewise constant boundary data. 
Let $\gamma_u$ be the level line of $h$ with height $u\in\R$ starting from $x\in\partial\HH$ targeted at $y\in\partial\HH$. 
From the analysis in this section, we have the following observations. 
\begin{enumerate}
\item [(1)] Suppose that $h$ has boundary value $c_R$ to the right of $x$ and $c_L$ to the left of $x$. To have non-trivial $\gamma_u$ (i.e. $\gamma_u$ has strictly positive length), we must have
\[c_L+u<\lambda,\quad \text{and }\quad c_R+u>-\lambda.\]
\item [(2)] Suppose that $h$ has boundary value $c_R$ to the right of $y$ and $c_L$ to the left of $y$. Then the probability of $\gamma_u$ to reach $y$ is zero if one of the following two conditions holds:
\begin{itemize}
\item $c_L+u\ge\lambda$ and $c_R+u\ge\lambda$;
\item $c_L+u\le -\lambda$ and $c_R+u\le -\lambda$.
\end{itemize}
\item [(3)] Suppose that $h$ has boundary value $c$ on some open interval $I$ which does neither contain $x$ nor $y$. If $\gamma_u$ hits $I$ with strictly positive probability, then $c+u\in (-\lambda, \lambda)$.
\end{enumerate}
\end{remark}
Combining Theorems \ref{thm::boundary_levelline_gff_deterministic} and \ref{thm::sle_chordal_continuity_transience}, we can get the following properties of the height-varying level line.
\begin{corollary}
Suppose that $h$ is a $\GFF$ on $\HH$ with piecewise constant boundary data. Fix heights $u_1,...,u_k$ such that,
\[|u_i-u_j|<2\lambda,\quad \text{for } 1\le i<j\le k.\]
Let $\gamma$ be the height-varying level line of $h$ with heights $u_1,...,u_k$. Then, almost surely, $\gamma$ is determined by $h$ and is continuous up to and including the continuation threshold.
\end{corollary}
\begin{proof}Prove by induction on $k$.
\end{proof}

\subsection{Proof of Theorem \ref{thm::boundary_levelline_gff_interacting}---general case}\label{subsec::boundary_levellines_interacting_general}
We start by the generalizations of previous conclusions in the non-boundary-intersecting case to the case when the level lines can hit the boundary.

\begin{figure}[ht!]
\begin{subfigure}[b]{0.48\textwidth}
\begin{center}
\includegraphics[width=0.6\textwidth]{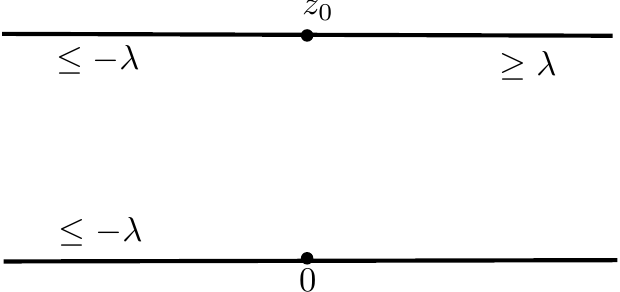}
\end{center}
\caption{The boundary value of the field in Lemma \ref{lem::boundary_levellines_case2_threshold}.}
\end{subfigure}
\begin{subfigure}[b]{0.48\textwidth}
\begin{center}\includegraphics[width=0.6\textwidth]{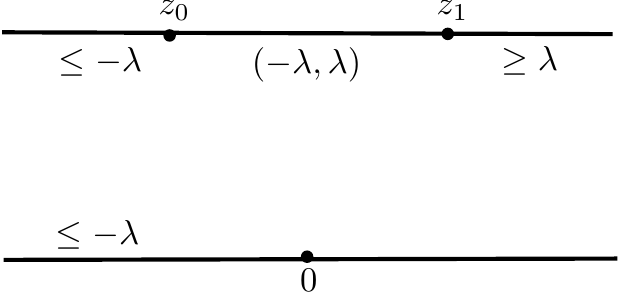}
\end{center}
\caption{The boundary value of the field in Lemma \ref{lem::boundary_levellines_case3_threshold}.}
\end{subfigure}
\caption{\label{fig::boundary_levellines_case23_threshold} The boundary values of the fields in Lemma \ref{lem::boundary_levellines_case2_threshold} and Lemma \ref{lem::boundary_levellines_case3_threshold}.}
\end{figure}

\begin{lemma}\label{lem::boundary_levellines_case2_threshold}[Generalization of Lemma \ref{lem::boundary_levellines_deterministic_case2}]
Suppose that $h$ is a $\GFF$ on $\T$. Assume that the boundary value of $h$ is piecewise constant, and is
\[\begin{array}{ll}
\text{at most } -\lambda \text{ to the left of } z_0 \text{ on } \partial_U\T, & \text{at least } \lambda \text{ to the right of } z_0 \text{ on } \partial_U\T, \\
\text{at most } -\lambda \text{ to the left of } 0 \text{ on } \partial_L\T. &
\end{array}\]
Let $\gamma$ be the level line of $h$ starting from 0 targeted at $z_0$. Then, almost surely, either $\gamma$ hits the continuation threshold before reaches $\partial_U\T$, or $\gamma$ hits $\partial_U\T$ at $z_0$ without otherwise hitting $\partial_U\T$.
\end{lemma}
\begin{proof}
We prove the conclusion by induction on the number of changes in $(0,\infty)$ on $\partial_L\T$. Suppose that
\[0=x_1<x_2<\cdots<x_n<x_{n+1}=\infty,\]
and that the boundary value of $h$ on $(x_j,x_{j+1})$ is $b_j$ for $j=1,...,n$ and assume that $b_1>-\lambda$. Lemma \ref{lem::boundary_levellines_continuity_tworight_forcepoints_generalization_1} implies that the conclusion holds for $n=1$. Suppose that the conclusion holds for $n\le m$ for some $m\ge 1$ and we will prove that the conclusion also holds for $n=m+1$. Let $\tau$ be the first time that $\gamma$ hits $[x_2,\infty]$ and set $\tau=\infty$ if $\gamma$ hits $\partial_U\T$ before hits $[x_2,\infty)$.

If $\tau=\infty$, then the law of $\gamma$ is absolutely continuous with respect to the level line of the GFF on $\T$ whose boundary value is consistent with $h$ outside $(0,\infty)$ and is $b_1$ on $(0,\infty)$. Therefore the conclusion holds by Lemma \ref{lem::boundary_levellines_continuity_tworight_forcepoints_generalization_1}. In the following, we assume that $\tau<\infty$. If $\tau$ is the continuation threshold, then we are done. If not, define $C$ to be the connected component of $\T\setminus\gamma[0,\tau]$ that has $z_0$ on the boundary and let $\psi$ be the conformal map from $C$ onto $\T$ that sends $\gamma(\tau)$ to 0, $+\infty$ to $+\infty$, and $z_0$ to $z_0$. Then, given $\gamma[0,\tau]$, the curve $\psi(\gamma|_C)$ is the level line of the GFF on $\T$ whose boundary value changes at most $m$ times on $(0,\infty)$. Thus the conclusion holds by induction hypothesis.

\end{proof}
\begin{lemma}\label{lem::boundary_levellines_case3_threshold}[Generalization of Lemma \ref{lem::boundary_levellines_deterministic_case3}]
Suppose that $h$ is a $\GFF$ on $\T$. Assume that the boundary value of $h$ is piecewise constant, and is in $(-\lambda,\lambda)$ on $[z_0,z_1]$ and is
\[\begin{array}{ll}
\text{at most } -\lambda \text{ to the left of } z_0 \text{ on } \partial_U\T, & \text{at least } \lambda \text{ to the right of } z_1 \text{ on } \partial_U\T, \\
\text{at most } -\lambda \text{ to the left of } 0 \text{ on } \partial_L\T. &
\end{array}\]
Let $\gamma$ be the level line of $h$ starting from 0 targeted at $z_0$. Then, almost surely, either $\gamma$ hits its continuation threshold before reaches $\partial_U\T$, or $\gamma$ first hits $\partial_U\T$ in $[z_0,z_1]$ without otherwise hitting $\partial_U\T$.
\end{lemma}

\begin{proof} The conclusion can be proved by a similar proof as the proof of Lemma \ref{lem::boundary_levellines_case2_threshold} where Lemma \ref{lem::boundary_levellines_continuity_tworight_forcepoints_generalization_1} needs to be replaced by Lemma \ref{lem::boundary_levellines_continuity_tworight_forcepoints_generalization_2}.
\end{proof}

\begin{figure}[ht!]
\begin{subfigure}[b]{0.48\textwidth}
\begin{center}
\includegraphics[width=0.6\textwidth]{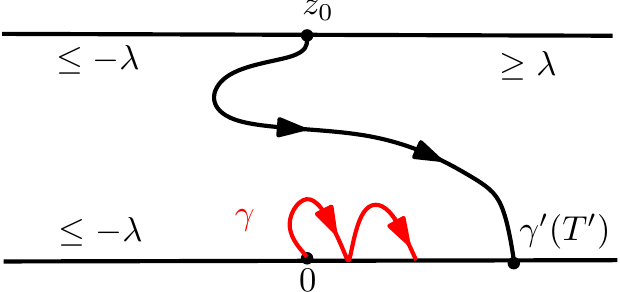}
\end{center}
\caption{}
\end{subfigure}
\begin{subfigure}[b]{0.48\textwidth}
\begin{center}\includegraphics[width=0.6\textwidth]{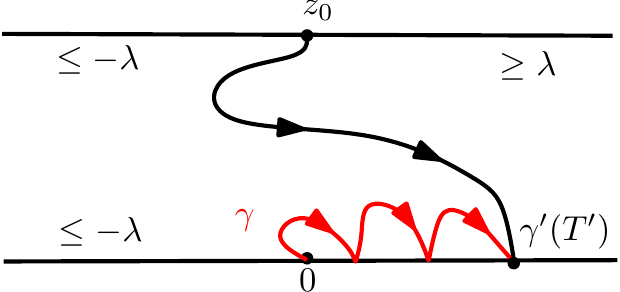}
\end{center}
\caption{}
\end{subfigure}
\caption{\label{fig::boundary_levellines_coincide_threshold} The explanation of the behaviour of the paths in Lemma \ref{lem::boundary_levellines_coincide_threshold}.}
\end{figure}

\begin{lemma}\label{lem::boundary_levellines_coincide_threshold}[Generalization of Lemma \ref{lem::boundary_levellines_deterministic_nonintersecting_coincide}]
Suppose the same the setting as in Lemma \ref{lem::boundary_levellines_case2_threshold}. Let $\gamma'$ be the level line of $-h$ starting from $z_0$ and define $T'$ to be the first time that $\gamma'$ hits $\partial_L\T$. Let $\gamma$ be the level line of $h$ starting from 0 targeted at $z_0$. Then, given $\gamma'[0,T']$, almost surely, either $\gamma$ hits the continuation threshold before hits $\gamma'[0,T']$, or $\gamma$ hits $\gamma'[0,T']$ at $\gamma'(T')$ and merges with $\gamma'$ afterwards. See Figure \ref{fig::boundary_levellines_coincide_threshold}.
\end{lemma}

\begin{proof} We can first prove the conclusion for any $\gamma'$-stopping time $\tau'<T'$. This can be proved by a similar proof as the proof of Lemma \ref{lem::boundary_levellines_deterministic_nonintersecting_coincide},
where Lemma \ref{lem::boundary_levellines_deterministic_case2} needs to be replaced by Lemma \ref{lem::boundary_levellines_case2_threshold}. Then use the continuity of $\gamma'$ to extend the conclusion to time $T'$.
\end{proof}

\begin{figure}[ht!]
\begin{subfigure}[b]{0.48\textwidth}
\begin{center}
\includegraphics[width=0.625\textwidth]{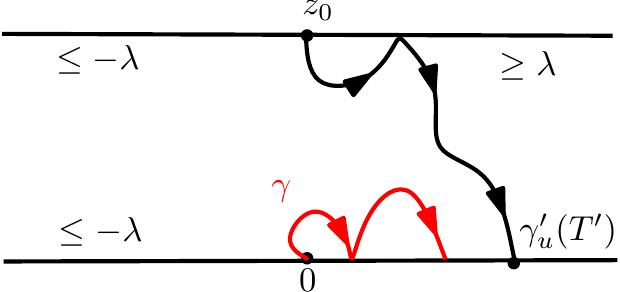}
\end{center}
\caption{Suppose $u>0$. The curve $\gamma$ may hit the continuation threshold before hit $\gamma'_u[0,T']$.}
\end{subfigure}
$\quad$
\begin{subfigure}[b]{0.48\textwidth}
\begin{center}\includegraphics[width=0.625\textwidth]{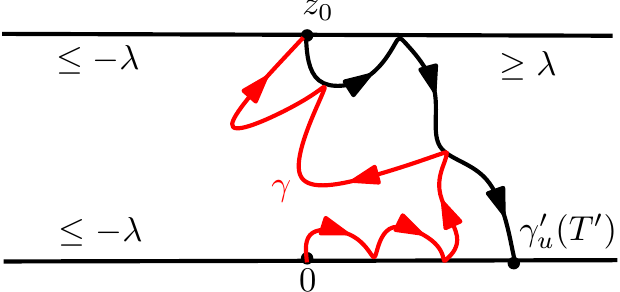}
\end{center}
\caption{Suppose that $u>0$. The curve $\gamma$ stays to the left of the set $\gamma'_u[0,T']$.}
\end{subfigure}
\begin{subfigure}[b]{0.48\textwidth}
\begin{center}
\includegraphics[width=0.625\textwidth]{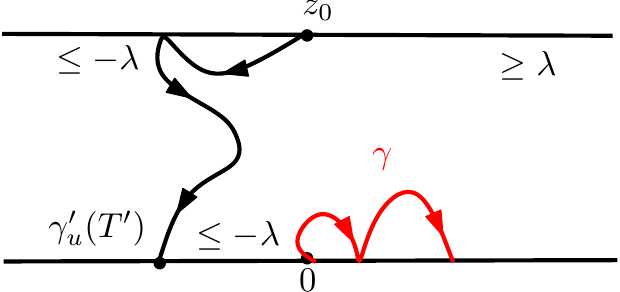}
\end{center}
\caption{Suppose $u<0$. The curve $\gamma$ may hit the continuation threshold before hit $\gamma'_u[0,T']$.}
\end{subfigure}
$\quad$
\begin{subfigure}[b]{0.48\textwidth}
\begin{center}\includegraphics[width=0.625\textwidth]{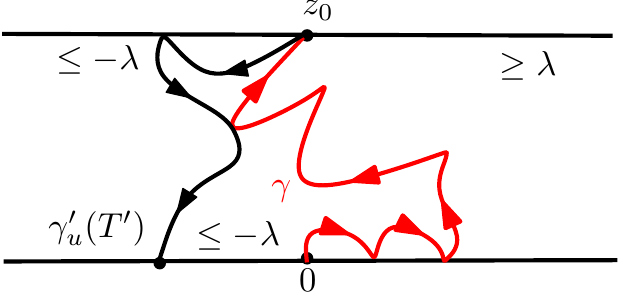}
\end{center}
\caption{Suppose $u<0$. On $E_u$, the curve $\gamma$ stays to the right of the set $\gamma'_u[0,T']$.}
\end{subfigure}

\caption{\label{fig::boundary_levellines_monotonicity_reverse_threshold} The explanation of the behaviour of the paths in Lemma \ref{lem::boundary_levellines_monotonicity_reverse_threshold}.}
\end{figure}

\begin{lemma}\label{lem::boundary_levellines_monotonicity_reverse_threshold}[Generalization of Proposition \ref{prop::boundary_levellines_nonintersecting_monotonicity_reverse}]
Suppose the same the setting as in Lemma \ref{lem::boundary_levellines_case2_threshold}. Fix $u\in\R$ and let $\gamma'_u$ be the level line of $-h$ with height $u$ starting from $z_0$. Assume that $\gamma'_u$ does not hit its continuation threshold before reaches $\partial_L\T$ and define $T'$ to be the first time that $\gamma'_u$ hits $\partial_L\T$. Let $\gamma$ be the level line of $h$ starting from 0 targeted at $z_0$. See Figure \ref{fig::boundary_levellines_monotonicity_reverse_threshold}.
\begin{enumerate}
\item [(1)] If $u>0$, then $\gamma$ almost surely stays to the left of $\gamma'_u[0,T']$.
\item [(2)] If $u<0$, define $E'_u$ to be the event that $\gamma'_u(T')$ is to the left of 0. Then, on $E'_u$, the level line $\gamma$ almost surely stays to the right of $\gamma'_{u}[0,T']$.
\end{enumerate}
\end{lemma}

\begin{proof}
We can first prove the conclusion for any $\gamma'_u$-stopping time $\tau'<T'$. This can be proved by a similar proof as the proof of Proposition \ref{prop::boundary_levellines_nonintersecting_monotonicity_reverse} where Lemma \ref{lem::boundary_levellines_deterministic_case3} needs to be replaced by Lemma \ref{lem::boundary_levellines_case3_threshold}. Then use the continuity of $\gamma'_u$ to extend the conclusion to time $T'$.
\end{proof}

\begin{lemma}\label{lem::boundary_levellines_monotonicity_threshold}[Generalization of Proposition \ref{prop::boundary_levellines_nonintersecting_monotonicity}]
Fix $u_2>u_1$.
Suppose that $h$ is a $\GFF$ on $\T$ whose boundary value is piecewise constant, and is at most $-\lambda-u_2$ to the left of $0$ on $\partial_L\T$. For $i=1,2$, let $\gamma_{i}$ be the level line of $h$ with height $u_i$ starting from 0 and let $\tau_i$ be the first time that $\gamma_{i}$ accumulates in $\partial_U\T$. Define $E_2$ to be the event that $\gamma_2[0,\tau_2]$ does not hit $(0,\infty)$ on $\partial_L\T$. Then, almost surely on $E_2$, $\gamma_{2}[0,\tau_2]$ lies to the left of $\gamma_{1}[0,\tau_1]$. We emphasize that there is no restriction for the boundary data of $h$ on $\partial_U\T$.
\end{lemma}

\begin{proof}
From the proof of Proposition \ref{prop::boundary_levellines_nonintersecting_monotonicity}, we know that the conclusion holds if we can prove the monotonicity for one properly chosen boundary value of $h$ on $\partial_U\T$. Thus it suffices to prove the monotonicity when the boundary value of $h$ is $-a'$ to the left of $z_0$ and is $b'$ to the right of $z_0$ on $\partial_U\T$ where
\[a'\ge \lambda+u_2,\quad b'\ge \lambda-u_1.\]
Let $\gamma'_2$ be the level line of $-h$ with height $-u_2$ starting from $z_0$ targeted at 0. By the choice of boundary value of $h$, we know that $\gamma_2'$ can neither hit $\partial_U\T$ nor $(-\infty,0)$ on $\partial_L\T$ except at points $z_0$ and $0$. Define $E'_2$ to be the event that $\gamma'_2$ does not hit $(0,\infty)$ on $\partial_L\T$. Then we have the following observations.
\begin{enumerate}
\item [(a)] By Lemma \ref{lem::boundary_levellines_monotonicity_reverse_threshold}, we know that, on $E'_2$, the path $\gamma_2'$ stays to the left of $\gamma_1$.
\item [(b)] On $E_2\cap E'_2$, both $\gamma_2$ and $\gamma'_2$ are non-boundary-intersecting, therefore the paths $\gamma_2$ and $\gamma'_2$ are equal.
\end{enumerate}
Combining these two facts, we obtain the conclusion.
\end{proof}

\begin{lemma}\label{lem::boundary_levellines_merge_threshold_bottom}
Fix two starting points $x_2<x_1$. Let \[\underline{y}=(y_1=x_2<y_2<\cdots<y_n<y_{n+1}=x_1)\] be a sequence of points in $[x_2,x_1]$. Suppose that $h$ is a $\GFF$ on $\HH$ whose boundary value is $-a$ on $(-\infty,x_2)$, is $c$ on $(x_1,\infty)$, and is $b_j$ on $(y_j,y_{j+1})$ for $j=1,...,n$. Assume that
\[a\ge \lambda,\quad c\ge \lambda, \quad b_1>-\lambda,\quad b_n<\lambda.\]
For $i=1,2$, let $\gamma_i$ be the level line of $h$ starting from $x_i$ targeted at $\infty$. Then, almost surely, either $\gamma_1$ or $\gamma_2$ hits the continuation threshold before they hit each other, or $\gamma_1$ and $\gamma_2$ merge upon intersecting.
\end{lemma}

\begin{proof}
Let $\gamma'$ be the level line of $-h$ starting from $\infty$ and define $T'$ to be the first time that $\gamma'$ hits $[x_2,x_1]$. From Lemma \ref{lem::boundary_levellines_coincide_threshold}, we have the following two observations.
\begin{enumerate}
\item [(a)] Given $\gamma'[0,T']$, either $\gamma_1$ hits its continuation threshold before hits $\gamma'[0,T']$, or $\gamma_1$ hits $\gamma'[0,T']$ at $\gamma'(T')$ and merges with $\gamma'$ afterwards.
\item [(b)] Given $\gamma'[0,T']$, either $\gamma_2$ hits its continuation threshold before hits $\gamma'[0,T']$, or $\gamma_2$ hits $\gamma'[0,T']$ at $\gamma'(T')$ and merges with $\gamma'$ afterwards.
\end{enumerate}
Combining these two facts, we obtain the conclusion.\end{proof}

\begin{lemma}\label{lem::boundary_levellines_monotonicity_threshold_bottom}
Fix two starting points $x_2<x_1$ and two heights $u_2>u_1$. Let \[\underline{y}=(y_1=x_2<y_2<\cdots<y_n<y_{n+1}=x_1)\] be a sequence of points in $[x_2,x_1]$. Suppose that $h$ is a $\GFF$ on $\HH$ whose boundary value is $-a$ on $(-\infty,x_2)$, is $c$ on $(x_1,\infty)$, and is $b_j$ on $(y_j,y_{j+1})$ for $j=1,...,n$. Assume that
\[a\ge \lambda+u_2,\quad c\ge \lambda-u_1, \quad b_1>-\lambda-u_2,\quad b_n<\lambda-u_1.\]
For $i=1,2$, let $\gamma_i$ be the level line of $h$ with height $u_i$ starting from $x_i$ targeted at $\infty$. Then $\gamma_2$ almost surely stays to the left of $\gamma_1$.
\end{lemma}

\begin{proof}
Pick $u\in (u_1,u_2)$ and let $\gamma'$ be the level line of $-h$ with height $-u$ starting from $\infty$. Define $T'$ to be the first time that $\gamma'$ hits $[x_2,x_1]$. From Lemma \ref{lem::boundary_levellines_monotonicity_reverse_threshold}, we have that $\gamma_2$ stays to the left of $\gamma'[0,T']$ and that $\gamma_1$ stays to the right of $\gamma'[0,T']$. This implies the conclusion.
\end{proof}

\begin{figure}[ht!]
\begin{center}
\includegraphics[width=0.355\textwidth]{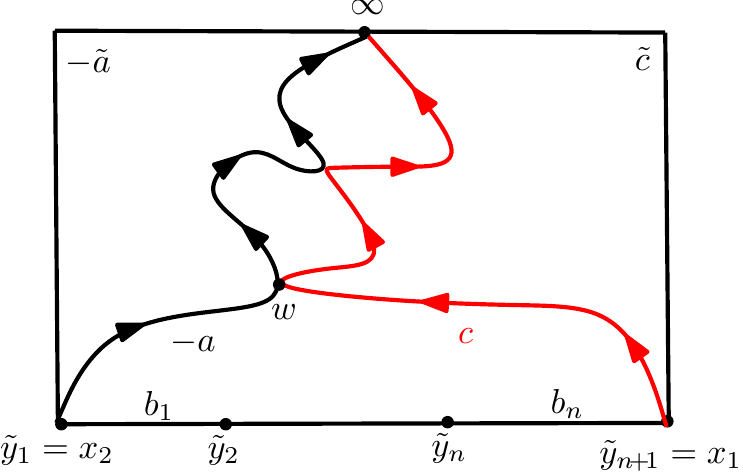}
\end{center}
\caption{\label{fig::boundary_levellines_monotonicity_threshold_bottom_general} The explanation of the behaviour of paths in the proof of Lemma \ref{lem::boundary_levellines_monotonicity_threshold_bottom_general}.}
\end{figure}

\begin{lemma}\label{lem::boundary_levellines_monotonicity_threshold_bottom_general}
The conclusion in Lemma \ref{lem::boundary_levellines_monotonicity_threshold_bottom} also holds when we replace the restriction on $a,c,b_1,b_n$ by the following:
\[a>-\lambda+u_2,\quad c>-\lambda-u_1,\quad b_1>-\lambda-u_2,\quad b_n<\lambda-u_1, \quad b_1\ge -a,\quad b_n\le c.\]
\end{lemma}

\begin{proof} Suppose that
\[\underline{\tilde{y}}=(\tilde{y}_1=x_2<\tilde{y}_2<\cdots<\tilde{y}_n<\tilde{y}_{n+1}=x_1)\]
is a sequence of points on the interval $[x_2,x_1]$.
Suppose that $\tilde{h}$ is a GFF on $\HH$ whose boundary value is $-\tilde{a}$ on $(-\infty,x_2)$, is $\tilde{c}$ on $(x_1,\infty)$, and is $b_j$ on $(\tilde{y}_j,\tilde{y}_{j+1})$ for $j=1,...,n$. Assume that
\[\tilde{a}\ge 2\lambda+a,\quad \tilde{c}\ge 2\lambda+c.\]
Set
\[\tilde{u}_2=\lambda+a>u_2,\quad \tilde{u}_1=-\lambda-c<u_1.\]
For $i=1,2$, let $\gamma_i$ be the level line of $\tilde{h}$ with height $u_i$ starting from $x_i$ targeted at $\infty$. For $i=1,2$, let $\tilde{\gamma}_i$ be the level line of $\tilde{h}$ with height $\tilde{u}_i$ starting from $x_i$ targeted at $\infty$. For $i=1,2$, define $\tilde{E}_i$ to be the event that $\tilde{\gamma}_i$ does not hit the open interval $(x_2,x_1)$. The restriction $b_1\ge -a$ guarantees that $\tilde{E}_2$ has positive probability, and the restriction $b_n\le c$ guarantees that $\tilde{E}_1$ has positive probability.
We summarize the relations between $\gamma_1,\gamma_2,\tilde{\gamma}_1$, and $\tilde{\gamma}_2$ as follows.
\begin{enumerate}
\item [(a)] By Lemma \ref{lem::boundary_levellines_monotonicity_threshold}, we know that, on $\tilde{E}_2$, the path $\tilde{\gamma}_2$ stays to the left of $\gamma_2$; and that, on $\tilde{E}_1$, the path $\tilde{\gamma}_1$ stays to the right of $\gamma_1$.
\item [(b)] By Lemma \ref{lem::boundary_levellines_monotonicity_threshold_bottom}, we know that $\tilde{\gamma}_2$ stays to the left of $\tilde{\gamma}_1$. On the event $\tilde{E}_1\cap\tilde{E}_2$, define $C$ to be the connected component of $\HH\setminus(\tilde{\gamma}_1\cup\tilde{\gamma}_2)$ that has the open interval $(x_2,x_1)$ on the boundary, and define $w$ to be the only point on $\partial C$ that is contained in $\tilde{\gamma}_1\cap\tilde{\gamma}_2$. See Figure \ref{fig::boundary_levellines_monotonicity_threshold_bottom_general}. Let $\psi$ be the conformal map from $C$ onto $\HH$ that sends $x_2$ to $x_2$, $x_1$ to $x_1$, and $w$ to $\infty$.
\item [(c)] By Lemma \ref{lem::boundary_levellines_monotonicity_threshold_bottom}, we know that $\gamma_2$ stays to the left of $\gamma_1$.
\end{enumerate}
Combining these three facts, given $(\tilde{\gamma}_1,\tilde{\gamma}_2)$ and on $\tilde{E}_1\cap\tilde{E}_2$, we have that $\psi(\gamma_2|_C)$ stays to the left of $\psi(\gamma_1|_C)$. For $\delta>0$, define
\[\tilde{F}_{\delta}=\bigcap_{j=2}^n\{\psi(\tilde{y}_j)\in (y_j-\delta,y_j+\delta)\}.\]
Clearly, the event $\tilde{F}_{\delta}$ depends on the pair $(\tilde{\gamma}_1,\tilde{\gamma}_2)$, and we can properly choose the sequence $\underline{\tilde{y}}$ so that the event $\tilde{F}_{\delta}\cap\tilde{E}_1\cap\tilde{E}_2$ has positive probability. From the above analysis, we know that, given $(\tilde{\gamma}_1,\tilde{\gamma}_2)$ and on $\tilde{F}_{\delta}\cap\tilde{E}_1\cap\tilde{E}_2$, the level lines of $\tilde{h}|_C\circ\psi^{-1}$ satisfy the desired monotonicity and note that the boundary value of $\tilde{h}|_C\circ\psi^{-1}$ is $-a$ on $(-\infty,x_2)$, is $c$ on $(x_1,\infty)$, and is $b_j$ on $(\psi(\tilde{y}_j),\psi(\tilde{y}_{j+1}))$ for $j=1,...,n$ where $\psi(\tilde{y}_j)\in (y_j-\delta,y_j+\delta)$. This implies the conclusion, see Remark \ref{rem::boundary_levellines_monotonicity_threshold_bottom_general}.
\end{proof}

\begin{remark}\label{rem::boundary_levellines_monotonicity_threshold_bottom_general}
Fix two starting points $x_2<x_1$ and two heights $u_2>u_1$. Let \[\underline{y}=(y_1=x_2<y_2<\cdots<y_n<y_{n+1}=x_1)\] be a sequence of points in $[x_2,x_1]$.
Suppose that $h$ is a $\GFF$ on $\HH$ whose boundary value is $-a$ on $(-\infty,x_2)$, is $c$ on $(x_1,\infty)$, and is $b_j$ on $(y_j,y_{j+1})$ for $j=1,...,n$.
For $i=1,2$, let $\gamma_i$ be the level line of $h$ with height $u_i$ starting from $x_i$ targeted at $\infty$.

Assume that, for any $\delta>0$, there exists a sequence
\[\underline{y}^{\delta}=(y_1^{\delta}=x_2<y_2^{\delta}<\cdots<y_n^{\delta}<y_{n+1}^{\delta}=x_1) \quad \text{where} \quad y_j^{\delta}\in (y_j-\delta,y_j+\delta)\quad\text{for }j=2,...,n,\]
such that the following is true. Let $h^{\delta}$ be a $\GFF$ on $\HH$ whose boundary value is $-a$ on $(-\infty,x_2)$, is $c$ on $(x_1,\infty)$, and is $b_j$ on $(y_j^{\delta},y_{j+1}^{\delta})$ for $j=1,...,n$.
For $i=1,2$, let $\gamma_i^{\delta}$ be the level line of $h^{\delta}$ with height $u_i$ starting from $x_i$ targeted at $\infty$; then $\gamma^{\delta}_2$ almost surely stays to the left of $\gamma_1^{\delta}$.

Then we have that $\gamma_2$ almost surely stays to the left of $\gamma_1$.
\end{remark}
\begin{proof}
For $i=1,2$, define $\tau^{\delta}_i$ to be the first time of $\gamma_i$ to get within distance $\delta$ of $\{y_2,...,y_n\}$. We have the following observations.
\begin{enumerate}
\item [(a)]The restriction of $h$ to $\HH\setminus (\cup_{j=2}^n B(y_j,\delta))$ is absolutely continuous with respect to the restriction of $h^{\delta}$ to the same domain.
\item [(b)]The path $\gamma^{\delta}_2$ stays to the left of $\gamma^{\delta}_1$ by the assumption.
\end{enumerate}
Combining these two facts, we know that $\gamma_2[0,\tau_2^{\delta}]$ almost surely stays to the left of $\gamma_1[0,\tau_1^{\delta}]$. This holds for any $\delta>0$, therefore $\gamma_2$ almost surely stays to the left of $\gamma_1$.
\end{proof}

\begin{figure}[ht!]
\begin{subfigure}[b]{0.48\textwidth}
\begin{center}
\includegraphics[width=0.74\textwidth]{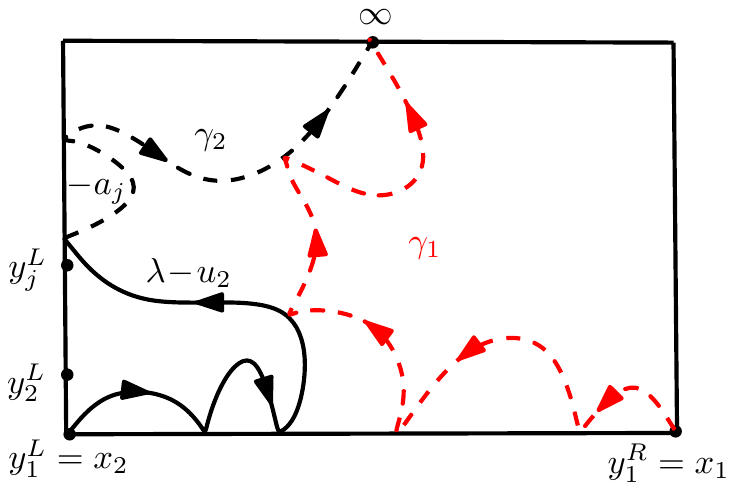}
\end{center}
\caption{Given $\gamma_2[0,T_2]$, the path $\gamma_2$ stays to the left of $\gamma_1$ by induction hypothesis.}
\end{subfigure}
$\quad$
\begin{subfigure}[b]{0.48\textwidth}
\begin{center}\includegraphics[width=0.75\textwidth]{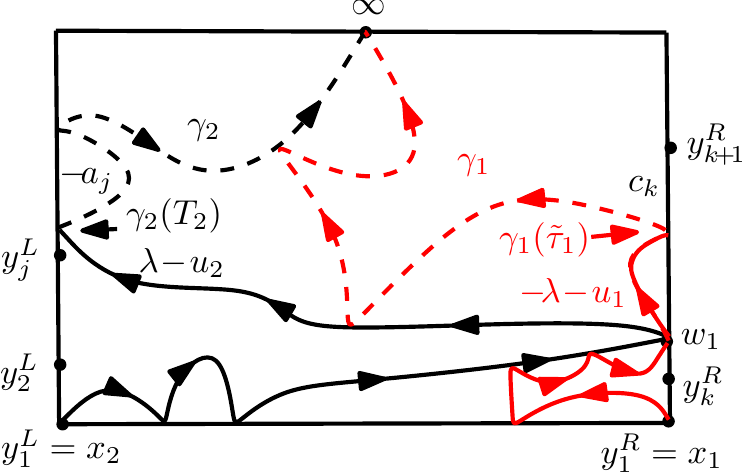}
\end{center}
\caption{Given $\gamma_2[0,T_2]\cup\gamma_1[0,\tilde{\tau}_1]$, the path $\gamma_2$ stays to the left of $\gamma_1$ by induction hypothesis.}
\end{subfigure}
\caption{\label{fig::boundary_levellines_monotonicity_threshold_bottomleftright} The explanation of the behaviour of the paths in Lemma \ref{lem::boundary_levellines_monotonicity_threshold_bottomleftright}.}
\end{figure}

\begin{lemma}\label{lem::boundary_levellines_monotonicity_threshold_bottomleftright}
Fix two starting points $x_2<x_1$ and two heights $u_2>u_1$. Let
\[\begin{array}{ll}
\underline{y} &=(y_1=x_2<y_2<\cdots<y_n<y_{n+1}=x_1)\\
\underline{y}^L &=(y_1^L=x_2>y_2^L>\cdots>y_l^L>y_{l+1}^L=-\infty)\\
\underline{y}^R &=(y_1^R=x_1<y_2^R<\cdots<y_r^R<y_{r+1}^R=\infty)
\end{array}\]
be three sequences of points on the boundary of $\HH$. Suppose that $h$ is a $\GFF$ on $\HH$ whose boundary value is
\[\begin{array}{rll}
-a_j &\text{on}\,  (y_{j+1}^L,y_j^L), &\text{for}\, j=1,...,l;\\
b_j &\text{on}\,  (y_j,y_{j+1}), &\text{for}\, j=1,...,n;\\
c_j &\text{on}\,  (y_{j}^R,y_{j+1}^R), &\text{for}\, j=1,...,r.\end{array}\]
Assume that
\[a_1>-\lambda+u_2,\quad c_1>-\lambda-u_1,\quad b_1>-\lambda-u_2,\quad b_n<\lambda-u_1\quad b_1\ge -a_1,\quad b_n\le c_1.\]
For $i=1,2$, let $\gamma_i$ be the level line of $h$ with height $u_i$ starting from $x_i$ targeted at $\infty$. Then $\gamma_2$ almost surely stays to the left of $\gamma_1$.
\end{lemma}

\begin{proof}
We will prove the conclusion by induction on $l+r$. Lemma \ref{lem::boundary_levellines_monotonicity_threshold_bottom_general} implies that the conclusion holds for $r=1, l=1$. We assume that the conclusion holds for $l+r\le m$ for some $m\ge 2$ and we will prove that the conclusion holds for $l+r=m+1$ and $l\ge 2$. Define $T_2$ to be the first time that $\gamma_2$ hits $(-\infty,y_2^L]$.
\smallbreak
\textit{First}, we argue that $\gamma_1$ almost surely stays to the right of $\gamma_2[0,T_2]$. For $i=1,2$, let $T_i^{\eps}$ be the first time that $\gamma_i$ gets within distance $\eps$ of $(-\infty,y_2^L]$. By induction hpothesis, we know that $\gamma_1[0,T_1^{\eps}]$ stays to the right of $\gamma_2[0,T_2^{\eps}]$. This holds for any $\eps>0$ and we know that $\gamma_1$ and $\gamma_2$ are continuous, thus $\gamma_1$ stays to the right of $\gamma_2[0,T_2]$ up to the first time $T_1$ that $\gamma_1$ hits $(-\infty,y_2^L]$. Since $\gamma_1$ is targeted at $\infty$, it continues towards $\infty$ in the connected component of $\HH\setminus\gamma_1[0,T_1]$ that has $\infty$ on the boundary. Therefore $\gamma_1$ stays to the right of $\gamma_2[0,T_2]$.
\smallbreak
\textit{Next}, we show that $\gamma_1$ almost surely stays to the right of $\gamma_2$ given $\gamma_2[0,T_2]$. If $T_2$ is $\infty$ or the continuation threshold of $\gamma_2$, we are done. In the following of the proof, we assume that $T_2$ is finite and it is not the continuation threshold of $\gamma_2$. Suppose that $\gamma_2(T_2)\in (y_{j+1}^L,y_j^L]$. Since $T_2$ is not the continuation threshold, we have that $a_j>-\lambda+u_2$. Define $C$ to be the connected component of $\HH\setminus\gamma_2[0,T_2]$ that has $\infty$ on the boundary. We only need to show that $\gamma_1|_C$ stays to the right of $\gamma_2|_C$ given $\gamma_2[0,T_2]$. There are two cases: Case 1. $\gamma_2[0,T_2]$ does not hit $(x_1,\infty)$, see Figure \ref{fig::boundary_levellines_monotonicity_threshold_bottomleftright}(a); Case 2. $\gamma_2[0,T_2]$ hits $(x_1,\infty)$, see Figure \ref{fig::boundary_levellines_monotonicity_threshold_bottomleftright}(b). We treat the two cases separately.
\smallbreak
Case 1. In this case, given $\gamma_2[0,T_2]$, the path $\gamma_2|_C$ is the level line with height $u_2$ starting from $\gamma_2(T_2)$ and the path $\gamma_1|_C$ is the level line with height $u_1$ starting from $x_1$. By induction hypothesis (note that the conditions $a_j>-\lambda+u_2$ and $b_n\le c_1$ guarantee that the induction hypothesis is applicable), we have that $\gamma_1|_C$ stays to the right of $\gamma_2|_C$.
\smallbreak
Case 2. In this case, define $w_1$ to be the last point of $\gamma_2|_{[0,T_2]}$ that is contained in $(x_1,\infty)$.
If $\gamma_1$ hits its continuation threshold before enters $C$, then we are done. In the following of the proof, we assume that $\gamma_1$ enters $C$, say at time $\tau_1$. Suppose that $w_1\in (y_k^R,y_{k+1}^R)$, then $c_k\in (-\lambda-u_1,\lambda-u_2)$. We run $\gamma_1$ up to a stopping time $\tilde{\tau}_1$ so that $\tau_1<\tilde{\tau}_1<T_1$ and $\gamma_1(\tilde{\tau}_1)\in (y_k^R,y_{k+1}^{R})$ (the condition $c_k\in (-\lambda-u_1,\lambda-u_2)$ guarantees that we are able to find such a time $\tilde{\tau}_1$). Define $\tilde{C}$ to be the connected component of $\HH\setminus(\gamma_2[0,T_2]\cup\gamma_1[0,\tilde{\tau}_1])$ that has $\infty$ on the boundary.
Given $\gamma_2[0,T_2]$ and $\gamma_1[0,\tilde{\tau}_1]$, the curve $\gamma_1|_{\tilde{C}}$ is the level line of $h|_{\tilde{C}}$ with height $u_1$ starting from $\gamma_1(\tilde{\tau}_1)$ and the curve $\gamma_2|_{\tilde{C}}$ is the level line of $h|_{\tilde{C}}$ with height $u_2$ starting from $\gamma_2(T_2)$. By induction hypothesis (note that the conditions $a_j>-\lambda+u_2$ and $c_k>-\lambda-u_1$ guarantee that the induction hypothesis is applicable), we know that $\gamma_1|_{\tilde{C}}$ stays to the right of $\gamma_2|_{\tilde{C}}$ given $\gamma_2[0,T_2]$ and $\gamma_1[0,\tilde{\tau}_1]$. This completes the proof.
\end{proof}

\begin{remark}\label{rem::boundary_levellines_monotonicity_threshold_bottomleftright}
The conclusion in Lemma \ref{lem::boundary_levellines_monotonicity_threshold_bottomleftright} also holds when we replace the restriction on $a_1,c_1,b_1,b_n$ by the following:
\[a_1>-\lambda+u_2,\quad c_1>-\lambda-u_1,\quad b_1>-\lambda-u_2,\quad b_n<\lambda-u_1.\]
\end{remark}

\begin{proof}
The conditions on $a_1,c_1,b_1,b_n$ guarantee that $\gamma_1$ and $\gamma_2$ exist for a positive time. Pick $\eps>0$ so that $\gamma_1,\gamma_2$ exist up to time $\eps$  and they do not hit each other up to $\eps$. We run both curves up to time $\eps$. To show the conclusion, we only need to show that, given $(\gamma_1[0,\eps],\gamma_2[0,\eps])$, the level line of $h$ (restricted to $\HH\setminus(\gamma_1[0,\eps]\cup\gamma_2[0,\eps])$) with height $u_1$ starting from $\gamma_1(\eps)$ targeted at $\infty$ stays to the right of the level line of $h$ (restricted to $\HH\setminus(\gamma_1[0,\eps]\cup\gamma_2[0,\eps])$) with height $u_2$ starting from $\gamma_2(\eps)$ targeted at $\infty$. Note that, given $(\gamma_1[0,\eps],\gamma_2[0,\eps])$, the field has boundary value $\underline{\tilde{a}}, \underline{\tilde{b}},\underline{\tilde{c}}$ where
\[\tilde{a}_1=\lambda+u_2,\quad \tilde{b}_1=\lambda-u_2,\quad \tilde{b}_n=-\lambda-u_1,\quad \tilde{c}_1=\lambda-u_1.\] This implies that we can apply Lemma \ref{lem::boundary_levellines_monotonicity_threshold_bottomleftright} to show the conclusion given $(\gamma_1[0,\eps],\gamma_2[0,\eps])$.
\end{proof}

\begin{proof}[Proof of Theorem \ref{thm::boundary_levelline_gff_interacting}]
Lemma \ref{lem::boundary_levellines_monotonicity_threshold_bottomleftright} implies the conclusion for $x_2<x_1$ and $u_2>u_1$. The conclusion for $x_2<x_1$ and $u_2=u_1$ can be proved by starting from Lemma \ref{lem::boundary_levellines_merge_threshold_bottom} and then using similar proofs of Lemmas \ref{lem::boundary_levellines_monotonicity_threshold_bottom_general} and  \ref{lem::boundary_levellines_monotonicity_threshold_bottomleftright}.

Finally, we explain the conclusion for $x_2=x_1=x$ and $u_2>u_1$.
Suppose that, in a small neighborhood of $x$, the boundary value of the field is $-a_1$ to the left of $x$ and is $c_1$ to the right of $x$. In order to get the two curves $\gamma_{u_1}^x$ and $\gamma_{u_2}^x$ started, we must have
\[a_1>-\lambda+u_2,\quad c_1>-\lambda-u_1.\]
Take $\eps>0$ small so that $\gamma_{u_1}^x[0,\eps]$ and $\gamma_{u_2}^x[0,\eps]$ are contained in this neighborhood of $x$. By Remark \ref{rem::boundary_levellines_deterministic_monotonicity_generalization}, we know that $\gamma_{u_2}^x[0,\eps]$ stays to the left of $\gamma_{u_1}^x[0,\eps]$. Given $(\gamma_{u_1}^x[0,\eps], \gamma_{u_2}^x[0,\eps])$, the remaining of $\gamma_{u_2}^x$ is the level line of the field with height $u_2$ starting from $\gamma_{u_2}^x(\eps)$ and the remaining of $\gamma_{u_1}^x$ is the level line of the field with height $u_1$ starting from $\gamma_{u_1}^x(\eps)$. Therefore we could use the conclusion for $x_2<x_1$ to show the monotonicity. This completes the proof.
\end{proof}

\begin{corollary}\label{cor::boundary_levellines_monotonicity_heightvarying}
Suppose that $h$ is a $\GFF$ on $\HH$ whose boundary value is piecewise constant. Fix $x_2\le x_1$ and two sequences of heights
\[\underline{u}_1=(u_{1,1},...,u_{1,l});\quad \underline{u}_2=(u_{2,1},...,u_{2,r})\]
satisfying
\[|u_{1,i}-u_{1,j}|<2\lambda,\quad\text{for }1\le i,j\le l;\quad |u_{2,i}-u_{2,j}|<2\lambda,\quad\text{for }1\le i,j\le r.\]
For $i=1,2$, let $\gamma_{i}$ be the height-varying level line of $h$ starting from $x_i$ targeted at $\infty$ with heights $\underline{u}_i$. Assume that
\[u_{1,j}<0,\quad\text{for }j=1,...,l;\quad u_{2,j}>0,\quad\text{for }j=1,...,r.\]
Then $\gamma_2$ stays to the left of $\gamma_1$ almost surely.
\end{corollary}

\subsection{Proof of Theorems \ref{thm::sle_chordal_reversibility} and \ref{thm::boundary_levellines_targetindependence}---general case}\label{subsec::boundary_levellines_reversibility_general}
\begin{figure}[ht!]
\begin{subfigure}[b]{0.48\textwidth}
\begin{center}
\includegraphics[width=0.6\textwidth]{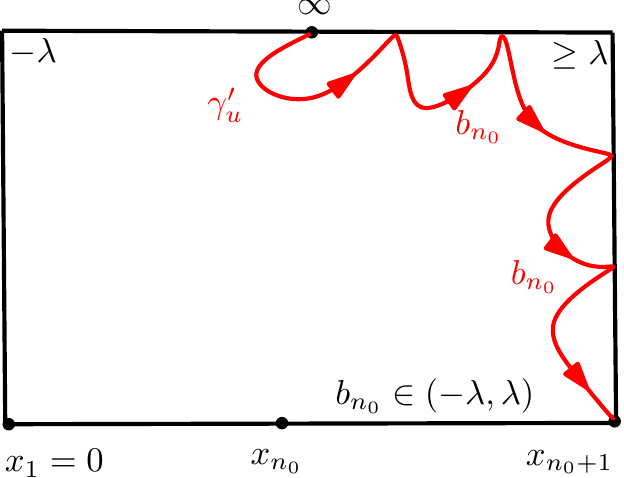}
\end{center}
\caption{Both $\gamma$ and $\gamma'$ stay to the left of $\gamma_u'$.}
\end{subfigure}
\begin{subfigure}[b]{0.48\textwidth}
\begin{center}\includegraphics[width=0.6\textwidth]{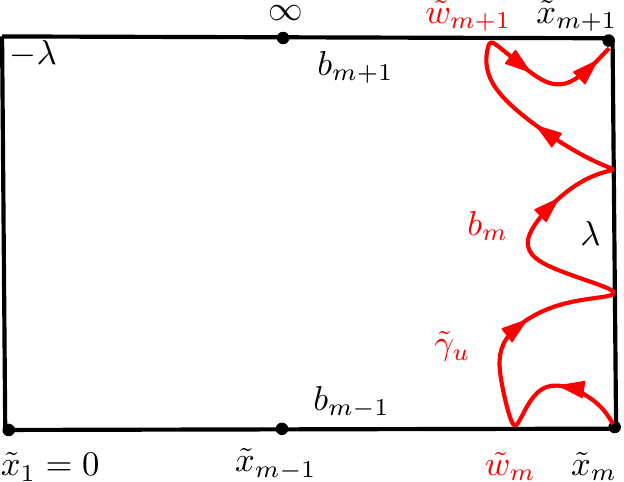}
\end{center}
\caption{$\tilde{\gamma}$ stays to the left of $\tilde{\gamma}_u$.}
\end{subfigure}
\caption{\label{fig::boundary_levellines_reversibility_right} The explanation of the behaviour of the paths in the proof of Lemma \ref{lem::boundary_levellines_reversibility_right}.}
\end{figure}

\begin{lemma}\label{lem::boundary_levellines_reversibility_right}
Suppose that
\[\underline{x}=(0=x_1<x_2<\cdots<x_n<x_{n+1}=\infty)\]
is a sequence of points along $\partial\HH$.
Suppose that $h$ is a $\GFF$ on $\HH$ whose boundary value is $-a$ on $\R_-$ and is $b_j$ on $(x_j,x_{j+1})$ for $j=1,...,n$. Assume that
\[a\ge \lambda,\quad b_j>-\lambda, \quad \text{for } j=1,...,n.\]
Let $\gamma$ be the level line of $h$ starting from 0 targeted at $\infty$ and let $\gamma'$ be the level line of $-h$ starting from $\infty$ targeted at 0. Then $\gamma'$ and $\gamma$ are equal almost surely.
\end{lemma}
\begin{proof}
We prove by induction on $n$. Proposition \ref{prop::boundary_levellines_tworight} implies that the conclusion holds for $n\le 2$. Assume that the conclusion holds for $n\le m$ for some $m\ge 2$, we will show that the conclusion holds for $n=m+1$. If all $b_j\ge \lambda$, then Proposition \ref{prop::boundary_levellines_deterministic_nonintersecting_coincide} and Remark \ref{rem::boundary_levellines_deterministic_nonintersecting_coincide} prove the conclusion. In the following, we assume that there exists some $b_j\in (-\lambda,\lambda)$.
\smallbreak
\textit{First}, we assume that $b_{m+1}\ge\lambda$. Define $n_0\le m$ to be the number so that $b_j\ge \lambda$ for $j>n_0$ and $b_{n_0}\in (-\lambda,\lambda)$. Set $u=\lambda+b_{n_0}\in (0,2\lambda)$. Let $\gamma'_u$ be the level line of $-h$ with height $u$ starting from $\infty$ targeted at $x_{n_0+1}$. See Figure \ref{fig::boundary_levellines_reversibility_right}(a). We summarize the relations between $\gamma,\gamma'$ and $\gamma_u'$ as follows.
\begin{enumerate}
\item [(a)] By Theorem \ref{thm::boundary_levelline_gff_interacting}, the path $\gamma'$ stays to the left of $\gamma'_u$.
\item [(b)] By Lemma \ref{lem::boundary_levellines_monotonicity_reverse_threshold}, the path $\gamma$ stays to the left of $\gamma'_u$.
\item [(c)] Define $C$ to be the connected component of $\HH\setminus\gamma'_u$ that has $\R_-$ on the boundary, and let $\psi$ be any conformal map from $C$ onto $\HH$ that sends $0$ to $0$ and $\infty$ to $\infty$. Given $\gamma_u'$, the law of $h|_C\circ\psi^{-1}$ is the same as a GFF on $\HH$ whose boundary value is $-\lambda$ on $\R_-$, is greater than $-\lambda$ on $\R_+$, and changes at most $n_0$ times on $\R_+$.
\end{enumerate}
Combining these three facts with the induction hypothesis, we know that, given $\gamma'_u$, the paths $\psi(\gamma)$ and $\psi(\gamma')$ are equal which implies that $\gamma'$ and $\gamma$ are equal.
\smallbreak
From the first step and symmetry, we know that the conclusion also holds for $b_1\ge\lambda$. In the following we assume that $b_1\in(-\lambda,\lambda)$ and $b_{m+1}\in (-\lambda,\lambda)$.
\smallbreak
\textit{Second}, we assume that there exists some $j\in\{2,...,m\}$ such that $b_j\ge \lambda$. For any $\eps>0$ define $H_{\eps}$ to be the open set obtained by removing from $\HH$ all points that are within distance $\eps$ of $[x_j,x_{j+1}]$. Define $E_{\eps}$ to be the event that the distance between $\gamma\cup\gamma'$ and $[x_j,x_{j+1}]$ is at least $\eps$. Let $\tilde{h}$ be a GFF on $\HH$  whose boundary value is consistent with $h$ except on $(x_j,x_{j+1})$ and is $b_{j-1}$ on $(x_j,x_{j+1})$. Let $\tilde{\gamma}$ be the level line of $\tilde{h}$ starting from 0 targeted at $\infty$ and let $\tilde{\gamma}'$ be the level line of $-\tilde{h}$ starting from $\infty$ targeted at 0. Then we have the following facts.
\begin{enumerate}
\item [(a)] By the induction hypothesis, we have that $\tilde{\gamma}'$ and $\tilde{\gamma}$ are equal.
\item [(b)] The law of $h$ restricted to $H_{\eps}$ is absolutely continuous with respect to the law of $\tilde{h}$ restricted $H_{\eps}$.
\end{enumerate}
Combining these two facts, we know that, on $E_{\eps}$, the curves $\gamma'$ and $\gamma$ are equal. Since that the boundary value on $(x_j,x_{j+1})$ is at least $\lambda$, the union $\gamma\cup\gamma'$ has positive distance from $(x_j,x_{j+1})$ almost surely. Therefore $\PP[E_{\eps}]\to 1$ as $\eps\to 0$, this implies the conclusion.
\smallbreak
\textit{Finally}, we assume that $b_j\in (-\lambda,\lambda)$ for all $j=1,...,m+1$. Suppose that
\[\underline{\tilde{x}}=(0=\tilde{x}_1<\cdots<\tilde{x}_{m+1}<\tilde{x}_{m+2}=\infty)\]
is a sequence of points on $\R_+$. Suppose that $\tilde{h}$ is a GFF on $\HH$ whose boundary value is $-\lambda$ on $\R_-$, is $b_j$ on $(\tilde{x}_{j},\tilde{x}_{j+1})$ for $j=1,...,m+1$ except for $j=m$, and is $\lambda$ on $(\tilde{x}_m,\tilde{x}_{m+1})$. Let $\tilde{\gamma}$ be the level line of $\tilde{h}$ starting from 0 targeted at $\infty$ and $\tilde{\gamma}'$ be the level line of $-\tilde{h}$ starting from $\infty$ targeted at $0$. Set $u=-\lambda-b_m\in (-2\lambda,0)$, and let $\tilde{\gamma}_u$ be the level line of $\tilde{h}$ with height $u$ starting from $\tilde{x}_m$ targeted at $\tilde{x}_{m+1}$. See Figure \ref{fig::boundary_levellines_reversibility_right}(b). We summarize the relations between $\tilde{\gamma}, \tilde{\gamma}'$ and $\tilde{\gamma}_u$ in the following.
\begin{enumerate}
\item [(a)] By the second step, we have that $\tilde{\gamma}$ and $\tilde{\gamma}'$ are equal.
\item [(b)] By Theorem \ref{thm::boundary_levelline_gff_interacting}, we have that $\tilde{\gamma}$ stays to the left of $\tilde{\gamma}_u$.
\item [(c)] Define $\tilde{E}_u$ to be the event that $\tilde{\gamma}_u$ reaches $[\tilde{x}_{m+1},\infty)$ before hits the continuation threshold and does not hit the union $(-\infty,(\tilde{x}_{m-1}+\tilde{x}_{m})/2)\cup(2\tilde{x}_{m+1},\infty)$. The conditions
     \[\lambda+u\in (-\lambda,\lambda),\quad b_{m-1}+u<\lambda, \quad b_{m+1}+u<\lambda\] guarantee that $\tilde{E}_u$ has positive probability. On $\tilde{E}_u$, let $C$ be the connected component of $\HH\setminus\tilde{\gamma}_u$ that has $\R_-$ on the boundary and let $\psi$ by any conformal map from $C$ onto $\HH$ that fixes $0$ and $\infty$. On $\tilde{E}_u$, define $\tilde{w}_m$ to be the last point of $\tilde{\gamma}_u$ that is contained in $(\tilde{x}_{m-1},\tilde{x}_m)$ and $\tilde{w}_{m+1}$ to be the first point of $\tilde{\gamma}_u$ that is contained in $(\tilde{x}_{m+1},\infty)$.
\end{enumerate}
Combining these three facts, we have that, given $\tilde{\gamma}_u$ and on $\tilde{E}_u$, the paths $\psi(\tilde{\gamma}')$ and $\psi(\tilde{\gamma})$ are equal. For $\delta>0$, define
\[\tilde{F}_{\delta}=\bigcap_{j=2}^{m-1}\{\psi(\tilde{x}_j)\in (x_j-\delta,x_j+\delta)\}\bigcap\{\psi(\tilde{w}_m)\in (x_m-\delta,x_m+\delta)\}\bigcap\{\psi(\tilde{w}_{m+1})\in (x_{m+1}-\delta,x_{m+1}+\delta)\}.\]
Clearly, the event $\tilde{F}_{\delta}$ depends on $\tilde{\gamma}_u$, and we can properly choose the sequence $\underline{\tilde{x}}$ so that the event $\tilde{E}_u\cap\tilde{F}_{\delta}$ has positive probability. From the above analysis, we know that, given $\tilde{\gamma}_u$ and on $\tilde{E}_u\cap\tilde{F}_{\delta}$, the level line of $\tilde{h}|_C\circ\psi^{-1}$ starting from 0 targeted at $\infty$ coincides with the level line of $-\tilde{h}|_C\circ\psi^{-1}$ starting from $\infty$ targeted at 0. Note that, on $\tilde{E}_u\cap\tilde{F}_{\delta}$, the boundary value of $\tilde{h}|_C\circ\psi^{-1}$ is $-a$ on $\R_-$, and is $b_j$ on $(\psi(\tilde{x}_j),\psi(\tilde{x}_{j+1}))$ for $j=1,...,m-2$, is $b_{m-1}$ on $(\psi(\tilde{x}_{m-1}),\psi(\tilde{w}_m))$, is $b_m$ on $(\psi(\tilde{w}_{m}),\psi(\tilde{w}_{m+1}))$, and is $b_{m+1}$ on $(\psi(\tilde{w}_{m+1}),\infty)$. This implies the conclusion, see Remark \ref{rem::boundary_levellines_reversibility_right_conditioningargument}.
\end{proof}
\begin{remark}\label{rem::boundary_levellines_reversibility_right_conditioningargument}
Suppose that
\[\underline{x}=(0=x_1<x_2<\cdots<x_n<x_{n+1}=\infty)\]
is a sequence of points along $\partial\HH$.
Suppose that $h$ is a $\GFF$ on $\HH$ whose boundary value is $-a$ on $\R_-$ and is $b_j$ on $(x_j,x_{j+1})$ for $j=1,...,n$ where $a>-\lambda, b_j>-\lambda$ for $j=1,...,n$.
Let $\gamma$ be the level line of $h$ starting from 0 targeted at $\infty$ and let $\gamma'$ be the level line of $-h$ starting from $\infty$ targeted at 0.

Assume that, for any $\delta>0$, there exists a sequence
\[\underline{x}^{\delta}=(0=x_1^{\delta}<x_2^{\delta}<\cdots<x_n^{\delta}<x_{n+1}^{\delta}=\infty)\quad \text{where}\quad x^{\delta}_j\in (x_j-\delta,x_j+\delta)\quad\text{for }j=2,...n,\]
such that the following is true. Let $h^{\delta}$ be a $\GFF$ on $\HH$ whose boundary value is $-a$ on $\R_-$ and is $b_j$ on $(x_j^{\delta},x_{j+1}^{\delta})$ for $j=1,...,n$.
Let $\gamma_{\delta}$ be the level line of $h^{\delta}$ starting from 0 targeted at $\infty$ and let $\gamma'_{\delta}$ be the level line of $-h^{\delta}$ starting from $\infty$ targeted at 0, then $\gamma'_{\delta}$ and $\gamma_{\delta}$ are equal almost surely.

Then $\gamma'$ and $\gamma$ are equal almost surely.
\end{remark}
\begin{proof}
Define $E_{\delta}$ (resp. $E'_{\delta}$) to be the event that the distance between $\gamma$ (resp. $\gamma'$) and $\{x_2,...,x_n\}$ is at least $\delta$. We have the following observations.
\begin{enumerate}
\item [(a)] The restriction of $h$ to $\HH\setminus (\cup_{j=2}^n B(x_j,\delta))$ is absolutely continuous with respect to the restriction of $h^{\delta}$ to the same domain.
\item [(b)] The paths $\gamma_{\delta}$ and $\gamma'_{\delta}$ are equal by the assumption.
\end{enumerate}
Combining these two facts, we know that, on the event $E_{\delta}\cap E'_{\delta}$, the paths $\gamma$ and $\gamma'$ are equal. This holds for any $\delta>0$. Note that, since $a>-\lambda, b_j>-\lambda$, the points $x_2,...,x_n$ are not continuation thresholds for $\gamma$ and $\gamma'$. Therefore, the distance between $\gamma\cup\gamma'$ and $\{x_2,...,x_n\}$ is almost surely positive. This implies the conclusion.
\end{proof}
\begin{remark}\label{rem::boundary_levellines_reversibility_right}
The conclusion in Lemma \ref{lem::boundary_levellines_reversibility_right} holds more generally when the boundary value of $h$ is piecewise constant, and is at most $-\lambda$ on $\R_-$.
\end{remark}

\begin{figure}[ht!]
\begin{center}
\includegraphics[width=0.3\textwidth]{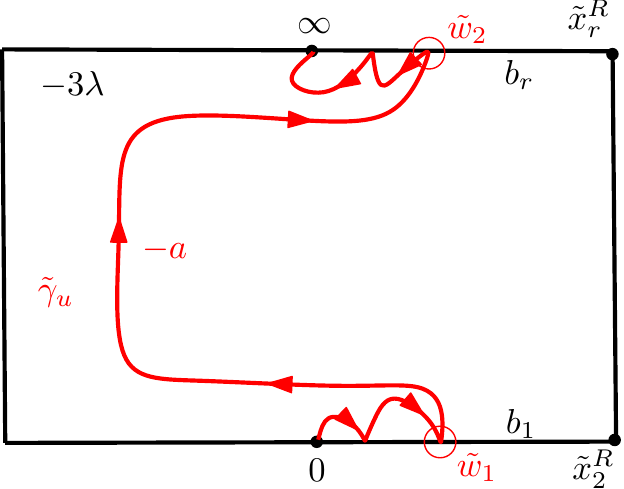}
\end{center}
\caption{\label{fig::boundary_levellines_reversibility_right_oneleft} The explanation of the behaviour of paths in the proof of Lemma \ref{lem::boundary_levellines_reversibility_right_oneleft}.}
\end{figure}

\begin{lemma}\label{lem::boundary_levellines_reversibility_right_oneleft}
The conclusion in Lemma \ref{lem::boundary_levellines_reversibility_right} also holds when we replace the restriction on $a$ by the following: $a>-\lambda$.
\end{lemma}
\begin{proof}
We may assume that $a\in (-\lambda,\lambda)$.
Suppose that
\[\underline{\tilde{x}}^R=(0=\tilde{x}^R_1<\tilde{x}^R_2<\cdots<\tilde{x}^R_r<\tilde{x}^R_{r+1}=\infty)\]
is a sequence of points on $\R_+$. Suppose that $\tilde{h}$ is a GFF on $\HH$ whose boundary value is $-3\lambda$ on $\R_-$, is $b_j$ on $(\tilde{x}_{j},\tilde{x}_{j+1})$ for $j=1,...,r$. Let $\tilde{\gamma}$ be the level line of $\tilde{h}$ starting from 0 targeted at $\infty$ and $\tilde{\gamma}'$ be the level line of $-\tilde{h}$ starting from $\infty$ targeted at $0$. Set $u=\lambda+a\in (0, 2\lambda)$, and let $\tilde{\gamma}_u$ be the level line of $\tilde{h}$ with height $u$ starting from $0$ targeted at $\infty$. See Figure \ref{fig::boundary_levellines_reversibility_right_oneleft}. We summarize the relations between $\tilde{\gamma}, \tilde{\gamma}'$ and $\tilde{\gamma}_u$ in the following.
\begin{enumerate}
\item [(a)] By Lemma \ref{lem::boundary_levellines_reversibility_right}, we have that $\tilde{\gamma}$ and $\tilde{\gamma}'$ are equal.
\item [(b)] By Theorem \ref{thm::boundary_levelline_gff_interacting}, we have that $\tilde{\gamma}_u$ stays to the left of $\tilde{\gamma}$.
\item [(c)] Define $\tilde{E}_u$ to be the event that $\tilde{\gamma}_u$ does not hit the interval $(\tilde{x}^R_2/2,2\tilde{x}^R_r)$. The conditions
\[-\lambda+u<\lambda, \quad b_1+u>-\lambda, \quad b_r+u>-\lambda\] guarantee that $\tilde{E}_u$ has positive probability. On $\tilde{E}_u$, define $C$ to be the connected component of $\HH\setminus\tilde{\gamma}_u$ that has $(\tilde{x}^R_2/2,2\tilde{x}^R_r)$ on the boundary, define $\tilde{w}_1$ to be the last point of $\tilde{\gamma}_u$ that is contained in $[0,\tilde{x}^R_2)$ and $\tilde{w}_2$ to be the first point of $\tilde{\gamma}_u$ that is contained in $(\tilde{x}^R_r,\infty]$, and let $\psi$ be any conformal map from $C$ onto $\HH$ that sends $\tilde{w}_1$ to $0$ and $\tilde{w}_2$ to $\infty$.
\end{enumerate}
Combining these three facts, we have that, given $\tilde{\gamma}_u$ and on $\tilde{E}_u$, the paths $\psi(\tilde{\gamma}|_C)$ and $\psi(\tilde{\gamma}'|_C)$ are equal. For $\delta>0$, define
\[\tilde{F}_{\delta}=\bigcap_{j=2}^{r}\{\psi(\tilde{x}^R_j)\in (x^R_j-\delta,x^R_j+\delta)\}.\]
Clearly, the event $\tilde{F}_{\delta}$ depends on $\tilde{\gamma}_u$, and we can properly choose the sequence $\underline{\tilde{x}}^R$ so that the event $\tilde{E}_u\cap\tilde{F}_{\delta}$ has positive probability. From the above analysis, we know that, given $\tilde{\gamma}_u$ and on $\tilde{E}_u\cap\tilde{F}_{\delta}$, the level line of $\tilde{h}|_C\circ\psi^{-1}$ starting from 0 targeted at $\infty$ coincides with the level line of $-\tilde{h}|_C\circ\psi^{-1}$ starting from $\infty$ targeted at 0. Note that, on $\tilde{E}_u\cap\tilde{F}_{\delta}$, the boundary value of $\tilde{h}|_C\circ\psi^{-1}$ is $-a$ on $\R_-$, and is $b_j$ on $(\psi(\tilde{x}^R_j),\psi(\tilde{x}^R_{j+1}))$ for $j=1,...,r$ where $\psi(\tilde{x}^R_j)\in (x^R_j-\delta,x^R_j+\delta)$. By a similar analysis as in Remark \ref{rem::boundary_levellines_reversibility_right_conditioningargument}, we obtain the conclusion.\end{proof}

\begin{proposition}\label{prop::boundary_levellines_reversibility_leftright}
Suppose that
\[
\underline{x}^L=(x^L_1=0>x^L_2>\cdots>x^L_l>x^L_{l+1}=-\infty),\quad
\underline{x}^R=(x^R_1=0<x^R_2<\cdots<x^R_r<x^R_{r+1}=\infty)\]
are two sequences of points along $\partial\HH$. Suppose that $h$ is a $\GFF$ on $\HH$ whose boundary value is
\[-a_j \quad\text{on }(x^L_{j+1},x^L_j),\quad \text{for }j=1,...,l;\quad b_j \quad \text{on }(x^R_{j},x^R_{j+1}),\quad \text{for }j=1,...,r.\]
Assume that
\[a_j >-\lambda,\quad \text{for }j=1,...,l;\quad b_j >-\lambda,\quad \text{for }j=1,...,r.\]
Let $\gamma$ be the level line of $h$ starting from 0 targeted at $\infty$ and $\gamma'$ be the level line of $-h$ starting from $\infty$ targeted at $0$. Then $\gamma'$ and $\gamma$ are equal almost surely.
\end{proposition}

By Lemma \ref{lem::boundary_levellines_reversibility_right_oneleft}, we know that the conclusion in Proposition \ref{prop::boundary_levellines_reversibility_leftright} holds for $l=1$. Before we prove Proposition \ref{prop::boundary_levellines_reversibility_leftright}, we first prove that the conclusion holds for $l=2$ and $l=3$.

\begin{figure}[ht!]
\begin{subfigure}[b]{0.48\textwidth}
\begin{center}
\includegraphics[width=0.6\textwidth]{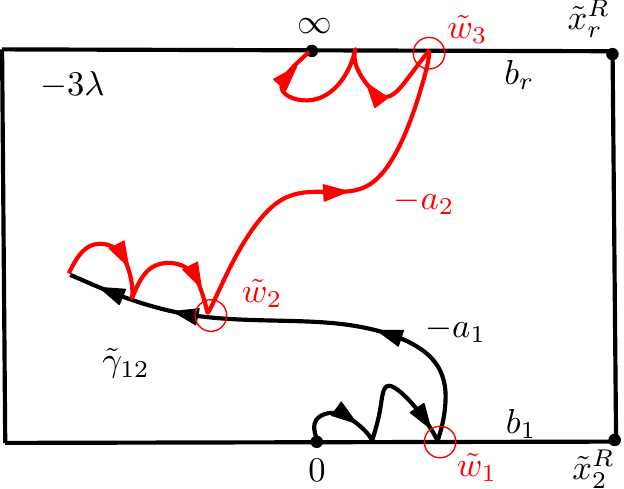}
\end{center}
\caption{$\tilde{\gamma}_{12}$ stays to the left of $\tilde{\gamma}$.}
\end{subfigure}
\begin{subfigure}[b]{0.48\textwidth}
\begin{center}\includegraphics[width=0.65\textwidth]{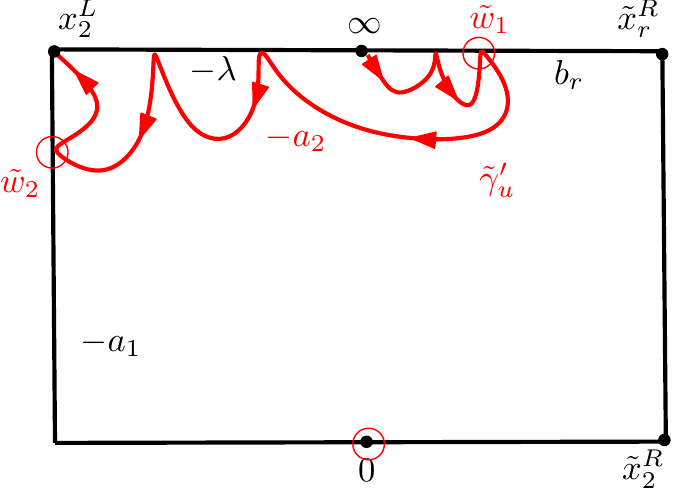}
\end{center}
\caption{$\tilde{\gamma}'_u$ stays to the left of $\tilde{\gamma}'$.}
\end{subfigure}
\caption{\label{fig::boundary_levellines_reversibility_right_twoleft} The explanation of the behaviour of the paths in the proof of Lemma \ref{lem::boundary_levellines_reversibility_right_twoleft}.}
\end{figure}

\begin{lemma}\label{lem::boundary_levellines_reversibility_right_twoleft}
The conclusion in Proposition \ref{prop::boundary_levellines_reversibility_leftright} holds when $l=2$.
\end{lemma}
\begin{proof}
If $a_1\ge\lambda$ and $a_2\ge\lambda$, the conclusion holds by Remark \ref{rem::boundary_levellines_reversibility_right}. In the following, we assume that $a_2\in (-\lambda,\lambda)$.
\smallbreak
\textit{First}, we assume that $a_1\in(-\lambda,\lambda)$ and $a_2\in (-\lambda,\lambda)$.
Suppose that \[\underline{\tilde{x}}^R=(0=\tilde{x}^R_1<\tilde{x}^R_2<\cdots<\tilde{x}^R_r<\tilde{x}^R_{r+1}=\infty)\] is a sequence of points along $\partial\HH$. Suppose that $\tilde{h}$ is a GFF on $\HH$ whose boundary value is $-3\lambda$ on $\R_-$, is $b_j$ on $(\tilde{x}^R_j,\tilde{x}^R_{j+1})$ for $j=1,...,r$. Let $\tilde{\gamma}$ be the level line of $\tilde{h}$ starting from 0 targeted at $\infty$ and $\tilde{\gamma}'$ be the level line of $-\tilde{h}$ starting from $\infty$ targeted at 0. Set
\[u_1=\lambda+a_1\in (0,2\lambda),\quad u_2=\lambda+a_2\in (0,2\lambda).\]
Let $\tilde{\gamma}_{12}$ be the height-varying level line of $\tilde{h}$ starting from 0 targeted at $\infty$ with heights $u_1,u_2$ and height change time $1$. We summarize the relations between $\tilde{\gamma}$, $\tilde{\gamma}'$, and $\tilde{\gamma}_{12}$ in the following. See Figure \ref{fig::boundary_levellines_reversibility_right_twoleft}(a).
\begin{enumerate}
\item [(a)] By Lemma \ref{lem::boundary_levellines_reversibility_right}, we know that $\tilde{\gamma}'$ and $\tilde{\gamma}$ are equal.
\item [(b)] By Corollary \ref{cor::boundary_levellines_monotonicity_heightvarying}, we know that $\tilde{\gamma}_{12}$ stays to the left of $\tilde{\gamma}$.
\item [(c)] Define $\tilde{E}_{12}$ to be the event that $\tilde{\gamma}_{12}$ does not hit the interval $(\tilde{x}^R_2/2,2\tilde{x}^R_r)$ and that $\tilde{\gamma}_{12}[0,1]$ does not hit the interval $(\tilde{x}^R_r,\infty)$. The conditions
    \[b_1+u_1>-\lambda,\quad b_r+u_2>-\lambda\] guarantee that $\tilde{E}_{12}$ has positive probability. On $\tilde{E}_{12}$, define $\tilde{w}_1$ to be the last point of $\tilde{\gamma}_{12}$ that is contained in $[0,\tilde{x}^R_2)$; define $\tilde{w}_2$ to be the last point of $\tilde{\gamma}_{12}|_{(1,\infty)}$ that is contained in $\tilde{\gamma}_{12}[0,1]$ if $u_2<u_1$ and $\tilde{w}_2$ to be $\tilde{\gamma}_{12}(1)$ if $u_2>u_1$; define $\tilde{w}_3$ to be the first point of $\tilde{\gamma}_{12}$ that is contained in $(\tilde{x}^R_r,\infty]$. On $\tilde{E}_{12}$, define $C$ to be the connected component of $\HH\setminus\tilde{\gamma}_{12}$ that has $(\tilde{x}^R_2/2,2\tilde{x}^R_r)$ on the boundary, and define $\psi$ to be the conformal map from $C$ onto $\HH$ that sends $\tilde{w}_1$ to 0, $\tilde{w}_2$ to $x^L_2$, and $\tilde{w}_3$ to $\infty$.
\end{enumerate}
Combining these three facts, we know that, given $\tilde{\gamma}_{12}$ and on $\tilde{E}_{12}$, the paths $\psi(\tilde{\gamma}|_C)$ and $\psi(\tilde{\gamma}'|_C)$ are equal. For $\delta>0$, define
\[
\tilde{F}_{\delta}=\bigcap_{j=2}^r\{\psi(\tilde{x}^R_j)\in (x^R_j-\delta,x^R_j+\delta)\}.
\]
We can properly choose the sequence $\underline{\tilde{x}}^R$ so that $\tilde{E}_{12}\cap\tilde{F}_{\delta}$ has positive probability. From the above analysis, we know that, given $\tilde{\gamma}_{12}$ and on $\tilde{E}_{12}\cap\tilde{F}_{\delta}$, the level line of $\tilde{h}|_C\circ\psi^{-1}$ starting from 0 targeted at $\infty$ coincides with the level line of $-\tilde{h}|_C\circ\psi^{-1}$ starting from $\infty$ targeted at 0.  Note that, on $\tilde{E}_{12}\cap\tilde{F}_{\delta}$, the boundary value of $\tilde{h}|_C\circ\psi^{-1}$ is $-a_2$ on $(-\infty,x^L_2)$, is $-a_1$ on $(x^L_2,0)$, is $b_j$ on $(\psi(\tilde{x}^R_j),\psi(\tilde{x}^R_{j+1}))$ for $j=1,...,r$ where $\psi(\tilde{x}^R_j)\in (x^R_j-\delta,x^R_j+\delta)$. By a similar analysis as in Remark \ref{rem::boundary_levellines_reversibility_right_conditioningargument}, we obtain the conclusion.
\smallbreak
\textit{Next}, we assume that $a_1\ge\lambda$ and $a_2\in(-\lambda,\lambda)$.
Suppose that
\[\underline{\tilde{x}}^R=(0=\tilde{x}^R_1<\tilde{x}^R_2<\cdots<\tilde{x}^R_r<\tilde{x}^R_{r+1}=\infty)\]
is a sequence of points along $\partial\HH$. Suppose that $\tilde{h}$ is a GFF on $\HH$ whose boundary value is $-\lambda$ on $(-\infty,x^L_2)$, is $-a_1$ on $(x^L_2,0)$, and is $b_j$ on $(\tilde{x}^R_j,\tilde{x}^R_{j+1})$ for $j=1,...,r$. Let $\tilde{\gamma}$ be the level line of $\tilde{h}$ starting from 0 targeted at $\infty$ and $\tilde{\gamma}'$ be the level line of $-\tilde{h}$ starting from $\infty$ targeted at 0. Set
\[u=-\lambda-a_2\in (-2\lambda,0).\]
Let $\tilde{\gamma}'_u$ be the level line of $-\tilde{h}$ with height $u$ starting from $\infty$ targeted at $x^L_2$. We summarize the relations between $\tilde{\gamma}$, $\tilde{\gamma}'$, and $\tilde{\gamma}'_u$ in the following. See Figure \ref{fig::boundary_levellines_reversibility_right_twoleft}(b).
\begin{enumerate}
\item [(a)] By Remark \ref{rem::boundary_levellines_reversibility_right}, we know that $\tilde{\gamma}'$ and $\tilde{\gamma}$ are equal.
\item [(b)] By Theorem \ref{thm::boundary_levelline_gff_interacting}, we know that $\tilde{\gamma}'_u$ stays to the left of $\tilde{\gamma}'$.
\item [(c)] Define $\tilde{E}'_u$ to be the event that $\tilde{\gamma}'_u$ reaches $x^L_2$ before hits its continuation threshold and that $\tilde{\gamma}'_u$ does not hit the interval $(x^L_2/2,2\tilde{x}^R_r)$. The conditions
    \[\lambda+u>-\lambda,\quad -b_r+u<\lambda\] guarantee that $\tilde{E}'_u$ has positive probability. On $\tilde{E}'_u$, define $C$ to be the connected component of $\HH\setminus\tilde{\gamma}_u'$ that has $(x^L_2/2,2\tilde{x}^R_r)$ on the boundary, define $\tilde{w}_1$ to be the last point of $\tilde{\gamma}_u'$ that is contained in $(\tilde{x}^R_r,\infty]$ and $\tilde{w}_2$ to be the first point of $\tilde{\gamma}_u'$ that is contained in $(0,x^L_2]$, and let $\psi$ be the conformal map from $C$ onto $\HH$ that sends $0$ to $0$, $\tilde{w}_1$ to $\infty$, and $\tilde{w}_2$ to $x^L_2$.
\end{enumerate}
Combining these three facts, we know that, given $\tilde{\gamma}'_u$ and on $\tilde{E}'_u$, the paths $\psi(\tilde{\gamma}'|_C)$ and $\psi(\tilde{\gamma}|_C)$ are equal.
For $\delta>0$, define
\[
\tilde{F}_{\delta}=\bigcap_{j=2}^r\{\psi(\tilde{x}^R_j)\in (x^R_j-\delta,x^R_j+\delta)\}.
\]
We can properly choose the sequence $\underline{\tilde{x}}^R$ so that $\tilde{E}_u'\cap\tilde{F}_{\delta}$ has positive probability. From the above analysis, we know that, given $\tilde{\gamma}'_u$ and on $\tilde{E}'_u\cap\tilde{F}_{\delta}$, the level line of $\tilde{h}|_C\circ\psi^{-1}$ starting from 0 targeted at $\infty$ coincides with the level line of $-\tilde{h}|_C\circ\psi^{-1}$ starting from $\infty$ targeted at 0.  Note that, on $\tilde{E}'_u\cap\tilde{F}_{\delta}$, the boundary value of $\tilde{h}|_C\circ\psi^{-1}$ is $-a_2$ on $(-\infty,x^L_2)$, is $-a_1$ on $(x^L_2,0)$, is $b_j$ on $(\psi(\tilde{x}^R_j),\psi(\tilde{x}^R_{j+1}))$ for $j=1,...,r$ where $\psi(\tilde{x}^R_j)\in (x^R_j-\delta,x^R_j+\delta)$. By a similar analysis as in Remark \ref{rem::boundary_levellines_reversibility_right_conditioningargument}, we obtain the conclusion.
\end{proof}

\begin{figure}[ht!]
\begin{center}
\includegraphics[width=0.325\textwidth]{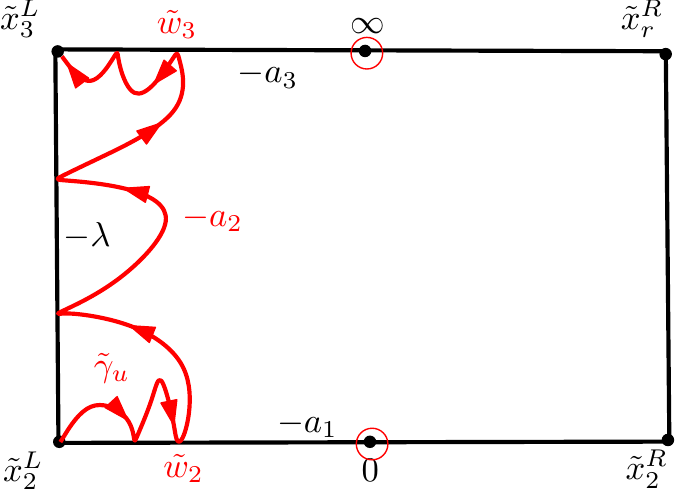}
\end{center}
\caption{\label{fig::boundary_levellines_reversibility_leftright} The explanation of the behaviour of the paths in the proof of Lemma \ref{lem::boundary_levellines_reversibility_right_threeleft}.}
\end{figure}

\begin{figure}[ht!]
\begin{subfigure}[b]{0.3\textwidth}
\begin{center}
\includegraphics[width=\textwidth]{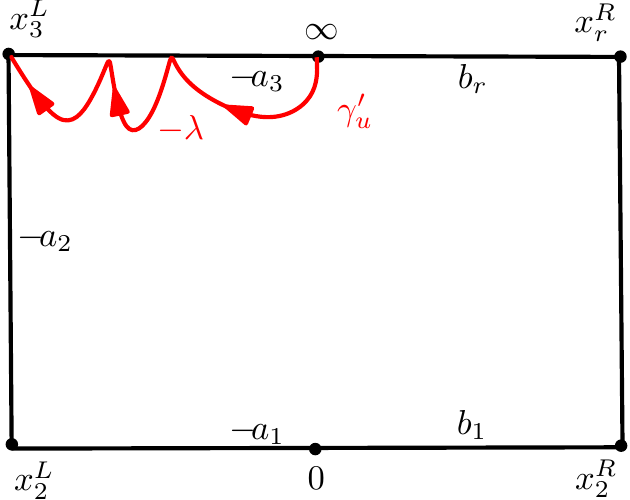}
\end{center}
\caption{$\gamma_u'$ may not hit $[x^L_2,0]$.}
\end{subfigure}
$\quad$
\begin{subfigure}[b]{0.3\textwidth}
\begin{center}\includegraphics[width=\textwidth]{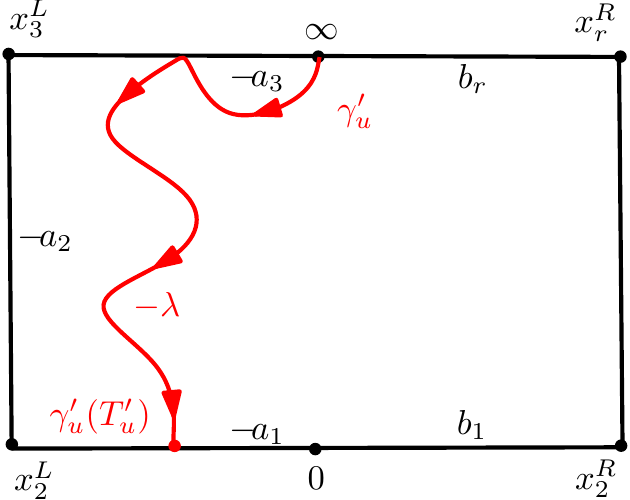}
\end{center}
\caption{$\gamma_u'$ may hit $[x^L_2,0]$.}
\end{subfigure}
$\quad$
\begin{subfigure}[b]{0.3\textwidth}
\begin{center}\includegraphics[width=\textwidth]{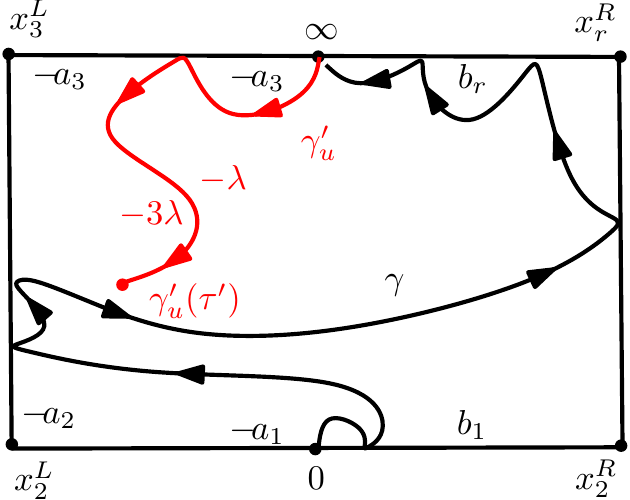}
\end{center}
\caption{$\gamma_u'$ stays to the left of $\gamma$.}
\end{subfigure}
\caption{\label{fig::boundary_levellines_reversibility_right_threeleft} The explanation of the behaviour of paths in the proof of Lemma \ref{lem::boundary_levellines_reversibility_right_threeleft}.}
\end{figure}

\begin{lemma}\label{lem::boundary_levellines_reversibility_right_threeleft}
The conclusion in Proposition \ref{prop::boundary_levellines_reversibility_leftright} holds when $l=3$.
\end{lemma}
\begin{proof} We divide the proof into three cases according to the boundary value: Case 1. $a_2\ge\lambda$; Case 2. $a_2\in (-\lambda,\lambda), a_1<a_2+2\lambda$; Case 3. $a_2\in (-\lambda,\lambda)$ and $a_1,a_3\ge 2\lambda+a_2$. We treat these three cases separately.
\smallbreak
\textit{Case 1.} We assume that $a_2\ge\lambda$. By a similar proof as the second step in the proof of Lemma \ref{lem::boundary_levellines_reversibility_right}, we have that the conclusion holds by Lemma \ref{lem::boundary_levellines_reversibility_right_twoleft}.
\smallbreak
\textit{Case 2.} We assume that $a_2\in (-\lambda,\lambda)$ and $a_1<a_2+2\lambda$. Suppose that
\[
\underline{\tilde{x}}^L=(\tilde{x}^L_1=0>\tilde{x}^L_2>\tilde{x}^L_3>\tilde{x}^L_{4}=-\infty),\quad
\underline{\tilde{x}}^R=(\tilde{x}^R_1=0<\tilde{x}^R_2<\cdots<\tilde{x}^R_r<\tilde{x}^R_{r+1}=\infty)\]
are two sequences of points along $\partial\HH$. Suppose that $\tilde{h}$ is a GFF on $\HH$ whose boundary value is
\[-a_3\quad \text{on }(-\infty,\tilde{x}^L_3),\quad -\lambda \quad \text{on }(\tilde{x}^L_3,\tilde{x}^L_2),\quad -a_1\quad \text{on }(\tilde{x}^L_2,0),\]
and is
\[ b_j \quad \text{on }(\tilde{x}^R_{j},\tilde{x}^R_{j+1}),\quad \text{for }j=1,...,r.\]
Let $\tilde{\gamma}$ be the level line of $\tilde{h}$ starting from 0 targeted at $\infty$ and $\tilde{\gamma}'$ be the level line of $-\tilde{h}$ starting from $\infty$ targeted at 0. Set
\[u=\lambda+a_2\in (0, 2\lambda).\]
Let $\tilde{\gamma}_u$ be the level line of $\tilde{h}$ with height $u$ starting from $\tilde{x}^L_2$ targeted at $\tilde{x}^L_3$. See Figure \ref{fig::boundary_levellines_reversibility_leftright}. We summarize the relations between $\tilde{\gamma}$, $\tilde{\gamma}'$, and $\tilde{\gamma}'_u$ in the following.
\begin{enumerate}
\item [(a)] By Case 1, we have that $\tilde{\gamma}$ and $\tilde{\gamma}'$ are equal.
\item [(b)] By Theorem \ref{thm::boundary_levelline_gff_interacting}, we have that $\tilde{\gamma}$ stays to the right of $\tilde{\gamma}_u$.
\item [(c)] Define $\tilde{E}_u$ to be the event that $\tilde{\gamma}_u$ reaches $\tilde{x}^L_3$ before hits the continuation threshold and that $\tilde{\gamma}_u$ does not hit the union of the intervals $(-\infty,\tilde{x}^L_3-1)\cup(\tilde{x}^L_2+1,0)\cup [0,\infty]$. The conditions
    \[-\lambda+u<\lambda,\quad -a_1+u>-\lambda\] guarantee that $\tilde{E}_u$ has positive probability. On $\tilde{E}_u$, define $C$ to be the connected component of $\HH\setminus\tilde{\gamma}_u$ that has $\R_+$ on the boundary, define $\tilde{w}_2$ to be the last point of $\tilde{\gamma}_u$ that is contained in $[\tilde{x}^L_2,0)$ and $\tilde{w}_3$ to be the first point of $\tilde{\gamma}_u$ that is contained in $(-\infty,\tilde{x}^L_3]$, and let $\psi$ be any conformal map from $C$ onto $\HH$ that sends $0$ to $0$, $\infty$ to $\infty$.
\end{enumerate}
Combining these three facts, we know that, given $\tilde{\gamma}_u$ and on $\tilde{E}_u$, the paths $\psi(\tilde{\gamma}'|_C)$ and $\psi(\tilde{\gamma}|_C)$ are equal.
For $\delta>0$, define
\[
\tilde{F}_{\delta}=\{\psi(\tilde{w}_2)\in (x^L_2-\delta,x^L_2+\delta)\}
\bigcap\{\psi(\tilde{w}_3)\in (x^L_3-\delta,x^L_3+\delta)\}
\bigcap_{j=2}^r\{\psi(\tilde{x}^R_j)\in (x^R_j-\delta,x^R_j+\delta)\}.\]
We can properly choose the sequences $\underline{\tilde{x}}^L, \underline{\tilde{x}}^R$ so that $\tilde{E}_u\cap\tilde{F}_{\delta}$ has positive probability. From the above analysis, we know that, given $\tilde{\gamma}_u$ and on $\tilde{E}_u\cap\tilde{F}_{\delta}$, the level line of $\tilde{h}|_C\circ\psi^{-1}$ starting from 0 targeted at $\infty$ coincides with the level line of $-\tilde{h}|_C\circ\psi^{-1}$ starting from $\infty$ targeted at 0.  Note that, on $\tilde{E}_u\cap\tilde{F}_{\delta}$, the boundary value of $\tilde{h}|_C\circ\psi^{-1}$ is $-a_1$ on $(\psi(\tilde{w}_2),0)$, is $-a_2$ on $(\psi(\tilde{w}_3),\psi(\tilde{w}_2))$, is $-a_3$ on $(-\infty,\psi(\tilde{w}_3))$, and is $b_j$ on $(\psi(\tilde{x}^R_j),\psi(\tilde{x}^R_{j+1}))$ for $j=1,...,r$ where $\psi(\tilde{w}_2)\in (x^L_2-\delta,x^L_2+\delta)$, $\psi(\tilde{w}_3)\in (x^L_3-\delta,x^L_3+\delta)$ and $\psi(\tilde{x}^R_j)\in (x^R_j-\delta,x^R_j+\delta)$. By a similar analysis as in Remark \ref{rem::boundary_levellines_reversibility_right_conditioningargument}, we obtain the conclusion.
\smallbreak
\textit{Case 3.} We assume that  $a_2\in (-\lambda,\lambda), a_1\ge2\lambda+a_2, a_3\ge2\lambda+a_2$.
Set $u=-2\lambda$ and let $\gamma_u'$ be the level line of $-h$ with height $u$ starting from $\infty$ targeted at $x^L_3$.
\smallbreak
\textit{First}, we analyze the behavior of $\gamma_u'$. The conditions
\[a_3+u>-\lambda,\quad -b_r+u<-\lambda,\]
guarantee the existence of $\gamma'_u$. The conditions
\[a_2+u<-\lambda,\quad -b_j+u<-\lambda,\quad \text{for }j=1,...,r,\]
guarantee that $\gamma'_u$ can not hit the union $(x^L_3,x^L_2)\cup(0,\infty)$. There are two possibilities: either $\gamma'_u$ reaches $x^L_3$ without hitting $[x^L_2,0]$, see Figure \ref{fig::boundary_levellines_reversibility_right_threeleft}(a); or $\gamma'_u$ hits $[x^L_2,0]$, see Figure \ref{fig::boundary_levellines_reversibility_right_threeleft}(b). In the former case, define $T'_u$ to be $\infty$; and in the latter case, define $T_u'$ to be the first time that $\gamma_u'$ hits $[x^L_2,0]$.
\smallbreak
\textit{Second}, we argue that $\gamma'_u[0,T'_u]$ stays to the left of $\gamma$. Suppose that $\tau'$ is any $\gamma_u'$-stopping time such that $\tau'<T_u'$. Given $\gamma_u'[0,\tau']$, the conditional law of $\gamma$ is the same as the level line of $h$ restricted to $\HH\setminus\gamma_u'[0,\tau']$ whose boundary value is as depicted in Figure \ref{fig::boundary_levellines_reversibility_right_threeleft}(c). Therefore, given $\gamma_u'[0,\tau']$, the path $\gamma$ can not hit the union $(-\infty,x^L_3)\cup\gamma_u'[0,\tau']$. This implies that $\gamma_u'[0,\tau']$ stays to the left of $\gamma$. This holds for any $\tau'<T_u'$. By the continuity of $\gamma_u'$ and $\gamma$, we know that $\gamma_u'[0,T_u']$ stays to the left of $\gamma$.
\smallbreak
\textit{Finally}, we show that $\gamma'$ and $\gamma$ are equal.
We summarize the relation between $\gamma,\gamma'$ and $\gamma'_u$ as follows.
\begin{enumerate}
\item [(a)] By Theorem \ref{thm::boundary_levelline_gff_interacting}, we have that $\gamma'$ stays to the right of $\gamma'_u$.
\item [(b)] By the above analysis, we have that $\gamma$ stays to the right of $\gamma'_u$.
\item [(c)] Define $C$ to be the connected component of $\HH\setminus\gamma'_u$ that has $\R_+$ on the boundary. Given $\gamma'_u$, the boundary value of $h|_C$ is as depicted in Figure \ref{fig::boundary_levellines_reversibility_right_threeleft}(a) and (b).
\end{enumerate}
Combining these three facts, given $\gamma'_u$, in the case that $\gamma'_u$ does not hit $[x^L_2,0]$ (Figure \ref{fig::boundary_levellines_reversibility_right_threeleft}(a)), the paths $\gamma'$ and $\gamma$ are equal by Case 2.; in the case that $\gamma'_u$ hits $[x^L_2,0]$ (Figure \ref{fig::boundary_levellines_reversibility_right_threeleft}(b)), the paths $\gamma'$ and $\gamma$ are equal by Lemma \ref{lem::boundary_levellines_reversibility_right_twoleft}. This completes the proof.
\end{proof}

\begin{proof}[Proof of Proposition \ref{prop::boundary_levellines_reversibility_leftright}]
We will prove by induction on $l$. Lemmas \ref{lem::boundary_levellines_reversibility_right_oneleft}, \ref{lem::boundary_levellines_reversibility_right_twoleft} and \ref{lem::boundary_levellines_reversibility_right_threeleft} imply that the conclusion holds for $l\le 3$. Assume that the conclusion holds for $l\le m$ for some $m\ge 3$, we will show that the conclusion holds for $l=m+1$.
\smallbreak
\textit{First}, we assume that there exists some $j\in\{2,...,m\}$ such that $a_j\ge\lambda$. By a similar proof as the second step in the proof of Lemma \ref{lem::boundary_levellines_reversibility_right}, we have that the conclusion holds by induction hypothesis.
\smallbreak
\textit{Second}, we assume that there exists $j_0\in\{1,...,m\}$ such that
\[|a_{j_0}-a_{j_0+1}|<2\lambda.\]
By a similar proof of Case 2 in the proof of Lemma \ref{lem::boundary_levellines_reversibility_right_threeleft}, we have that the conclusion holds.
\smallbreak
\textit{Finally}, we point out that, if we are not in the case of the first step, then we have that $a_2, a_3\in (-\lambda,\lambda)$ which implies that $|a_2-a_3|<2\lambda$, therefore we are in the case of the second step. Thus, the above two steps address all cases and complete the proof.
\end{proof}

\begin{proof}[Proof of Theorem \ref{thm::sle_chordal_reversibility}]
Theorem \ref{thm::sle_chordal_reversibility} is a direct consequence of Proposition \ref{prop::boundary_levellines_reversibility_leftright}.
\end{proof}

\begin{proof}[Proof of Theorem \ref{thm::boundary_levellines_targetindependence}]
\textit{First}, we assume that the boundary value of $h$ is at most $-\lambda$ on $(y_2,x)$, is at least $\lambda$ on $(x,y_1)$, and is in $(-\lambda,\lambda)$ on $(y_1,\infty]\cup[\infty,y_2)$. For $i=1,2$, let $\gamma'_i$ be the level line of $-h$ starting from $y_i$ targeted at $x$. We summarize the relations between $\gamma_1,\gamma_2,\gamma_1'$ and $\gamma_2'$ in the following.
\begin{enumerate}
\item [(a)] By Lemma \ref{lem::boundary_levellines_reversibility_right}, we have that $\gamma_1'$ and $\gamma_1$ are equal and that $\gamma_2'$ and $\gamma_2$ are equal.
\item [(b)] By Lemma \ref{lem::boundary_levellines_coincide_threshold}, we have that, given $\gamma_1$, the path $\gamma_2'$ first hits $\gamma_1$ at $\gamma_1(T_1)$ and merges with $\gamma_1$ afterwards.
\end{enumerate}
Combining these two facts, we obtain the conclusion.
\smallbreak
\textit{Next}, for the general boundary value, the conclusion can be proved by similar proof as the proofs of Lemmas \ref{lem::boundary_levellines_monotonicity_threshold_bottom_general} and \ref{lem::boundary_levellines_monotonicity_threshold_bottomleftright}.
\end{proof}

\section{The coupling between $\GFF$ and $\CLE_4$ with time parameter}\label{sec::interior_levellines}
\subsection{Radial $\SLE$}\label{subsec::radial_sle}
%\subsubsection{Radial $\SLE_{\kappa}$, $\SLE_{\kappa}(\rho)$, $\SLE_{\kappa}(\rho^L;\rho^R)$}
The radial Loewner equation describes the evolution of hulls growing from the boundary of the unit disc $\U$ towards the origin. Suppose that $(W_{t}, t\ge 0)$ is a continuous function from $[0,\infty)$ to $\partial\U$. For each $z\in\overline{\U}$, define the function $g_{t}(z)$ to be the solution to \textit{Radial Loewner Equation}
\[\partial_{t}g_{t}(z)=g_{t}(z)\frac{W_{t}+g_{t}(z)}{W_{t}-g_{t}(z)}, \quad g_{0}(z)=z.\]
This is well-defined as long as $W_{t}-g_{t}(z)$ does not hit $0$. Define
\[T(z)=\sup\{t>0: \min_{s\in[0,t]}|W_{s}-g_{s}(z)|>0\}.\]
This is the largest time up to which $g_{t}(z)$ is well-defined.
Set
\[K_{t}=\{z\in\overline{\U}: T(z)\le t\},\quad U_{t}=\U\setminus K_{t}. \]
Then $g_{t}$ is the unique conformal map from $U_{t}$ onto $\mathbb{U}$ such that $g_{t}(0)=0$ and $g'_{t}(0)>0$. In fact, $g'_t(0)=e^t$.
In other words, the time is parameterized by minus the log conformal radius of $U_{t}$ seen from the origin.

The family $(K_{t}, t\ge 0)$ is called the \textbf{radial Loewner chain} driven by $(W_{t}, t\ge 0)$.
\textbf{Radial $\SLE_{\kappa}$} for $\kappa\ge 0$ is the radial Loewner chain driven by $W_{t}=\exp(i\sqrt{\kappa}B_{t})$ where $B$ is one-dimensional Brownian motion.

\begin{proposition}
For all $\kappa\in[0,4]$, radial $\SLE_{\kappa}$ is almost surely generated by a simple continuous curve $\gamma$, i.e. there exists a simple continuous curve $\gamma$ such that $K_{t}=\gamma[0,t]$ for all $t\ge 0$.
\end{proposition}
\begin{proof}
\cite[Proposition 4.2]{LawlerSchrammWernerExponent2}.
\end{proof}

Define
\begin{equation}\label{eqn::radial_sle_auxiliary_definition_1}
\Psi(w,z)=-z\frac{z+w}{z-w},\quad \widetilde{\Psi}(z,w)=\frac{1}{2}(\Psi(z,w)+\Psi(1/\bar{z},w)).\end{equation}
Suppose $\kappa>0$, $\rho\in\R$, and $V_{0}\in\partial\mathbb{U}$. Radial $\SLE_{\kappa}(\rho)$ with force point $V_{0}$ is the Loewner chain driven by $W$ which is the solution to the following SDE:
\begin{equation}\label{eqn::sde_radial_sle_kapparho}
dW_{t}=\mathcal{G}(W_{t}, dB_{t}, dt)+\frac{\rho}{2}\widetilde{\Psi}(V_{t},W_{t})dt, \quad dV_{t}=\Psi(W_{t}, V_{t})dt, \end{equation}
where
\begin{equation}\label{eqn::radial_sle_auxiliary_definition_2}
\mathcal{G}(W_{t}, dB_{t}, dt)=-\frac{\kappa}{2}W_{t}dt+i\sqrt{\kappa}W_{t}dB_{t}.\end{equation}
Note that the processes $W$ and $V$ take values in $\partial\U$.
We say that the process $(W_t,V_t)$ describes the \textbf{radial $\SLE_{\kappa}(\rho)$} process with force point $V_{0}$ and weight $\rho$.

We first explain the existence of the solution $(W_t,V_t)$ to Equation (\ref{eqn::sde_radial_sle_kapparho}) for $\rho>-2$. Set $\theta_{t}=\arg W_{t}- \arg V_{t}$, and assume that $\theta_{0}\in(0,2\pi)$, then $\theta_{t}$ satisfies the SDE:
\begin{equation}\label{eqn::sde_radial_sle_argdifference}
d\theta_{t}=\frac{\rho+2}{2}\cot(\theta_{t}/2)dt+\sqrt{\kappa}dB_{t}.
\end{equation}
The process is well defined up to $T=\inf\{t\ge 0: \theta_{t}\in\{0, 2\pi\}\}$. Note that, as $\theta\rightarrow 0$, we have $\cot(\theta/2)=(2/\theta)+O(\theta)$. Hence, when $\theta_{t}$ is close to 0, the evolution of $\theta_{t}$ is absolutely continuous with respect to $\sqrt{\kappa}$ times a Bessel process of dimension $d=1+2(\rho+2)/\kappa>1$. Similarly, when $\theta_{t}$ is close to $2\pi$, the process $2\pi-\theta_{t}$ is absolutely continuous with respect to $\sqrt{\kappa}$ times a Bessel process of the same dimension.

Note that
\[dV_{t}=-V_{t}\cot(\frac{\theta_{t}}{2})dt,\quad W_{t}=e^{i\theta_{t}}V_{t}.\]
Thus the existence and uniqueness of the solution to Equation (\ref{eqn::sde_radial_sle_argdifference}) guarantees the existence and uniqueness of the solution $(W_{t}, V_{t})$ to Equation (\ref{eqn::sde_radial_sle_kapparho}).

\begin{comment}
The radial $\SLE_{\kappa}(\rho)$ process has the domain Markov property.
\begin{proposition}
Let $K_{t}$ be a radial $\SLE_{\kappa}(\rho)$ process, $(g_{t})$ be the corresponding family of conformal maps, and $(W_t,V_t)$ be the driving process.
Suppose that $\tau$ is any almost surely finite stopping time for $K_{t}$. Then $g_{\tau}(K_{t}\setminus K_{\tau})$ is a radial $\SLE_{\kappa}(\rho)$ process whose driving function $(\tilde{W}, \tilde{V})$ has initial condition $(\tilde{W}_{0}, \tilde{V}_{0})=(W_{\tau}, V_{\tau})$.
\end{proposition}
\begin{proof}
\cite[Proposition 2.2]{MillerSheffieldIG4}.
\end{proof}
\end{comment}

Next we explain the geometric meaning of the process $(W_{t}, V_{t})$: $W_{t}$ is the image of the tip of $K_{t}$ under $g_{t}$. For $V_{t}$, there are two different cases: if $V_{0}\not\in K_{t}$, then $V_{t}$ is the image of $V_{0}$ under $g_{t}$; if $V_{0}\in K_{t}$, then $V_{t}$ is the image of the last point on the boundary that $K$ hits by time $t$ under $g_{t}$. See Figure \ref{fig::radial_sle_kapparho_explanation}.

\begin{figure}[ht!]
\begin{subfigure}[b]{0.47\textwidth}
\begin{center}
\includegraphics[width=\textwidth]{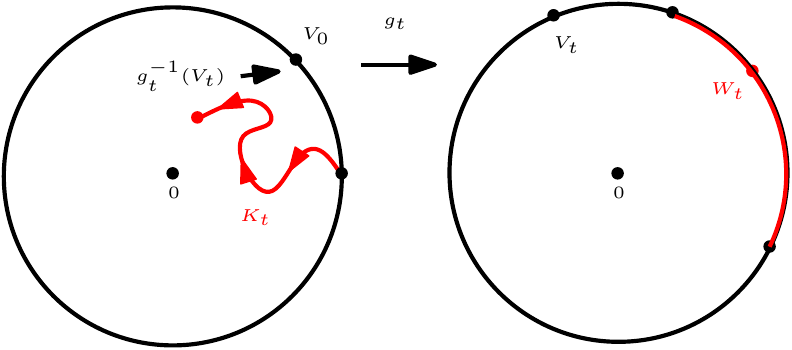}
\end{center}
\caption{If $V_0\not\in K_t$, then $V_t$ is the image of $V_0$ under $g_t$. If $\rho\ge \kappa/2-2$, this holds for all time.}
\end{subfigure}
$\quad$
\begin{subfigure}[b]{0.47\textwidth}
\begin{center}\includegraphics[width=\textwidth]{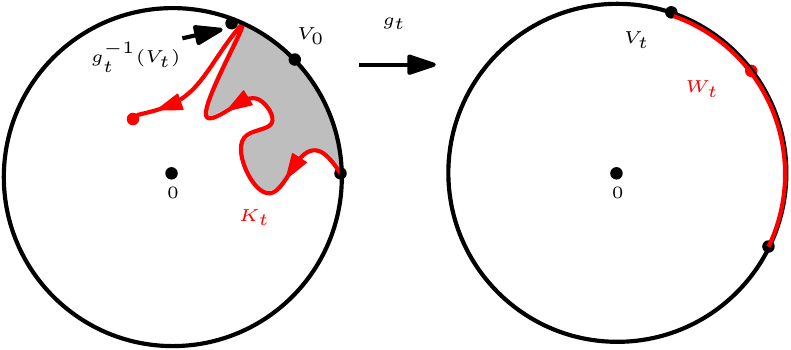}
\end{center}
\caption{If $V_0\in K_t$, then $V_t$ is the image of the last point in $\partial\U\cap K_t$ under $g_t$.}
\end{subfigure}
\caption{\label{fig::radial_sle_kapparho_explanation} The geometric meaning of $(W_t, V_t)$ in radial $\SLE_{\kappa}(\rho)$ process.}
\end{figure}

The relation between radial $\SLE_{\kappa}(\rho)$ for different $\rho$'s is as follows: Suppose that $\kappa> 0$, $\rho>-2$, $\tilde{\rho}>-2$, $V_{0}\in\partial\mathbb{U}$. Let $(K_{t}, t\ge 0)$ be the radial Loewner chain corresponding to the radial $\SLE_{\kappa}(\rho)$ process with force point $V_{0}$, and $(g_{t}, t\ge 0)$ be the corresponding family of conformal maps. Define
\[M_{t}=g_{t}'(0)^{(\tilde{\rho}-\rho)(\tilde{\rho}+\rho+4)/(8\kappa)}\times|g_{t}(V_{0})-W_{t}|^{(\tilde{\rho}-\rho)/\kappa}\times|g_{t}'(V_{0})|^{(\tilde{\rho}-\rho)(\tilde{\rho}+\rho+4-\kappa)/(4\kappa)}.\]
Then $M_{t}$ is well-defined up to the first time $W_{t}$ collides with $V_{t}$, and is a local martingale. The law of $K$ weighted by $M$ is the same as that of a radial $\SLE_{\kappa}(\widetilde{\rho})$ process with force point $V_{0}$ as long as one stops at a bounded stopping time that occurs before $W_{t}$ gets within some fixed distance of $V_{t}$.

One can also consider the radial $\SLE_{\kappa}(\rho^{L};\rho^{R})$ with two boundary force points $V_{0}^{L}, V_{0}^{R}\in\partial\mathbb{U}$.
\begin{definition}\label{def::radial_sle_twoforcepoints}
Fix $\kappa> 0$, $\rho^{L},\rho^{R}\in\mathbb{R}$, and three boundary points $V^{L}_{0}, W_{0}, V^{R}_{0}$ which are located on $\partial\mathbb{U}$ in counterclockwise order. Let $B_{t}$ be a standard Brownian motion. We will say that the process $(W_{t}, V^{L}_{t}, V^{R}_{t})$ describes a \textbf{radial $\SLE_{\kappa}(\rho^{L};\rho^{R})$ process with force points $(V_{0}^{L}; V_{0}^{R})$} if they are adapted to the filtration of $B$ and the following hold:
\begin{enumerate}
\item [(1)] The processes $W_{t}, V^{L}_{t}, V^{R}_{t}$ and $B_{t}$ satisfy the following SDE system on the time intervals on which $W_{t}$ does not collide with any of $V^{L}_{t}$, $V^{R}_{t}$.
\[dW_{t}=\mathcal{G}(W_{t}, dB_{t}, dt)+\frac{\rho^{L}}{2}\widetilde{\Psi}(V^{L}_{t},W_{t})dt+\frac{\rho^{R}}{2}\widetilde{\Psi}(V^{R}_{t},W_{t})dt\]
\[dV^{L}_{t}=\Psi(W_{t}, V^{L}_{t})dt,\quad dV^{R}_{t}=\Psi(W_{t}, V^{R}_{t})dt\]
where the functions $\Psi,\widetilde{\Psi},\mathcal{G}$ are defined as in Equations (\ref{eqn::radial_sle_auxiliary_definition_1}) and (\ref{eqn::radial_sle_auxiliary_definition_2}).
\item [(2)] We have instantaneous reflection of $W_{t}$ off of the $V^{L}_{t}$ and $V^{R}_{t}$.
\item [(3)] We also have almost surely that
\[V^{L}_{t}=V_{0}^{L}+\int_{0}^{t}\Psi(W_{s}, V^{L}_{s})ds,\quad V^{R}_{t}=V_{0}^{R}+\int_{0}^{t}\Psi(W_{s}, V^{R}_{s})ds.\]
\end{enumerate}
\end{definition}

Define the \textbf{continuation threshold} for the process $(W_{t}, V^{L}_{t}, V^{R}_{t})$:
\begin{itemize}
\item if $\rho^{L}\le -2<\rho^{R}$, the continuation threshold is the infimum of $t$ for which $W_{t}=V^{L}_{t}$;
\item if $\rho^{R}\le -2<\rho^{L}$, the continuation threshold is the infimum of $t$ for which $W_{t}=V^{R}_{t}$;
\item if $\rho^{L}\le -2, \rho^{R}\le -2$, the continuation threshold is the infimum of $t$ for which either $W_{t}=V^{L}_{t}$ or $W_{t}=V^{R}_{t}$;
\item if $\rho^{L}>-2,\rho^{R}>-2$, $\rho^{L}+\rho^{R}\le -2$, the continuation threshold is the infimum of $t$ for which $V^L_t=W_t=V^R_t$;
\item if $\rho^{L}>-2,\rho^{R}>-2,\rho^{L}+\rho^{R}>-2$, the continuation threshold is never reached.
\end{itemize}

\begin{lemma}\label{lem::radial_sle_twoforcepoints_existence}
In the setting of Definition \ref{def::radial_sle_twoforcepoints}, the joint law of $(B_{t}, W_{t}, V^{L}_{t}, V^{R}_{t})$ is uniquely determined up to the continuation threshold. Under this law, $(B_{t}, W_{t}, V^{L}_{t}, V^{R}_{t})$ is a continuous multidimensional Markovian process indexed by $t$.
\end{lemma}
\begin{proof}
This can be proved by a similar proof as the proof of Proposition  \ref{prop::chordal_sle_existence_uniqueness}.
\begin{comment} 
We only need to point out the martingale in Case 2. Suppose that $V^L_0=W_0\neq V^R_0$ and $\rho^L>-2$. Suppose that $(W_t,V^L_t)$ describes radial $\SLE_{\kappa}(\rho^L)$ process with force point $V^L_0$. Let $(g_t,t\ge 0)$ be the corresponding sequence of the conformal maps of the radial Loewner chain. Define
\[M_t=g_t'(0)^{\rho^R(2\rho^L+\rho^R+4)/(8\kappa)}\times g_t'(V_0^R)^{\rho^R(\rho^R+4-\kappa)/(4\kappa)}\times |g_t(V_0^R)-W_t|^{\rho^R/\kappa}\times |g_t(V_0^R)-V^L_t|^{\rho^L\rho^R/(2\kappa)}.\]
The process $M$ is well-defined up to the first time $t$ that $W_t$ collides with $g_t(V^R_0)$, and $M$ is a local martingale. The law of $(W_t,V_t^L,V^R_t:=g_t(V^R_0))$ weighted by $M$ describes radial $\SLE_{\kappa}(\rho^L;\rho^R)$ process with force points $(V^L_0;V^R_0)$.
\end{comment}
\end{proof}
\begin{comment}
Next, we explain the geometric meaning of the process $(V^L_t,W_t,V^R_t)$: $W_t$ is the image of the tip of $K_t$ under $g_t$. For $V^L_t,V^R_t$, there are several different cases. We explain one possible scenario in Figure \ref{fig::radial_sle_twoforcepoints_explanation}, and leave the other cases to interested readers.

\begin{figure}[ht!]
\begin{center}
\includegraphics[width=0.47\textwidth]{figures/radial_sle_twoforcepoints_explanation}
\end{center}
\caption{\label{fig::radial_sle_twoforcepoints_explanation} The geometric meaning of $(W_t, V_t^L, V^R_t)$ in radial $\SLE_{\kappa}(\rho^L;\rho^R)$ process.}
\end{figure}
\end{comment}
%\subsubsection{Relation between radial $\SLE$ and chordal $\SLE$}

There is a close relation between radial $\SLE$ and chordal $\SLE$. To describe this relation, we need to introduce chordal $\SLE$ with interior force point. We will not address the most general case, we only
introduce the process that we will use in the current paper: chordal $\SLE$ with two boundary force points and one interior force point.

\begin{definition}
Fix $\kappa\ge 0$, $\rho^L,\rho^R,\rho^I\in\R$, and $x^L\le 0\le x^R$, $z\in\HH$. Chordal $\SLE_{\kappa}(\rho^L;\rho^R;\rho^I)$ with force points $(x^L;x^R;z)$ is the chordal Loewner chain driven by $W$ which is the solution to the following SDE:
\[dW_t=\sqrt{\kappa}dB_t+\frac{\rho^Ldt}{W_t-V^L_t}+\frac{\rho^Rdt}{W_t-V^R_t}+\Re\frac{\rho^I}{W_t-V^I_t}dt,\]
\[dV^L_t=\frac{2dt}{V^L_t-W_t},\quad dV^R_t=\frac{2dt}{V^R_t-W_t},\quad dV^I_t=\frac{2dt}{V^I_t-W_t},\quad W_0=0, V^L_0=x^L, V^R_0=x^R, V^I_0=z.\]
\end{definition}

We can define chordal $\SLE$ in the unit disc via conformal transformations. Fix four boundary points $V^L_0,W_0,V^R_0,V^{\infty}$ along $\partial\U$ which are located in counterclockwise order. Define $\phi$ to be the conformal map from $\HH$ onto $\U$ such that $\phi(0)=W_0,\phi(\infty)=V^{\infty}$.
Define chordal $\SLE_{\kappa}(\rho^L;\rho^R;\rho^I)$ in $\U$ starting from $W_0$ targeted at $V^{\infty}$ with force points $(V^L_0;V^R_0;0)$ to be the image of chordal $\SLE_{\kappa}(\rho^L;\rho^R;\rho^I)$ in $\HH$ with force points $(\phi^{-1}(V^L_0);\phi^{-1}(V^R_0); \phi^{-1}(0))$ under the conformal map $\phi$.
Chordal $\SLE$ with interior force point and radial SLE are closely related.

\begin{lemma}\label{lem::sle_chordal_radial_equivalence}
Fix $\kappa\ge 0$, $\rho^L,\rho^R,\rho^I\in\R$. Fix four boundary points $V^L_0,W_0,V^R_0,V^{\infty}$ along $\partial\U$ which are located in counterclockwise order. Let $\gamma_1$ be the chordal $\SLE_{\kappa}(\rho^L;\rho^R;\rho^I)$ process in $\U$ starting from $W_0$ targeted at $V^{\infty}$ with force points $(V^L_0;V^R_0;0)$. Define $\tau_1$ to be the first time that $V^{\infty}$ and the origin are disconnected by $\gamma_1$. Let $\gamma_2$ be the radial $\SLE_{\kappa}(\rho^L;\rho^R)$ starting from $W_0$ with force points $(V^L_0;V^R_0)$. Define $\tau_2$ be the first time that $V^{\infty}$ and the origin are disconnected by $\gamma_2$. Assume that
\[\rho^L+\rho^R+\rho^I=\kappa-6.\]
Then the path $\gamma_1$ stopped at $\tau_1$ has the same law as the path $\gamma_2$ stopped at $\tau_2$ (up to time change).
\end{lemma}

\begin{proof}
Combining Proposition \ref{prop::chordal_sle_existence_uniqueness}, Lemma \ref{lem::radial_sle_twoforcepoints_existence} and \cite[Theorem 3]{SchrammWilsonSLECoordinatechanges}.
\end{proof}

\subsection{Conformal Loop Ensemble}\label{subsec::cle}
\subsubsection{Carath\'eodory topology}
We say that a sequence of functions $f_n$ on a domain $D$ converges to $f$ \textit{uniformly on compact sets} if, for any compact $K\subset D$, the functions $f_n$ converge to $f$ uniformly on $K$.

Suppose that $(D_n,n\ge 1)$ is a sequence of simply connected domains other than $\C$ containing the origin, and let $f_n$ be the conformal map from $\U$ onto $D_n$ such that $f_n(0)=0,f_n'(0)>0$. We define \textit{convergence in Carath\'eodory topology} as follows:
\begin{enumerate}
\item [(1)] $D_n$ converges to $\C$ if $f_n'(0)\to\infty$.
\item [(2)] $D_n$ converges to $\{0\}$ if $f_n'(0)\to 0$.
\item [(3)] $D_n$ converges to $D$ (which is simply connected other than $\C$) if $f_n$ converges to $f$ uniformly on compact sets where $f$ is the conformal transformation from $\U$ onto $D$ such that $f(0)=0,f'(0)>0$.
\end{enumerate}
If $(D_n,n\ge 1)$ is a sequence of simply connected domains other than $\C$ containing some fixed $z$. We say that $D_n$ converges to $D$ in Carath\'eodory topology seen from $z$ if $D_n-z$ converges to $D-z$ in the above sense.

\begin{comment}
\begin{proposition}\label{prop::caratheodory_kernel}
Suppose that $D_n$ and $D$ are simply connected domains other than $\C$ containing the origin. Then $D_n$ converges to $D$ in Carath\'eodory topology if and only if the following holds. For each subsequence $\{n_j\}$, call the kernel of the subsequence the largest domain $\tilde{D}$ containing the origin such that, for all compact $K\subset\tilde{D}$, we have $K\subset D_{n_j}$ for all but finitely many $j$. Then the kernel of every subsequence is $D$.
\end{proposition}
\begin{proof}
\cite[Proposition 3.63]{LawlerConformallyInvariantProcesses}.
\end{proof}
\end{comment}

\begin{lemma}\label{lem::caratheodory_decreasing}
Suppose that $D_n$ and $D$ are simply connected domains other than $\C$ containing the origin. Assume that $(D_n,n\ge 1)$ is decreasing:
\[D_{n+1}\subseteq D_n,\quad \text{for all } n\ge 1;\]
and that
\[D\subseteq D_n,\quad \text{for all }n\ge 1.\]
Then $(D_n,n\ge 1)$ converges to $D$ in Carath\'eodory topology if and only if the sequence of conformal radius $(\CR(D_n),n\ge 1)$ converges to $\CR(D)$.
\end{lemma}
\begin{proof}
We only need to show that the convergence in conformal radii implies the convergence in Carath\'eodory topology.
Let $f_n$ be the conformal map from $\U$ onto $D_n$ such that $f_n(0)=0,f_n'(0)>0$, and $f$ be the conformal map from $\U$ onto $D$ such that $f(0)=0, f'(0)>0$.

From \cite[Proposition 3.61]{LawlerConformallyInvariantProcesses}, we know that there exist a conformal map $\tilde{f}$ from $\U$ onto some $\tilde{D}=\tilde{f}(\U)$ and a subsequence $f_{n_j}$ such that $f_{n_j}$ converges to $\tilde{f}$ uniformly on compact sets. Therefore, the sequence $D_{n_j}$ converges to $\tilde{D}$ in Carath\'eodory topology.
Furthermore $f_{n_j}'(0)\to \tilde{f}'(0)$. We have the following observations.
\begin{enumerate}
\item [(a)] Since $D\subseteq D_n$ for all $n$, we have that $D\subseteq\tilde{D}$.
\item [(b)] Since $\CR(D_n)\to \CR(D)$ as $n\to\infty$, we have that $f'(0)=\tilde{f}'(0)$.
\end{enumerate}
Combining these two facts, we have that $D=\tilde{D}$. Therefore, $D_{n_j}$ converges to $D$ in Carath\'eodory topology. Since the sequence $(D_n,n\ge 1)$ is decreasing, we have that $D_n$ converges to $D$ in Carath\'eodory topology.
\end{proof}

\subsubsection{$\CLE$ and the exploration process}\label{subsubsec::cle_exploration}
In this section, we recall some features of CLE and we refer to \cite{SheffieldWernerCLE} for details and proofs of the statements. CLE in $\U$ is a collection $\Gamma$ of non-nested disjoint simple loops $(\gamma_j,j\in J)$ in $\U$ that possesses the following two properties.
\begin{enumerate}
\item [(1)][Conformal Invariance] For any M\"obius transformation $\Phi$ of $\U$ onto itself, the laws of $\Gamma$ and  $\Phi(\Gamma)$ are the same. This makes it possible to define, for any non-trivial simply connected domain $D$ (that can therefore be viewed as the conformal image of $\U$ via some map $\tilde \Phi$), the law of CLE in $D$ as the distribution of $\tilde \Phi (\Gamma)$ (because this distribution does not depend on the actual choice of conformal map $\tilde \Phi$ from $\U$ onto $D$).
\item [(2)][Domain Markov Property] For any non-trivial simply connected domain $D \subset \U$, define the set $D^* = D^* (D, \Gamma)$ obtained by removing from $D$ all the loops (and their interiors) of $\Gamma$ that do not entirely lie in $D$. Then, conditionally on $D^*$, and for each connected component $U$ of $D^*$, the law of those loops of $\Gamma$ that do stay in $U$ is exactly that of a CLE in $U$.
\end{enumerate}
The loops in CLE are $\SLE_{\kappa}$-type loops for some $\kappa\in (8/3,4]$. In fact, for each such value of $\kappa$, there exists exactly one CLE distribution that has $\SLE_{\kappa}$-type loops. We denote the corresponding CLE by $\CLE_{\kappa}$ for $\kappa\in (8/3,4]$.

We call $l$ a bubble in $\U$ if $l\subset \overline{\U}$ is homeomorphic to the unit circle and $l\cap\partial\U$ contains exactly one point; we call the point in $l\cap\partial\U$ the root of $l$, denoted by $R(l)$.

In \cite{SheffieldWernerCLE}, the authors introduce a discrete exploration process of CLE loop configuration. The conformal invariance and the domain Markov property make the discrete exploration easy to control. Consider a CLE in $\U$, draw a small disc $B(x,\eps)$ with center $x\in\partial\U$, let $l^{\eps}$ be the loop that intersects $B(x,\eps)$ with largest radius. Define the quantity
\[u(\eps)=\PP[l^{\eps}\text{ surrounds the origin}].\]
In fact, $u(\eps)=\eps^{\beta+o(1)}$ as $\eps$ goes to zero where $\beta=8/\kappa-1$.
\begin{lemma}\label{lem::cle_discrete_bubble}
The law of $l^{\eps}$ normalized by $1/u(\eps)$ converges towards a limit measure, denoted by $M(x)$.
\end{lemma}
\begin{proof}
\cite[Section 4]{SheffieldWernerCLE}.
\end{proof}

\begin{figure}[ht!]
\begin{subfigure}[b]{0.64\textwidth}
\includegraphics[width=\textwidth]{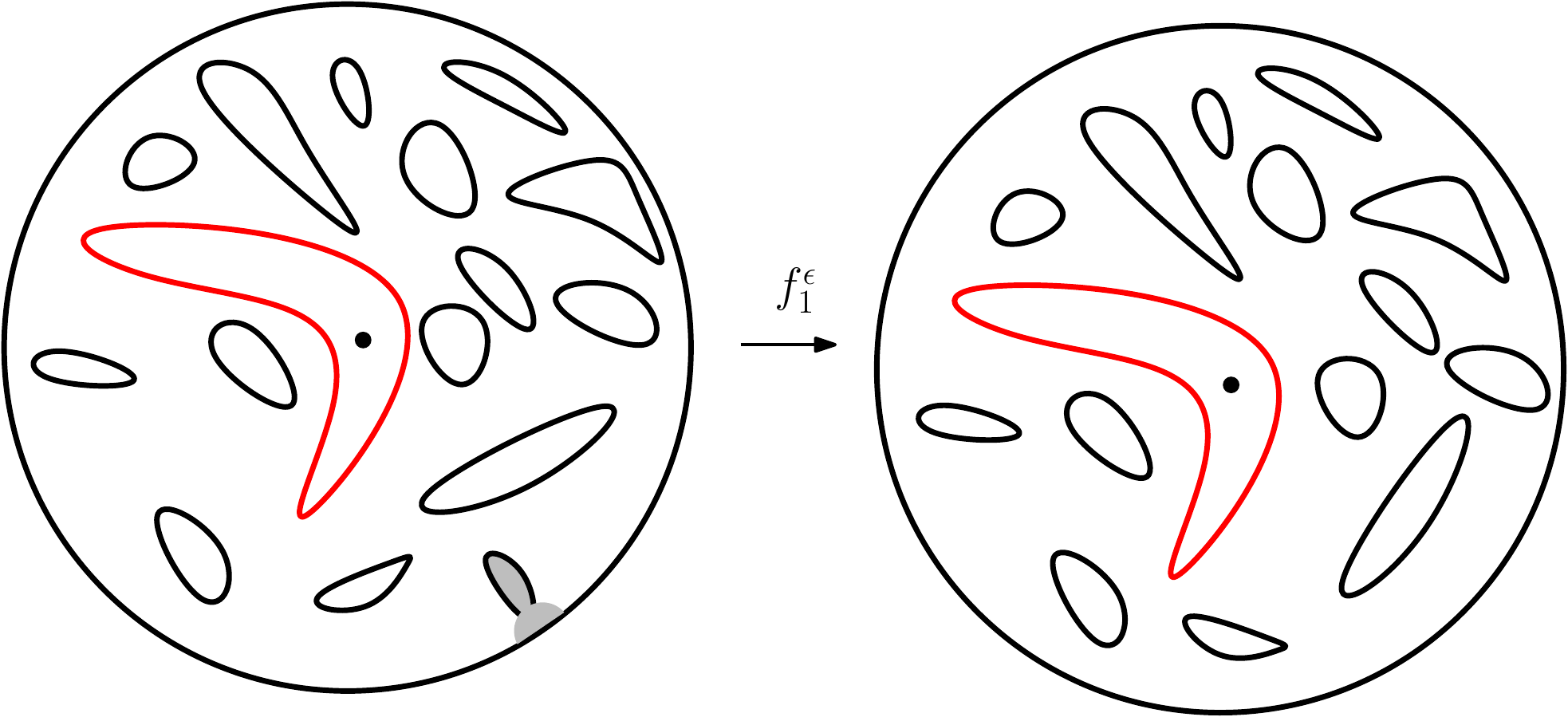}
\caption{If we do not discover the loop surrounding the origin, define $f_1^{\eps}$ to be the conformal map from the to-be-explored domain onto the unit disc such that $f_1^{\eps}(0)=0$, $(f_1^{\eps})'(0)>0$.}
\end{subfigure}
$\quad$
\begin{subfigure}[b]{0.32\textwidth}
\includegraphics[width=0.9\textwidth]{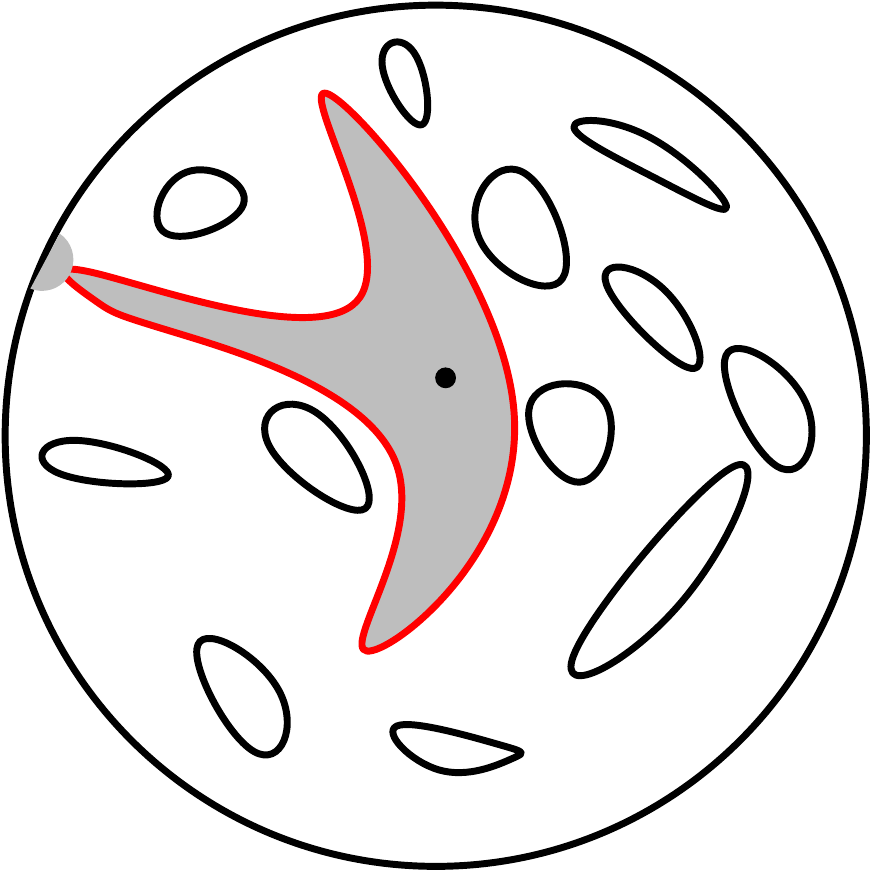}
\caption{After a finite number of steps, we will discover the loop surrounding the origin.}
\end{subfigure}
\caption{\label{fig::discrete_exploration_cle} Explanation of the discrete exploration process of CLE.}
\end{figure}

Because of the conformal invariance and the domain Markov property, we can repeat the ``small semi-disc exploration" until we discover the loop containing the origin: Suppose we have a CLE loop configuration in the unit disc $\U$. We draw a small semi-disc of radius $\eps$ whose center is uniformly chosen on the unit circle. The loops that intersect this small semi-disc are the loops discovered. If we do not discover the loop surrounding the origin, we refer to the connected component of the remaining domain that contains the origin as the \textit{to-be-explored domain}. Let $f_1^{\eps}$ be the conformal map from the to-be-explored domain onto $\U$ such that $f^{\eps}_1(0)=0, (f^{\eps}_1)'(0)>0$.
We also define $l_1^{\eps}$ to be the loop discovered with largest radius. Because of the conformal invariance and the domain Markov property of CLE, the image of the loops in the to-be-explored domain under the conformal map $f_1^{\eps}$ has the same law as simple CLE in the unit disc. Thus we can repeat the same procedure for the image of the loops under $f_1^{\eps}$. We draw a small semi-disc of radius $\eps$ whose center is uniformly chosen on the unit circle. The loops that intersect the small semi-disc are the loops discovered at the second step. If we do not discover the loop surrounding the origin, define the conformal map $f_2^{\eps}$ from the to-be-explored domain onto $\U$ such that $f^{\eps}_2(0)=0, (f^{\eps}_2)'(0)>0$. The image of the loops in the to-be-explored domain under $f_2^{\eps}$ has the same law as CLE in the unit disc, etc. At some finite step $N$, we discover the loop surrounding the origin, we define $l^{\eps}_{N}$ to be the loop surrounding the origin discovered at this step and stop the exploration. We summarize the properties and notations in this discrete exploration below. See Figure \ref{fig::discrete_exploration_cle}.
\begin{enumerate}
\item [(a)] Before $N$, all steps of discrete exploration are i.i.d.
\item [(b)] The number of the step $N$, when we discover the loop surrounding the origin, has the geometric distribution:
\[\PP[N>n]=\PP[l^{\eps} \text{does not surround the origin}]^n=(1-u(\eps))^n.\]
\item [(c)] Define the conformal map
\[\Phi^{\eps}=f^{\eps}_{N-1}\circ \cdots\circ f_2^{\eps}\circ f_1^{\eps}.\]
\end{enumerate}
As $\eps$ goes to zero, the discrete exploration will converge to a Poisson point process of bubbles with intensity measure given by
\[M=\int_{\partial\U}dx M(x)\]
where $dx$ is Lebesgue length measure on $\partial\U$.

Now we can reconstruct CLE loops from the Poisson point process of SLE bubbles. Let $(l_t,t\ge 0)$ be a Poisson point process with intensity $M$.
Namely, let $((l_j,t_j),j\in J)$ be a Poisson point process with intensity $M\times[0,\infty)$, and then arrange the bubble according to the time $t_j$, i.e. denote $l_t$ as the bubble $l_j$ if $t=t_j$, and $l_t$ is empty set if there is no $t_j$ that equals $t$. Clearly, there are only countably many bubbles in $(l_t,t\ge 0)$ that are not empty set.
Define
\[\tau=\inf\{t: l_t\text{ surrounds the origin}\}.\]
For each $t<\tau$, the bubble $l_t$ does not surround the origin. Define $f_t$ to be the conformal map from the connected component of $\U\setminus l_t$ containing the origin onto the unit disc such that $f_t(0)=0,f_t'(0)>0$. For this Poisson point process, we have the following properties:
\begin{enumerate}
\item [(a)] $\tau$ has the exponential law: $\PP[\tau>t]=e^{-t}$.
\item [(b)] For $\delta>0$ small, let $t_1(\delta),t_2(\delta), ...,t_j(\delta)$ be the times $t$ before $\tau$ at which the bubble $l_t$ has radius greater than $\delta$. Define $\Psi^{\delta}=f_{t_j(\delta)}\circ\cdots\circ f_{t_1(\delta)}$. Then $\Psi^{\delta}$ almost surely converges towards some conformal map $\Psi$ in Carath\'eodory topology seen from the origin  as $\delta$ goes to zero. Therefore $\Psi$ can be interpreted as $\Psi=\circ_{t<\tau}f_t$.
\item [(c)] Generally, for each $t\le\tau$, we can define $\Psi_t=\circ_{s<t}f_s$. Then
\[(L_t:=\Psi^{-1}_t(l_t),0\le t\le\tau)\] is a collection of loops in the unit disc and $L_{\tau}$ is a loop surrounding the origin.
\end{enumerate}

The relation between this Poisson point process of bubbles and the discrete exploration process we described above is given via the following result.
\begin{proposition}\label{prop::caratheodory_convergence}
$\Phi^{\eps}$ converges in distribution to $\Psi$ in Carath\'eodory topology seen from the origin.
Moreover, the loop $L_{\tau}$ has the same law as the loop of $\CLE$ in $\U$ that surrounds the origin.
\end{proposition}
\begin{proof}
\cite[Section 7]{SheffieldWernerCLE}.
\end{proof}

The results in this section hold for all $\kappa \in (8/3,4]$. In the next section, we will point out a particular property that only holds for $\kappa=4$.

\subsubsection{$\CLE_4$ with time parameter}\label{subsubsec::cle4_timeparameter}
In this section, we will introduce $\CLE_4$ with time parameter. We refer to \cite{WernerWuCLEExploration} for the details and proofs of the statements.
Throughout this section, we fix $\kappa=4$.
Recall that $M(x)$ is the SLE bubble measure rooted at $x\in\partial\U$ defined in Lemma \ref{lem::cle_discrete_bubble}, and $M$ is defined by
\[M=\int_{x\in\partial\U}dx M(x).\]
The following property only holds for $\kappa=4$ and it is the most important ingredient in the construction of $\CLE_4$ with time parameter.

\begin{lemma}\label{lem::sle4_bubble_conformalinvariance}
When $\kappa=4$, the bubble measure $M$ is conformally invariant: for any M\"obius transformation $\phi$ of $\U$, we have
\[\phi\circ M=M.\]
\end{lemma}
\begin{proof}
\cite[Lemma 6]{WernerWuCLEExploration}.
\end{proof}
We call $M$ the \textbf{$\SLE_4$ bubble measure} uniformly rooted over the boundary.
Let $(l_t,t\ge 0)$ be a Poisson point process with intensity $M$.
The sequence of loops $(L_t,0\le t\le \tau)$ is obtained from $(l_t,t\ge 0)$ by targeting at the origin described in Section \ref{subsubsec::cle_exploration}.
In fact, we can construct a sequence of loops $(L^z_t, 0\le t\le \tau^z)$ from $(l_t,t\ge 0)$ by targeting at any interior point $z\in \U$. Given $(l_t,t\ge 0)$, define
\[\tau^z=\inf\{t: l_t\text{ surrounds }z\}.\]
For each $t<\tau^z$, let $f^z_t$ be the conformal map from the connected component of $\U\setminus l_t$ that contains $z$ onto $\U$ such that $f^z_t(z)=z, (f^z_t)'(z)>0$.
For each $t\le \tau^z$, define $\Psi^z_t=\circ_{s<t}f^z_s$. Then
\[\left(L^z_t=(\Psi^z_t)^{-1}(l_t), 0\le t\le \tau^z\right)\]
is a collection of loops in the unit disc and $L^z_{\tau^z}$ has the same law as the loop in CLE that surrounds $z$.

The conformal invariance in $M$ leads to the following ``target-independent property". Suppose that we have two distinct target points $z,w\in\U$, the process $(L^z_t, 0\le t\le \tau^z)$ and the process $(L^w_t, 0\le t\le \tau^w)$, up to the first time that $z$ and $w$ are disconnected, have the same law \cite[Lemma 8]{WernerWuCLEExploration}. Therefore, we can couple the two processes in the following way: up to the first time that $z$ and $w$ are disconnected, the two processes of loops coincide; and after the disconnecting time, the two processes evolve independently towards their target points respectively.
Consequently, it is possible to couple $(L^z_t, 0\le t\le \tau^z)$ for all $z\in\U$ simultaneously in the way that, for any two points $z,w\in\U$, the previous statement holds.
This is the \textit{conformal invariant growing mechanism} in $\CLE_4$ constructed in \cite{WernerWuCLEExploration}. From this growing process, we obtain a collection of $\CLE_4$ loops $\Gamma$ and, moreover, each loop has a time parameter: $((L,t_L), L\in\Gamma)$. We call $((L,t_L), L\in\Gamma)$ \textbf{$\CLE_4$ with time parameter}. It satisfies conformal invariance in the following sense.
\begin{proposition}
For $t\ge 0$, let $U_t$ be the domain obtained by removing from $\U$ all the loops $L\in\Gamma$ with $t_L\le t$. Then, for any M\"obius transformation $\phi$ of $\U$, the process $(\phi(U_t),t\ge 0)$ has the same law as the process $(U_t,t\ge 0)$.
\end{proposition}
\begin{proof}
\cite[Proposition 9]{WernerWuCLEExploration}.
\end{proof}

\subsection{Level lines targeted at interior points}\label{subsec::interiror_levellines_interior}
In Section \ref{sec::boundary_levellines}, we studied the level line of GFF starting from a boundary point and targeted at a distinct boundary point. In this section, we will study the level line of GFF starting from a boundary point and targeted at an interior point.

\begin{figure}[ht!]
\begin{subfigure}[b]{0.31\textwidth}
\begin{center}
\includegraphics[width=\textwidth]{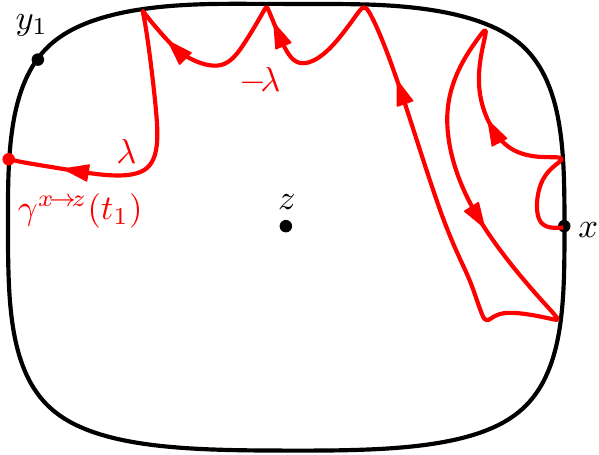}
\end{center}
\caption{Start the curve by the level line starting from $x$ targeted at $y_1$ stopped at the first disconnecting time $t_1$.}
\end{subfigure}
$\quad$
\begin{subfigure}[b]{0.31\textwidth}
\begin{center}\includegraphics[width=\textwidth]{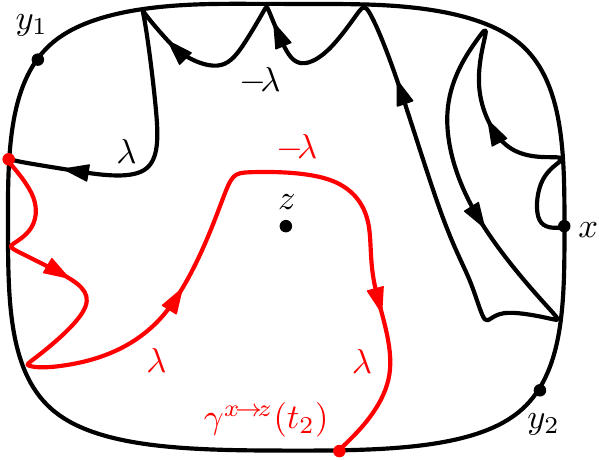}
\end{center}
\caption{Continue by the level line starting from $\gamma^{x\!\to\! z}(\!t_1\!)$ targeted at $y_2$ stopped at the first disconnecting time $t_2$.}
\end{subfigure}
$\quad$
\begin{subfigure}[b]{0.31\textwidth}
\begin{center}\includegraphics[width=\textwidth]{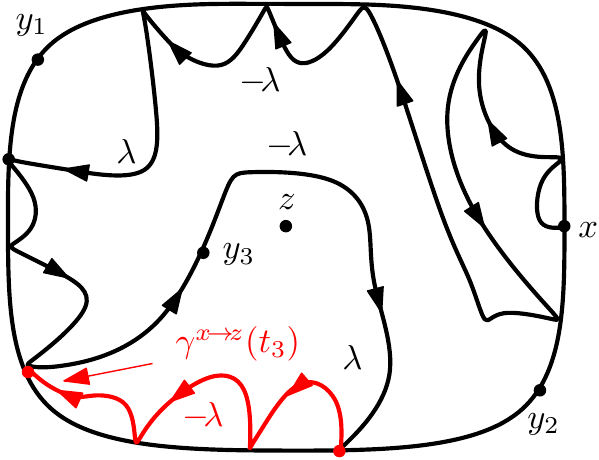}
\end{center}
\caption{Continue by the level line starting from $\gamma^{x\!\to\! z}(\!t_2\!)$ targeted at $y_3$ stopped at the first disconnecting time $t_3$.}
\end{subfigure}
\caption{\label{fig::levellines_targeted_interior} The explanation of generating level lines targeted at interior point. In this figure, the continuation threshold is hit at $t_3$.}
\end{figure}

Suppose that $h$ is a GFF on $\U$ whose boundary value is piecewise constant. Fix a boundary point $x\in\partial\U$ and an interior point $z\in\U$. We define \textbf{the level line of $h$ starting from $x$ targeted at $z$}, denoted by $\gamma^{x\to z}$, in the following way. See Figure \ref{fig::levellines_targeted_interior}.

Pick a point $y_1\in\partial\U$ different from $x$. We start $\gamma^{x\to z}$ by the level line of $h$ starting from $x$ targeted at $y_1$. We parameterize the curve by minus the log of the conformal radius of $\U\setminus\gamma^{x\to z}[0,t]$ seen from $z$. Namely, the curve $\gamma^{x\to z}$ is parameterized so that
\[\CR(\U\setminus\gamma^{x\to z}[0,t];z)=e^{-t}.\]
Define $t_1$ to be the first disconnecting time: the first time $t$ that $y_1$ is not on the boundary of the connected component of $\U\setminus\gamma^{x\to z}[0,t]$ that contains $z$. Denote by $U_1$ the connected component of $\U\setminus\gamma^{x\to z}[0,t_1]$ that contains $z$. Generally, when $\gamma^{x\to z}[0,t_k]$ and $U_k$ are defined for some $k\ge 1$, we pick $y_{k+1}$ on the boundary of $U_k$ different from $\gamma^{x\to z}(t_k)$. Given $\gamma^{x\to z}[0,t_k]$, we continue the curve by the level line of $h$ restricted to $U_k$ starting from $\gamma^{x\to z}(t_k)$ targeted at $y_{k+1}$ and parameterize the curve by minus the log of the conformal radius seen from $z$. Define $t_{k+1}$ to be the first disconnecting time: the first time $t$ that $y_{k+1}$ is not on the boundary of the connected component of $U_{k}\setminus\gamma^{x\to z}[t_k,t]$ that contains $z$. Denote by $U_{k+1}$ the connected component of $U_k\setminus\gamma^{x\to z}[t_k,t_{k+1}]$ that contains $z$. We continue this procedure until the curve hits the continuation threshold $\tau$. We summarize some basic properties of $\gamma^{x\to z}$ in the following.
\begin{enumerate}
\item [(a)] The curve $\gamma^{x\to z}$ is parameterized by minus the log of the conformal radius:
\[\CR(\U\setminus\gamma^{x\to z}[0,t];z)=e^{-t}.\]
\item [(b)] By the target-independent property of the level lines in Theorem \ref{thm::boundary_levellines_targetindependence}, we have that the curve $(\gamma^{x\to z}(t), 0\le t\le \tau)$ is independent of the choice of the sequence $(y_k,k\ge 1)$.
\item [(c)] The curve $(\gamma^{x\to z}(t), 0\le t\le \tau)$ is almost surely determined by the field $h$ and is almost surely continuous up to and including the continuation threshold.
\end{enumerate}

Generally, for $u\in\R$, \textbf{the level line of $h$ with height $u$ starting from $x$ targeted at $z$}, denoted by $\gamma^{x\to z}_u$, is the level line of $h+u$ starting from $x$ targeted at $z$.

\begin{figure}[ht!]
\begin{subfigure}[b]{0.48\textwidth}
\begin{center}
\includegraphics[width=0.65\textwidth]{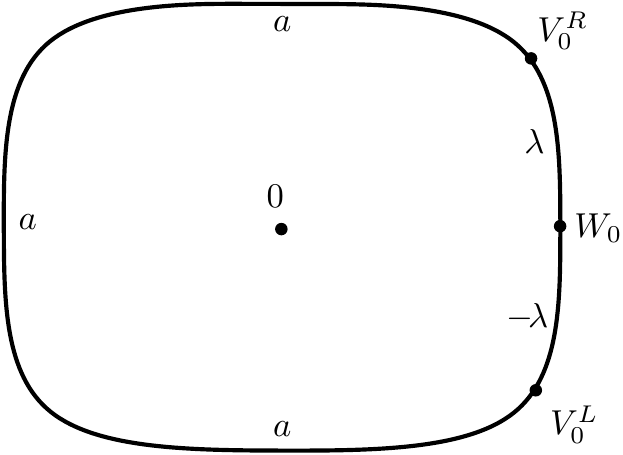}
\end{center}
\caption{The boundary value of the initial field.}
\end{subfigure}
$\quad$
\begin{subfigure}[b]{0.48\textwidth}
\begin{center}\includegraphics[width=0.65\textwidth]{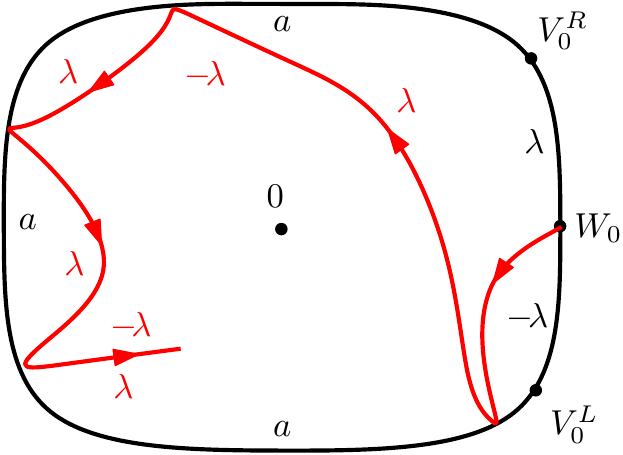}
\end{center}
\caption{The boundary value of the field given the level line.}
\end{subfigure}
\caption{\label{fig::levellines_targeted_interior_radialsle4} The boundary value of the field in Proposition \ref{prop::levellines_targeted_interior_radialsle4}.}
\end{figure}

By Lemma \ref{lem::sle_chordal_radial_equivalence}, the level line targeted at interior point is radial $\SLE_{4}(\underline{\rho})$ process and we record this fact in the following proposition.
\begin{proposition}\label{prop::levellines_targeted_interior_radialsle4}
Fix $a\in (-\lambda,\lambda)$ and three boundary points $V_0^L,W_0,V^R_0$ that are located on $\partial\U$ in counterclockwise order. Suppose that $h$ is a $\GFF$ on $\U$ whose boundary value is $-\lambda$ on $(V^L_0,W_0)$, is $\lambda$ on $(W_0,V^R_0)$, and is $a$ on $(V^R_0,V^L_0)$ (the intervals are counterclockwise). See Figure \ref{fig::levellines_targeted_interior_radialsle4}(a).

Suppose that $\gamma$ is a radial $\SLE_4(\rho^L;\rho^R)$ process starting from $W_0$ with two force points $(V^L_0;V^R_0)$ and the corresponding weights
\[\rho^L=-a/\lambda-1,\quad \rho^R=a/\lambda-1.\] Note that the continuation threshold of $\gamma$ is hit at the following time
\[\tau=\inf\{t>0: V^L_t=W_t=V^R_t\},\]
which is almost surely finite.

There exists a coupling between $h$ and $\gamma$ so that the following is true. Suppose that $\tilde{\tau}$ is any $\gamma$-stopping time less than the continuation threshold for $\gamma$. Then, given $\gamma[0,\tilde{\tau}]$, the conditional law of $h$ restricted to $\U\setminus\gamma[0,\tilde{\tau}]$ is $\GFF$ in each connected component whose boundary value is consistent with $h$ on $\partial\U$ and is $\lambda$ to the right of $\gamma$ and is $-\lambda$ to the left of $\gamma$. See Figure \ref{fig::levellines_targeted_interior_radialsle4}(b).

Furthermore, in this coupling, the path $\gamma$ is almost surely determined by $h$ and is continuous up to and including the continuation threshold.
\end{proposition}

\begin{figure}[ht!]
\begin{subfigure}[b]{0.48\textwidth}
\begin{center}
\includegraphics[width=0.65\textwidth]{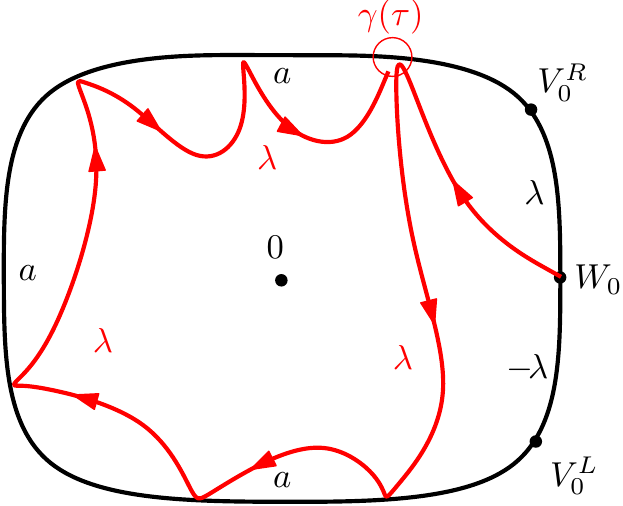}
\end{center}
\caption{The conditional mean is $\lambda$. As $t\uparrow \tau$, we have $\arg(W_t)-\arg(V^L_t)\to 0$, $\arg(V^R_t)-\arg(W_t)\to 2\pi$.}
\end{subfigure}
$\quad$
\begin{subfigure}[b]{0.48\textwidth}
\begin{center}\includegraphics[width=0.65\textwidth]{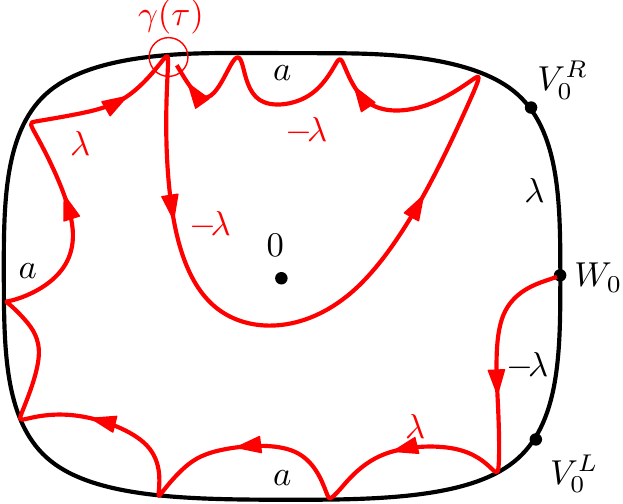}
\end{center}
\caption{The conditional mean is $-\lambda$. As $t\uparrow \tau$, we have $\arg(W_t)-\arg(V^L_t)\to 2\pi$, $\arg(V^R_t)-\arg(W_t)\to 0$.}
\end{subfigure}
\caption{\label{fig::levellines_interior_conditionalmean}The explanation of the conditional mean of the field given $\gamma[0,\tau]$ in Corollary \ref{cor::levellines_interior_conditionalmean}.}
\end{figure}

\begin{corollary}\label{cor::levellines_interior_conditionalmean}
Assume the same setting as in Proposition \ref{prop::levellines_targeted_interior_radialsle4}. Denote by $\LC$ the mean of $h$. Given $\gamma[0,\tau]$, denote by $U_{\tau}$ the connected component of $\U\setminus\gamma[0,\tau]$ that contains the origin and denote by $\LC_{\gamma[0,\tau]}$ the conditional mean of $h$ restricted to $U_{\tau}$. Then there are two possibilities (see Figure \ref{fig::levellines_interior_conditionalmean}):
\[\text{either}\quad \LC_{\gamma[0,\tau]}(0)=\lambda,\quad\text{or}\quad \LC_{\gamma[0,\tau]}(0)=-\lambda.\]
Furthermore,
\[\PP\left[\LC_{\gamma[0,\tau]}(0)=\lambda\right]=\frac{\lambda+\LC(0)}{2\lambda},\quad \PP\left[\LC_{\gamma[0,\tau]}(0)=-\lambda\right]=\frac{\lambda-\LC(0)}{2\lambda}.\]
\end{corollary}
\begin{proof}
For $\eps>0$ small, recall that the average of $h$ in $B(z,\eps)$, denoted by $h_{\eps}(z)$, is defined through Equation (\ref{eqn::gff_localmean}). We have the following.
\begin{enumerate}
\item [(a)] The variable $h_{\eps}(0)$ is a Gaussian with mean $\LC(0)$ and variance $-\log\eps$.
\item [(b)] Given $\gamma[0,\tau]$ and on the event $\left[B(0,\eps)\subset U_{\tau}\right]$, the variable $h_{\eps}(0)$ is a Gaussian with mean $\LC_{\gamma[0,\tau]}(0)$ and variance $-\log\eps-\tau$ (recall that $\tau$ is minus the log of the conformal radius of $U_{\tau}$).
\end{enumerate}
Combining these two facts and letting $\eps$ go to zero, we have that
\[\E\left[\LC_{\gamma[0,\tau]}(0)\right]=\LC(0).\]
This implies the conclusion.
\end{proof}

\begin{figure}[ht!]
\begin{subfigure}[b]{0.48\textwidth}
\begin{center}
\includegraphics[width=0.63\textwidth]{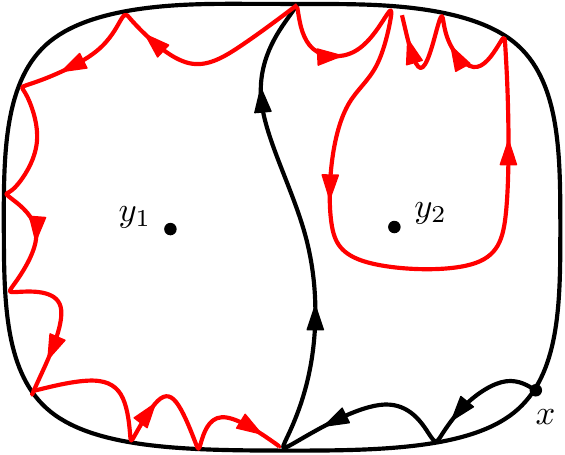}
\end{center}
\caption{The two paths coincide up to and including the first disconnecting time.}
\end{subfigure}
$\quad$
\begin{subfigure}[b]{0.48\textwidth}
\begin{center}\includegraphics[width=0.63\textwidth]{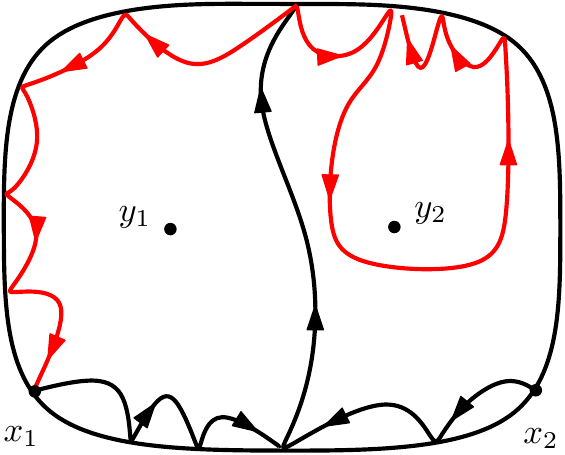}
\end{center}
\caption{The two paths merge upon intersecting, and continue together until the first disconnecting time.}
\end{subfigure}
\caption{\label{fig::levellines_interior_targetindependence} The explanation of the behavior of the paths in Proposition \ref{prop::levellines_interior_targetindependence} and in Remark \ref{rem::levellines_interior_targetindependence}.}
\end{figure}

By the construction of the level lines targeted at interior points, we have that these level lines also satisfy target-independent property.
\begin{proposition}\label{prop::levellines_interior_targetindependence}[Generalization of Theorem \ref{thm::boundary_levellines_targetindependence}]
Suppose that $h$ is a $\GFF$ on $\U$ whose boundary value is piecewise constant. Fix three points $x\in\partial\U$ and $y_1,y_2\in\U$. For $i=1,2$, let $\gamma_i$ be the level line of $h$ starting from $x$ targeted at $y_i$; define $T_i$ to be the first disconnecting time: $T_i$ is the inf of time $t$ such that $y_1,y_2$ are not in the same connected component of $\U\setminus\gamma_i[0,t]$. See Figure \ref{fig::levellines_interior_targetindependence}(a). Then, almost surely, the paths $\gamma_1$ and $\gamma_2$ coincide up to and including the first disconnecting time (up to time change); given $(\gamma_1[0,T_1],\gamma_2[0,T_2])$, the two paths continue towards their target points independently.

Furthermore, the same conclusion holds when one or two of the target points are on the boundary.
\end{proposition}

\begin{remark}\label{rem::levellines_interior_targetindependence}
A similar conclusion as in Proposition \ref{prop::levellines_interior_targetindependence} also holds when the two starting points are distinct. Suppose that $h$ is a $\GFF$ on $\U$ whose boundary value is piecewise constant. Fix two distinct starting points $x_1,x_2\in\partial\U$ and two distinct target points $y_1,y_2\in\U$. For $i=1,2$, let $\gamma_i$ be the level line of $h$ starting from $x_i$ targeted at $y_i$. On the event that the two paths $\gamma_1,\gamma_2$ hit each other, the two paths will merge upon intersecting, continue together until the first disconnecting time after which the two paths continue towards their own target points independently. See Figure \ref{fig::levellines_interior_targetindependence}(b).
\end{remark}

\begin{figure}[ht!]
\begin{subfigure}[b]{0.48\textwidth}
\begin{center}
\includegraphics[width=0.63\textwidth]{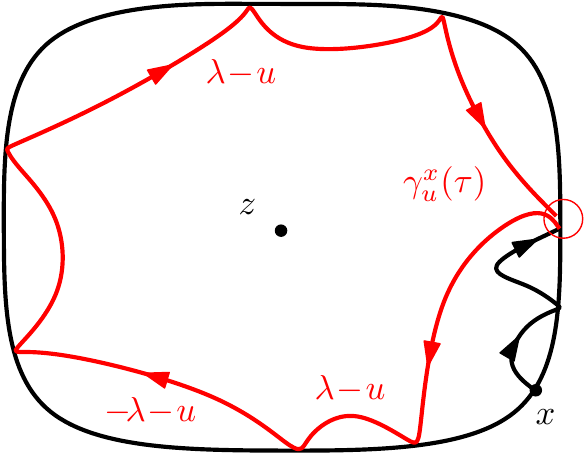}
\end{center}
\caption{The level line $\gamma^x_u$ with height $u$ starting from a boundary point $x$ targeted at $z$.}
\end{subfigure}
$\quad$
\begin{subfigure}[b]{0.48\textwidth}
\begin{center}\includegraphics[width=0.63\textwidth]{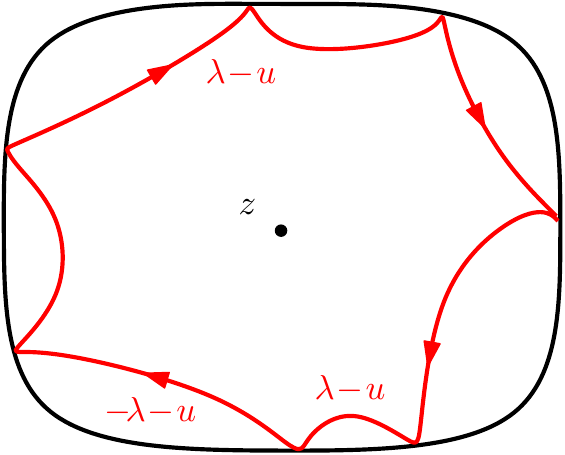}
\end{center}
\caption{The level loop with height $u$ starting from the boundary targeted at $z$.}
\end{subfigure}
\caption{\label{fig::levellines_interior_levelloop} The boundary value of the field in Lemma \ref{lem::levellines_interior_levelloop}.}
\end{figure}

Suppose that $h$ is a zero-boundary GFF on $\U$. Fix a height $u\in (-\lambda,\lambda)$ and a target point $z\in\U$. Let $\gamma^x_u$ be the level line of $h$ with height $u$ starting from $x\in\partial\U$ targeted at $z$, and define $\tau$ to be its continuation threshold. Let $U_{\tau}$ be the connected component of $\U\setminus\gamma^x_u[0,\tau]$ that contains $z$, and let $L_u$ be the oriented boundary of $U_{\tau}$.
By Remark \ref{rem::levellines_interior_targetindependence}, we know that $L_u$ is independent of the starting point $x$. In other words, the level line $\gamma^x_u$ depends on $x$ but $L_u$ does not. We call $L_u$ the \textbf{level loop} of $h$ starting from the boundary targeted at $z$. We denote by $\inte(L_u)$ the connected component of $\U\setminus L_u$ that contains $z$. We record some basic properties of the level loop in the following lemma.
\begin{lemma}\label{lem::levellines_interior_levelloop}
Suppose that $h$ is a zero-boundary $\GFF$ on $\U$. Fix a height $u\in (-\lambda,\lambda)$ and a target point $z\in\U$. Let $L_u$ be the level loop of $h$ starting from the boundary $\partial\U$ targeted at $z$. Then we have the following.
\begin{enumerate}
\item [(1)] $L_u$ is oriented either clockwise or counterclockwise and is homeomorphic to the unit disc.
\item [(2)] $L_u\cap\partial\U\neq\emptyset$.
\item [(3)] Given $L_u$, the conditional law of $h$ restricted to each connected component of $\U\setminus L_u$ is the same as $\GFF$'s whose boundary value is zero on $\partial\U$, is $\lambda-u$ to the right of $L_u$, and is $-\lambda-u$ to the left of $L_u$. See Figure \ref{fig::levellines_interior_levelloop}. In particular, the conditional law of $h$ restricted to $\inte(L_u)$ is the same as a $\GFF$ whose boundary value is
\[\left\{
    \begin{array}{ll}
      \lambda-u, & \text{if }L_u \text{ is clockwise}, \\
      -\lambda-u, & \text{if }L_u \text{ is counterclockwise}.
    \end{array}
  \right.\]
\end{enumerate}
Moreover, the loop $L_u$ is almost surely determined by $h$ and
\[\PP\left[L_u\text{ is clockwise}\right]=\frac{\lambda+u}{2\lambda},\quad \PP\left[L_u\text{ is counterclockwise}\right]=\frac{\lambda-u}{2\lambda}.\]
\end{lemma}

In fact, the three properties in Lemma \ref{lem::levellines_interior_levelloop} characterize the level loop of $\GFF$.
\begin{lemma}\label{lem::levellines_interior_levelloop_characterization}
Suppose that $h$ is a zero-boundary $\GFF$ on $\U$. Fix a height $u\in (-\lambda,\lambda)$ and a target point $z\in\U$. Let $L_u$ be the level loop of $h$ starting from the boundary $\partial\U$ targeted at $z$. Assume that $\tilde{L}$ is an oriented loop in $\U$ satisfying the three properties in Lemma \ref{lem::levellines_interior_levelloop}. Then almost surely $\tilde{L}$ and $L_u$ are equal.
\end{lemma}
\begin{proof}
For $x\in\partial\U$, let $\gamma^x_u$ be the level line of $h$  with height $u$ starting from $x$ targeted at $z$. We have the following observations.
\begin{enumerate}
\item [(a)] On the event that $\gamma^x_u$ hits $\tilde{L}$, the level line $\gamma^x_u$ will merge with $\tilde{L}$ (since $\tilde{L}$ satisfies the third property in Lemma \ref{lem::levellines_interior_levelloop}) upon intersecting; therefore, $\tilde{L}$ and $L_u$ are equal.
\item [(b)] $\tilde{L}\cap\partial\U\neq\emptyset$.
\end{enumerate}
Combining these two facts, we obtain the conclusion.
\end{proof}
\begin{figure}[ht!]
\begin{subfigure}[b]{0.48\textwidth}
\begin{center}
\includegraphics[width=0.75\textwidth]{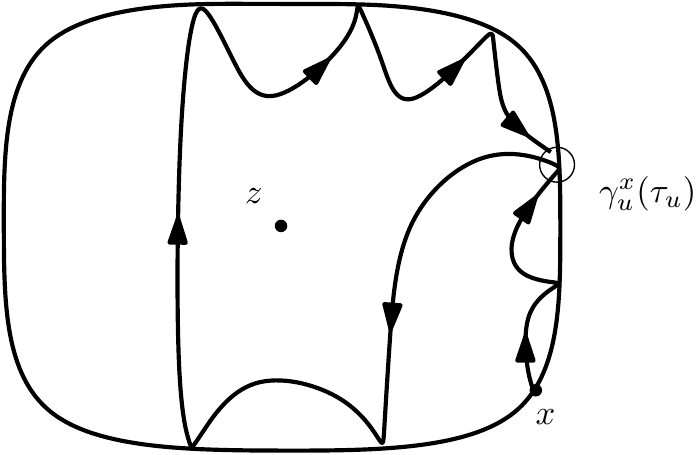}
\end{center}
\caption{$L_u$ is the boundary of the connected component of $\U\setminus\gamma^x_u[0,\tau_u]$ that contains $z$.}
\end{subfigure}
$\quad$
\begin{subfigure}[b]{0.48\textwidth}
\begin{center}\includegraphics[width=0.75\textwidth]{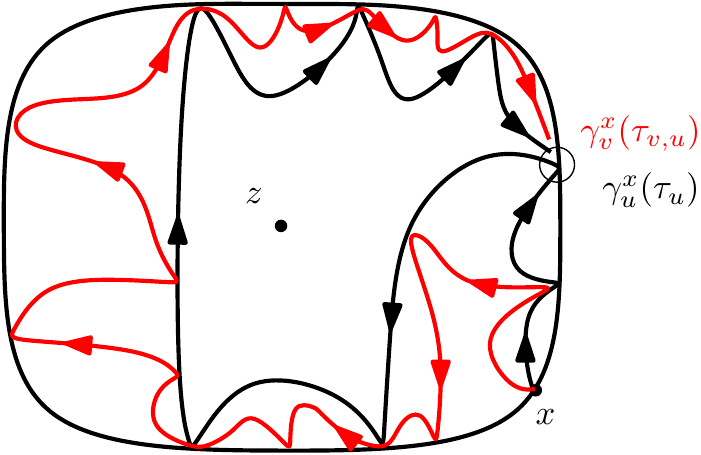}
\end{center}
\caption{Since $\gamma^x_v$ stays to the left of $\gamma^x_u$, the path $\gamma^x_v[0,\tau_{v,u}]$ is outside of $L_u$.}
\end{subfigure}
\caption{\label{fig::levellines_interior_loop_interaction} The explanation of the behavior of the paths in Proposition \ref{prop::levellines_interior_loop_interaction}.}
\end{figure}

The following proposition addresses the interaction between two level loops.
\begin{proposition}\label{prop::levellines_interior_loop_interaction}
Suppose that $h$ is a zero-boundary $\GFF$ on $\U$. Fix a target point $z\in\U$. For $a\in (-\lambda,\lambda)$, denote by $L_a$ the level loop of $h$ with height $a$ starting from the boundary $\partial\U$ targeted at $z$. Fix $u\in(-\lambda,\lambda)$.

If $L_u$ is clockwise, then, for any $v\in (u,\lambda)$, the loop $L_v$ is outside of $L_u$ and is also clockwise. Moreover, given $L_v$, the conditional law of $L_u$ is the same as the level loop of a zero-boundary $\GFF$ on $\inte(L_v)$ with height $\lambda-v+u$ starting from the boundary $L_v$ targeted at $z$, conditioned on the event that the loop is clockwise.

If $L_u$ is counterclockwise, then, for any $v\in (-\lambda, u)$, the loop $L_v$ is outside of $L_u$ and is also counterclockwise. Moreover, given $L_v$, the conditional law of $L_u$ is the same as the level loop of a zero-boundary $\GFF$ on $\inte(L_v)$ with height $-\lambda-v+u$ starting from the boundary $L_v$ targeted at $z$, conditioned on the event that the loop is counterclockwise.
\end{proposition}
\begin{proof}
We only need to prove the conclusion when $L_u$ is clockwise. Let $\gamma^x_u$ (resp. $\gamma^x_v$) be the level line of $h$ with height $u$ (resp. height $v$) starting from $x\in\partial\U$ targeted at $z$ and $\tau_u$ (resp. $\tau_v$) be its continuation threshold. By the construction of level lines targeted at interior point and Theorem \ref{thm::boundary_levelline_gff_interacting}, we know that $\gamma_v^x$ stays to the left of $\gamma^x_u$ until the time that $\gamma_v^x$ reaches $\gamma^x_u(\tau_u)$, say at time $\tau_{v,u}<\tau_v$. See Figure \ref{fig::levellines_interior_loop_interaction}. This implies that $\gamma^x_v[0,\tau_{v,u}]$ is outside of $L_u$. This holds for any $x$ and we know that $L_v$ is independent of $x$, thus $L_v$ is outside of $L_u$.
\end{proof}

Suppose that $h$ is a zero-boundary GFF on $\U$ and $u\in (-\lambda,\lambda)$. Fix a target point $z\in\U$.
Let $L_u$ be the level loop of $h$ with height $u$ starting from the boundary targeted at $z$.
The rest of this section is devoted to the study of the asymptotic behavior of the level loop $L_u$ as $u\to -\lambda$.

\begin{figure}[ht!]
\begin{subfigure}[b]{0.48\textwidth}
\begin{center}
\includegraphics[width=0.61\textwidth]{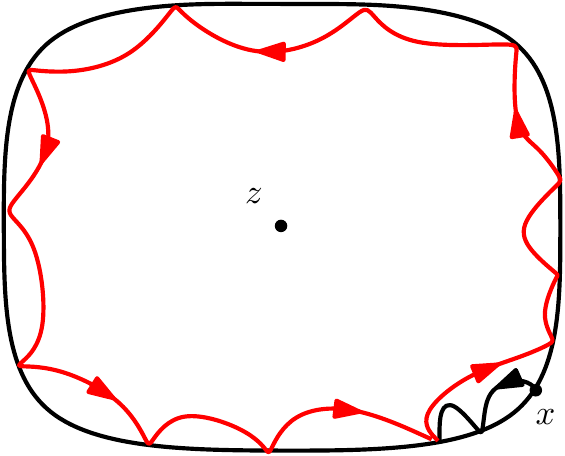}
\end{center}
\caption{When $L_u$ is counterclockwise, the loop converges to $\partial\U$ as $u\to -\lambda$.}
\end{subfigure}
$\quad$
\begin{subfigure}[b]{0.48\textwidth}
\begin{center}\includegraphics[width=0.61\textwidth]{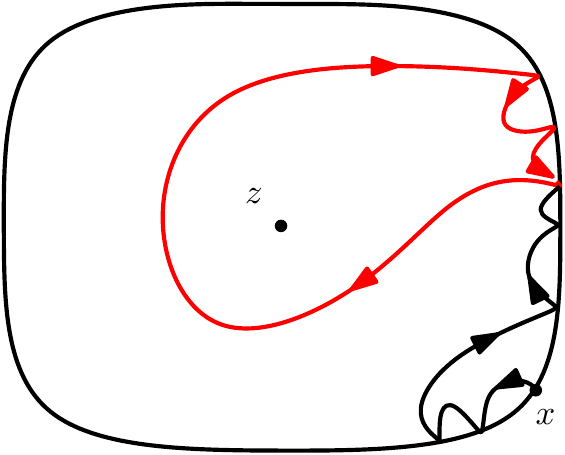}
\end{center}
\caption{When $L_u$ is clockwise, the loop converges to some bubble as $u\to-\lambda$.}
\end{subfigure}
\caption{\label{fig::levellines_interiror_limit} The explanation of the behavior of the paths in Lemmas \ref{lem::levellines_interior_limit_counterclockwise} and \ref{lem::levellines_interior_limit_clockwise}.}
\end{figure}

There are two possibilities for the orientation of $L_u$: counterclockwise (with probability $(\lambda-u)/2\lambda$) or clockwise (with probability $(\lambda+u)/2\lambda$). Lemma \ref{lem::levellines_interior_limit_counterclockwise} addresses the first case: when $L_u$ is counterclockwise, the loop will converge to $\partial\U$ as $u\to-\lambda$. See Figure \ref{fig::levellines_interiror_limit}(a). Lemma \ref{lem::levellines_interior_limit_clockwise} addresses the second case: when (conditioned on) $L_u$ is clockwise, the loop will converge to some bubble as $u\to-\lambda$. See Figure \ref{fig::levellines_interiror_limit}(b).

\begin{lemma}\label{lem::levellines_interior_limit_counterclockwise}
Suppose that $h$ is a zero-boundary $\GFF$ on $\U$. Fix a target point $z\in\U$. Let $L_u$ be the level loop of $h$ with height $u\in (-\lambda,\lambda)$ starting from the boundary targeted at $z$. Then, almost surely,
\[\CR(\inte(L_u);z)\to 1,\quad\text{as }u\to-\lambda.\]
Moreover, there exists a universal constant $c$ such that
\begin{equation}\label{eqn::log_cr_additive}
\E[-\log\CR(\inte(L_u);z)\cond L_u \text{ is counterclockwise}]=c(\lambda+u).
\end{equation}
\end{lemma}
\begin{figure}[ht!]
\begin{center}
\includegraphics[width=0.63\textwidth]{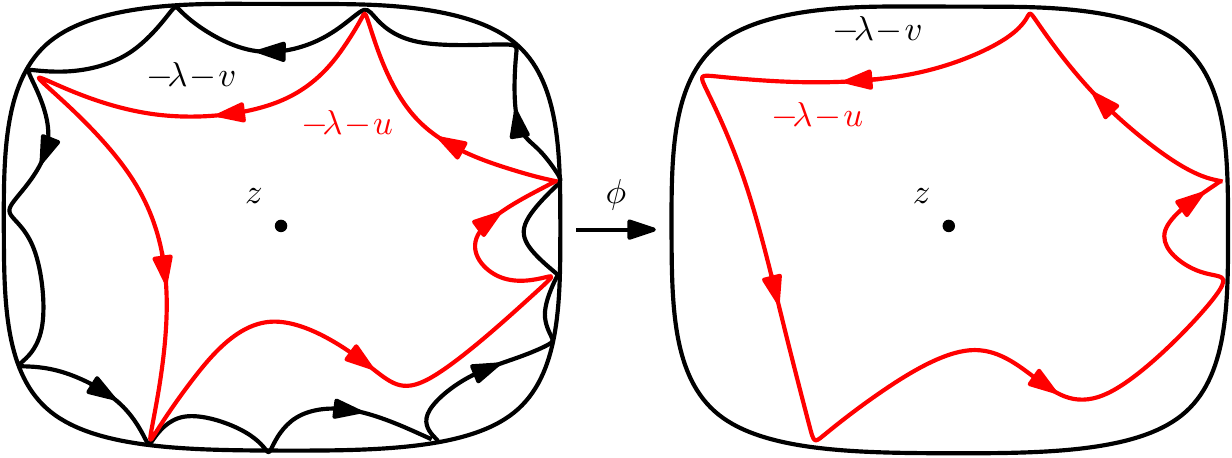}
\end{center}
\caption{\label{fig::levellines_interior_log_cr_additive} In the proof of Equation (\ref{eqn::log_cr_additive}), let $\phi$ be the conformal map from $\inte(L_v)$ onto $\U$, then $\phi(L_u)$ has the same law as the level loop with height $-\lambda+\delta-\eps$.}
\end{figure}
\begin{proof}
\textit{First}, we show that $\CR(\inte(L_u);z)$ is almost surely monotone as $u\to-\lambda$. For $u_0\in (-\lambda,\lambda)$, suppose that $L_{u_0}$ is counterclockwise. By Proposition \ref{prop::levellines_interior_loop_interaction}, we know that, for $u_0>u>v>-\lambda$, the loops $L_u$ and $L_v$ are counterclockwise and that $L_v$ is outside of $L_u$. Therefore,
\[\CR(\inte(L_u);z)\le\CR(\inte(L_v);z)\le 1.\]
Thus, on the event $[L_{u_0}\text{ is counterclockwise}]$, the sequence $(\CR(\inte(L_u);z), u\in (-\lambda,u_0))$ is monotone. Note that the event $[L_{u_0}\text{ is counterclockwise}]$ has probability $(\lambda-u_0)/2\lambda$. Letting $u_0\to-\lambda$, we obtain the conclusion.
\smallbreak
\textit{Second}, we have the following observations.
\begin{enumerate}
\item [(a)] The sequence $\CR(\inte(L_u);z)$ is almost surely monotone as $u\to -\lambda$.
\item [(b)] The sequence $\CR(\inte(L_u);z)$ converges to 1 in distribution as $u\to -\lambda$.
\end{enumerate}
Combining these two facts, we have that, almost surely, $\CR(\inte(L_u);z)\to 1$ as $u\to -\lambda$.
\smallbreak
\textit{Finally}, we show Equation (\ref{eqn::log_cr_additive}). For $u_0\in (-\lambda,\lambda)$, suppose that $L_{u_0}$ is counterclockwise. Then, for $u\in (-\lambda,u_0)$, the loop $L_u$ is counterclockwise; define
\[F(\lambda+u)=-\log\CR(\inte(L_u);z).\]
Let $\delta>\eps>0$ be small and set $u=-\lambda+\delta$, $v=-\lambda+\eps$. We have the following observations.
\begin{enumerate}
\item [(a)] The loops $L_u,L_v$ are counterclockwise and $L_v$ is outside of $L_u$. Define $\phi$ to be the conformal map from $\inte(L_v)$ onto $\U$ such that $\phi(z)=z,\phi'(z)>0$. Then
\[\CR(\inte(L_u);z)=\CR(\inte(L_v);z)\times\CR(\inte(\phi(L_u));z).\]
\item [(b)] Given $L_v$, the loop $\phi(L_u)$  has the same law as the level loop of zero-boundary GFF in $\U$ with height $-\lambda+\delta-\eps$. See Figure \ref{fig::levellines_interior_log_cr_additive}. Thus
\[-\log\CR(\inte(\phi(L_u));z)\overset{d}{=}F(\delta-\eps).\]
\end{enumerate}
Combining these two facts, we have that
\[F(\delta)\overset{d}{=}F(\eps)+\tilde{F}(\delta-\eps)\]
where $\tilde{F}(\delta-\eps)$ is a copy of $F(\delta-\eps)$ and is independent of $F(\eps)$. Therefore, there exists a universal constant $c$ such that
\[E[F(\lambda+u)\cond L_u\text{ is counterclockwise}]=c(\lambda+u).\]
\end{proof}

\begin{lemma}\label{lem::levellines_interior_limit_clockwise}
Suppose that $h$ is a zero-boundary $\GFF$ on $\U$. Fix a target point $z\in\U$. Let $L_u$ be the level loop of $h$ with height $u\in (-\lambda,\lambda)$ starting from the boundary targeted at $z$. Denote by $\mu_u^z$ the law of $L_u$ conditioned on the event $[L_u\text{ is clockwise}]$. Then, as $u\to -\lambda$, the measure $\mu_u^z$ converges to some limit measure, denoted by $\mu^z$, in Carath\'eodory topology seen from $z$.

Moreover, the limit measure $\mu^z$ is conformal invariant: for any M\"{o}bius transformation $\phi$ of $\U$, we have that
\[\phi\circ \mu^{z}=\mu^{\phi(z)}.\]
\end{lemma}
\begin{proof}
For $u>v>-\lambda$, on the event $[L_v \text{ is clockwise}]$, we know that $L_u$ is also clockwise and that $L_v$ is inside of $L_u$. By Lemma \ref{lem::caratheodory_decreasing}, to show the convergence in Carath\'eodory topology, we only need to show the convergence in $\CR(\inte(L_u);z)$.

Suppose that $F$ is any bounded Lipschitz function.
\begin{eqnarray*}
\lefteqn{|\mu_u^z\left[F\left(-\log\CR(\inte(L);z)\right)\right]-\mu_v^z\left[F\left(-\log\CR(\inte(L);z)\right)\right]|}\\
&=&|\E[F(-\log\CR(\inte(L_u);z))-F(-\log\CR(\inte(L_v);z)) \cond L_v\text{ is clockwise}]|\\
&\lesssim& \E\left[\log \frac{\CR(\inte(L_u);z)}{\CR(\inte(L_v);z)} \cond L_v\text{ is clockwise}\right].
\end{eqnarray*}
Given $[L_v\text{ is clockwise}]$ and $L_u$, the conditional law of $L_v$ is the same as the level loop of a zero-boundary GFF with height $\lambda-u+v$. Define $\phi$ to be the conformal map from $\inte(L_u)$ onto $\U$ such that $\phi(z)=z, \phi'(z)>0$. Then we have that
\begin{eqnarray*}
\lefteqn{\E\left[\log \frac{\CR(\inte(L_u);z)}{\CR(\inte(L_v);z)} \cond L_v\text{ is clockwise}\right]}\\
&=&\E[-\log\CR(\inte(\phi(L_v));z)\cond L_v\text{ is clockwise}]\\
&\lesssim& |u-v| \qquad \text{(By Equation (\ref{eqn::log_cr_additive}))}.
\end{eqnarray*}
This implies the convergence.
\smallbreak
The conformal invariance in the limit measure is inherited from the conformal invariance in the level loops: for any M\"obius transformation $\phi$ of $\U$, the loop $\phi(L_u)$ has the same law as the level loop with height $u$ targeted at $\phi(z)$.
\end{proof}

We give some properties of the limit measure $\mu^z$ defined in Lemma \ref{lem::levellines_interior_limit_clockwise}.
Recall that $l$ is a bubble in $\U$ if $l\subset\overline{\U}$ is homeomorphic to the unit circle and $l\cap\partial\U$ contains exactly one point; and the point in $l\cap\partial\U$ is called the root of $l$, denoted by $R(l)$.
The (probability) measure $\mu^z$ is supported on clockwise bubbles in $\U$. For any $z\in\U$, let $l^z$ be a bubble with law $\mu^z$. By the conformal invariance in $\mu^z$, we know that, for any M\"obius transformation $\phi$, we have
\[\phi(l^z)\overset{d}{=}l^{\phi(z)}.\]
In particular, the root $R(l^0)$ is uniform over $\partial\U$.

\begin{figure}[ht!]
\begin{center}
\includegraphics[width=0.31\textwidth]{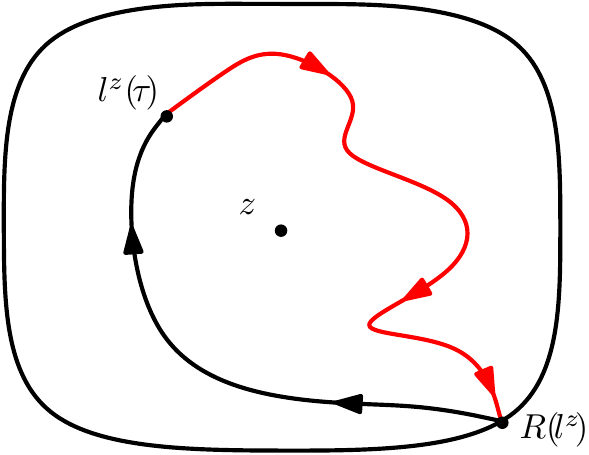}
\end{center}
\caption{\label{fig::levellines_interior_bubble_domainMarkov} Given $l^z[0,\tau]$, the conditional law of the remainder of the bubble is the same as chordal $\SLE_4$ conditioned on the event that $z$ is to the right of the path.}
\end{figure}
\smallbreak
The measure $\mu^z$ inherits the ``domain Markov property" from the level loop in the following sense. Suppose that $l^z$ is a bubble with law $\mu^z$. We parameterize $l^z$ clockwise by minus the log of the conformal radius seen from $z$:
\[l^z(0)=R(l^z);\quad \CR(\U\setminus l^z[0,t];z)=e^{-t}, 0\le t\le T;\quad l^z(T)=R(l^z).\]
Then the bubble $l^z$ satisfies the following property. For any stopping time $\tau<T$, given $l^z[0,\tau]$, the conditional law of $(l^z(t), \tau\le t\le T)$ is the same as chordal $\SLE_4$ (up to time change) in $\U\setminus l^z[0,\tau]$ from $l^z(\tau)$ to $R(l^z)$ conditioned on the event that $z$ is to the right of the path. See Figure \ref{fig::levellines_interior_bubble_domainMarkov}.
\smallbreak
We can define an infinite measure $\mu$ on bubbles: it is the measure on bubbles such that, for any $z\in\U$, it coincides with $\mu^z$ on the set of bubbles that surround $z$.

\textit{First}, we explain that $\mu$ is well-defined. For any two points $z,w\in\U$, we need to show that the two measures $\mu^z$ and $\mu^w$ coincide on the set of bubbles that surround both $z$ and $w$. To this end, we only need to show that, for any $u\in (-\lambda,\lambda)$, the two measures $\mu^z_u$ and $\mu^w_u$ coincide on the set of loops that surround both $z$ and $w$. Suppose that $h$ is a zero-boundary GFF on $\U$ and that $L^z_u$ (resp. $L^w_u$) is the level loop of $h$ with height $u$ starting from the boundary targeted at $z$ (resp. targeted at $w$). Then, when restricted to the set of loops that surround both $z$ and $w$, we have that
\begin{eqnarray*}
\lefteqn{\mu_u^z[\cdot]=\E[\cdot \cond L^z_u\text{ is clockwise}]}\\
&=& \E\left[\cdot 1_{\{L^z_u\text{ is clockwise}\}}\right]\times\frac{2\lambda}{\lambda+u}=\E\left[\cdot 1_{\{L^w_u\text{ is clockwise}\}}\right]\times\frac{2\lambda}{\lambda+u}\\
&=& \E[\cdot \cond L^w_u\text{ is clockwise}]=\mu_u^w[\cdot].
\end{eqnarray*}
%\[\mu_u^z[\cdot]=\E[\cdot \cond L^z_u\text{ is clockwise}]=\E\left[\cdot 1_{[L^z_u\text{ is clockwise}]}\right]\times\frac{2\lambda}{\lambda+u}=\E\left[\cdot 1_{[L^w_u\text{ is clockwise}]}\right]\times\frac{2\lambda}{\lambda+u}=\E[\cdot \cond L^w_u\text{ is clockwise}]=\mu_u^w[\cdot].\]
This implies that $\mu^z_u$ and $\mu^w_u$ coincide on the set of loops that surround both $z$ and $w$.

Furthermore, the requirement that $\mu$ coincides with the $\mu^z$'s full determines $\mu$. Thus $\mu$ is well-defined.
\smallbreak
\textit{Next}, we explain that $\mu$ is the same as the measure $M$ of $\SLE_4$ bubble measure defined in Section \ref{subsubsec::cle4_timeparameter}. We record this in the following lemma.

\begin{lemma}\label{lem::levellines_interior_limit_identify}
The infinite measure $\mu$ equals $M$, $\SLE_4$ bubble measure (oriented to be clockwise) uniformly rooted over the boundary.
\end{lemma}
\begin{proof}
By the conformal invariance of $\mu^z$, we also have the conformal invariance of $\mu$: for any M\"obius transformation $\phi$ of $\U$, we have that
\[\phi\circ\mu=\mu.\]
To show that $\mu$ equals $M$, we only need to show that $\mu^0$, which is $\mu$ restricted to the bubbles that surround the origin, equals $M^0$, which is $M$ restricted to the bubbles that surround the origin. In fact, the domain Markov property of $\mu^0$ characterizes the bubble measure \cite[Section 6]{SheffieldWernerCLE}. Therefore $\mu^0$ is a multiple of $M^0$. Since the total mass of $\mu^0$ is 1 and the total mass of $M^0$ is also 1, we have that $\mu^0$ equals $M^0$.
\end{proof}

\subsection{Upward height-varying level lines}\label{subsec::interior_levellines_upward}
We will introduce height-varying level lines targeted at interior points. We do not plan to address the general case, we only focus on one particular type of height-varying level lines: upward height-varying level lines. Suppose that $h$ is a zero-boundary $\GFF$ on $\U$. Fix $r\in (0,1)$, a boundary point $x\in\partial\U$ and a target point $z\in\U$.
We define the upward height-varying level line of $h$ starting from $x$ targeted at $z$ with height difference $r\lambda$, denoted by $\gamma^{x\to z}_{(r)\uparrow}$ or $\gamma_{(r)\uparrow}$, in the following way.

For $k\ge 1$, set \[u_k=-\lambda+kr\lambda.\]
We start $\gamma_{(r)\uparrow}$ by the level line of $h$ with height $u_1$ starting from $x$ targeted at $z$, and define $\tau_1$ to be its continuation threshold. Let $U_1$ be the connected component of $\U\setminus\gamma_{(r)\uparrow}[0,\tau_1]$ that contains $z$. By Corollary \ref{cor::levellines_interior_conditionalmean}, we know that, given $\gamma_{(r)\uparrow}[0,\tau_1]$, the conditional mean $m_1$ of $h$ restricted to $U_1$ is either $2\lambda-r\lambda$ or $-r\lambda$; moreover,
\[\PP[m_1=2\lambda-r\lambda]=r/2,\quad \PP[m_1=-r\lambda]=1-r/2.\]
By Proposition \ref{prop::levellines_targeted_interior_radialsle4}, we know that the law of $\gamma_{(r)\uparrow}$ is the same as radial $\SLE_4(-r;-2+r)$. Denote by $(V^L_t,W_t,V^R_t)_{t\in [0,\tau_1]}$ the corresponding radial Loewner evolution. Note that
\[\tau_1=\inf\{t>0: V^L_t=W_t=V^R_t\}.\]
Assume $m_1=-r\lambda$, then, as $t\uparrow\tau_1$, we have that
\[\arg(W_t-z)-\arg(V^L_t-z)\to 2\pi,\quad \arg(V^R_t-z)-\arg(W_t-z)\to 0.\]
If $m_1=2\lambda-r\lambda$, we stop and set $N=1$, $T=\tau_1$. If $m_1=-r\lambda$, we continue.

Generally, given $\gamma_{(r)\uparrow}[0,\tau_k]$ and $m_k=-kr\lambda$ for some $k\ge 1$, we continue $\gamma_{(r)\uparrow}$ by the level line with height $u_{k+1}$ starting from $\gamma_{(r)\uparrow}(\tau_k)$ targeted at $z$, and define $\tau_{k+1}$ to be its continuation threshold. Let $U_{k+1}$ be the connected component of $\U\setminus\gamma_{(r)\uparrow}[0,\tau_{k+1}]$ that contains $z$. Given $\gamma_{(r)\uparrow}[0,\tau_{k+1}]$, the conditional mean $m_{k+1}$ of $h$ restricted to $U_{k+1}$ is either $2\lambda-(k+1)r\lambda$ (with chance $r/2$) or $-(k+1)r\lambda$ (with chance $1-r/2$).
Denote by $(V^L_t,W_t,V^R_t)_{t\in [\tau_k,\tau_{k+1}]}$ the corresponding radial Loewner evolution. Note that
\[\tau_{k+1}=\inf\{t>\tau_k: V^L_t=W_t=V^R_t\}.\]
Assume $m_{k+1}=-(k+1)r\lambda$, then, as $t\uparrow\tau_{k+1}$, we have that
\[\arg(W_t-z)-\arg(V^L_t-z)\to 2\pi,\quad \arg(V^R_t-z)-\arg(W_t-z)\to 0.\]
If $m_{k+1}=2\lambda-(k+1)r\lambda$, we stop and set $N=k+1$ and $T=\tau_{k+1}$. If $m_{k+1}=-(k+1)r\lambda$, we continue.

At each step, we have chance $r/2$ to stop. Therefore, we will stop at some finite step $N$ almost surely. When we stop, we have $T=\tau_{N}$ and \[m_N=2\lambda-Nr\lambda.\] Moreover, when $t\uparrow T$, we have that
\[\arg(W_t-z)-\arg(V^L_t-z)\to 0,\quad \arg(V^R_t-z)-\arg(W_t-z)\to 2\pi.\]

This path $\gamma_{(r)\uparrow}$ is called the \textbf{upward height-varying level line of $h$} with height difference $r\lambda$ starting from $x$ targeted at $z$. We call $N$ the \textbf{transition step} and $T$ the \textbf{transition time}. We summarize some basic properties of $\gamma_{(r)\uparrow}$ in the following.
\begin{enumerate}
\item [(a)] The path $\gamma_{(r)\uparrow}$ is parameterized by minus the log of the conformal radius:
\[\CR(\U\setminus\gamma_{(r)\uparrow}[0,t];z)=e^{-t}.\]
\item [(b)] The path $\gamma_{(r)\uparrow}$ is almost surely determined by $h$ and is almost surely continuous up to and including the transition time.
\item [(c)] The transition step $N$ satisfies geometric distribution:
\[\PP[N>n]=(1-r/2)^n,\quad \text{for all } n\ge 0.\]
\item [(d)] Suppose that $(\tau_k,1\le k\le N)$ is the sequence of height change times and that $(V^L_t,W_t,V^R_t)_{t\in [0,T]}$ is the corresponding Loewner evolution. Then, for $0\le k\le N-1$, the process $(V^L_t,W_t,V^R_t)_{t\in [\tau_k,\tau_{k+1}]}$ satisfies the SDE for radial $\SLE_4(-r;-2+r)$.
\end{enumerate}

Suppose that $h$ is a zero-boundary GFF. Fix $r\in (0,1)$ and a target point $z\in\U$. Let $\gamma$ be the upward height-varying level line of $h$ with height difference $r\lambda$ starting from some $x\in\partial\U$ targeted at $z$. Let $N$ be the transition step, $T$ be the transition time, and $(\tau_k, 1\le k\le N)$ be the sequence of height change times. Let $L_1$ be the level loop of $h$ with height $u_1=-\lambda+r\lambda$ starting from $\partial\U$ targeted at $z$. We know that $L_1$ is part of $\gamma[0,\tau_1]$. Denote by $\inte(L_1)$ the connected component of $\U\setminus L_1$ that contains $z$. Generally, given $(\gamma[0,\tau_k], L_1,...,L_k)$ for some $1\le k<N$, let $L_{k+1}$ be the level loop of $h$ restricted to $\inte(L_k)$ with height $u_{k+1}=-\lambda+(k+1)r\lambda$ starting from $L_k$ targeted at $z$. Denote by $\inte(L_{k+1})$ the connected component of $\inte(L_k)\setminus L_{k+1}$ that contains $z$. We know that $L_{k+1}$ is part of $\gamma[\tau_k,\tau_{k+1}]$.

In this way, we obtain a sequence of level loops $(L_1,...,L_N)$ which we call the \textbf{upward height-varying sequence of level loops of $h$} with height difference $r\lambda$ starting from $L_0=\partial\U$ targeted at $z$; and we call $N$ the transition step. We summarize some basic properties of the upward height-varying sequence of level loops $(L_1,...,L_N)$ in the following.
\begin{enumerate}
\item [(a)] The sequence $(L_1,...,L_N)$ is almost surely determined by $h$. The loops $L_1,...,L_{N-1}$ are counterclockwise and the loop $L_N$ is clockwise.
\item [(b)] For $1\le k< N$, the loop $L_{k+1}$ is contained in the closure of $\inte(L_k)$ and $L_{k+1}\cap L_k\neq\emptyset$.
\item [(c)] Given $(L_1,...,L_k)$ for $1\le k<N$, the conditional law of $h$ restricted to $\inte(L_k)$ is the same as a GFF with boundary value $-kr\lambda$.
\item [(d)] Given $(L_1,...,L_N)$, the conditional law of $h$ restricted to $\inte(L_N)$ is the same as a GFF with boundary value $2\lambda-Nr\lambda$.
\end{enumerate}

The following lemma addresses the interaction between two upward height-varying sequences of level loops.
\begin{lemma}\label{lem::levellines_interior_levelloops_upward}
Suppose that $h$ is a zero-boundary $\GFF$ on $\U$. Fix $r\in (0,1)$ and a target point $z\in\U$. Let $(L_n, 1\le n\le N)$ be the upward height-varying sequence of level loops of $h$ with height difference $r\lambda$ where $N$ is the transition step. Let $(\tilde{L}_n, 1\le n\le \tilde{N})$ be the upward height-varying sequence of level loops of $h$ with height difference $r\lambda/2$ where $\tilde{N}$ is the transition step. Then, almost surely, we have that
\[\tilde{L}_{2n}=L_n,\quad\text{for }1\le n\le N-1.\]
Moreover, there are two possibilities for $\tilde{L}_{2N-1}$: clockwise or counterclockwise.

If $\tilde{L}_{2N-1}$ is clockwise, then
\[\tilde{N}=2N-1, \quad \tilde{L}_{\tilde{N}}\subset\overline{\inte(L_N)}.\]

If $\tilde{L}_{2N-1}$ is counterclockwise, then
\[\tilde{N}=2N,\quad \tilde{L}_{\tilde{N}}=L_N.\]
\end{lemma}
\begin{proof}
\textit{First}, we show that $\tilde{L}_{2n}=L_n$ for $1\le n\le N-1$. Suppose that $N>1$. We will explain the conclusion for $n=1$. We have the following observations.
\begin{enumerate}
\item [(a)] The loop $L_1$ is the level loop of $h$ with height $-\lambda+r\lambda$ and it is counterclockwise.
\item [(b)] The loop $\tilde{L}_1$ is the level loop of $h$ with height $-\lambda+r\lambda/2$.
\item [(c)] Given $\tilde{L}_1$, the loop $\tilde{L}_2$ is the level loop of $h$ restricted to $\inte(L_1)$ with height $-\lambda+r\lambda$. By Lemma \ref{lem::levellines_interior_levelloop_characterization}, we know that $\tilde{L}_2$ is the level loop of $h$ with height $-\lambda+r\lambda$ (without conditioning on $\tilde{L}_1$).
\end{enumerate}
Combining these three facts and Proposition \ref{prop::levellines_interior_loop_interaction}, we have that, given $L_1$ and on the event $[N>1]$, the loop $\tilde{L}_1$ is outside of $L_1$ and $\tilde{L}_2=L_1$. Iterating the same proof, we have that $\tilde{L}_{2n}=L_n$ for $1\le n \le N-1$.
\smallbreak
\textit{Next}, we discuss the relation between $\tilde{L}_{\tilde{N}}$ and $L_N$. Given $(L_1,...,L_{N-1}, L_N)$, we know that $\tilde{L}_{2N-2}=L_{N-1}$, and that $L_N$ is clockwise. Then there are two possibilities for $\tilde{L}_{2N-1}$: clockwise or counterclockwise.

Case 1. Assume that $\tilde{L}_{2N-1}$ is clockwise. Then we have that $\tilde{N}=2N-1$. We have the following observations.
\begin{enumerate}
\item [(a)] Given $(L_1,...,L_{N-1})$, the loop $\tilde{L}_{\tilde{N}}$ is the level loop of $h$ restricted to $\inte(L_{N-1})$ with height $-\lambda+(N-1)r\lambda+r\lambda/2$ and it is clockwise.
\item [(b)] Given $(L_1,...,L_{N-1})$, the loop $L_N$ is the level loop of $h$ restricted to $\inte(L_{N-1})$ with height $-\lambda+(N-1)r\lambda+r\lambda$ and it is clockwise.
\end{enumerate}
Combining these two facts and Proposition \ref{prop::levellines_interior_loop_interaction}, we have that $L_N$ is outside of $\tilde{L}_{\tilde{N}}$.

Case 2. Assume that $\tilde{L}_{2N-1}$ is counterclockwise. We have the following observations.
\begin{enumerate}
\item [(a)] Given $(L_1,...,L_{N-1})$, the loop $L_N$ is the level loop of $h$ restricted to $\inte(L_{N-1})$ with height $-\lambda+Nr\lambda$.
\item [(b)] Given $(L_1,...,L_{N-1})$, the loop $\tilde{L}_{2N}$ is the level loop of $h$ restricted to $\inte(L_{N-1})$ with height $-\lambda+Nr\lambda$.
\end{enumerate}
Combining these two facts, we have that $\tilde{L}_{2N}=L_N$. In particular, $\tilde{L}_{2N}$ is clockwise and $\tilde{N}=2N$.
\end{proof}

\begin{proposition}\label{prop::levelloops_upward_limit}
Suppose that $h$ is a zero-boundary $\GFF$ on $\U$ and that $z\in\U$ is a fixed target point. For $k\ge 1$, let $(L^k_n, 1\le n\le N^k)$ be the upward height-varying sequence of level loops of $h$ with height difference $2^{-k}\lambda$ where $N^k$ is the transition step. Define $L^k(z)$ to be the last loop in the sequence:
\[L^k(z)=L^k_{N^k}.\]
Then we have the following.
\begin{enumerate}
\item [(1)] The sequence $(2^{-k-1}N^k,k\ge 1)$ converges almost surely to some quantity denoted by $t^{\infty}(z)$; moreover, the quantity is almost surely determined by $h$ and satisfies the exponential distribution:
    \[\PP\left[t^{\infty}(z)>t\right]=e^{-t},\quad \text{for all }t\ge 0.\]
\item [(2)] The sequence $(L^k(z),k\ge 1)$ converges almost surely to some loop denoted by $L^{\infty}(z)$ in Carath\'eodory topology seen from $z$; moreover, the loop $L^{\infty}(z)$ is almost surely determined by $h$ and has the same law as the loop in $\CLE_4$ that surrounds $z$.
\item [(3)] Given $(L^k_n, 1\le n\le N^k)$ for all $k\ge 1$, the conditional law of $h$ restricted to $\inte(L^{\infty}(z))$ is the same as $\GFF$ with boundary value $2\lambda(1-t^{\infty}(z))$.
\end{enumerate}
\end{proposition}
\begin{proof}
\textit{First}, we show the convergence of the sequence $(2^{-k-1}N^k, k\ge 1)$. By Lemma \ref{lem::levellines_interior_levelloops_upward}, for any $k\ge 1$, we have that, almost surely,
\[\text{either}\quad N^{k+1}=2N^k,\quad \text{or}\quad N^{k+1}=2N^k-1.\]
Therefore, almost surely,
\[0\le 2^{-k-1}N^k-2^{-k-2}N^{k+1}\le 2^{-k-2},\quad \text{for all }k\ge 1.\]
This implies the almost sure convergence of the sequence. Since that $N^k$ satisfies the geometric distribution
\[\PP[N^k>n]=(1-2^{-k-1})^n,\quad\text{for all }n\ge 0,\]
we know that the limit quantity $t^{\infty}(z)$ satisfies the exponential distribution.
\smallbreak
\textit{Second}, we show the convergence in distribution of the sequence $(L^k(z),k\ge 1)$. Consider the sequence $(L^k_n, 1\le n\le N^k)$, define $\Psi^k$ to be the conformal map from $\inte(L^k_n)$ onto $\U$, where $n=N^k-1$, such that $\Psi^k(z)=z, (\Psi^k)'(z)>0$. Note that the sequence $((\Psi^k)^{-1}(\U), k\ge 1)$ is decreasing.

Let $(l_t,t\ge 0)$ be the Poisson point process with intensity $\mu$. Define
\[t(z)=\inf\{t: l_t\text{ surrounds }z\}.\]
For each $t<t(z)$, let $f_t$ be the conformal map from the connected component of $\U\setminus l_t$ that contains $z$ onto $\U$ such that $f_t(z)=z,f_t'(z)>0$. From Section \ref{subsubsec::cle_exploration}, we know that the iterated conformal map $\Psi=\circ_{s<t(z)}f_s$ is well-defined.

We can show that $\Psi^k$ converges in distribution to $\Psi$ in Carath\'eodory topology as $k\to\infty$. Therefore, the loop $L^k(z)$ converges in distribution to $\Psi^{-1}(l_{t(z)})$ which has the same law as the loop in $\CLE_4$ that surrounds $z$. This implies the conclusion.
\smallbreak
\textit{Third}, we have the following observations.
\begin{enumerate}
\item [(a)] By the second step, we know that $L^k(z)$ converges in law to the loop in $\CLE_4$ that surrounds $z$.
\item [(b)] By Lemma \ref{lem::levellines_interior_levelloops_upward}, we know that, almost surely for all $k\ge 1$,
\[\inte(L^{k+1}(z))\subseteq \inte(L^k(z)).\]
\end{enumerate}
Combining these two facts, we have that $L^k(z)$ converges in Carath\'eodory topology almost surely to some limit $L^{\infty}(z)$ which has the same law as the loop in $\CLE_4$ that surrounds $z$.
\smallbreak
\textit{Finally}, we explain the conditional law of $h$ restricted to $\inte(L^{\infty}(z))$. Given $((L^k_n, 1\le n\le N^k), 1\le k\le m)$ for $m\ge 1$, we know that the conditional law of $h$ restricted to $\inte(L^m(z))$ is a GFF with boundary value $2\lambda(1-2^{-m-1}N^m)$. This holds for any $m\ge 1$. Combining this with the almost sure convergence of $2^{-m-1}N^m$ and $L^m(z)$, we obtain the conclusion.
\end{proof}

\subsection{Upward height-varying exploration trees}\label{subsec::interior_levellines_tree}
In this section, we start by analyzing the relation between two level loops with the same height targeted at distinct target points. Suppose that $h$ is a zero-boundary $\GFF$ on $\U$. Fix two interior points $y_1,y_2\in\U$.

For $i=1,2$, let $\gamma^{x\to y_i}_u$ be the level line of $h$ with height $u\in (-\lambda,\lambda)$ starting from $x\in\partial\U$ targeted at $y_i$, and let $\tau^{y_i}$ be its continuation threshold. From Proposition \ref{prop::levellines_interior_targetindependence}, we know that the two paths coincide up to and including the first disconnecting time after which the two paths continue towards their target points independently until they reach their continuation thresholds respectively.

For $i=1,2$, let $L_u^{y_i}$ be the level loop of $h$ with height $u$ starting from the boundary targeted at $y_i$; we know that $L_u^{y_i}$ is the boundary of the connected component of $\U\setminus\gamma^{x\to y_i}_u[0,\tau^{y_i}]$ that contains $y_i$. From the relation between $\gamma_u^{x\to y_1}$ and $\gamma_u^{x\to y_2}$, we know that there are two possibilities for the relation between $L_u^{y_1}$ and $L_u^{y_2}$: either $L_u^{y_1}=L_u^{y_2}$ (this happens when $\gamma_u^{x\to y_1}$ and $\gamma_u^{x\to y_2}$ hit their continuation threshold before the first disconnecting time, see Figure \ref{fig::levelloop_interaction_distincttargets}(a)) or $\inte(L_u^{y_1})\cap \inte(L_u^{y_2})=\emptyset$ (this happens when $\gamma_u^{x\to y_1}$ and $\gamma_u^{x\to y_2}$ hit the first disconnecting time before the continuation thresholds, see Figure \ref{fig::levelloop_interaction_distincttargets}(b)(c)).

\begin{figure}[ht!]
\begin{subfigure}[b]{0.31\textwidth}
\begin{center}
\includegraphics[width=\textwidth]{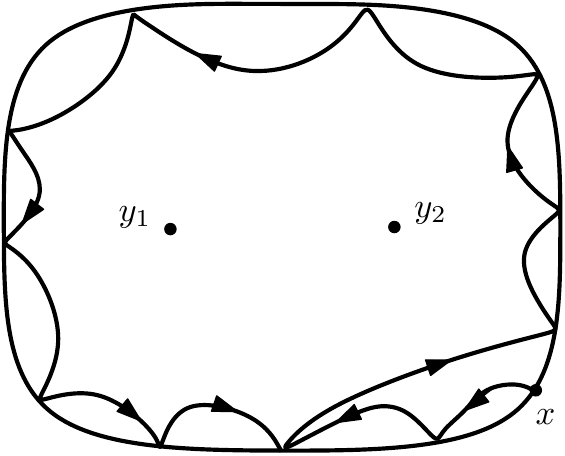}
\end{center}
\caption{$L_u^{y_1}=L_u^{y_2}$. The two paths hit the continuation threshold before the first disconnecting time.}
\end{subfigure}
$\quad$
\begin{subfigure}[b]{0.31\textwidth}
\begin{center}\includegraphics[width=\textwidth]{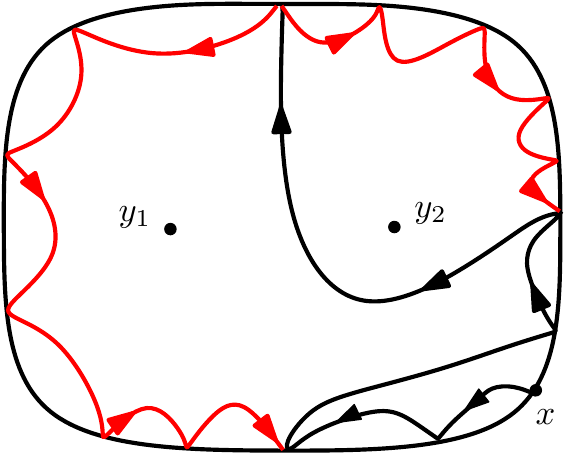}
\end{center}
\caption{The intersection $\inte(L_u^{y_1})\cap\inte(L_u^{y_2})$ is empty, and the two loops have different orientations.}
\end{subfigure}
$\quad$
\begin{subfigure}[b]{0.31\textwidth}
\begin{center}\includegraphics[width=\textwidth]{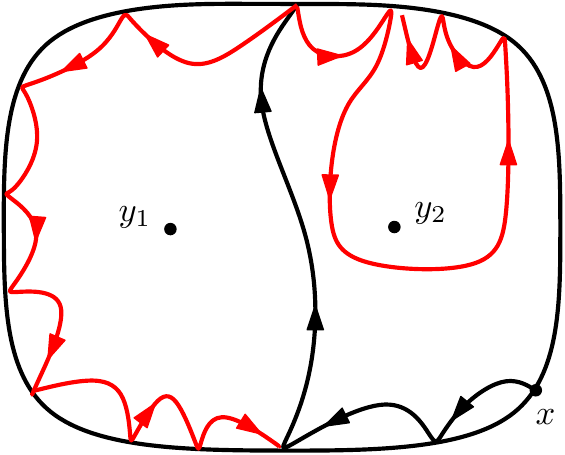}
\end{center}
\caption{The intersection $\inte(L_u^{y_1})\cap\inte(L_u^{y_2})$ is empty and the two loops have the same orientation.}
\end{subfigure}
\caption{\label{fig::levelloop_interaction_distincttargets} The relation between two level loops with the same height $u$ targeted at distinct target points $y_1,y_2$. }
\end{figure}

Fix $r\in (0,1)$. For $i=1,2$, let $(L_n^{y_i}, 1\le n\le N(y_i))$ be the upward height-varying sequence of level loops of $h$ with height difference $r\lambda$ starting from the boundary targeted at $y_i$ where $N(y_i)$ is the transition step. From the above analysis, we know that there exists a number $M\le N(y_1)\wedge N(y_2)+1$ such that
\[L_n^{y_1}=L_n^{y_2},\quad \text{for all } n\le M-1;\quad \text{and }\inte(L_n^{y_1})\cap\inte(L_n^{y_2})=\emptyset,\quad\text{for }n=M.\]
(If $L_n^{y_1}=L_n^{y_2}$ for all $n$, then we set $M=N(y_1)+1(=N(y_2)+1)$.)
Given $(L_n^{y_1}, L_n^{y_2}, n\le M)$, the two sequences continue towards their target points independently. We call $M$ the first disconnecting step for the two sequences $(L_n^{y_1}, 1\le n\le N(y_1))$ and $(L_n^{y_2}, 1\le n\le N(y_2))$.

Fix $r\in (0,1)$, let $\LZ$ be a fixed countable dense subset of $\U$. For $z\in\LZ$, let $(L_n^z, 1\le n \le N(z))$ be the upward height-varying sequence of level loops of $h$ with height difference $r\lambda$ targeted at $z$. We call the union of all loops $(L_n^z, 1\le n \le N(z))$ for all $z\in\LZ$ the \textbf{upward height-varying exploration tree of $h$} with height difference $r\lambda$, denoted by $\tree(r\lambda)$. We summarize some basic properties of the exploration tree $\tree(r\lambda)$.
\begin{enumerate}
\item [(a)] The tree $\tree(r\lambda)$ is almost surely independent of the choice of $\LZ$. Furthermore, it is almost surely determined by $h$.
\item [(b)] The tree $\tree(r\lambda)$ is conformal-invariant: for any M\"obius transformation $\phi$ of $\U$, we have
\[\phi(\tree(r\lambda))\overset{d}{=}\tree(r\lambda).\]
\item [(c)] For any two points $y_1,y_2\in\LZ$, the two upward height-varying sequences of level loops $(L_n^{y_1}, 1\le n\le N(y_1))$ and $(L_n^{y_2}, 1\le n\le N(y_2))$ satisfy the following property: the level loops coincide up to the first disconnecting step, after which the two sequences continue towards their target points independently.
\item [(d)] For any $z\in\LZ$, we denote by $L^r(z)$ the connected component of $\U\setminus\tree(r\lambda)$ that contains $z$, or equivalently $L^r(z)=L^z_n$ where $n=N(z)$. For any $z\in\LZ$, given $\tree(r\lambda)$, the conditional law of $h$ restricted to $\inte(L^r(z))$ is the same as GFF with boundary value \[2\lambda(1-\frac{r}{2}N(z)).\]
    For any $y_1,y_2\in\LZ$, given $\tree(r\lambda)$ and on the event that $[\inte(L^r(y_1))\cap\inte(L^r(y_2))=\emptyset]$, the restrictions of $h$ to $\inte(L^r(y_1))$ and to $\inte(L^r(y_2))$ are conditionally independent.
\end{enumerate}

Now, we are ready to complete the proof of Theorem \ref{thm::coupling_gff_cle4time}.
\begin{proof}[Proof of Theorem \ref{thm::coupling_gff_cle4time}] Suppose that $h$ is a zero-boundary GFF on $\U$. Fix a countable dense subset $\LZ$ of $\U$. For $k\ge 1$, let $\tree^k$ be the upward height-varying exploration tree of $h$ with height difference $2^{-k}\lambda$. For $z\in\LZ$, let $L^k(z)$ be the connected component of $\U\setminus\tree^k$ that contains $z$, and let $N^k(z)$ be the number such that, given $\tree^k$, the restriction of $h$ to $\inte(L^k(z))$ has boundary value $2\lambda(1-2^{-k-1}N^k(z))$.  From Proposition \ref{prop::levelloops_upward_limit}, we have the following observations.
\begin{enumerate}
\item [(a)] For all $z\in\LZ$, the sequence $(2^{-k-1}N^k(z),k\ge 1)$ converges almost surely to some quantity, denoted by $t^{\infty}(z)$.
\item [(b)] For all $z\in\LZ$, the sequence $(L^k(z),k\ge 1)$ converges almost surely to some loop, denoted by $L^{\infty}(z)$, in Carath\'eodory topology seen from $z$.
\item [(c)] For any $z\in\LZ$, given $(\tree^k, k\ge 1)$, the conditional law of $h$ restricted to $\inte(L^{\infty}(z))$ is the same as GFF with boundary value $2\lambda(1-t^{\infty}(z))$.
\item [(d)] For any $y_1,y_2\in\LZ$, given $(\tree^k, k\ge 1)$ and on the event that $[\inte(L^{\infty}(y_1))\cap\inte(L^{\infty}(y_2))=\emptyset]$, the restrictions of $h$ to $\inte(L^{\infty}(y_1))$ and to $\inte(L^{\infty}(y_2))$ are conditionally independent.
\end{enumerate}
Combining these four facts, we have that $h$ and $((L^{\infty}(z),t^{\infty}(z)), z\in\LZ)$ are coupled in the way described in Theorem \ref{thm::coupling_gff_cle4time}. From Proposition \ref{prop::levelloops_upward_limit}, the collection $((L^{\infty}(z),t^{\infty}(z)), z\in\LZ)$ has the same law as $\CLE_4$ with time parameter. This implies the existence of the coupling.
\end{proof}

\bibliographystyle{alpha}
\bibliography{bibliography}

\end{document}